\documentclass[12pt]{report}

\newenvironment{preliminary} 
{\pagestyle{plain}\pagenumbering{roman}}
{\pagenumbering{arabic}}

\usepackage{setspace}

\usepackage[margin=1in]{geometry}
\usepackage[utf8]{inputenc}  

\usepackage{amsthm}
\usepackage{amsfonts}
\usepackage{amsmath}
\usepackage{amssymb}
\usepackage{bbold}
\usepackage{dsfont}

\usepackage{mathtools} 
\usepackage{stmaryrd}
\usepackage{marvosym}
\usepackage{skull}
\usepackage{wasysym}
\usepackage{upgreek}
\usepackage{mathtools}
\usepackage{relsize}
\usepackage{etex}
\usepackage{lmodern}
\usepackage[T1]{fontenc}
\usepackage[utf8]{inputenc}

\usepackage{xcolor}
\definecolor{darkgreen}{rgb}{0,0.45,0}
\definecolor{darkred}{rgb}{0.75,0,0}
\definecolor{darkblue}{rgb}{0,0,0.6}
\usepackage[colorlinks,citecolor=darkgreen,linkcolor=darkblue,urlcolor=darkred]{hyperref}
\usepackage{breakurl}
\usepackage[mathscr]{eucal}
\usepackage{enumerate} 
\usepackage[capitalize]{cleveref}
\DeclareMathAlphabet\mathbfcal{OMS}{cmsy}{b}{n}

\usepackage{tikz-cd}
\usetikzlibrary{arrows}
\usetikzlibrary{cd}
\usepackage{verbatim}
\usepackage[all,cmtip]{xy}
\xyoption{2cell}
\xyoption{curve}
\UseTwocells
\input{diagxy}  

\theoremstyle{plain}
\newtheorem{theorem}{Theorem}[section]

\newtheorem*{theorem*}{Theorem}

\newtheorem{proposition}[theorem]{Proposition}
\newtheorem{lemma}[theorem]{Lemma}
\newtheorem{corollary}[theorem]{Corollary}

\theoremstyle{definition}
\newtheorem{definition}[theorem]{Definition}

\newtheorem{example}[theorem]{Example}

\newtheorem{notation}[theorem]{Notation}

\newtheorem{remark}[theorem]{Remark}

\newtheorem*{notation*}{Notation}

\newcommand{\A}{\mathrm{A}} 
\newcommand{\B}{\mathrm{B}} 
\newcommand{\C}{\mathrm{C}} 
\newcommand{\cD}{\mathrm{D}} 
\newcommand{\cE}{\mathrm{E}} 
\newcommand{\J}{\mathrm{J}} 
\newcommand{\K}{\mathrm{K}} 
\newcommand{\M}{\mathscr{M}} 
\newcommand{\N}{\mathscr{N}} 
\newcommand{\D}{\mathscr{D}} 
\newcommand{\E}{\mathscr{E}} 
\newcommand{\dA}{\mathscr{A}} 
\newcommand{\dB}{\mathscr{B}} 
\newcommand{\dC}{\mathscr{C}} 
\newcommand{\eA}{\underline{\mathscr{A}}} 
\newcommand{\eB}{\underline{\mathscr{B}}} 
\newcommand{\eC}{\underline{\mathscr{C}}} 
\newcommand{\eD}{\underline{\mathscr{D}}} 
\newcommand{\eE}{\underline{\mathscr{E}}} 
\newcommand{\0}{\mathbf{[0]}} 
\newcommand{\Ab}{\mathbf{Ab}} 
\newcommand{\Cat}{\mathbf{Cat}} 
\newcommand{\CAT}{\mathbf{CAT}} 
\newcommand{\Set}{\mathbf{Set}} 
\newcommand{\sSet}{\mathbf{sSet}} 
\newcommand{\Spt}{\mathbf{Spt}} 
\newcommand{\PDer}{\mathbf{PDer}} 
\newcommand{\Der}{\mathbf{Der}} 
\newcommand{\ECat}{\mathscr{E}\text{-}\mathbf{Cat}} 
\newcommand{\EPDer}{\mathscr{E}\text{-}\mathbf{PDer}} 

\newcommand{\id}{\mathrm{id}} 
\newcommand{\p}{\mathrm{p}} 
\newcommand{\h}{\mathrm{h}} 
\newcommand{\Hom}{\underline{\operatorname{Hom}}} 
\newcommand{\map}{\operatorname{map}} 
\newcommand{\tw}{\operatorname{tw}} 
\newcommand{\op}{^{\mathrm{op}}} 
\newcommand{\Ho}{\operatorname{Ho}} 
\newcommand{\cHom}{\underline{\operatorname{Hom}}_!} 
\newcommand{\dia}{\operatorname{dia}} 
\newcommand{\hocolim}{\operatorname{hocolim}} 
\newcommand{\holim}{\operatorname{holim}} 

\newcommand{\dHom}{\mathscr{H}\mathrm{om}} 
\newcommand{\dHo}{\mathscr{H}\mathrm{o}} 
\newcommand{\cdHom}{\mathscr{H}\mathrm{om}_!} 
\newcommand{\Prof}{\mathscr{P}\mathrm{rof}} 

\newcommand{\iso}{\cong} 
\renewcommand{\equiv}{\simeq} 

\begin{document}
\begin{preliminary}
\pagenumbering{roman}

  \phantomsection
	\addcontentsline{toc}{chapter}{Abstract}
	\Large\begin{center}\textbf{\textlarger{Enriched Derivators}} \\ \textsmaller{James Richardson}\end{center}\normalsize

The theory of derivators provides a convenient abstract setting for computing with homotopy limits and colimits. In enriched homotopy theory, the analogues of homotopy (co)limits are weighted homotopy (co)limits. In this thesis, we develop a theory of derivators and, more generally, prederivators enriched over a monoidal derivator $\E$.  In parallel to the unenriched case, these $\E$-prederivators provide a framework for studying the constructions of enriched homotopy theory, in particular weighted homotopy (co)limits.

As a precursor to $\E$-(pre)derivators, we study $\E$-categories, which are categories enriched over a bicategory $\Prof(\E)$ associated to $\E$. We prove a number of fundamental results about $\E$-categories, which parallel classical results for enriched categories. In particular, we prove an $\E$-category Yoneda lemma, and study representable maps of $\E$-categories. 

In any $\E$-category, we define notions of weighted homotopy limits and colimits. We define $\E$-derivators to be $\E$-categories with a number of further properties; in particular, they admit all weighted homotopy (co)limits. We show that the closed $\E$-modules studied by Groth, Ponto and Shulman give rise to associated $\E$-derivators, so that the theory of $\E$-(pre)derivators captures these examples. However, by working in the more general context of $\E$-prederivators, we can study weighted homotopy (co)limits in other settings, in particular in settings where not all weighted homotopy (co)limits exist.

Using the $\E$-category Yoneda lemma, we prove a representability theorem for $\E$-prederivators. We show that we can use this result to deduce representability theorems for closed $\E$-modules from representability results for their underlying categories.

	\tableofcontents

\newpage

\end{preliminary}

\chapter{Introduction}
\pagenumbering{arabic}
\setcounter{page}{1}

\subsection*{Derivators}

The theory of derivators is one of several current approaches to homological and homotopical algebra. Derivators were introduced independently in~\cite{Heller88} and~\cite{Grothendieck90}, along with similar theories developed in~\cite{Keller91} and~\cite{Franke96}. To motivate their definition, we begin by recalling a well-known deficiency of another approach to homotopical algebra, the theory of triangulated categories. 

Triangulated categories are a prominent and important axiomatisation of stable homotopy theory. See~\cite{Krause08,Neeman14} for a comprehensive introduction. Given any stable model category or stable quasicategory, its homotopy category is naturally triangulated. In general, information is lost when we pass to the homotopy category; however, in many situations, enough is retained that the original homotopy theory can be studied, and computations can be made, using only the underlying triangulated homotopy category. For example, if $\mathlarger{\mathscr{T}}$ is a triangulated category, we may study homotopy cofibre sequences in $\mathlarger{\mathscr{T}}$. In particular, any morphism in $\mathlarger{\mathscr{T}}$ has a homotopy cofibre. However, this construction cannot be made into a functor of the form $\mathlarger{\mathscr{T}}^{[1]}\rightarrow\mathlarger{\mathscr{T}}$, where $\mathlarger{\mathscr{T}}^{[1]}$ is the category of arrows in $\mathlarger{\mathscr{T}}$. When $\mathlarger{\mathscr{T}}=\Ho(\M)$ is the homotopy category of a stable model category, this problem reflects the fact that the homotopy cofibre construction in a model category induces a functor $\Ho(\M^{[1]})\rightarrow\Ho(\M)$, rather than an functor $\Ho(\M)^{[1]}\rightarrow\Ho(\M)$.

This problem can be addressed by considering the derivator associated to $\M$, rather than only the homotopy category. Introductions to derivators can be found in~\cite{Groth13,Groth16,GPS14a, GR19}, and the first chapter of~\cite{CN08}. The essential idea is to consider a family of categories, one for each small category $\A$, rather than a single category. In the case of a model category $\M$, this corresponds to keeping track of the family of homotopy categories $\Ho(\M^\A)$, for every $\A\in\Cat$, rather than the single homotopy category $\Ho(\M)$, where these homotopy categories are formed with respect to pointwise weak equivalences. (Note that, in general, it is not obvious that these homotopy categories are locally small. See~\cite{Cisinski03} for a proof.) Moreover, for any functor $u:\A\rightarrow\B$, we keep track of the (derived) pullback functor $u^*:\Ho(\M^\B)\rightarrow\Ho(\M^\A)$. It is possible to show that this functor has both adjoints. Using these, we can study homotopy limits and colimits in $\M$ and, in particular, if $\M$ is pointed we can recover the homotopy cofibre functor $\Ho(\M^{[1]})\rightarrow\Ho(\M)$.

We now discuss derivators in more detail. We can break the definition into two steps. First, a \textbf{prederivator} (see \cref{Prederivator definition}) is a $2$-functor $\D:\Cat\op\rightarrow\CAT$. We denote its values as follows:
\begin{center}
\begin{tikzcd}
\A\arrow[rr, bend left=50, "u" above,""{name=U, below}]\arrow[rr, bend right=50, "v" below, ""{name=D}]
&& \B
& \longmapsto
& \D(\A)
& \D(\B)\arrow[l,bend left=50,"v^*" below,""{name=L,above}]\arrow[l,bend right=50,"u^*" above,""{name=R,below}]
\arrow[Rightarrow,from=U,to=D,shorten >=0.1cm,shorten <=0.1cm,"\kappa"]
\arrow[Rightarrow,from=R,to=L,shorten >=0.1cm,shorten <=0.1cm,"\kappa^*"]
\end{tikzcd}
\end{center}
Given a functor $u:\A\rightarrow\B$, we call the functor $u^*:\D(\B)\rightarrow\D(\A)$ the \textbf{pullback functor along $u$}. If $u^*$ has a left adjoint $u_!:\D(\A)\rightarrow\D(\B)$, we call this functor the \textbf{homotopy left Kan extension along $u$}. Dually, if $u^*$ has a right adjoint $u_*:\D(\A)\rightarrow\D(\B)$, we call this functor the \textbf{homotopy right Kan extension along $u$}. In the case of the unique map $\p:\A\rightarrow\0$ to the terminal category, we call the right and left homotopy Kan extensions \textbf{homotopy limits} and \textbf{homotopy colimits} respectively. \textbf{Derivators} are prederivators that satisfy four additional axioms, which we recall in \cref{Derivator definition}. In particular, if $\D$ is a derivator,  one axiom, \textbf{Der 3}, implies that any functor $u:\A\rightarrow\B$ admits both a left and right homotopy Kan extension in $\D$. Another axiom, \textbf{Der 4}, gives a formula for calculating homotopy Kan extensions.

We have already alluded to the fact that any model category $\M$ gives rise to a derivator $\dHo(\M)$, whose value at $\A\in\Cat$ is the homotopy category $\Ho(\M^\A)$. This is the main theorem of~\cite{Cisinski03}. Similarly, given any complete and cocomplete quasicategory $Q$, we can form a derivator $\dHo(Q)$, whose value at $\A\in\Cat$ is $\Ho(Q^{N\A})$, where $N\A$ denotes the nerve of $\A$, and $\Ho(Q^{N\A})$ denotes the homotopy category of the quasicategory $Q^{N\A}$. See~\cite{Lenz18} for a proof that this defines a derivator. Thus, the passage from either a model category or a (co)complete quasicategory to its homotopy category factors through an associated derivator. 

We have seen that the derivator associated to a stable model category retains more information than the triangulated homotopy category. However, derivators do not retain all of the homotopical information that is available in model categories or quasicategories. See~\cite[Section 2.5]{Toen03} for a discussion of what information is lost. Thus, derivators cannot be thought of as a replacement for model categories or quasicategories. However, in certain settings, there are advantages to working with derivators rather than these other models of homotopy theory, where carrying around all of the available information can result in technical difficulties. For example, working in a derivator rather than a model category does away with the need to manage fibrant and cofibrant replacements and zigzags of weak equivalences. At the derivator level, we only have access to information that is homotopically meaningful, and homotopy invariant.

Derivators retain enough information to define homotopy Kan extensions, in particular homotopy limits and colimits, using universal properties, and provide enough tools to carry out elegant formal computations. Since these universal properties characterise homotopy limits and colimits in other models of homotopy theory, results we prove in derivators must hold, in particular, in model categories and quasicategories. Thus, derivators provide a convenient abstract setting in which we can manage homotopy coherence and compute with homotopy Kan extensions. This also carries over to morphisms of derivators: to prove, for example, that a given left Quillen functor commutes with a particular homotopy limit, it suffices to show this for the associated derivator map. This perspective on derivators is developed and exemplified in~\cite{GPS14,GPS14a,GS17}.

\subsection*{Actions of derivators}

In this thesis, we study enrichment of prederivators and derivators, and develop formal methods for studying the constructions of enriched category theory (see~\cite{Kelly82}) in homotopical settings. In particular, just as derivators are a tool for studying homotopy limits and colimits, enriched derivators provide a setting for studying \textbf{weighted homotopy limits and colimits}, which are the enriched analogue. 

There are a number of approaches to enriched homotopy theory. Simplicial enrichments, in particular, are well-studied and fundamentally important. Simplicial model categories were introduced in~\cite{Quillen67}, and, for any model category, constructions of simplicial mapping spaces were defined and studied in~\cite{DK80}. See~\cite{Hirschhorn03} for a textbook treatment. Weighted homotopy limits and colimits in simplicial model categories are studied in~\cite{Gambino10}. For a survey of other approaches to enrichment in homotopy theory, and a unified treatment of weighted homotopy limits and colimits, see~\cite{Shulman06}. 

The properties of derivators that make them convenient for studying other aspects of homotopy theory also make them an elegant setting for studying enrichment. For example, managing cofibrant and fibrant replacement in a model category can vastly complicate proofs that are relatively straightforward on the derivator level. This is an important advantage when it comes to studying enrichment, since we often want to verify lists of coherence conditions, and this can become difficult or impossible if we have to keep track of cofibrant and fibrant replacements. We give the following relevant example. In~\cite{Hovey07} it was conjectured that, for any monoidal model category $\M$, its homotopy category $\Ho(\M)$ is naturally a central algebra over the homotopy category of simplicial sets $\Ho(\sSet)$. All of this, except for the centrality condition, was shown using model categorical methods in~\cite{Hovey07}. However, the final coherence condition was only successfully checked in~\cite{Cisinski08}, using the associated derivator $\dHo(\M)$.

We will now outline some relevant previous work on enrichment in derivators.  A \textbf{monoidal derivator} $\E$ is a derivator equipped with a tensor product $\otimes:\E\times\E\rightarrow\E$ that is coherently unital and associative, and which, in an appropriate sense, preserves homotopy colimits in both variables. Note that, for any category $\A$, the tensor product induces a monoidal structure on the category $\E(\A)$. Monoidal derivators are studied in~\cite{Cisinski08,GPS14}; we recall the definition in \cref{Section Monoidal derivators}. Given any monoidal model category $\M$, its associated derivator $\dHo(\M)$ is monoidal.

Given a monoidal derivator $\E$, an \textbf{action} of $\E$ on a derivator $\D$ is a coherently associative and unital map $\otimes:\E\times\D\rightarrow\D$, which preserves homotopy colimits in both variables. See \cref{Definition of E-module}. We call a derivator $\D$ an \textbf{$\E$-module} if it is equipped with an $\E$-action. A fundamentally important result, proved in~\cite{Cisinski08}, is that any derivator $\D$ has a unique $\dHo(\sSet)$-module structure.  This theorem is conceptually important, but it also has practical implications for calculating homotopy colimits. See \cite[Section 7]{PS16} for such an application; using the $\dHo(\sSet)$-action, homotopy colimits in an arbitrary derivator can be computed from homotopy colimits in $\dHo(\sSet)$. This is true in general: computations involving homotopy colimits in an $\E$-module can be reduced to computations in $\E$. This approach is used to characterise stable derivators in~\cite{GS17}. 

Derivator \textbf{two-variable adjunctions} are studied in~\cite{GPS14}. We recall the definition in \cref{Section Two-variable adjunctions}. An $\E$-module $\D$ is called a \textbf{closed $\E$-module} if the $\E$-action $\otimes:\E\times\D\rightarrow\D$ is part of a two-variable adjunction. For example, if $\M$ is a monoidal model category and $\N$ is an $\M$-enriched model category, then the derivator $\dHo(\N)$ is a closed $\dHo(\M)$-module. As a part of the structure, a closed $\E$-module is equipped with two additional maps:
\begingroup
\addtolength{\jot}{0.5em}
\begin{align*}
\map_\D(-,-)&:\D\op\times\D\rightarrow\E\\
\vartriangleleft&:\D\times\E\op\rightarrow\D
\end{align*}
\endgroup
Using the first of these maps, closed $\E$-modules have a notion of mapping objects, which take values in $\E$. For this reason, closed $\E$-modules are called $\E$-enriched derivators in~\cite{GR19,GS17}; however, we will reserve this terminology for a different, though related, concept. Given a closed $\E$-module $\D$, the map $\otimes:\E\times\D\rightarrow\D$ can be used to define weighted homotopy colimits in $\D$, and $\vartriangleleft:\D\times\E\op\rightarrow\D$ can be used to define weighted homotopy limits. Examples of weighted homotopy colimits include the pullback functors in $\D$, and ordinary left Kan extensions. In~\cite{GR19,GS17}, computations with weighted homotopy colimits are reduced to computations in $\E$ with the corresponding weights.

\subsection*{Enriched prederivators and derivators}

In this work, we establish an alternative approach to enrichment in derivators, which incorporates the closed $\E$-modules of~\cite{GR19,GS17}, as well as a number of other examples. In particular, our framework allows us to study enrichment of general prederivators, rather than being restricted to derivators. Moreover, using this approach, we can formulate local definitions of weighted homotopy limits and colimits, which agree with the global definitions in~\cite{GR19,GS17} when all weighted homotopy (co)limits exist. This allows us to study weighted homotopy (co)limits in a broader range of settings, including in situations where only certain weighted homotopy (co)limits exist. 

We develop the theory of enriched derivators in a series of steps, starting with the concept of $\E$-categories. We may then add extra structure to obtain $\E$-prederivators and, finally, $\E$-derivators. We will now outline this process.

Any monoidal derivator $\E$ gives rise to an associated bicategory, which we denote by $\Prof(\E)$ and call the \textbf{bicategory of profunctors in $\E$}. This bicategory is defined in~\cite{GPS14}, and we recall its definition in \cref{Definition of Prof(E)}. See~\cite{Leinster98} for basic bicategorical definitions. An \textbf{$\E$-category} $\eA$ is defined, in \cref{E-category definition}, to be a category enriched over the bicategory $\Prof(\E)$. In particular, this includes the following data:
\begin{itemize}
\item For each small category $\A$, a (large) set of objects $\dA_0(\A)$.
\item For any two objects $X\in\dA_0(\A)$ and $Y\in\dA_0(\B)$, an object $\widetilde{\map}_{\dA}(X,Y)\in\E(\A\op\times\B)$.
\end{itemize}
In addition, $\eA$ is equipped with notions of \textbf{composition} and \textbf{units}, subject to natural axioms that express associativity and unitality of composition. In this way, the definition of $\E$-categories is analogous to the familiar definition of enriched categories in~\cite{Kelly82}; in fact, for any category $\A$, an $\E$-category $\eA$ gives rise to an $\E(\0)$-category $\eA(\A)$, which we describe in \cref{A([0]) is E([0])-enriched} and \cref{Any A(J) is E(0)-enriched for an E-prederivator}. 

Our development of the basic theory of $\E$-categories mirrors the classical development of enriched category theory. For example, given an $\E$-category $\eA$, a category $\A$ and an object $X\in\dA_0(\A)$, the mapping objects induce a \textbf{representable} $\E$-category morphism
\[
\widetilde{\map}_{\dA}(X,-):\eA\rightarrow\eE^{\A\op}
\]
where $\eE^{\A\op}$ is an $\E$-category associated to the \textbf{shifted prederivator} $\E^{\A\op}$ of \cref{Shifted prederivator definition}. This $\E$-category is described in \cref{E-modules are E-categories}; note, in particular, that for any category $\B$, the set $\E^{\A\op}_0(\B)$ is the set of objects in the category $\E(\A\op\times\B)$. 

Representable maps play a vital role in the theory of $\E$-categories, and in our study of $\E$-prederivators and $\E$-derivators. In part, this is a consequence of \cref{Yoneda lemma for E-categories}, the $\E$-category Yoneda lemma:
\begin{theorem*}
Let $\eA$ be an $\E$-category, let $\A$ be a category, and let $X\in\dA_0(\A)$. Let $F:\eA\rightarrow\eE^{\A\op}$ be an $\E$-category map. We have a natural bijection:
\[
\ECat(\eA,\eE^{\A\op})(\widetilde{\map}_{\dA}(X,-),F)\iso\E(\A\op\times\A)(\h_\A,FX)
\]
\end{theorem*}
Here $\ECat(\eA,\eE^{\A\op})$ is the hom-category in the $2$-category of $\E$-categories, which is defined in \cref{2-category of E-categories}. The object $\h_\A\in\E(\A\op\times\A)$ is called the \textbf{identity profunctor}; we recall its definition in \cref{Definition of h on objects and morphisms}. It is the unit object in a monoidal structure on $\E(\A\op\times\A)$, which we describe in \cref{Section The cancelling tensor product}. 

We define $\E$-prederivators in \cref{E-prederivator definition}. An $\E$-category $\eA$ is called an \textbf{$\E$-prederivator} if, among other conditions, it is equipped with a notion of pullback along functors. In particular, given a functor $u:\A\rightarrow\B$ and an object $X\in\dA_0(\B)$, we have an object $u^*X\in\dA_0(\A)$. We show, in \cref{E-prederivators induce prederivators}, that any $\E$-prederivator $\eA$ gives rise to a prederivator $\dA$, which we call the \textbf{prederivator induced by $\eA$}.

In \cref{Weighted colimits definition}, we define weighted homotopy limits and colimits in $\E$-categories. Given an $\E$-category $\eA$, categories $\A$ and $\B$, and objects $X\in\dA_0(\A)$ and $W\in\E(\A\op\times\B)$, the \textbf{homotopy colimit of $X$ weighted by $W$}, if it exists, is an object $W\otimes_\A X\in\dA_0(\B)$. This object must represent the $\E$-category map below: 
\begin{center}
\begin{tikzcd}[column sep=small]
\eA\arrow{rrrrr}{\widetilde{\map}_{\dA}(X,-)} &&&&& \eE^{\A\op}\arrow{rrrrr}{\widetilde{\map}_{\E^{\A\op}}(W,-)} &&&&& \eE^{\B\op}
\end{tikzcd}
\end{center} 
Thus, for any category $\C$ and any $Z\in\dA_0(\C)$, the weighted homotopy colimit is equipped with isomorphisms
\[
\widetilde{\map}_{\dA}(W\otimes_\A X,Z)\iso\widetilde{\map}_{\E^{\A\op}}(W,\widetilde{\map}_{\dA}(X,Z))
\]
in $\E(\B\op\times\C)$, which are \textbf{$\E$-natural} in $Z$ (see \cref{E-natural transformation definition}). This is a local definition for the weighted homotopy colimit; if, given an object $X\in\dA_0(\A)$, the weighted homotopy colimit of $X$ exist for all possible weights, then we can obtain a global characterisation in the manner of~\cite{GR19,GS17}. Specifically, the weighted homotopy colimits assemble into an $\E$-category map $-\otimes_\A X$, which forms part of the following $\E$-category \textbf{adjunction} (see \cref{Section The yoneda lemma and adjunctions for E-categories}): 
\begin{center}
\begin{tikzcd}
\eE^{\A\op}\arrow[rr, bend left=40, "-\;\otimes_\A X" above,""{name=U, below}]\arrow[rr,leftarrow, bend right=40, "\widetilde{\map}_{\dA}(X\text{,}-)" below, ""{name=D}]
&& \;\eA
\arrow[phantom,from=U,to=D,"\bot"]
\end{tikzcd}
\end{center}

In \cref{(Left) E-derivator definition}, we define \textbf{$\E$-derivators}. These are $\E$-prederivators that, in particular, admit all weighted homotopy limits and colimits. We show, in \cref{(Left) E-derivators induce (left) derivators}, that if $\eA$ is an $\E$-derivator, the induced prederivator $\dA$ is a derivator. 

In \cref{Closed E-modules induce E-derivators}, we show that any closed $\E$-module $\D$ gives rise to an $\E$-derivator $\eD$. In this way, the theory of closed $\E$-modules is encompassed by the theory of $\E$-prederivators and $\E$-derivators. However, in general, $\E$-prederivators need not admit all weighted homotopy limits and colimits: for instance, in \cref{Compact objects give an E-prederivator}, we discuss the prederivator of compact objects in a triangulated derivator, which we show admits certain, but not all, $\dHo(\Spt)$-weighted homotopy colimits. In this way, working with $\E$-prederivators gives us the flexibility to study natural examples of enriched homotopy theories that are not captured by closed $\E$-modules. Moreover, even if an $\E$-prederivator $\eA$ does not admit all weighted homotopy (co)limits, if it admits enough, we can still manipulate these as in~\cite{GR19,GS17}. In particular, we can carry out computations with weighted homotopy (co)limits by doing computations with the weights in $\E$. 

Working in the $2$-category of $\E$-prederivators has other advantages. For example, the $\E$-category Yoneda lemma of \cref{Yoneda lemma for E-categories} is extremely useful. We give an application of the Yoneda lemma in \cref{E-morphism is representable if underlying morphism is}, in which we prove the following representability theorem for $\E$-prederivator maps:
\begin{theorem*}
Let $\eA$ be an $\E$-prederivator. An $\E$-category morphism $F:\eA\rightarrow\eE$ is representable if and only if the $\E(\0)$-functor
\[
F:\eA(\0)\rightarrow\eE(\0)
\]
is representable as an $\E(\0)$-functor.
\end{theorem*}
The $\E(\0)$-functor $F:\eA(\0)\rightarrow\eE(\0)$ that appears in this theorem is defined in \cref{A([0]) is E([0])-enriched}. In this way, even if we are only interested in studying concepts that can be phrased on the level of unenriched derivators, it can be beneficial to work in the $2$-category of $\dHo(\sSet)$-prederivators, particularly since derivators associated to model categories and quasicategories are naturally $\dHo(\sSet)$-enriched.

\subsection*{Organisation}

In \cref{Chapter Derivators}, we survey the elementary theory of prederivators and derivators. Most of the content of this section can be found in~\cite{Groth13,Groth16,GPS14}. We recall the basic definitions and a number of important examples in \cref{Section Prederivators and derivators}. In \cref{Section Cocontinuous maps and adjunctions}, we discuss preservation of homotopy Kan extensions and study adjunctions between derivators. In~\cref{Section Two-variable maps}, we recall some simple aspects of the theory of two-variable derivator maps. Finally, in~\cref{Section Pointed and stable derivators}, we recall the definitions of pointed and triangulated derivators. We also recall some important results from the theory of triangulated categories, which we use in \cref{Chapter Actions of monoidal derivators} and \cref{Chapter Enriched derivators} to study enriched triangulated derivators.
 
In \cref{Chapter Actions of monoidal derivators}, we discuss actions of monoidal derivators. First, in~\cref{Section Ends and coends}, we recall the definitions of ends and coends in a derivator; these play an essential role in the theory of derivator two-variable adjunctions, which we recall in~\cref{Section Two-variable adjunctions}. In~\cref{Section Monoidal derivators} and~\cref{Section Closed actions of monoidal derivators}, we recall the definitions of monoidal derivators and their actions. Much of the material in these first sections can be found in~\cite{GPS14}; beyond this point, unless otherwise specified, the results are new. In~\cref{Section The cancelling tensor product} and~\cref{Section The maps otimes u and h u}, we collect a number of coherence results for structure arising from the action of a monoidal derivator. These results provide important ingredients for the development of $\E$-categories and $\E$-prederivators in~\cref{Chapter E-categories} and~\cref{Chapter Enriched derivators}. We also use them, in~\cref{Section Cotensors in closed E-modules}, to study cotensors in closed $\E$-modules. In particular, in \cref{Left adjoint preserves tensors iff right adjoint preserves cotensors}, we prove that a left adjoint between closed $\E$-modules preserves tensors if and only if its right adjoint preserves cotensors. In \cref{Section A representability theorem for triangulated E-modules}, we use this result to prove \cref{Representability for perfectly generated triangulated derivator}, a representability theorem for triangulated closed modules over triangulated monoidal derivators. 

In \cref{Chapter E-categories}, we introduce $\E$-categories. In~\cref{Section Basic definitions of E-categories} we develop their basic theory and give a number of examples; in particular, in~\cref{E-modules are E-categories}, we prove that any closed $\E$-module gives rise to an associated $\E$-category. We prove the $\E$-category Yoneda lemma, \cref{Yoneda lemma for E-categories}, in \cref{Section The yoneda lemma and adjunctions for E-categories}, and use this to study $\E$-category adjunctions. In \cref{Section Transferring enrichments}, we study monoidal morphisms, and prove, in \cref{Transferring enrichment along a monoidal left adjoint}, that we may transfer enrichment along monoidal adjunctions. 

In \cref{Chapter Enriched derivators}, we introduce $\E$-prederivators and $\E$-derivators. We study $\E$-prederivators in \cref{Section E-prederivators}, which we introduce in \cref{E-prederivator definition}. In \cref{Section The 2-category of E-prederivators}, we define the $2$-category of $\E$-prederivators, and prove \cref{E-morphism is representable if underlying morphism is}, a representability theorem for $\E$-prederivator maps. In \cref{Section E-derivators}, we define weighted homotopy limits and colimits in an $\E$-category, and give the definition of $\E$-derivators in \cref{(Left) E-derivator definition}. In \cref{Closed E-modules induce E-derivators}, we show that the $\E$-category associated to a closed $\E$-module is an $\E$-derivator. Finally, in \cref{(Left) E-derivators induce (left) derivators}, we show that any $\E$-derivator induces a derivator.

\begin{notation*}
We will use the following notation throughout:
\begin{itemize}
\item We will write $\mathbf{[n]}=\{ 0\rightarrow 1\rightarrow\dots\rightarrow n \}$ for the ordinal number $n+1$ regarded as a category. In particular, $\0$ is the terminal category.

\item We will denote small categories with upright font. For example, $\A,\B,...\in\Cat$.

\item Given categories $\A$ and $\B$, we will denote the symmetry isomorphism for the product by $\sigma:\A\times\B\xrightarrow{\;\iso\;}\B\times\A$.

\item Given categories $\A$ and $\B$, we will write $\p_\A:\A\times\B\rightarrow\B$ for the canonical projection onto $\B$. We will also write $\p_\A:\B\times\A\rightarrow\B$. 

More generally, we will write $\p_\A$ for any map projecting away the category $\A$. For example, we also write $\p_\A:\A\rightarrow\0$ for the canonical map to the terminal category.

\item When we have multiple copies of the same category appearing in a product, for example $\A$ in the product $\A\times\A$, we will write $\A=\A_1=\A_2$ to keep track of the maps $\sigma$ and $\p_\A$. In this example, the two projections are $\p_{\A_1}:\A_1\times\A_2\rightarrow\A_2$ and $\p_{\A_2}:\A_1\times\A_2\rightarrow\A_1$.
\end{itemize}
\end{notation*}

\chapter{Derivators}\label{Chapter Derivators}

In this chapter we recall the elementary theory that we will need in later chapters. We begin, in~\cref{Section Prederivators and derivators}, with the definition of the $2$-category $\Der$ of derivators. This section includes a number of important definitions and basic results; in particular, we introduce homotopy Kan extensions and homotopy exact squares. In~\cref{Section Cocontinuous maps and adjunctions}, we discuss preservation of homotopy Kan extensions and study adjunctions in $\Der$. In~\cref{Section Two-variable maps}, we look at some aspects of the theory of two-variable maps; we will return to this in~\cref{Chapter Actions of monoidal derivators} once we have recalled the theory of ends and coends. Finally, in~\cref{Section Pointed and stable derivators}, we briefly discuss pointed and triangulated derivators; in later chapters, these will provide important examples of enriched derivators. In this section we also recall some basic results about triangulated categories, which play a significant role in the theory of triangulated derivators.  Much of the material in this chapter can be found in~\cite{Groth13}. There are also useful introductions to basic derivator theory in~\cite{Groth16}, \cite[Section 2]{GPS14} and  \cite[Chapter 1]{CN08}.

\section{Prederivators and derivators}\label{Section Prederivators and derivators}

This section contains the definitions and basic elements of the theory of derivators. Most of the material in the section can be found in~\cite[Chapter 1,2]{Groth13}. We begin with the definition of prederivators, and develop the language that we require to state the additional axioms that define derivators. These are defined towards the middle of the section, in~\cref{Derivator definition}. We also introduce a number of examples, some of which we will revisit repeatedly. 

We will denote the $2$-category of small categories by $\Cat$, and the $2$-category of large categories by $\CAT$. We take the same size conventions as~\cite{Lagkas18a}.

\begin{definition}\label{Prederivator definition}
A \textbf{prederivator} is a $2$-functor $\D:\Cat\op\rightarrow\CAT$. We denote its values as follows:
\begin{center}
\begin{tikzcd}
\A\arrow[rr, bend left=50, "u" above,""{name=U, below}]\arrow[rr, bend right=50, "v" below, ""{name=D}]
&& \B
& [-15pt] \longmapsto
& [-15pt] \D(\A)
& \D(\B)\arrow[l,bend left=50,"v^*" below,""{name=L,above}]\arrow[l,bend right=50,"u^*" above,""{name=R,below}]
\arrow[Rightarrow,from=U,to=D,shorten >=0.1cm,shorten <=0.1cm,"\kappa"]
\arrow[Rightarrow,from=R,to=L,shorten >=0.1cm,shorten <=0.1cm,"\kappa^*"]
\end{tikzcd}
\end{center}
\end{definition}

\begin{definition}\label{Pseudonatural transformation definition}
Let $\D_1$ and $\D_2$ be prederivators. We call a \textbf{pseudonatural transformation} $F:\D_1\rightarrow\D_2$ a \textbf{morphism of prederivators}. Explicitly, this consists of functors $F:\D_1(\A)\rightarrow\D_2(\A)$ for any category $\A$, and natural isomorphisms
\begin{center}
\begin{tikzcd}
 \D_1(\B) \arrow[rr,"F"]\arrow[dd,"u^*" left] && 
 \D_2(\B)\arrow[Rightarrow,ddll,"\gamma" above left,"\iso" below right,shorten >=1cm,shorten <=1cm, shift left] \arrow[dd,"u^*"] \\
 &\\
\D_1(\A)\arrow[rr,"F" below] && 
 \D_2(\A)
\end{tikzcd}
\end{center}
for any functor $u:\A\rightarrow\B$. This data must satisfy the following equalities:
\begin{enumerate}
\item For any category $\A$, we have:
\begin{center}
\begin{tikzcd}
 \D_1(\A) \arrow[rr,"F"]\arrow[dd,"\id_\A^*" left] && 
 \D_2(\A)\arrow[Rightarrow,ddll,"\gamma" above left,"\iso" below right,shorten >=1.1cm,shorten <=1.1cm, shift left] \arrow[dd,"\id_\A^*"] \\
 &&&[-25pt] = &[-25pt] \id_F\\
\D_1(\A)\arrow[rr,"F" below] && 
 \D_2(\A)
\end{tikzcd}
\end{center}
\item For any composable maps $\A\xrightarrow{\;\;u\;\;}\B\xrightarrow{\;\;v\;\;}\C$, we have:
\begin{center}
\begin{tikzcd}
 \D_1(\C) \arrow[rr,"F"]\arrow[dd,"v^*" left] && 
 \D_2(\C)\arrow[Rightarrow,ddll,"\gamma" above left,"\iso" below right,shorten >=1.1cm,shorten <=1.1cm, shift left] \arrow[dd,"v^*"] \\
&&&[-15pt] &[-15pt] \D_1(\C)\arrow[rr,"F"]\arrow[dd,"(v\circ u)^*" left] &&\D_2(\C)\arrow[dd,"(v\circ u)^*"]\arrow[Rightarrow,ddll,"\gamma" above left,"\iso" below right,shorten >=1.1cm,shorten <=1.1cm, shift left]  
\\
\D_1(\B)\arrow[rr,"F" below]\arrow[dd,"u^*" left] && 
 \D_2(\B)\arrow[Rightarrow,ddll,"\gamma" above left,"\iso" below right,shorten >=1.1cm,shorten <=1.1cm, shift left] \arrow[dd,"u^*"]
 & =
 \\
&&&& \D_1(\A)\arrow[rr,"F" below] &&\D_2(\A) 
\\
\D_1(\A)\arrow[rr,"F" below] && 
 \D_2(\A)
\end{tikzcd}
\end{center}
\item For any natural transformation $\kappa:u\Rightarrow v$, we have:
\begin{center}
\begin{tikzcd}
 \D_1(\B) \arrow[rr,"F"]\arrow[dd,"u^*" right,bend left=40,""{name=R,left}]\arrow[dd,"v^*" left,bend right=40,""{name=L,right}] && 
 \D_2(\B)\arrow[Rightarrow,ddll,"\gamma" above left,"\iso" below right,shorten >=1cm,shorten <=1cm, shift left] \arrow[dd,"u^*"] 
 &[-15pt] &[-15pt] \D_1(\B) \arrow[rr,"F"]\arrow[dd,"v^*" left] 
 && \D_2(\B)\arrow[Rightarrow,ddll,"\gamma" above left,"\iso" below right,shorten >=1cm,shorten <=1cm, shift left] \arrow[dd,"v^*" left,bend right=40,""{name=U,right}]\arrow[dd,"u^*" right,bend left=40,""{name=D,left}]
 \\
&&& = \\
\D_1(\A)\arrow[rr,"F" below] 
&& \D_2(\A)
&& \D_1(\A)\arrow[rr,"F" below] 
&& \D_2(\A)
\arrow[Rightarrow,from=R,to=L,shorten >=0.1cm,shorten <=0.1cm,"\;\;\kappa^*" above]
\arrow[Rightarrow,from=D,to=U,shorten >=0.1cm,shorten <=0.1cm,"\;\;\kappa^*" above]
\end{tikzcd}
\end{center}
\end{enumerate}

\end{definition}

\begin{definition}
Given prederivator maps $F,G:\D_1\rightarrow\D_2$, a \textbf{modification} $\theta:F\Rightarrow G$ consists of natural transformations
\begin{center}
\begin{tikzcd}
\D_1(\A)\arrow[r, bend left=50, "F" above,""{name=U, below}]\arrow[r, bend right=50, "G" below, ""{name=D}]
& \D_2(\A)
\arrow[Rightarrow,from=U,to=D,shorten >=0.1cm,shorten <=0.1cm,"\theta"]
\end{tikzcd}
\end{center}
such that, for any functor $u:\A\rightarrow\B$, we have:
\begin{center}
\begin{tikzcd}
 \D_1(\B) \arrow[rr,"F",bend left=30,""{name=R,below}]\arrow[dd,"u^*" left]\arrow[rr,"G" below,bend right=15,""{name=L,above}] && 
 \D_2(\B)\arrow[Rightarrow,ddll,"\gamma" above left,"\iso" below right,shorten >=1cm,shorten <=1cm, shift left] \arrow[dd,"u^*"] 
 &[-15pt] &[-15pt] \D_1(\B) \arrow[rr,"F"]\arrow[dd,"u^*" left] 
 && \D_2(\B)\arrow[Rightarrow,ddll,"\gamma" above left,"\iso" below right,shorten >=1cm,shorten <=1cm, shift left]\arrow[dd,"u^*" right]
 \\
&&& = \\
\D_1(\A)\arrow[rr,"G" below] 
&& \D_2(\A)
&& \D_1(\A)\arrow[rr,"F" above,bend left=15,""{name=D,below}] \arrow[rr,"G" below,bend right=30,""{name=U,above}]
&& \D_2(\A)
\arrow[Rightarrow,from=R,to=L,shorten >=0.1cm,shorten <=0.1cm,"\theta\;" left]
\arrow[Rightarrow,from=D,to=U,shorten >=0.1cm,shorten <=0.1cm,"\theta\;" left]
\end{tikzcd}
\end{center}
\end{definition}
Prederivators, morphisms and modifications form a $2$-category $\PDer$. Given prederivators $\D_1$ and $\D_2$, we write $\Hom(\D_1,\D_2)$ for the category of morphisms from $\D_1$ to $\D_2$.

\begin{example}\label{Representable and constant prederivators}
Given any (large) category $\mathcal{C}$, we can form its \textbf{represented prederivator}:  
\begin{align*}
y(\mathcal{C}):\Cat\op&\rightarrow \; \CAT\\
 \A \;\; &\mapsto \;\;\; \mathcal{C}^\A
\end{align*}
We may also form the \textbf{constant prederivator}:
\begin{align*}
c(\mathcal{C}):\Cat\op&\rightarrow \; \CAT\\
 \A \;\; &\mapsto \;\;\; \mathcal{C}
\end{align*}
\end{example}

In analogy with representable prederivators, for any prederivator $\D$, we call $\D(\0)$ the \textbf{underlying category} of $\D$, and for any $u:\A\rightarrow\B$, we call $u^*:\D(\B)\rightarrow\D(\A)$ the \textbf{restriction} or \textbf{pullback functor along $u$}. 

\begin{example}
Given any model category $\M$, we can form a prederivator
\begin{align*}
\dHo(\M):\Cat\op&\rightarrow \; \CAT\\
 \A \;\; &\mapsto \;\;\; \Ho(\M^\A)
\end{align*}
where the homotopy category $\Ho(\M^\A)$ is formed with respect to the pointwise weak equivalences in $\M^\A$. Note that $\M^\A$ may not carry a model structure in general; the fact that the localisation can still be formed without moving to a larger universe follows from~\cite{Cisinski03}. Moreover, given model categories $\M$ and $\N$, and a left Quillen functor $F:\M\rightarrow\N$, the derived functors induce a prederivator map $\mathds{L}F:\dHo(\M)\rightarrow\dHo(\N)$. Similarly, any right Quillen functor $G:\N\rightarrow\M$, induces a prederivator map $\mathds{R}G:\dHo(\N)\rightarrow\dHo(\M)$.

Similarly, given a quasicategory $Q$, we have a prederivator
\begin{align*}
\dHo(Q):\Cat\op&\rightarrow \; \CAT\\
 \A \;\; &\mapsto \;\;\; \Ho(Q^{N\A})
\end{align*}
where $N\A$ is the nerve of $\A$, and $\Ho(Q^{N\A})$ is the homotopy category of the quasicategory $Q^{N\A}$.
\end{example}

\begin{example}\label{Shifted prederivator definition}
For any prederivator $\D$ and any category $\J$, we can form the \textbf{shifted prederivator}:
\begin{align*}
\D^\J:\Cat\op&\rightarrow \; \CAT\\
 \A \; &\mapsto \;\D(\J\times\A)
\end{align*}
Moreover, any functor $u:\J\rightarrow\K$ induces a prederivator map $u^*:\D^\K\rightarrow\D^\J$, with component at $\A$ given by
\[
(u\times\A)^*:\D(\K\times\A)\rightarrow\D(\J\times\A),
\]
and any natural transformation $\kappa:u\Rightarrow v$ induces modification $\omega^*:u^*\Rightarrow v^*$. These organise into a $2$-functor: 
\begin{align*}
\Cat\op\times\PDer&\rightarrow \; \PDer\\
 (\J,\D) \;\;\;\;\;\; &\mapsto \;\;\; \D^\J
\end{align*}
\end{example}

\begin{remark}
Given a prederivator map $F:\D_1\rightarrow\D_2$, and a functor $u:\J\rightarrow\K$, the structure isomorphisms of $F$ induce a modification:
\begin{center}
\begin{tikzcd}
 \D_1^\K \arrow[rr,"F"]\arrow[dd,"u^*" left] && 
 \D_2^\K \arrow[Rightarrow,ddll,"\gamma" above left,"\iso" below right,shorten >=0.8cm,shorten <=0.8cm, shift left] \arrow[dd,"u^*"] \\
 &\\
\D_1^\J\arrow[rr,"F" below] && 
 \D_2^\J
\end{tikzcd}
\end{center}
Moreover, given $\theta:F\Rightarrow G$, the modification condition lifts to an equality:
\begin{center}
\begin{tikzcd}
 \D_1^\K \arrow[rr,"F",bend left=30,""{name=R,below}]\arrow[dd,"u^*" left]\arrow[rr,"G" below,bend right=30,""{name=L,above}] && 
 \D_2^\K\arrow[Rightarrow,ddll,"\gamma" above left,"\iso" below right,shorten >=1cm,shorten <=1cm, shift left=2] \arrow[dd,"u^*"] 
 &[-10pt] &[-10pt] \D_1^\K \arrow[rr,"F"]\arrow[dd,"u^*" left] 
 && \D_2^\K\arrow[Rightarrow,ddll,"\gamma" above left,"\iso" below right,shorten >=1cm,shorten <=1cm, shift right=2]\arrow[dd,"u^*" right]
 \\
&&& = \\
\D_1^\J\arrow[rr,"G" below] 
&& \D_2^\J
&& \D_1^\J\arrow[rr,"F" above,bend left=30,""{name=D,below}] \arrow[rr,"G" below,bend right=30,""{name=U,above}]
&& \D_2^\J
\arrow[Rightarrow,from=R,to=L,shorten >=0.1cm,shorten <=0.1cm,"\theta" left]
\arrow[Rightarrow,from=D,to=U,shorten >=0.1cm,shorten <=0.1cm,"\theta" left]
\end{tikzcd}
\end{center}
\end{remark}

Often it will be convenient to state and prove results at the level of shifted prederivators. For example, rather than prove a result about the component $F:\D_1(\A)\rightarrow\D_2(\A)$ of a prederivator map $F$, it may be more convenient to work with the shifted map $F:\D_1^\A\rightarrow\D_2^\A$ in $\PDer$.

\begin{definition}
Given any prederivator $\D$, we may define its \textbf{opposite} prederivator:
\begin{align*}
\D\op:\Cat\op&\rightarrow \; \CAT\\
 \A \;\; &\mapsto \; \D(\A\op)\op
\end{align*}
\end{definition}

\begin{example}
For any category $\mathcal{C}$, we have a canonical isomorphism $y(\mathcal{C})\op\iso y(\mathcal{C}\op)$, induced by the isomorphisms $(\mathcal{C}^{\A\op})\op\iso (\mathcal{C}\op)^\A$. Similarly, for any model category $\M$, we have $\dHo(\M)\op\iso\dHo(\M\op)$.
\end{example}

\begin{example}
For any prederivator $\D$ and any category $\J$, we have $(\D^\J)\op=(\D\op)^{\J\op}$.
\end{example}

\begin{definition}
Let $\D$ be a prederivator, and let $\A$ be a category. Suppose we have a map $f:a\rightarrow b$  in $\A$, classified by a natural transformation
\begin{center}
\begin{tikzcd}
\0\arrow[r, bend left=50, "a" above,""{name=U, below}]\arrow[r, bend right=50, "b" below, ""{name=D}]
& \A.
\arrow[Rightarrow,from=U,to=D,shorten >=0.1cm,shorten <=0.1cm,"f"]
\end{tikzcd}
\end{center}
Applying $\D$ to this diagram, we obtain
\begin{center}
\begin{tikzcd}
\D(\0)\arrow[leftarrow,r, bend left=50, "a^*" above,""{name=U, below}]\arrow[leftarrow,r, bend right=50, "b^*" below, ""{name=D}]
& \D(\A).
\arrow[Rightarrow,from=U,to=D,shorten >=0.1cm,shorten <=0.1cm,"f^*"]
\end{tikzcd}
\end{center}
For any object $X\in\D(\A)$, write $X_a$ for $a^*X$, and write $X_f:X_a\rightarrow X_b$ for the component of $f^*$ at $X$. These assignments define a functor:
\begin{align*}
\dia_\A(X):\A&\rightarrow \; \D(\0)\\
 \;\;\;\;\;\; a  &\mapsto \;\; X_a
\end{align*}
We call this the \textbf{underlying diagram of $X$}. This construction induces a functor \linebreak$\dia_\A:\D(\A)\rightarrow\D(\0)^\A$. Similarly, the underlying diagram functors of the shifted prederivators $\D^\J$ induce \textbf{partial underlying diagram functors} $\dia_\A^\J:\D(\J\times\A)\rightarrow\D(\J)^\A$ for any $\J$.
\end{definition}

\begin{definition}\label{Homotopy Kan extension definition}
Let $\D$ be a prederivator, and let $u:\A\rightarrow\B$ be a functor. If the pullback functor $u^*:\D(\B)\rightarrow\D(\A)$ has a left adjoint $u_!:\D(\A)\rightarrow\D(\B)$, we call this map the \textbf{left (homotopy) Kan extension} along $u$. If $u^*$ admits a right adjoint $u_*:\D(\A)\rightarrow\D(\B)$, we call this the \textbf{right (homotopy) Kan extension} along $u$.

In particular, in the case of the unique map $\p:\A\rightarrow\0$, we call these the \textbf{(homotopy) colimit} and \textbf{(homotopy) limit}, and denote them by:
\[
\hocolim:=\p_!:\D(\A)\rightarrow\D(\0)
\]
\[\holim:=\p_*:\D(\A)\rightarrow\D(\0)
\]
\end{definition}

\begin{example}
Given a category $\mathcal{C}$, homotopy Kan extensions in the representable prederivator $y(\mathcal{C})$ are exactly Kan extensions. If $\M$ is a model category, homotopy Kan extensions in $\dHo(\M)$ recover the familiar homotopy Kan extensions. 
\end{example}

\begin{definition}\label{D-exact squares}
Let $\D$ be a prederivator, and suppose we have a natural transformation:
\begin{center}
\begin{tikzcd}
 \A \arrow[rr,"u"]\arrow[dd,"v" left] && 
 \B \arrow[Rightarrow,ddll,"\kappa" above left,shorten >=0.7cm,shorten <=0.8cm, shift left] \arrow[dd,"w"] \\
 &\\
\J\arrow[rr,"z" below] && 
 \K
\end{tikzcd}
\end{center}
Suppose that each of the functors above admits left and right homotopy Kan extensions in $\D$. Then the square above induces canonical maps:
\begin{center}
\begin{tikzcd}
&[-25pt] &[-25pt] \D(\J)\arrow[ddrr,bend right,equal,""{name=L,above right}]\arrow[rr,leftarrow,"v_!" above] && \D(\A)\arrow[Rightarrow,to=L,"\epsilon" above left,shorten >=0.5cm,shorten <=0.7cm] \arrow[leftarrow,rr,"u^*"]\arrow[leftarrow,dd,"v^*" left] && 
 \D(\B)\arrow[ddrr,bend left,equal,""{name=R,below left}] \arrow[Rightarrow,ddll,"\kappa^*" above left,shorten >=1cm,shorten <=1cm, shift left] \arrow[leftarrow,dd,"w^*"]\\
 
 \kappa_! & := \\
 
&&&&\D(\J)\arrow[leftarrow,rr,"z^*" below] && 
 \D(\K)\arrow[Leftarrow,to=R,"\eta" above left,shorten >=0.5cm,shorten <=0.7cm] \arrow[rr,leftarrow,"w_!" below]
 && \D(\B) 
\end{tikzcd}
\end{center}

\begin{center}
\begin{tikzcd}
&[-25pt] &[-25pt] \D(\B)\arrow[ddrr,bend left,equal,""{name=U,below left}]\arrow[dd,leftarrow,"u_*" left] \\
\\
&&
 \D(\A)\arrow[Leftarrow,to=U,"\eta" above left,shorten >=0.5cm,shorten <=0.5cm] \arrow[leftarrow,rr,"u^*"]\arrow[leftarrow,dd,"v^*" left] && 
 \D(\B) \arrow[Rightarrow,ddll,"\kappa^*" above left,shorten >=1cm,shorten <=1cm, shift left] \arrow[leftarrow,dd,"w^*"] \\
\kappa_*& := \\
&&\D(\J)\arrow[ddrr,bend right,equal,""{name=D,above right}]\arrow[leftarrow,rr,"z^*" below] && 
 \D(\K)\arrow[Rightarrow,to=D,"\epsilon" above left,shorten >=0.5cm,shorten <=0.5cm] \\
 \\
&&&& \D(\J)\arrow[uu,"z_*" right]
\end{tikzcd}
\end{center}
These two transformations are \textbf{mates} of the natural transformation $\kappa^*$ in the sense of~\cite[Section 2]{KS74}. An introduction to mates can be found in~\cite[Appendix A]{GPS14a}. It follows that $\kappa_!$ and $\kappa_*$ are \textbf{conjugate}, as in~\cite[Chapter IV.7]{MacLane13}. In particular, $\kappa_!$ is an isomorphism if and only if $\kappa_*$ is. If this is the case, we say the square $\kappa$ is \textbf{$\D$-exact}.
\end{definition}

\begin{remark}\label{Canonical map to colimit}
Given a commutative square, we can apply the constructions of \cref{D-exact squares} to the identity transformation, taking the convention that the transformation goes from the top composite the bottom. Note that, even in this case, the direction of the $2$-cell is important. For example, we have a commutative square:
\begin{center}
\begin{tikzcd}
 \A \arrow[rr,"u"]\arrow[dd,"\p" left] && 
 \B \arrow[Rightarrow,ddll,"\id" above left,shorten >=0.8cm,shorten <=0.9cm]\arrow[dd,"\p"] \\
 &\\
\0\arrow[rr, equal] && 
 \0
\end{tikzcd}
\end{center}
For any prederivator $\D$ admitting the relevant homotopy colimits we get a canonical map:
\begin{center}
\begin{tikzcd}
 \D(\B)\arrow[dd,"u^*" left]\arrow[rrd,"\p_!" above right,bend left,""{name=U}] \\
 
 && \D(\0) \\
 
 \D(\A) \arrow[rur,bend right,"\p_!" below right]
 
 \arrow[Rightarrow,from=3-1,to=U,shorten >=1cm,shorten <=1cm,shift right=2]
 
\end{tikzcd}
\end{center} 
On the other hand, considering the identity as a map in the opposite direction
\begin{center}
\begin{tikzcd}
 \A \arrow[rr,"\p"]\arrow[dd,"u" left] && 
 \0 \arrow[Rightarrow,ddll,"\id" above left,shorten >=0.8cm,shorten <=0.9cm] \arrow[dd, equal] \\
 &\\
\B\arrow[rr, "\p" below] && 
 \0
\end{tikzcd}
\end{center}
and taking a prederivator $\D$ admitting homotopy limits, we get a canonical map:
\begin{center}
\begin{tikzcd}
 \D(\B)\arrow[dd,"u^*" left]\arrow[rrd,"\p_*" above right,bend left,""{name=U}] \\
 
 && \D(\0) \\
 
 \D(\A) \arrow[rur,bend right,"\p_*" below right]
 
 \arrow[Leftarrow,from=3-1,to=U,shorten >=1cm,shorten <=1cm,shift right=2]
\end{tikzcd}
\end{center} 
For this reason, even if the $2$-cell in a square is the identity, we will indicate its direction when we discuss $\D$-exact squares.
\end{remark}

We now give the definition of derivator. Following~\cite{Groth13}, we list four derivator axioms \textbf{Der 1-4}. Other sources, for example~\cite{Tabuada08}, add a fifth derivator axiom, \textbf{Der 5}, which we  recall in \cref{Strong prederivator definition}. We will call prederivators that satisfy all five axioms strong derivators.

\begin{definition}\label{Derivator definition}
A prederivator $\D$ is a \textbf{derivator} if it satisfies the following axioms:
\setlength{\leftmargini}{3.8em}
\begin{itemize}
\item[\textbf{Der 1}]$\D$ preserves coproducts. That is, the canonical map 
\[
\D(\coprod_{i}\A_i)\rightarrow\prod_{i}\D(\A_i)
\] 
is an equivalence. In particular, $\D(\emptyset)\equiv\0$.
\item[\textbf{Der 2}]A map $f:X\rightarrow Y$ in $\D(\A)$ is an isomorphism if and only if $f_a:X_a\rightarrow Y_a$ is an isomorphism in $\D(\0)$, for every $a\in\A$.
\item[\textbf{Der 3}] Any functor $u:\A\rightarrow\B$ admits both a left and right Kan extension in $\D$.
\item[\textbf{Der 4}] For any functor $u:\A\rightarrow\B$ and any $b\in\B$, the squares below are $\D$-exact:
\begin{center}
\begin{tikzcd}
 \A\downarrow b \arrow[rr,"\text{pr}"]\arrow[dd,"\p" left] && 
 \A \arrow[Rightarrow,ddll,shorten >=1cm,shorten <=1cm] \arrow[dd,"u"] &
 b\downarrow\A \arrow[rr,"\text{pr}"]\arrow[dd,"\p" left] && 
 \A \arrow[Leftarrow,ddll,shorten >=1cm,shorten <=1cm] \arrow[dd,"u"]\\
 &\\
\0\arrow[rr,"b" below] && 
 \B &
 \0\arrow[rr,"b" below] && 
 \B
\end{tikzcd}
\end{center}
Thus, for any $X\in\D(\A)$, we have canonical isomorphisms $\hocolim(\text{pr}^*X)\iso (u_!X)_b$ and $(u_*X)_b\iso\holim(\text{pr}^*X)$.
\end{itemize}
We denote the full sub-$2$-category of $\PDer$ on derivators by $\Der$. A prederivator map between derivators will often be called a \textbf{derivator map}.
\end{definition}

In \cref{Derivator definition}, the axioms \textbf{Der 1} and \textbf{Der 2} give conditions on the values of a prederivator that make them behave like the homotopy categories of diagram categories. We call a prederivator a \textbf{semiderivator} if it satisfies these two axioms. 

The axiom \textbf{Der 3} is a completeness condition. \textbf{Der 4} allows us to calculate the underlying diagram of a homotopy Kan extension entirely in terms of homotopy limits and colimits. A semiderivator $\D$ is called a \textbf{left derivator} if it satisfies the parts of \textbf{Der 3} and \textbf{Der 4} that deal with left Kan extensions. \textbf{Right derivators} are defined dually. Note that this terminology agrees with the terminology of~\cite{GS17}, but reverses the terminology of~\cite{Cisinski08}.

\begin{remark}
A square 
\begin{center}
\begin{tikzcd}
 \A \arrow[rr,"u"]\arrow[dd,"v" left] && 
 \B \arrow[Rightarrow,ddll,"\kappa" above left,shorten >=0.9cm,shorten <=0.9cm, shift left] \arrow[dd,"w"] \\
 &\\
\J\arrow[rr,"z" below] && 
 \K
\end{tikzcd}
\end{center}
is called \textbf{homotopy exact} if it is $\D$-exact for every derivator $\D$. By definition, every comma square of the form given in \textbf{Der 4} is homotopy exact. A complete characterisation of homotopy exact squares appears in~\cite[Section 3]{GPS14a}.
\end{remark}

\begin{definition}\label{Strong prederivator definition}
A prederivator $\D$ is called \textbf{strong} if it satisfies the following axiom:
\setlength{\leftmargini}{3.8em}
\begin{itemize}
\item[\textbf{Der 5}]For each category $\A$, the partial underlying diagram functor
\[
\dia_{[1]}^\A:\D(\A\times [1])\rightarrow\D(\A)^{[1]}
\]
is full and essentially surjective.
\end{itemize}
\end{definition}

This axiom gives an important connection between the maps in a category $\D(\A)$ and the objects in $\D(\A\times [1])$. It is analogous to the triangulated category axioms (\textbf{TR1} and \textbf{TR3} in~\cite{Neeman14}) that allow us to extend a map in a triangulated category to a distinguished triangle, and to extend a commutative square to a map of distinguished triangles. 

\begin{example}
A prederivator $\D$ is a (strong) derivator if and only if $\D\op$ is. 
\end{example}

\begin{example}
A representable prederivator $y(\mathcal{C})$ is a derivator if and only if $\mathcal{C}$ is complete and cocomplete. In this case, \textbf{Der 4} says that Kan extensions in $\mathcal{C}$ are computed pointwise. Any representable prederivator is strong, since all underlying diagram functors are equivalences.

The constant prederivator $c(\mathcal{C})$ is not a derivator for any $\mathcal{C}\neq\0$, since in this case \textbf{Der 1} fails.
\end{example}

\begin{example}[Cisinski]
For any model category $\M$, the prederivator $\dHo(\M)$ is a strong derivator. For a general model category, the proof is technical; this is the main result of~\cite{Cisinski03}. For a simpler proof in the case of combinatorial model categories, see~\cite[Section 1.3]{Groth13}.
\end{example}

\begin{example}
Let $Q$ be a quasicategory. The prederivator $\dHo(Q)$ is strong. If $Q$ is complete and cocomplete, then $\dHo(Q)$ is a derivator. See~\cite{GPS14a,Lenz18} for proofs of this fact.
\end{example}

\begin{example}\label{Shifted derivator is a derivator}
For any (strong) derivator $\D$ and any category $\A$, the shifted prederivator $\D^\A$ is a (strong) derivator.  See~\cite[Section 1.3]{Groth13} for a proof.
\end{example}

Shifted derivators are extremely useful when proving general theorems about categories that arise as the values of derivators: if a statement holds for the underlying category of every derivator $\D$, then for any category $\A$ the statement must also be true for $\D(\A)$, since this is the underlying category of $\D^\A$. The following remark gives a simple example:

\begin{remark}
If $\D$ is a derivator, \textbf{Der 1} implies that homotopy products and coproducts coincide with products and coproducts in $\D(\0)$. By \textbf{Der 3}, then, $\D(\0)$ admits all products and coproducts. Using~\cref{Shifted derivator is a derivator}, it follows that $\D(\A)$ has all products and coproducts, for any category $\A$.
\end{remark}

The following lemma is a simple but important consequence of \textbf{Der 2}:

\begin{lemma}\label{modifications are isomorphisms iff they are on underlying category}
Suppose we have prederivators $\D_1$ and $\D_2$, and a modification 
\begin{center}
\begin{tikzcd}
\D_1\arrow[r, bend left=50, "F" above,""{name=U, below}]\arrow[r, bend right=50, "G" below, ""{name=D}]
& \D_2.
\arrow[Rightarrow,from=U,to=D,shorten >=0.1cm,shorten <=0.1cm,"\theta"]
\end{tikzcd}
\end{center}
If $\D_2$ satisfies \textbf{Der 2}, then $\theta$ is an isomorphism if and only if the component on underlying categories is an isomorphism.
\end{lemma}
\begin{proof}
For any category $\A$, and any $X\in\D_1(\A)$ consider the component $\theta_X:FX \rightarrow GX$ in $\D_2(\A)$. By \textbf{Der 2}, this is an isomorphism if and only if, for every $a\in\A$, the map $(\theta_X)_a:(FX)_a \rightarrow (GX)_a$ is an isomorphism in $\D_2(\0)$. But the modification condition for $\theta$ implies that the diagram below commutes:
\begin{center}
\begin{tikzcd}
 (FX)_a \arrow[rr,"\iso" below,"\gamma"]\arrow[dd,"(\theta_X)_a" left] && 
 F(X_a) \arrow[dd,"\theta_{X_a}"] \\\\
 (GX)_a\arrow[rr,"\iso","\gamma" below] && 
 G(X_a)
\end{tikzcd}
\end{center}
Thus, $\theta$ is an isomorphism if and only if $\theta_x$ is an isomorphism for every $x\in\D_1(\0)$.
\end{proof}

\section{Cocontinuous maps and adjunctions}\label{Section Cocontinuous maps and adjunctions}

We begin this section with a discussion of the interaction between derivator maps and homotopy Kan extensions, proving some simple results that we will need in~\cref{Chapter Actions of monoidal derivators} and~\cref{Chapter E-categories}. Other results along similar lines can be found in~\cite{Groth16}. In the second half of the chapter we recall the basic theory of adjunctions from~\cite[Section 2]{Groth13}. This will form a basis for the discussion of two-variable adjunctions in~\cref{Chapter Actions of monoidal derivators}.

\begin{definition}\label{Cocontinuous derivator map}
Suppose we have derivators $\D_1$ and $\D_2$, and a morphism $F:\D_1\rightarrow\D_2$. For any functor $u:\A\rightarrow\B$ we have a canonical transformation:
\begin{center}
\begin{tikzcd}
\D_1(\A)\arrow[ddrr,bend right,equal,""{name=L,above right}]\arrow[rr,rightarrow,"u_!" above] && \D_1(\B)\arrow[Leftarrow,to=L,"\eta" above left,shorten >=0.5cm,shorten <=0.7cm] \arrow[rightarrow,rr,"F"]\arrow[rightarrow,dd,"u^*" left] && 
 \D_2(\B)\arrow[ddrr,bend left,equal,""{name=R,below left}] \arrow[Leftarrow,ddll,"\gamma^{-1}" above left, "\iso" below right, shorten >=1cm,shorten <=1cm, shift left] \arrow[rightarrow,dd,"u^*"]\\
 
\\
 
&&\D_1(\A)\arrow[rightarrow,rr,"F" below] && 
 \D_2(\A)\arrow[Rightarrow,to=R,"\epsilon" above left,shorten >=0.5cm,shorten <=0.7cm] \arrow[rr,rightarrow,"u_!" below]
 && \D_2(\B) 
\end{tikzcd}
\end{center}
We say $F$ \textbf{preserves the left homotopy Kan extension along $u$} if this map is an isomorphism. If $F$ preserves all left homotopy Kan extensions, we say $F$ is \textbf{cocontinuous}. We denote the full subcategory of $\Hom(\D_1,\D_2)$ on the cocontinuous maps by $\cHom(\D_1,\D_2)$.

Dually, we can define \textbf{continuous} maps; denote the category of these by $\Hom_*(\D_1,\D_2)$.
\end{definition}

We record the following fact, whose proof can be found in~\cite[Section 2]{Groth13}:

\begin{lemma}
A derivator map $F:\D_1\rightarrow\D_2$ is cocontinuous if and only if it preserves homotopy colimits.
\end{lemma}

The following lemma and its dual can be found in~\cite[Section 3]{Groth16}:

\begin{lemma}\label{Modifications respect left Kan extensions}
Let $F,G:\D_1\rightarrow\D_2$ be derivator maps, and let $\theta:F\Rightarrow G$ be a modification. Then, for any functor $u:\A\rightarrow\B$ and any $X\in\D_1(\A)$, the diagram below commutes, where the vertical arrows are the canonical maps of~\cref{Cocontinuous derivator map}:
\begin{center}
\begin{tikzcd}
 u_! FX \arrow[rr,"u_!(\theta_X)"]\arrow[dd] && 
  u_! GX \arrow[dd] \\\\
Fu_! X\arrow[rr,"\theta_{u_! X}" below] && 
 Gu_! X
\end{tikzcd}
\end{center}
\end{lemma}
\begin{proof}
The commutativity of this diagram for any $X\in\D_1(\A)$ expresses the equality of the two pasting diagrams below:
\begin{center}
\begin{tikzcd}
\D_1(\A)\arrow[ddrr,bend right,equal,""{name=L,above right}]\arrow[rr,rightarrow,"u_!" above] && \D_1(\B)\arrow[Leftarrow,to=L,"\eta" above left,shorten >=0.5cm,shorten <=0.7cm] \arrow[rightarrow,rr,"G"]\arrow[rightarrow,dd,"u^*" left] && 
 \D_2(\B)\arrow[ddrr,bend left,equal,""{name=R,below left}] \arrow[Leftarrow,ddll,"\gamma^{-1}" above left, "\iso" below right, shorten >=1cm,shorten <=1cm, shift right=5] \arrow[rightarrow,dd,"u^*"]\\
 
\\
 
&&\D_1(\A)\arrow[rightarrow,rr,"G" above, bend left,""{name=A}]\arrow[rightarrow,rr,"F" below, bend right,""{name=B}] && 
 \D_2(\A)\arrow[Rightarrow,to=R,"\epsilon" above left,shorten >=0.5cm,shorten <=0.7cm] \arrow[rr,rightarrow,"u_!" below]
 && \D_2(\B) 
 
\arrow[Rightarrow, from=B, to=A, "\theta" right,shorten >=0.3cm,shorten <=0.2cm]
\end{tikzcd}
\end{center}
\begin{center}
\begin{tikzcd}
\D_1(\A)\arrow[ddrr,bend right,equal,""{name=L,above right}]\arrow[rr,rightarrow,"u_!" above] && \D_1(\B)\arrow[Leftarrow,to=L,"\eta" above left,shorten >=0.5cm,shorten <=0.7cm] \arrow[rightarrow,rr,"F" below,bend right,""{name=B}]\arrow[rightarrow,rr,"G",bend left,""{name=A}]\arrow[rightarrow,dd,"u^*" left] && 
 \D_2(\B)\arrow[ddrr,bend left,equal,""{name=R,below left}] \arrow[Leftarrow,ddll,"\gamma^{-1}" above left, "\iso" below right, shorten >=1cm,shorten <=1cm, shift left=5] \arrow[rightarrow,dd,"u^*"]\\
 
\\
 
&&\D_1(\A)\arrow[rightarrow,rr,"F" below] && 
 \D_2(\A)\arrow[Rightarrow,to=R,"\epsilon" above left,shorten >=0.5cm,shorten <=0.7cm] \arrow[rr,rightarrow,"u_!" below]
 && \D_2(\B) 
 
 \arrow[Rightarrow, from=B, to=A, "\theta" right,shorten >=0.3cm,shorten <=0.2cm]
\end{tikzcd}
\end{center}
This follows immediately from the modification condition for $\theta$.
\end{proof}

\begin{lemma}\label{Derivator maps respect homotopy exact squares}
Suppose we have a derivator map $F:\D_1\rightarrow\D_2$ and a natural transformation:
\begin{center}
\begin{tikzcd}
 \A \arrow[rr,"u"]\arrow[dd,"v" left] && 
 \B \arrow[Rightarrow,ddll,"\kappa" above left,shorten >=0.7cm,shorten <=0.8cm, shift left] \arrow[dd,"w"] \\
 &\\
\J\arrow[rr,"z" below] && 
 \K
\end{tikzcd}
\end{center}
Consider the natural transformation $\kappa_!:v_!\circ u^*\Rightarrow z^*\circ w_!$ of~\cref{D-exact squares}. For any \linebreak$X\in\D_1(\B)$, the diagram below commutes, where the vertical maps are induced by the canonical map in~\cref{Cocontinuous derivator map}, and the structure isomorphisms of $F$:
\begin{center}
\begin{tikzcd}
 v_!u^*FX \arrow[rr,"(\kappa_!)_{FX}"]\arrow[dd] && 
 z^*w_!FX\arrow[dd] \\
 &\\
Fv_!u^*X\arrow[rr,"F(\kappa_!)_X" below] && 
Fz^*w_!X
\end{tikzcd}
\end{center}
\end{lemma}
\begin{proof}
The commutative diagram above expresses the equality of certain pasting diagrams. Using the triangle equality to cancel instances of units and counits, we can reduce these to the following:
\begin{center}
\begin{tikzcd}
&&&& \D_1(\J)\arrow[ddrr,"F" above right,bend left=25,""{name=A,below left,near start}]\\\\
\D_1(\B)\arrow[ddrr,bend right,equal,""{name=L,above right}]\arrow[rr,"w_!" above] && \D_1(\K)\arrow[uurr,"z^*" above left, bend left=25]\arrow[Leftarrow,to=L,"\eta" above left,shorten >=0.5cm,shorten <=0.7cm] \arrow[rr,"F",""{name=B,above right}]\arrow[dd,"w^*" left] && 
 \D_2(\K)\arrow[dd,"w^*" left]\arrow[rr,"z^*" above]\arrow[Leftarrow,ddll,"\gamma^{-1}" above left, "\iso" below right, shorten >=1cm,shorten <=1cm, shift left]
 &&\D_2(\J)\arrow[ddrr,bend left,equal,""{name=R,below left}]\arrow[Leftarrow,ddll,"\kappa^*" above left, shorten >=1cm,shorten <=1cm, shift left]  \arrow[dd,"v^*"]\\
 
\\
 
&&\D_1(\B)\arrow[rr,"F" below] && 
 \D_2(\B)\arrow[rr,"u^*" below] &&\D_2(\A)\arrow[Rightarrow,to=R,"\epsilon" above left,shorten >=0.5cm,shorten <=0.7cm] \arrow[rr,"v_!" below]
 && \D_2(\J) 
 
\arrow[Rightarrow,from=B,to=A,"\gamma^{-1}" above left, "\iso" below right, shorten >=1.2cm,shorten <=1.2cm] 
\end{tikzcd}
\end{center}

\begin{center}
\begin{tikzcd}
\D_1(\B)\arrow[ddrr,bend right,equal,""{name=L,above right}]\arrow[rr,"w_!" above] && \D_1(\K)\arrow[Leftarrow,to=L,"\eta" above left,shorten >=0.5cm,shorten <=0.7cm] \arrow[rr,"z^*"]\arrow[dd,"w^*" left] && 
 \D_1(\J)\arrow[dd,"v^*" left]\arrow[rr,"F" above]\arrow[Leftarrow,ddll,"\kappa^*" above left, shorten >=1cm,shorten <=1cm, shift left]
 &&\D_2(\J)\arrow[Leftarrow,ddll,"\gamma^{-1}" above left, "\iso" below right, shorten >=1cm,shorten <=1cm, shift left]\arrow[ddrr,bend left,equal,""{name=R,below left}]  \arrow[dd,"v^*"]\\
 
\\
 
&&\D_1(\B)\arrow[rrdd,"F" below left,bend right=25,""{name=B,above right,near end}]\arrow[rr,"u^*" below] && 
 \D_1(\A)\arrow[rr,"F" below,""{name=A,below left}] &&\D_2(\A)\arrow[Rightarrow,to=R,"\epsilon" above left,shorten >=0.5cm,shorten <=0.7cm] \arrow[rr,"v_!" below]
 && \D_2(\J) \\\\
 &&&& \D_2(\B)\arrow[rruu,"u^*" below right,bend right=25]
 
 \arrow[Rightarrow,from=B,to=A,"\gamma^{-1}" above left, "\iso" below=2.5, shorten >=1.4cm,shorten <=1cm,shift right] 
\end{tikzcd}
\end{center}
That these two diagrams are equal follows easily from the coherence conditions for the derivator map $F$, as in~\cref{Pseudonatural transformation definition}.
\end{proof}

\begin{definition}
An \textbf{adjunction} between derivators is an adjunction in the $2$-category $\Der$, in the sense of~\cite[Section 2]{KS74}.
\end{definition}

We will make regular use of the following lemma, which characterises derivator left adjoints:

\begin{lemma}\label{Cocontinuous pointwise left adjoint is a left adjoint}
A derivator map $F:\D_1\rightarrow\D_2$ is a left adjoint if and only if it is cocontinuous and each component functor $F:\D_1(\A)\rightarrow\D_2(\A)$ has a right adjoint $G:\D_2(\A)\rightarrow\D_1(\A)$. Moreover, a derivator map is an equivalence if and only if it is a pointwise equivalence.
\end{lemma}
\begin{proof}
A proof of this lemma can be found in~\cite[Section 2]{Groth13}. We give a brief outline. 
Suppose $F:\D_1\rightarrow\D_2$ is a derivator map such that each component functor $F:\D_1(\A)\rightarrow\D_2(\A)$ has a right adjoint $G:\D_2(\A)\rightarrow\D_1(\A)$. Given $u:\A\rightarrow\B$, consider the following pasting diagram:
\begin{center}
\begin{tikzcd}
 && \D_1(\A)\arrow[Rightarrow,ddrr,"\gamma^{-1}" above right, "\iso" below left, shorten >=1cm,shorten <=1cm, shift right] \arrow[leftarrow,rr,"u^*"]\arrow[rightarrow,dd,"F" left] && 
 \D_1(\B)\arrow[rr,leftarrow,"G" above]  \arrow[rightarrow,dd,"F"]
&& \D_2(\B)\arrow[ddll,bend left,equal,""{name=R,above left}] 
\\
\\
\D_1(\A)\arrow[uurr,bend left,equal,""{name=L,below right}]\arrow[rr,leftarrow,"G" below]&&\D_2(\A)\arrow[Leftarrow,to=L,"\eta" above right,shorten >=0.5cm,shorten <=0.7cm]\arrow[leftarrow,rr,"u^*" below] && 
 \D_2(\B)
 
\arrow[Rightarrow,from=1-5,to=R,"\epsilon" above right,shorten >=0.5cm,shorten <=0.7cm] 
  
\end{tikzcd}
\end{center}

If these maps are isomorphisms for every $u:\A\rightarrow\B$, then these form the structure isomorphisms for a derivator map $G:\D_2\rightarrow\D_1$, which is then right adjoint to $F$.

This transformation is conjugate to the canonical map in \cref{Cocontinuous derivator map}. Thus, this is an isomorphism if and only if $F$ is cocontinuous. Moreover, it is clearly an isomorphism when $F$ is a pointwise equivalence, since in that case it is the pasting of three isomorphisms.
\end{proof}

\begin{remark}\label{Left adjoint is an equivalence iff underlying map is}
A derivator left adjoint $F:\D_1\rightarrow\D_2$ is an equivalence if and only if the underlying functor $F:\D_1(\0)\rightarrow\D_2(\0)$ is an equivalence. This follows by applying \cref{modifications are isomorphisms iff they are on underlying category} to the unit and counit of the adjunction.
\end{remark}


\begin{example}
Let $\M$ and $\N$ be model categories, and $F:\M\rightarrow\N$ be a left Quillen functor, with right adjoint $G:\N\rightarrow\M$. Then the derived functors induce an adjunction:
\begin{center}
\begin{tikzcd}
\dHo(\M)\arrow[r, bend left=40, "\mathds{L} F" above,""{name=U, below}]\arrow[leftarrow,r, bend right=40, "\mathds{R} G" below, ""{name=D}]
& \dHo(\N)
\arrow[phantom,from=U,to=D,shorten >=0.1cm,shorten <=0.1cm,"\bot" above, near end]
\end{tikzcd}
\end{center} 
By \cref{Left adjoint is an equivalence iff underlying map is}, a left Quillen functor $F:\M\rightarrow\N$ induces an equivalence on derivators if and only if it is a Quillen equivalence. 
\end{example}


We record one more fact, whose proof is in~\cite[Section 2]{Groth13}:

\begin{lemma}\label{u^* is cocontinuous}
Let $\D$ be a derivator, and let $u:\A\rightarrow\B$ be a functor. Then the derivator map $u^*:\D^\B\rightarrow\D^\A$ is continuous and cocontinuous.
\end{lemma}

Combining \cref{Cocontinuous pointwise left adjoint is a left adjoint} and \cref{u^* is cocontinuous}, for any functor $u:\A\rightarrow\B$ and any derivator $\D$, the map $u^*:\D^\B\rightarrow\D^\A$ admits both a left adjoint $u_!:\D^\A\rightarrow\D^\B$ and a right adjoint $u_*:\D^\A\rightarrow\D^\B$ in $\Der$.

\begin{remark}\label{Cocontinuity as a modification}
Given a derivator $\D$ and a functor $u:\A\rightarrow\B$, we can lift the canonical transformations of \cref{Cocontinuous derivator map} to a modification:
\begin{center}
\begin{tikzcd}
\D_1^\A\arrow[ddrr,bend right,equal,""{name=L,above right}]\arrow[rr,rightarrow,"u_!" above] && \D_1^\B\arrow[Leftarrow,to=L,"\eta" above left,shorten >=0.5cm,shorten <=0.7cm] \arrow[rightarrow,rr,"F"]\arrow[rightarrow,dd,"u^*" left] && 
 \D_2^\B\arrow[ddrr,bend left,equal,""{name=R,below left}] \arrow[Leftarrow,ddll,"\gamma^{-1}" above left, "\iso" below right, shorten >=1cm,shorten <=1cm, shift left] \arrow[rightarrow,dd,"u^*"]\\
 
\\
 
&&\D_1^\A\arrow[rightarrow,rr,"F" below] && 
 \D_2^\A\arrow[Rightarrow,to=R,"\epsilon" above left,shorten >=0.5cm,shorten <=0.7cm] \arrow[rr,rightarrow,"u_!" below]
 && \D_2^\B 
\end{tikzcd}
\end{center}
The map $F$ is cocontinuous if and only if each of these modifications is an isomorphism. 
\end{remark}

\begin{definition}
Given prederivators $\D_1$ and $\D_2$, we can form a prederivator $\dHom(\D_1,\D_2)$ as follows:
\begin{align*}
\dHom(\D_1,\D_2):\Cat\op&\rightarrow \; \CAT\\
 \A \;\;\; &\mapsto \;\Hom(\D_1,\D_2^\A)\\
 (u:\A\rightarrow\B) &\mapsto \;(u^*\circ -:\Hom(\D_1,\D_2^\B)\rightarrow\;\Hom(\D_1,\D_2^\A))
\end{align*}

Similarly, if $\D_1$ and $\D_2$ are derivators, we have a prederivator:
\begin{align*}
\cdHom(\D_1,\D_2):\Cat\op&\rightarrow \; \CAT\\
 \A \;\;\; &\mapsto \;\cHom(\D_1,\D_2^\A)\\
 (u:\A\rightarrow\B) &\mapsto \;(u^*\circ -:\cHom(\D_1,\D_2^\B)\rightarrow\;\cHom(\D_1,\D_2^\A))
\end{align*}
This is well-defined by \cref{u^* is cocontinuous}.
\end{definition}

\begin{remark}
If $\D_2$ is a derivator, then so is $\dHom(\D_1,\D_2)$. Given a functor $u:\A\rightarrow\B$, the Kan extensions along $u$ in $\dHom(\D_1,\D_2)$ are given by postcomposition with $u_!:\D^\A\rightarrow\D^\B$ and $u_*:\D^\A\rightarrow\D^\B$. On the other hand, $\cdHom(\D_1,\D_2)$ is a left derivator, but not a derivator in general. See~\cite[Section 5]{Cisinski08}.
\end{remark}

We end this section with an important theorem, describing the universal property of the derivator of spaces, which appears in~\cite{Cisinski08}.

\begin{theorem}[Cisinski]\label{Universal property of spaces}
For any left derivator $\D$, the map
\[
\cHom(\dHo(\sSet),\D)\rightarrow\D(\0),
\]
given by evaluation at the point $\Delta^0\in\Ho(\sSet)$, is an equivalence.
\end{theorem}

\section{Two-variable maps}\label{Section Two-variable maps}

In this section, we review some aspects of two-variable derivator maps, primarily following~\cite{GPS14}. These have a more complicated theory than single-variable maps, resulting in part from their external variants, which we discuss at the beginning of this section, and in part from the cancelling variants, which we will discuss in~\cref{Chapter Actions of monoidal derivators}.

Given derivators $\D_1$ and $\D_2$, the product $\D_1\times\D_2$ is  formed pointwise: for any $\A\in\Cat$, we have
\[
(\D_1\times\D_2)(\A)=\D_1(\A)\times\D_2(\A).
\]
Suppose we have a derivator map $\otimes:\D_1\times\D_2\rightarrow\D_3$. For each category $\A$, this map has a component functor of the form
\[
\otimes:\D_1(\A)\times\D_2(\A)\rightarrow\D_3(\A).
\]
From this, we can construct an \textbf{external} version of $\otimes$ as follows:
\begin{align*}
\widetilde{\otimes}:\D_1(\A)\times\D_2(\B)&\rightarrow \; \D_3(\A\times\B)\\
 (X,Y) \;\;\;\;\;\;\; &\mapsto \; \p^*_{\B}X\otimes\p^*_{\A}Y
\end{align*}
For a fixed object $X\in\D_1(\A)$ or $Y\in\D_2(\B)$, the external product induces derivator maps: 
\[
X\;\widetilde{\otimes}\;- :\D_2\rightarrow\D_3^{\A}
\]
\[
-\;\widetilde{\otimes}\;Y:\D_1\rightarrow\D_3^{\B}
\]
Note that our notation for the external two-variable map differs from that in~\cite{GPS14}.

\begin{remark}\label{Lifting the external product to a derivator map}
The external product lifts to a derivator map:
\[
\widetilde{\otimes}:\D_1^\A\times\D_2^\B\xrightarrow{\;\;\p_\B^*\times\p_\A^*\;\;}\D_1^{\A\times\B}\times\D_2^{\A\times\B}\xrightarrow{\;\;\otimes\;\;}\D_3^{\A\times\B}
\]

Given functors $u:\A\rightarrow\C$ and $v:\B\rightarrow\cD$, the structure isomorphism of $\otimes$ induces a natural isomorphism:
\begin{center}
\begin{tikzcd}
 \D_1^{\C}\times\D_2^{\cD} \arrow[rr,"\widetilde{\otimes}"]\arrow[dd,"u^*\times v^*" left] && 
 \D_3^{\C\times\cD}\arrow[Rightarrow,ddll,"\gamma" above left,"\iso" below right,shorten >=1cm,shorten <=1cm, shift left] \arrow[dd,"(u\times v)^*"] \\
 &\\
\D_1^{\A}\times\D_2^{\B}\arrow[rr,"\widetilde{\otimes}" below] && 
 \D_3^{\A\times\B}
\end{tikzcd}
\end{center}
\end{remark}

\begin{definition}\label{Cocontinuous two-variable map definition}
A two-variable map $\otimes$ is called \textbf{cocontinuous} if, for every $X\in\D_1(\A)$ and $Y\in\D_2(\B)$, the maps
\[
X\;\widetilde{\otimes}\;- :\D_2\rightarrow\D_3^{\A}
\]
\[
-\;\widetilde{\otimes}\;Y:\D_1\rightarrow\D_3^{\B}
\]
are cocontinuous. Let $\cHom(\D_1,\D_2;\D_3)$ denote the full subcategory of $\Hom(\D_1\times\D_2,\D_3)$ on the cocontinuous maps.
\end{definition}

\begin{remark}\label{Opposite of external product}
Given a two-variable map $\otimes:\D_1\times\D_2\rightarrow\D_3$, its opposite is a map
\[
\otimes\op:\D_1\op\times\D_2\op\rightarrow\D_3\op.
\]
For any $X\in\D_1(\A)$ we have an isomorphism
\[
X\;\widetilde{\otimes\op}\;-\iso (X\;\widetilde{\otimes}\;-)\op:\D_2\op\rightarrow(\D_3\op)^{\A\op}.
\]
So $\otimes$ is cocontinuous if and only if $\otimes\op$ is \textbf{continuous} in the obvious sense.
\end{remark}

\begin{remark}\label{Cocontinuity of tensor as a modification}
As in \cref{Cocontinuity as a modification}, we may lift the isomorphisms associated to a cocontinuous two-variable map $\otimes:\D_1\times\D_2\rightarrow\D_3$ to modifications. Given any functor $u:\A\rightarrow\B$, and any category $\J$, we have a canonical map
\begin{center}
\begin{tikzcd}
\D_1^\J\times\D_2^\A\arrow[ddrr,bend right,equal,""{name=L,above right}]\arrow[rr,rightarrow,"\D_1^\J\times u_!" above] && \D_1^\J\times\D_2^\B\arrow[Leftarrow,to=L,"\eta" above left,shorten >=0.7cm,shorten <=0.8cm] \arrow[rightarrow,rrr,"\widetilde{\otimes}"]\arrow[rightarrow,dd,"\D_1^\J\times u^*" right] &&&
 \D_3^{\J\times\B}\arrow[ddrr,bend left,equal,""{name=R,below left}] \arrow[Leftarrow,ddlll,"\gamma^{-1}" above left, "\iso" below right, shorten >=1.2cm,shorten <=1.2cm, shift left] \arrow[rightarrow,dd,"(\J\times u)^*" left]\\
 
\\
 
&&\D_1^\J\times\D_2^\A\arrow[rightarrow,rrr,"\widetilde{\otimes}" below] &&& 
 \D_3^{\J\times\A}\arrow[Rightarrow,to=R,"\epsilon" above left,shorten >=0.5cm,shorten <=0.7cm] \arrow[rr,rightarrow,"(\J\times u)_!" below]
 && \D_3^{\J\times\B}
\end{tikzcd}
\end{center}
and a similar map where the Kan extension is in the first variable. The two-variable map $\otimes$ is cocontinuous if and only if each of these modifications is an isomorphism.
\end{remark}

\begin{lemma}\label{Cartesian closure of PDer}
There is an equivalence of categories 
\[
\Hom(\D_1\times\D_2,\D_3)\simeq\Hom(\D_1,\dHom(\D_2,\D_3))
\]
which restricts to an equivalence:
\[
\cHom(\D_1,\D_2;\D_3)\simeq\cHom(\D_1,\cdHom(\D_2,\D_3))
\]
\end{lemma}

\begin{proof}
See~\cite[Section 5] {Cisinski08} for a complete proof. For convenience, we record the maps in both directions.

Given a map $\otimes\in\Hom(\D_1\times\D_2,\D_3)$, the corresponding map in $\Hom(\D_1,\dHom(\D_2,\D_3))$ has components:
\begin{center}
\begin{tikzcd}[row sep=small]
\D_1(\A)\arrow[r] & \Hom(\D_2,\D_3^\A)\\
X \arrow[mapsto,r,shorten >=0.7cm,shorten <=0.5cm] & X\;\widetilde{\otimes}\;-
\end{tikzcd}
\end{center}
For the inverse, suppose we have a map $\varphi\in\Hom(\D_1,\dHom(\D_2,\D_3))$. For any $X\in\D_1(\A)$, $\varphi$ gives us a map $\varphi(X):\D_2\rightarrow\D_3^\A$. The morphism corresponding to $\varphi$ in $\Hom(\D_1\times\D_2,\D_3)$ has components given by
\begin{center}
\begin{tikzcd}[row sep=small]
\D_1(\A)\times\D_2(\A)\arrow[r] & \D_3(\A\times\A)\arrow[r,"\delta^*"] &\D_3(\A)\\
(X,Y) \arrow[mapsto,r,shorten >=0.3cm,shorten <=0.8cm] 
& \varphi(X)(Y) \arrow[mapsto,r,shorten >=0.2cm,shorten <=0.3cm] 
& \delta^*(\varphi(X)(Y))
\end{tikzcd}
\end{center}
where $\delta:\A\rightarrow\A\times\A$ is the diagonal map.
\end{proof}

\cref{Cartesian closure of PDer} carries over to the following statement about two-variable maps. For a proof, see~\cite[Theorem 3.11]{GPS14}.

\begin{remark}\label{Bimorphisms}
Let $\D_1$, $\D_2$ and $\D_3$ be prederivators. Consider the $2$-functor below:
\begin{center}
\begin{tikzcd}[row sep=small]
\D_1\;\overline{\times}\;\D_2:\Cat\op\times\Cat\op\arrow[r] & \CAT\\
\;\;\;\;\;\;\;\;\;\;\;\;\;\;\;\;\;\;(\A,\B) \arrow[mapsto,r,shorten >=0.3cm,shorten <=0.8cm] 
& \D_1(\A)\times\D_2(\B)
\end{tikzcd}
\end{center}
We may also form the following composite, where the first map takes a pair of categories $\A$ and $\B$ to their product $\A\times\B$:
\begin{center}
\begin{tikzcd}[row sep=small]
\D_3\circ\times:\Cat\op\times\Cat\op\arrow[r,"\times"] & \Cat\op\arrow[r,"\D_3"] & \CAT
\end{tikzcd}
\end{center}
The maps of \cref{Cartesian closure of PDer} induce an equivalence of categories
\[
\Hom(\D_1\times\D_2,\D_3)\simeq \text{Psnat}(\D_1\;\overline{\times}\;\D_2,\D_3\circ\times),
\]
where the second category has objects given by the pseudonatural transformations from $\D_1\;\overline{\times}\;\D_2$ to $\D_3\circ\times$, and maps given by modifications. 

Explicitly, a prederivator map $\otimes:\D_1\times\D_2\rightarrow\D_3$ corresponds to a pseudonatural transformation with component at $(\A,\B)$ given as follows:
\begin{center}
\begin{tikzcd}[row sep=tiny]
\D_1(\A)\times\D_2(\B)\arrow[r] & \D_3(\A\times\B)\\
(X,Y) \arrow[mapsto,r,shorten >=0.5cm,shorten <=0.8cm] 
& X\;\widetilde{\otimes}\;Y
\end{tikzcd}
\end{center}

\end{remark}

For any derivator $\D$, we can apply \cref{Universal property of spaces} and \cref{Cartesian closure of PDer} to the left derivator $\cdHom(\D,\D)$, to get an equivalence:
\[
\cHom(\D,\D)\equiv\cHom(\dHo(\sSet),\cdHom(\D,\D))\equiv\cHom(\dHo(\sSet),\D;\D)
\]
Under this equivalence, $\id\in\cHom(\D,\D)$ gives us a canonical cocontinuous map:
\[
\otimes:\dHo(\sSet)\times\D\rightarrow\D
\]
This is the essentially unique cocontinuous map such that
\[
\Delta^0\;\widetilde{\otimes}\;-\iso\id:\D\rightarrow\D.
\] 
For the derivator $\dHo(\M)$ associated to a model category $\M$, this map agrees with the familiar action of $\sSet$ on $\M$, constructed, for example, in~\cite[Chapter 5]{Hovey07}. We will discuss actions of derivators in detail in~\cref{Chapter Actions of monoidal derivators}.

\section{Pointed and triangulated derivators}\label{Section Pointed and stable derivators}

In this section we include a brief overview of pointed, stable and triangulated derivators, most of which can be found in~\cite[Section 3,4]{Groth13}. These are the derivator analogues of pointed and stable model categories, and a number of the concepts that are important in that setting can also be developed in derivators, for example a theory of homotopy fibres and cofibres. Note that, in contrast to~\cite{Groth13}, we do not assume stable derivators are strong. Thus, the definition of stable derivator in~\cite{Groth13} is what we call a triangulated derivator in~\cref{triangulated derivator definition}. To study triangulated derivators, we will use a number of results from the theory of triangulated categories, which we recall at the end of this section. Using these, we will be able to prove representability theorems for enriched triangulated derivators in \cref{Chapter Actions of monoidal derivators} and \cref{Chapter Enriched derivators}.

\begin{definition}
We say a derivator $\D$ is \textbf{pointed} if the underlying category has a zero object $0\in\D(\0)$. 
\end{definition}

If $\D$ is a pointed derivator, note that $\p_\A^*0\in\D(\A)$ is a zero object, for any category $\A$.

\begin{example}
Let $\M$ be a model category. Then the derivator $\dHo(\M)$ is pointed if and only if the unique map from the initial object in $\M$ to the terminal object is a weak equivalence. In particular, if $\M$ is a pointed model category then $\dHo(\M)$ is pointed.
\end{example}

We have an analogue of \cref{Universal property of spaces} for pointed derivators:

\begin{theorem}[Cisinski]\label{Universal property of pointed spaces}
For any pointed derivator $\D$, the map
\[
\cHom(\dHo(\sSet_*),\D)\rightarrow\D(\0)
\]
given by evaluation at $\mathrm{S}^0\in\Ho(\sSet_*)$ is an equivalence.
\end{theorem}

For any pointed derivator $\D$, this universal property induces a canonical cocontinuous map
\[
\wedge:\dHo(\sSet_*)\times\D\rightarrow\D.
\]
This is the essentially unique cocontinuous map such that
\[
\mathrm{S}^0\;\widetilde{\wedge}\;-\iso\id:\D\rightarrow\D.
\] 
See~\cite{Cisinski08} for a complete proof of this fact.

\begin{definition}
For any pointed derivator $\D$, the action of $\dHo(\sSet_*)$ on $\D$ induces the \textbf{suspension} map:
\[
\Sigma:=\mathrm{S}^1\;\widetilde{\wedge}\;-:\D\rightarrow\D
\]
We call a pointed derivator $\D$ \textbf{stable} if this map is an equivalence. 
\end{definition}

Given a pointed derivator, it is possible to define suspension in elementary terms, without appealing to the action of $\dHo(\sSet_*)$. See~\cite{Groth13} for this approach. See~\cite{GS17} for a number of equivalent characterisations of stability.

\begin{example}
Given a pointed model category $\M$, the derivator $\dHo(\M)$ is stable if and only if $\M$ is a stable model category.
\end{example}

\begin{definition}\label{triangulated derivator definition}
Let $\D$ be a stable derivator. If in addition $\D$ is strong, we call $\D$ a \textbf{triangulated} derivator.
\end{definition}

The following theorem is the motivation for the term triangulated derivator. \textbf{Der 5} is essential in its proof.  

\begin{theorem}[Groth, Maltsiniotis]\label{Triangulated derivators induce triangulated categories}
Let $\D$ be a triangulated derivator. Then, for any category $\A$, the category $\D(\A)$ has a canonical triangulated structure, and for any $u:\A\rightarrow\B$, the functor $u^*:\D(\B)\rightarrow\D(\A)$ is canonically exact. 

Moreover, suppose $\D_1$ and $\D_2$ are triangulated derivators, and let $F:\D_1\rightarrow\D_2$ be a derivator map that preserves initial objects and homotopy pushouts. Then, for each category $\A$, the functor $F:\D_1(\A)\rightarrow\D_2(\A)$ has a canonical exact structure. The same is true if $F$ preserves terminal objects and homotopy pullbacks.
\end{theorem}

See~\cite{Groth13} for a proof of this theorem, and for the definition of \textbf{homotopy pushouts}, which are called \textbf{cocartesian squares} in~\cite{Groth13}. We will not recall the definition here: for our purposes, it suffices to note that all cocontinuous derivator maps preserve initial objects and homotopy pushouts.

We now recall some concepts from the theory of triangulated categories, which we will use to study triangulated derivators. See~\cite{Krause08,Neeman14} for more details. There is also a useful survey in the first section of~\cite{Stovicek09}.

\begin{definition}
Let $\mathlarger{\mathscr{T}}$ be a triangulated category that admits all coproducts. An object $x\in\mathlarger{\mathscr{T}}$ is called \textbf{compact} if the functor it represents
\[
\mathscr{T}(x,-):\mathscr{T}\rightarrow\Ab
\]
preserves coproducts.

A triangulated category $\mathlarger{\mathscr{T}}$ is called \textbf{compactly generated} if it admits all coproducts and there is a set $\mathcal{C}$ of objects in $\mathlarger{\mathscr{T}}$ with the following properties:
\begin{enumerate}
\item Every object $x\in\mathcal{C}$ is compact.
\item If $y\in\mathlarger{\mathscr{T}}$ has the property that $\mathscr{T}(x,y)=0$ for every $x\in\mathcal{C}$, then $y=0$. 
\end{enumerate}
We call $\mathcal{C}$ a \textbf{set of compact generators for $\mathlarger{\mathscr{T}}$}.
\end{definition}

Compactly generated triangulated categories are historically important and well-studied. However, the following generalisation, introduced in~\cite{Krause02}, retains several important properties of compactly generated triangulated categories, and admits many more examples:

\begin{definition}
A triangulated category $\mathlarger{\mathscr{T}}$ is called \textbf{perfectly generated} if it admits all coproducts and there is a set $\mathcal{P}$ of objects in $\mathlarger{\mathscr{T}}$ with the following properties:
\begin{enumerate}
\item Let $f_i:y_i\rightarrow z_i$ be a family of maps in $\mathlarger{\mathscr{T}}$. Suppose that, for every $x\in\mathcal{P}$ and every $f_i$, the map
\[
\mathscr{T}(x,f_i):\mathscr{T}(x,y_i)\rightarrow\mathscr{T}(x,z_i)
\]
is surjective. Then, for every object $x\in\mathcal{P}$, the map
\[
\mathscr{T}(x,\coprod f_i):\mathscr{T}(x,\coprod y_i)\rightarrow\mathscr{T}(x,\coprod z_i)
\]
is surjective.
\item If $y\in\mathlarger{\mathscr{T}}$ has the property that $\mathscr{T}(x,y)=0$ for every $x\in\mathcal{P}$, then $y=0$. 
\end{enumerate}
We call $\mathcal{P}$ a \textbf{set of perfect generators for $\mathlarger{\mathscr{T}}$}.
\end{definition}

We will now record some important properties of perfectly generated triangulated categories. First, we recall two more definitions:

\begin{definition}
Let $\mathlarger{\mathscr{T}}$ be a triangulated category. A functor $F:\mathlarger{\mathscr{T}}\op\rightarrow\Ab$ is called a \textbf{cohomological functor} if it takes any distinguished triangle in $\mathlarger{\mathscr{T}}$ to an exact sequence in $\Ab$. 
\end{definition}

\begin{definition}
Let $\mathlarger{\mathscr{T}}$ be a triangulated category that admits all coproducts. We say that $\mathlarger{\mathscr{T}}$ \textbf{satisfies Brown representability} if any cohomological functor $F:\mathlarger{\mathscr{T}}\op\rightarrow\Ab$ that preserves products (that is, takes coproducts in $\mathlarger{\mathscr{T}}$ to products in $\Ab$) is representable.
\end{definition}

Versions of the following theorem have been proven in a number of settings. The version we give below is proved in~\cite{Krause02}:

\begin{theorem}\label{Perfectly generated triangulated categories satisfy representability}
Perfectly generated triangulated categories satisfy Brown representability.
\end{theorem}

The following is a useful property of triangulated categories that satisfy Brown representability. A simple proof can be found in~\cite[Chapter 8]{Neeman14}.

\begin{lemma}\label{Left adjoints from perfectly generated triangulated categories} 
Let $\mathlarger{\mathscr{T}}$ and $\mathlarger{\mathscr{S}}$ be triangulated categories, and suppose $\mathlarger{\mathscr{T}}$ satisfies Brown representability. An exact functor $F:\mathlarger{\mathscr{T}}\rightarrow\mathlarger{\mathscr{S}}$ has a right adjoint if and only if it preserves coproducts.
\end{lemma}

We will now apply these theorems to triangulated derivators. A proof of the following lemma can be found in~\cite[Section 3]{Hormann17}: 

\begin{lemma}\label{Perfectly generated triangulated derivator}
Let $\D$ be a triangulated derivator. Suppose $\mathcal{G}$ is a set of compact (resp. perfect) generators for $\D(\0)$. Then, for any category $\A$, the set
\[
\mathcal{G}_\A = \{a_!x \;|\; a\in\A, x\in\mathcal{G}\}
\]
is a compact (resp. perfect) generating set for $\D(\A)$.
\end{lemma}

By~\cref{Perfectly generated triangulated categories satisfy representability} and~\cref{Perfectly generated triangulated derivator}, the following lemma applies to any triangulated derivator $\D_1$ whose underlying category is perfectly generated:

\begin{proposition}\label{Adjoint functor theorem for triangulated derivators}
Let $\D_1$ and $\D_2$ be triangulated derivators, and suppose $\D_1(\A)$ satisfies Brown representability, for any category $\A$. Then a derivator map $F:\D_1\rightarrow\D_2$ has a right adjoint if and only if it is cocontinuous.
\end{proposition}
\begin{proof}
Let $\A$ be a category. If $F:\D_1\rightarrow\D_2$ is cocontinuous then the component functor $F:\D_1(\A)\rightarrow\D_2(\A)$ preserves coproducts, and, by~\cref{Triangulated derivators induce triangulated categories}, is exact. Thus, the result follows by~\cref{Cocontinuous pointwise left adjoint is a left adjoint} and~\cref{Left adjoints from perfectly generated triangulated categories}.
\end{proof}

We conclude this section with an analogue of \cref{Universal property of spaces} for stable derivators:

\begin{theorem}[Heller]\label{Universal property of spectra}
For any stable derivator $\D$, the map
\[
\cHom(\dHo(\Spt),\D)\rightarrow\D(\0)
\]
given by evaluation at the sphere spectrum $\mathds{S}\in\Ho(\Spt)$ is an equivalence.
\end{theorem}

For derivators satisfying a stronger analogue of \textbf{Der 5}, this theorem is essentially proved in~\cite{Heller97}; see~\cite[Section 9]{Tabuada08}, however, for a proof that the derivator constructed in~\cite{Heller97} is equivalent to $\dHo(\Spt)$. See~\cite{Coley18} for a similar proof that does not use \textbf{Der 5}. 

Using this universal property, for any stable derivator $\D$ we obtain a canonical cocontinuous map
\[
\wedge:\dHo(\Spt)\times\D\rightarrow\D,
\]
essentially unique such that
\[
\mathds{S}\;\widetilde{\wedge}\;-\iso\id:\D\rightarrow\D.
\] 

For the derivator $\dHo(\M)$ associated to a stable model category $\M$, this map recovers the action constructed in~\cite{SS02}.

\chapter{Actions of Monoidal Derivators}\label{Chapter Actions of monoidal derivators}

In this chapter we discuss the theory of modules over monoidal derivators. Closed modules play a particularly important role in~\cref{Chapter E-categories} and~\cref{Chapter Enriched derivators}, and in this chapter we give a number of examples. Moreover, we show that, given a closed module $\D$ over a symmetric monoidal derivator, the shifted derivators $\D^\A$ and the opposite derivator $\D\op$ are also closed modules. In this way, we build a collection of examples, which, as we will show in~\cref{Chapter Enriched derivators}, induce important examples of enriched derivators. In addition, we prove a number of results in this chapter that will contribute to our development of the theory in~\cref{Chapter E-categories} and~\cref{Chapter Enriched derivators}. 

In order to study monoidal derivators and their actions, we recall the theory of ends and coends in~\cref{Section Ends and coends}, and use this to discuss derivator two-variable adjunctions in~\cref{Section Two-variable adjunctions}. In~\cref{Section Monoidal derivators} and~\cref{Section Closed actions of monoidal derivators}, we recall the definition of monoidal derivators and their actions. Up to this point, much of this material can be found in~\cite{GPS14}, although our presentation differs in a number of ways.~\cref{Section The cancelling tensor product} and~\cref{Section The maps otimes u and h u} contain an in-depth discussion of the structure arising from the action of a monoidal derivator. The results we prove in these sections form the groundwork for a number of the results in~\cref{Chapter E-categories} and~\cref{Chapter Enriched derivators}. We also use the results of these sections in~\cref{Section Cotensors in closed E-modules}, which, given a closed $\E$-module $\D$, discusses the action on the opposite derivator $\D\op$. Finally, in \cref{Section A representability theorem for triangulated E-modules}, we prove a representability theorem for triangulated closed modules over triangulated monoidal derivators.
 
\section{Ends and coends}\label{Section Ends and coends}

In this section, we recall from~\cite[Section 5]{GPS14} the definition of (homotopy) ends and coends in a derivator, and prove a number of simple results that we will need in the later sections of~\cref{Chapter Actions of monoidal derivators} and in~\cref{Chapter E-categories}. The definition of ends and coends that we give here goes via the twisted arrow category. See~\cite[Appendix A]{GPS14} for a discussion of equivalent definitions.

Given a category $\A$, recall that the \textbf{twisted arrow category} $\tw(\A)$ is the category of elements of the hom-functor $\hom:\A\op\times\A\rightarrow\Set$. Explicitly, objects of $\tw(\A)$ are arrows $f:a\rightarrow b$ in $\A$, and an arrow from $f:a\rightarrow b$ to $g:c\rightarrow d$ is a commutative square:
\begin{center}
\begin{tikzcd}
 a \arrow[dd,"f" left] && 
 c \arrow[ll,"h" above] \arrow[dd,"g" right] \\\\
 b \arrow[rr,"k" below] && 
 d
\end{tikzcd}
\end{center}
We have a canonical map $\text{(s,t)}:\tw(\A)\rightarrow\A\op\times\A$. Taking the opposite of this map, we get:
\[
(\mathrm{t}\op,\;\mathrm{s}\op):\tw(\A)\op\xrightarrow{\;\;\text{(s,t)}\op\;\;}\A\times\A\op\iso\A\op\times\A
\]
 
\begin{definition}\label{Ends and coends definition}
For any derivator $\D$ and any category $\A$, the \textbf{(homotopy) coend} over $\A$ is the composite
\[
\int^\A:\D^{\A\op\times\A}\xrightarrow{\;\;(\mathrm{t}\op,\;\mathrm{s}\op)^*\;\;} \D^{\tw(\A)\op}\xrightarrow{\;\;\p_!\;\;} \D.
\]
We will denote the right adjoint of this map by
\[
\partial^\A:\D\xrightarrow{\;\;\p^*\;\;} \D^{\tw(\A)\op}\xrightarrow{\;\;(\mathrm{t}\op,\;\mathrm{s}\op)_*\;\;} \D^{\A\op\times\A}.
\]
Dually, the \textbf{(homotopy) end} over $\A$ is the composite
\[
\int_\A:\D^{\A\op\times\A}\xrightarrow{\;\;\text{(s,t)}^*\;\;} \D^{\tw(\A)}\xrightarrow{\;\;\p_*\;\;} \D
\]
and its left adjoint will be denoted
\[
\partial_\A:\D\xrightarrow{\;\;\p^*\;\;} \D^{\tw(\A)}\xrightarrow{\;\;\text{(s,t)}_!\;\;} \D^{\A\op\times\A}.
\]
\end{definition}

If $\D=y(\mathcal{C})$ is a represented derivator, these constructions recover the usual end and coend.

\begin{definition}\label{Maps that preserve delta}
Let $F:\D_1\rightarrow\D_2$ be a derivator map, and let $\A$ be a category. We say $F$ \textbf{preserves $\partial_\A$} if the canonical pasting 
\begin{center}
\begin{tikzcd}
 \D_1 \arrow[rr,"\p^*"]\arrow[dd,"F" left] && 
 \D_1^{\tw(\A)} \arrow[Leftarrow,ddll,shorten >=1cm,shorten <=1cm, "\gamma" above left, "\iso" below right] \arrow[dd,"F"]
\arrow[rr,"\text{(s,t)}_!"] && 
 \D_1^{\A\op\times\A} \arrow[Leftarrow,ddll,shorten >=1cm,shorten <=1cm] \arrow[dd,"F"]\\\\
\D_2\arrow[rr,"\p^*" below] &&
\D_2^{\tw(\A)}\arrow[rr,"\text{(s,t)}_!" below] && 
\D_2^{\A\op\times\A}
\end{tikzcd}
\end{center}
is an isomorphism, where the second square is the modification of~\cref{Cocontinuity as a modification}. Note that this is the case if and only if $F$ preserves the left Kan extension along $\text{(s,t)}$. Similarly, we can define functors that preserve ends, coends and $\partial^\A$.
\end{definition}

\begin{remark}\label{Modifications respect coends}
Let $F,G:\D_1\rightarrow\D_2$ be derivator maps, and let $\theta:F\Rightarrow G$ be a modification. Using~\cref{Modifications respect left Kan extensions}, it follows immediately that $\theta$ respects the constructions of~\cref{Ends and coends definition}. For example, given any category $\A$ and $X\in\D_1(\A\op\times\A)$, the diagram below commutes, where the vertical arrows are the canonical maps, as in~\cref{Maps that preserve delta}:
\begin{center}
\begin{tikzcd}
 \int^\A FX \arrow[rr,"\int^\A\theta_X"]\arrow[dd] && 
 \int^\A GX \arrow[dd] \\\\
F\int^\A X\arrow[rr,"\theta_{\int^\A X}" below] && 
 G\int^\A X
\end{tikzcd}
\end{center}
\end{remark}

\begin{definition}\label{tw(u) definition}
Suppose we have a functor $u:\A\rightarrow\B$. This induces a functor on twisted arrow categories $\tw(u):\tw(\A)\rightarrow\tw(\B)$ that makes the diagram below commute:
\begin{center}
\begin{tikzcd}
 \tw(\A) \arrow{rr}[above]{\tw(u)}\arrow{dd}[left]{(\mathrm{s},\mathrm{t})} && 
 \tw(\B)  \arrow{dd}[right]{(\mathrm{s},\mathrm{t})} \\\\
 \A\op\times\A \arrow{rr}[below]{u\op\times u} && 
 \B\op\times\B
\end{tikzcd}
\end{center}
\end{definition}

Using this commutative diagram, each of the constructions of \cref{Ends and coends definition} can be extended to act on the functor $u$. Since we use them repeatedly, we will describe $\partial_u$ and $\int^u$ explicitly:

\begin{definition}\label{Delta and end on morphisms}
For any functor $u:\A\rightarrow\B$, and any derivator $\D$, we have the following modifications:
\begin{center}
\begin{tikzcd}
 &[-25pt] &[-25pt] && 
 \D^{\tw(\B)} \arrow[dd,"\tw(u)^*"]
\arrow[rr,"\text{(s,t)}_!"] && 
 \D^{\B\op\times\B} \arrow[Leftarrow,ddll,shorten >=1.2cm,shorten <=1.2cm] \arrow[dd,"(u\op\times u)^*"]\\
 \partial_u &:= & \D\arrow[urr,"\p^*" above left,bend left=20]\arrow[drr,"\p^*" below left,bend right=20] \\
&&&&
\D^{\tw(\A)}\arrow[rr,"\text{(s,t)}_!" below] && 
\D^{\A\op\times\A}
\end{tikzcd}
\end{center}
\begin{center}
\begin{tikzcd}
 &[-25pt] & \D^{\B\op\times\B}\arrow[rr,"(\text{t}\op\text{, s}\op)^*"]\arrow[dd,"(u\op\times u)^*" left] && 
 \D^{\tw(\B)\op} \arrow[dd,"(\tw(u)\op)^*" left] \arrow[drr,"\p_!"above right,bend left=20,""{name=U}] \\
 
 \int^u &:= &&&&& \D \\
 
&& \D^{\A\times\A\op} \arrow[rr,"(\text{t}\op\text{, s}\op)^*" below] && 
 \D^{\tw(\A)\op} \arrow[urr,"\p_!" below right,bend right=20]
 
 \arrow[Leftarrow,from=U,to=3-5, shorten >=1cm,shorten <=1cm,shift left=2]
\end{tikzcd}
\end{center} 
The non-identity modifications in these pastings are obtained as in \cref{D-exact squares} and \cref{Canonical map to colimit}. The transformations $\partial^u$ and $\int_u$ are dual; that is, they can be obtained from these by moving to the opposite derivator $\D\op$.

Given any category $\C$ and $X\in\D(\C)$, rather than use a subscript as for most modifications, we will denote the component of $\partial_u$ at $X$ by $\partial_u X$. We will denote the others similarly.
\end{definition}

Note that $\int_u$ is dual to $\int^u$, and is a mate of $\partial_u$ under the adjunctions $\partial_\A\dashv\int_\A$ and $\partial_\B\dashv\int_\B$. Using this, the following lemmas have analogues for each of the constructions of \cref{Ends and coends definition}. Rather than record each version explicitly, we state each in the case that we will use most frequently in subsequent sections.

\begin{lemma}\label{Functoriality of delta}
Let $\D$ be a derivator, and let $\D\Downarrow\Der$ denote the category whose objects are arrows $F:\D\rightarrow\D'$ in $\Der$ and whose morphisms from $F$ to $F'$ are given by modifications:
\begin{center}
\begin{tikzcd}
 && 
 \D' \arrow[dd,"G" right]
\\
\D\arrow[urr,"F" above left,bend left=20]\arrow[drr,"F'" below left,bend right=20,""{name=U} above right] \\
&&
\D''
\arrow[Rightarrow,from=U,to=1-3,shorten >=0.7cm,shorten <=0.5cm]
\end{tikzcd}
\end{center}
Then we have a functor $\partial_-:\Cat\op\rightarrow\D\Downarrow\Der$, taking categories $\A$ to $\partial_\A$ and functors $u$ to $\partial_u$.
\end{lemma}
\begin{proof}
It is easy to check that forming the twisted arrow category preserves identities and composition. The result then follows by the functoriality of mates, as in~\cite[Section 2]{KS74}.
\end{proof}

\begin{lemma}\label{Interaction of delta with natural transformations}
Suppose we have a natural transformation:
\begin{center}
\begin{tikzcd}
\A\arrow[rr, bend left=45, "u" above,""{name=U, below}]\arrow[rr, bend right=45, "v" below, ""{name=D}]
&& \B
\arrow[Rightarrow,from=U,to=D,shorten >=0.1cm,shorten <=0.1cm,"\kappa", shift right]
\end{tikzcd}
\end{center}
For any derivator $\D$, any category $\C$, and any $X\in\D(\C)$ we have a commutative square:
\begin{center}
\begin{tikzcd}
\partial_\A X \arrow{rrr}[above]{\partial_u X}\arrow{dd}[left]{\partial_v X} &&& 
 (u\op\times u)^*\partial_\B X \arrow{dd}[right]{(u\op\times \kappa)^*_{\partial_\B X}} \\\\
 (v\op\times v)^*\partial_\B X \arrow{rrr}[below]{(\kappa\op\times v)^*_{\partial_\B X} } &&& 
 (u\op\times v)^*\partial_\B X 
\end{tikzcd}
\end{center}
\end{lemma}
\begin{proof}
Given the natural transformation $\kappa$, we can form a functor $\tw(u,v):\tw(\A)\rightarrow\tw(\B)$ given by:
\begin{center}
\begin{tikzcd}[row sep=small,column sep=small]
\tw(\A)\arrow[rr] && \tw(\B)\\
a \arrow[dd,"f" left]
&& u(a)\arrow[d,"u(f)" right] \\
\;\arrow[mapsto,rr,shorten >=0.4cm,shorten <=0.6cm] 
&& u(b)\arrow[d,"\kappa_b" right] \\
b && v(b)
\end{tikzcd}
\end{center}
This makes the diagram below commute:
\begin{center}
\begin{tikzcd}
 \tw(\A) \arrow{rr}[above]{\tw(u,v)}\arrow{dd}[left]{(\mathrm{s},\mathrm{t})} && 
 \tw(\B)  \arrow{dd}[right]{(\mathrm{s},\mathrm{t})} \\\\
 \A\op\times\A \arrow{rr}[below]{u\op\times v} && 
 \B\op\times\B
\end{tikzcd}
\end{center}
Moreover, we have natural transformations $\tw(u,\kappa):\tw(u)\Rightarrow\tw(u,v)$ and $\tw(\kappa,v):\tw(v)\Rightarrow\tw(u,v)$, which satisfy the following equalities:
\begin{center}
\begin{tikzcd}
 \tw(\A) \arrow[rr,"\tw(u)",bend left=30,""{name=R,below}]\arrow[dd,"\text{(s,t)}" left]\arrow[rr,"\tw(u\text{,}v)" below,bend right=30,""{name=L,above}] && 
 \tw(\B)\arrow[dd,"\text{(s,t)}"] 
 &[-10pt] &[-10pt] \tw(\A) \arrow[rr,"\tw(u)"]\arrow[dd,"\text{(s,t)}" left] 
 && \tw(\B) \arrow[dd,"\text{(s,t)}" right]
 \\
&&& = \\
\A\op\times\A\arrow[rr,"u\op\times v" below] 
&& \B\op\times\B
&& \A\op\times\A\arrow[rr,"u\op\times u" above,bend left=30,""{name=D,below}] \arrow[rr,"u\op\times v" below,bend right=30,""{name=U,above}]
&& \B\op\times\B

\arrow[Rightarrow,from=R,to=L,shorten >=0.1cm,shorten <=0.1cm,"\tw(u\text{,}\kappa)" right,shift right]
\arrow[Rightarrow,from=D,to=U,shorten >=0.1cm,shorten <=0.1cm,"u\op\times\kappa" right,shift right]
\end{tikzcd}
\end{center}
\begin{center}
\begin{tikzcd}
 \tw(\A) \arrow[rr,"\tw(v)",bend left=30,""{name=R,below}]\arrow[dd,"\text{(s,t)}" left]\arrow[rr,"\tw(u\text{,}v)" below,bend right=30,""{name=L,above}] && 
 \tw(\B)\arrow[dd,"\text{(s,t)}"] 
 &[-10pt] &[-10pt] \tw(\A) \arrow[rr,"\tw(v)"]\arrow[dd,"\text{(s,t)}" left] 
 && \tw(\B) \arrow[dd,"\text{(s,t)}" right]
 \\
&&& = \\
\A\op\times\A\arrow[rr,"u\op\times v" below] 
&& \B\op\times\B
&& \A\op\times\A\arrow[rr,"v\op\times v" above,bend left=30,""{name=D,below}] \arrow[rr,"u\op\times v" below,bend right=30,""{name=U,above}]
&& \B\op\times\B

\arrow[Rightarrow,from=R,to=L,shorten >=0.1cm,shorten <=0.1cm,"\tw(\kappa\text{,}v)" right,shift right]
\arrow[Rightarrow,from=D,to=U,shorten >=0.1cm,shorten <=0.1cm,"\kappa\op\times v" right,shift right]
\end{tikzcd}
\end{center}
These squares give rise to modifications, as in \cref{D-exact squares}. Pasting these to the modification
\begin{center}
\begin{tikzcd}
&& \D^{\tw(\B)} \arrow[dd,"\tw(u\text{,} v)^*" right,bend left=55,""{name=R,left}]\arrow[dd,"\tw(u)^*" left,bend right=55,""{name=L,right}] 
 &&[-10pt] & [-20pt]&& \D^{\tw(\B)}\arrow[dd,bend right=55,"\tw(v)^*" left,""{name=U,right}]\arrow[dd,"\tw(u\text{,} v)^*" right,bend left=55,""{name=D,left}]
 \\
\D\arrow[rru,bend left=25,"\p^*" above left]\arrow[rrd,bend right=25,"\p^*" below left,""{name=K,above right}] 
&&&& = & \D\arrow[rru,bend left=25,"\p^*" above left]\arrow[rrd,bend right=25,"\p^*" below left,""{name=J,above right}] \\
&&\D^{\tw(\A)}
&&&& &\D^{\tw(\A)}
\arrow[Leftarrow,from=R,to=L,shorten >=0.1cm,shorten <=0.1cm,"\tw(u\text{,}\kappa)^*" above,shift left]
\arrow[Leftarrow,from=D,to=U,shorten >=0.1cm,shorten <=0.1cm,"\tw(\kappa\text{,} v)^*" above,shift left]

\end{tikzcd}
\end{center}
gives the following equality:
\begin{center}
\begin{tikzcd}
&&& \D^{\B\op\times\B} \arrow[dd,"\mathsmaller{(u\op\times v)^*}" right,bend left=55,""{name=R,left}]\arrow[dd,"\mathsmaller{(u\op\times u)^*}" left,bend right=55,""{name=L,right}] 
 && [-10pt] & [-20pt]&&& \D^{\B\op\times\B}\arrow[dd,bend right=55,"\mathsmaller{(v\op\times v)^*}" left,""{name=U,right}]\arrow[dd,"\mathsmaller{(u\op\times v)^*}" right,bend left=55,""{name=D,left}]
 \\
\D\arrow[rrru,bend left=25,"\partial_\B" above left]\arrow[rrrd,bend right=25,"\partial_\A" below left,""{name=K,above right}] 
&&&&& = &\D\arrow[rrru,bend left=25,"\partial_\B" above left]\arrow[rrrd,bend right=25,"\partial_\A" below left,""{name=J,above right}] \\
&&&\D^{\A\op\times\A}
&&&&& &\D^{\A\op\times\A}
\arrow[Leftarrow,from=R,to=L,shorten >=0.1cm,shorten <=0.1cm,"\mathsmaller{(u\op\times\kappa)^*}" above,shift left]
\arrow[Leftarrow,from=D,to=U,shorten >=0.1cm,shorten <=0.1cm,"\mathsmaller{(\kappa\op\times v)^*}" above,shift left]
\arrow[Rightarrow,from=K,to=1-4,shorten >=1.2cm,shorten <=0.3cm,"\partial_u" above left,near start,shift left=10]
\arrow[Rightarrow,from=J,to=1-10,shorten >=1.2cm,shorten <=0.3cm,"\partial_v" above left,near start,shift left=10]
\end{tikzcd}
\end{center}
This is precisely the equality we require.
\end{proof}

\begin{lemma}\label{Coends over A and Aop}
For any category $\A$ and any derivator $\D$, there is a canonical isomorphism:
\begin{center}
\begin{tikzcd}
 \D^{\A\op\times\A}\arrow[dd,"\iso" right,"\sigma^*" left,""{name=L}] \arrow[rrd,"\int^\A" above right,""{name=R},bend left] & \\
 
 && \D\\
 
 \D^{\A\times\A\op} \arrow[Rightarrow,to=R,"\iso", shorten >=0.95cm,shorten <=0.95cm,shift right=2]\arrow[rur,"\int^{\A\op}" below right,bend right]
  
\end{tikzcd}
\end{center} 
\end{lemma}
\begin{proof}
We have an isomorphism $\tw(\A\op)\xrightarrow{\iso}\tw(\A)$ that makes the diagram below commute:
\begin{center}
\begin{tikzcd}
 \tw(\A\op) \arrow{rr}[above]{\iso}\arrow{dd}[left]{\text{(s,t)}} && 
 \tw(\A)  \arrow{dd}[right]{\text{(s,t)}} \\\\
 \A\times\A\op \arrow{rr}[below]{\sigma}[above]{\iso} && 
 \A\op\times\A
\end{tikzcd}
\end{center}
This induces the required isomorphism:
\begin{center}
\begin{tikzcd}
 \D^{\A\op\times\A}\arrow[rr,"\text{(t}\op\text{, s}\op)^*"]\arrow[dd,"\iso" right,"\sigma^*" left] && 
 \D^{\tw(\A)\op} \arrow[dd,"\iso" left] \arrow[drr,"\p_!"above right,bend left=20,""{name=U}] \\
 
 &&&& \D \\
 
 \D^{\A\times\A\op} \arrow[rr,"\text{(t}\op\text{, s}\op)^*" below] && 
 \D^{\tw(\A\op)\op} \arrow[urr,"\p_!" below right,bend right=20]
 
 \arrow[Rightarrow,from=3-3,to=U,"\iso" above left, shorten >=0.8cm,shorten <=0.8cm,shift right=2]
\end{tikzcd}
\end{center} 
\end{proof}

\begin{lemma}[Fubini Theorem for Derivators]\label{Fubini theorem}
For any categories $\A$ and $\B$, and any derivator $\D$, there are canonical  isomorphisms: 
\begin{center}
\begin{tikzcd}
 \D^{\A\op\times\A\times\B\op\times\B} \arrow[rrdd,"\int^\A\int^\B" above right,bend left,""{name=U}]\arrow[dd,"\iso" left] \\
 \\
 \D^{(\A\times\B)\op\times (\A\times\B)} \arrow[rr,"\int^{\A\times\B}"] && \D \\
 \\
 \D^{\B\op\times\B\times\A\op\times\A}\arrow[uu,"\iso" left]\arrow[rruu,"\int^\B\int^\A" below right,bend right,""{name=D}]
 \arrow[Rightarrow,from=3-1,to=U,"\iso", shorten >=0.8cm,shorten <=0.8cm,shift left]
 \arrow[Rightarrow,from=3-1,to=D,"\iso" below left, shorten >=0.8cm,shorten <=0.8cm,shift right]
\end{tikzcd}
\end{center} 
\end{lemma}
\begin{proof}
We have an isomorphism $\tw(\A\times\B)\xrightarrow{\iso}\tw(\A)\times\tw(\B)$, which makes the diagram below commute:
\begin{center}
\begin{tikzcd}
 \tw(\A\times\B) \arrow{rr}[above]{\iso}\arrow{dd}[left]{(\mathrm{s},\mathrm{t})} && 
 \tw(\A)\times\tw(\B)  \arrow{dd}[right]{(\mathrm{s},\mathrm{t})\times (\mathrm{s},\mathrm{t})} \\\\
 (\A\times\B)\op\times (\A\times\B) \arrow{rr}[below]{\iso} && 
 \A\op\times\A\times\B\op\times\B
\end{tikzcd}
\end{center}
As in the proof of \cref{Coends over A and Aop}, this induces the desired isomorphisms.
\end{proof}

\begin{notation}\label{A=A_1=A_2 notation}
In light of \cref{Coends over A and Aop} and \cref{Fubini theorem}, we will suppress instances of the symmetry isomorphism $\sigma^*:\D(\A\times\B)\xrightarrow{\iso}\D(\B\times\A)$ from our notation as much as possible. This will be convenient when there are multiple parameters indexing a shifted derivator: for example, we may write $\int^\A:\D^{\A\times\B\times\A\op}\rightarrow\D^\B$ without ambiguity.

On the other hand, if multiple copies of the same category appear in the index of a shifted derivator, we will often use different subscripts to denote the same category. For example, if we denote $\A=\A_1=\A_2=\A_3$ then the notation $\int^{\A_{1,3}}:\D^{\A_1\op\times\A_2\times\A_3}\rightarrow\D^{\A_2}$ is unambiguous. 
\end{notation}

\begin{remark}\label{Prederivator maps respect Fubini isomorphism}
Using the pasting properties of mates, it is easy to show that any derivator map $F:\D_1\rightarrow\D_2$ respects the isomorphism in~\cref{Fubini theorem}. That is, given any object $X\in\D_1(\A\op\times\A\times\B\op\times\B)$, the diagram below commutes:
\begin{center}
\begin{tikzcd}
 \int^\A\int^\B FX \arrow{rr}[above]{\iso}\arrow{dd} && 
 \int^{\A\times\B} FX  \arrow{dddd} \\\\
 \int^\A F\int^\B X \arrow{dd} \\\\
 F\int^\A\int^\B X \arrow{rr}[below]{\iso}
 && F\int^{\A\times\B}X
\end{tikzcd}
\end{center}
The vertical maps are the canonical morphisms, as in~\cref{Maps that preserve delta}.
\end{remark}

The following lemma is technical, but the isomorphism it provides is extremely important in the theory of monoidal derivators. In particular, it induces the unit isomorphisms in the bicategory of~\cref{Definition of Prof(E)} associated to any monoidal derivator.

\begin{proposition}\label{Interaction of delta and integral}
For any derivator $\D$, and any category $\A=\A_i$, we have an isomorphism:

\begin{center}
\begin{tikzcd}[row sep=tiny]
 && 
\D^{\A_1\times\A_2\op\times\A_3} \arrow[rrd,"\int^{\A_{1,2}}" above right, bend left=10]
\\
\D^{\A_1}\arrow[rrrr,equal,bend right=20,""{name=D}]\arrow[urr,"\partial_{\A_{2,3}}" above left,bend left=10]&&&&\D^{\A_3} 
\arrow[Rightarrow,from=1-3,to=D,"\iso",shorten >=0.2cm,shorten <=0.2cm]
\end{tikzcd}
\end{center}
\end{proposition}
\begin{proof}
We must show that the composite
\[
\D^{\A}\xrightarrow{(\A\times\p)^*}\D^{\A\times\tw(\A)}\xrightarrow{(\A\times(\mathrm{s},\;\mathrm{t}))_!}\D^{\A\times\A\op\times\A}\xrightarrow{((\mathrm{t}\op,\;\mathrm{s}\op)\times\A)^*}\D^{\tw(\A\op)\op\times\A}\xrightarrow{(\p\times\A)_!}\D^\A
\]
is isomorphic to the identity. 

This isomorphism is constructed in the proof of~\cite[Lemma B.1]{GPS14}. We will give an outline of the proof. 

Define the category $\mathrm{P}_\A$ to be the following pullback. In~\cite[Lemma B.1]{GPS14}, this square is shown to be homotopy exact:

\begin{center}
\begin{tikzcd}
 \mathrm{P}_\A\arrow[rrdd,phantom,"\lrcorner" very near start] \arrow{rr}[above]{\mathrm{m}_\A}\arrow{dd}[left]{\mathrm{n}_\A} && 
 \A\times\tw(\A)\arrow[Rightarrow,ddll,"\id" above=2, shorten >=1.3cm,shorten <=1.3cm]\arrow{dd}[right]{\A\times (\mathrm{s},\mathrm{t})} \\\\
 \tw(\A\op)\op\times\A \arrow{rr}[below]{(\mathrm{t}\op,\mathrm{s}\op)\times\A} && 
 \A\times\A\op\times\A

\end{tikzcd}
\end{center}
The category $\mathrm{P}_\A$ has objects given by pairs of composable maps $a\xrightarrow{\;\;g\;\;} b\xrightarrow{\;\;f\;\;} c$ in $\A$, with a morphism from $a_1\xrightarrow{\;\;g_1\;\;} b_1\xrightarrow{\;\;f_1\;\;} c_1$ to $a_2\xrightarrow{\;\;g_2\;\;} b_2\xrightarrow{\;\;f_2\;\;} c_2$ consisting of a commutative diagram:
\begin{center}
\begin{tikzcd}
 a_1\arrow{r}[above]{\alpha}\arrow{d}[left]{g_1} & 
 a_2 \arrow{d}[right]{g_2} \\
 b_1 \arrow{d}[left]{f_1} & 
 b_2 \arrow{d}[right]{f_2}\arrow{l}[below]{\beta}\\
 c_1 \arrow{r}[below]{\gamma} & c_2
\end{tikzcd}
\end{center}
The functor $\mathrm{m}_\A$ takes $a\xrightarrow{\;\;g\;\;} b\xrightarrow{\;\;f\;\;} c$ to the pair $(a,b\xrightarrow{\;\;f\;\;} c)$, and $\mathrm{n}_\A$ takes $a\xrightarrow{\;\;g\;\;} b\xrightarrow{\;\;f\;\;} c$ to $(a\xrightarrow{\;\;g\;\;} b,c)$. There is a natural transformation
\begin{center}
\begin{tikzcd}
 \mathrm{P}_\A\arrow{r}[above]{\mathrm{m}_\A}\arrow{d}[left]{\mathrm{n}_\A} & 
 \A\times\tw(\A)\arrow{r}[above]{\A\times \p} & \A\arrow[Rightarrow,ddll,"\vartheta_\A\;\;\;" above=2, shorten >=2.5cm,shorten <=2.5cm]\arrow[dd,equal] \\
 \tw(\A\op)\op\times\A\arrow{d}[left]{\p\times\A}
 \\
 \A\arrow[rr,equal] && 
 \A

\end{tikzcd}
\end{center}
with component at the object $a\xrightarrow{\;\;g\;\;} b\xrightarrow{\;\;f\;\;} c$ given by $a\xrightarrow{\;\;f\circ g\;\;} c$. In~\cite[Lemma B.1]{GPS14}, this square is also shown to be homotopy exact, so we obtain a canonical isomorphism:
\begin{center}
\begin{tikzcd}
\D^{\A}\arrow[rr,"(\A\times\p)^*"]\arrow[rrrrrrrrr,equal,bend right=50,""{name=D}] 
&& \D^{\A\times\tw(\A)}\arrow[rrdd,"\mathrm{m}^*" below left,bend right=15]\arrow[rr,"(\A\times(\mathrm{s}\text{,}\;\mathrm{t}))_!"]
&& \D^{\A\times\A\op\times\A}\arrow[Leftarrow,dd,"\iso",shorten >=0.4cm,shorten <=0.3cm]\arrow[rrr,"((\mathrm{t}\op\text{,}\;\mathrm{s}\op)\times\A)^*"]
&&& \D^{\tw(\A\op)\op\times\A}\arrow[rr, "(\p\times\A)_!"]
&& \D^{\A} \\\\

&&&& \D^{\mathrm{P}_\A}\arrow[Rightarrow,dd,"\iso",shorten >=0.4cm,shorten <=0.3cm]\arrow[rrruu,"\mathrm{n}_!" below right,bend right=10] \\\\
&&&& 

\end{tikzcd}
\end{center}
\end{proof}

\begin{remark}\label{Cancelling As and Fubini}
Given any category $\A=\A_i$, the pasting diagrams below are equal, for any derivator $\D$:

\begin{center}
\begin{tikzcd}[row sep=tiny]
 && 
\D^{\A_1\op\times\A_2\times\A_3\op\times\A_4} \arrow[rrddd,"\int^{\A_{1,4}}" above right, bend left=10]
\\\\\\
\D^{\A_1\op\times\A_2}\arrow[rrrr,equal,bend right=20,""{name=D}]\arrow[uuurr,"\partial_{\A_{3,4}}" above left,bend left=10]&&&&\D^{\A_3\op\times\A_2} \arrow[rr,"\int^{\A_{3,2}}" above] && \D 

\arrow[Rightarrow,from=1-3,to=D,"\iso",shorten >=0.4cm,shorten <=0.4cm]
\end{tikzcd}
\end{center}

\begin{center}
\begin{tikzcd}[row sep=tiny]
 &&&& \D^{\A_3\op\times\A_2}\arrow[rrdddd,"\int^{\A_{3,2}}" above right, bend left=10]\arrow[Rightarrow, dddd, "\iso" right,shorten >=0.3cm,shorten <=0.3cm]  \\
 && 
\D^{\A_1\op\times\A_2\times\A_3\op\times\A_4} \arrow[rrddd,"\int^{\A_{3,2}}" below left, bend left=10]\arrow[rru,"\int^{\A_{1,4}}" above left, bend left=10]
\\\\\\
\D^{\A_1\op\times\A_2}\arrow[rrrr,equal,bend right=20,""{name=D}]\arrow[uuurr,"\partial_{\A_{3,4}}" above left,bend left=10]&&&&\D^{\A_1\op\times\A_4} \arrow[rr,"\int^{\A_{1,4}}" below] && \D 

\arrow[Rightarrow,from=2-3,to=D,"\iso",shorten >=0.4cm,shorten <=0.4cm]
\end{tikzcd}
\end{center}

This is shown in the proof of~\cite[Lemma B.5]{GPS14}. Dually, the pasting diagrams below are also equal:
\begin{center}
\begin{tikzcd}[row sep=tiny]
 &&&& 
\D^{\A_1\op\times\A_2\times\A_3\op\times\A_4} \arrow[rrddd,"\int^{\A_{1,4}}" above right, bend left=10]
\\\\\\
\D\arrow[rr,"\partial_{\A_{1,2}}" above]  &&
\D^{\A_1\op\times\A_2}\arrow[rrrr,equal,bend right=20,""{name=D}]\arrow[uuurr,"\partial_{\A_{3,4}}" above left,bend left=10]&&&&\D^{\A_3\op\times\A_2}

\arrow[Rightarrow,from=1-5,to=D,"\iso",shorten >=0.4cm,shorten <=0.4cm]
\end{tikzcd}
\end{center}

\begin{center}
\begin{tikzcd}[row sep=tiny]
 && \D^{\A_1\op\times\A_2}\arrow[rrd,"\partial_{\A_{3,4}}" above right, bend left=15]\arrow[Rightarrow, dddd, "\iso" right,shorten >=0.3cm,shorten <=0.3cm]
\\ &&&& 
\D^{\A_1\op\times\A_2\times\A_3\op\times\A_4} \arrow[rrddd,"\int^{\A_{1,4}}" above right, bend left=10]
\\\\\\
\D\arrow[uuuurr,"\partial_{\A_{1,2}}" above left,bend left=15]\arrow[rr,"\partial_{\A_{3,4}}" below]  &&
\D^{\A_3\op\times\A_4}\arrow[rrrr,equal,bend right=20,""{name=D}]\arrow[uuurr,"\partial_{\A_{1,2}}" below right,bend left=10]&&&&\D^{\A_3\op\times\A_2}

\arrow[Rightarrow,from=2-5,to=D,"\iso",shorten >=0.4cm,shorten <=0.4cm]
\end{tikzcd}
\end{center}
\end{remark}

\begin{remark}\label{Derivator maps respect delta coend cancellation}
Let $F:\D_1\rightarrow\D_2$ be a derivator map. Using~\cref{Derivator maps respect homotopy exact squares}, it follows that $F$ respects the isomorphism of~\cref{Interaction of delta and integral}. That is, for any category $\A=\A_i$, and any $X\in\D_1(\A_1)$, the diagram below commutes, where the horizontal maps are the canonical morphisms of~\cref{Maps that preserve delta}:
\begin{center}
\begin{tikzcd}
 \int^{\A_{1,2}} \partial_{\A_{2,3}} FX\arrow[r]\arrow[rd,"\iso" below left,bend right=18] & \int^{\A_{1,2}} F\partial_{\A_{2,3}} X\arrow[r] & F\int^{\A_{1,2}} \partial_{\A_{2,3}} X\arrow[ld,"\iso" below right,bend left=18]  \\
 & FX
\end{tikzcd}
\end{center}
\end{remark}

\section{Two-variable adjunctions}\label{Section Two-variable adjunctions}

In this section, we recall the definition and some basic properties of two-variable adjunctions between derivators, which were introduced and studied in~\cite{GPS14}. Our presentation differs slightly from~\cite{GPS14}. In particular, we begin with~\cref{2-variable adjunction characterisation}, which collects all of the structure present in a derivator two-variable adjunction, and then point out that this structure is uniquely determined by far less. This inverts the approach of~\cite{GPS14}, but has the advantage of fixing notation from the outset.

Before we give the derivator analogue, we recall the definition of a two-variable left adjoint between categories. Given categories $\mathcal{C}_1$, $\mathcal{C}_2$ and $\mathcal{C}_3$, a functor $\otimes:\mathcal{C}_1\times\mathcal{C}_2\rightarrow\mathcal{C}_3$ is called a \textbf{two-variable left adjoint} if, for each $x\in\mathcal{C}_1$ and each $y\in\mathcal{C}_2$, the functors $x\otimes -:\mathcal{C}_2\rightarrow\mathcal{C}_3$ and $-\otimes\;y:\mathcal{C}_1\rightarrow\mathcal{C}_3$ have right adjoints. Equivalently, there are functors
\begin{align*}
\vartriangleright\;:\; & \mathcal{C}_2\op\times\mathcal{C}_3\rightarrow\mathcal{C}_1\\
\vartriangleleft\;:\; & \mathcal{C}_3\times\mathcal{C}_1\op\rightarrow\mathcal{C}_2
\end{align*}
and natural isomorphisms:
\[
\mathcal{C}_1(x,y\vartriangleright z)\iso\mathcal{C}_3(x\otimes y,z)\iso\mathcal{C}_2(y,z\vartriangleleft x)
\]

We now give the derivator analogue. Note that the following is not the definition given in~\cite{GPS14}, but is equivalent to it by~\cite[Lemma 8.8]{GPS14}.

\begin{definition}
We call a derivator map $\otimes:\D_1\times\D_2\rightarrow\D_3$ a \textbf{two-variable left adjoint} if, for any categories $\A$ and $\B$, and any $X\in\D_1(\A)$ and $Y\in\D_2(\B)$, the derivator maps 
\[
X\;\widetilde{\otimes}\;- :\D_2\rightarrow\D_3^{\A}
\]
\[
-\;\widetilde{\otimes}\;Y:\D_1\rightarrow\D_3^{\B}
\]
have right adjoints.
\end{definition}

The following theorem highlights the importance of~\cref{Interaction of delta and integral}. We will use this theorem repeatedly, especially in subsequent sections of~\cref{Chapter Actions of monoidal derivators}.

\begin{theorem}\label{2-variable adjunction characterisation}
Suppose we have a cocontinuous two-variable map
\[
\otimes:\D_1\times\D_2\rightarrow\D_3
\]
and two continuous derivator maps
\begin{align*}
\vartriangleright:\; & \D_2\op\times\D_3\rightarrow\D_1\\
\vartriangleleft:\; & \D_3\times\D_1\op\rightarrow\D_2.
\end{align*}
The following are equivalent:
\begin{enumerate}
\item For any $X\in\D_1(\A)$, the map $X\;\widetilde{\otimes}\;- :\D_2\rightarrow\D_3^{\A}$ has a right adjoint given by
\[
\D_3^{\A}\xrightarrow{\;\;-\;\widetilde{\vartriangleleft}\;X\;\;}\D_2^{\A\times\A\op}\xrightarrow{\;\;\int_{\A\op}\;\;}\D_2
\]
and for any $Y\in\D_2(\B)$, the map $-\;\widetilde{\otimes}\;Y:\D_1\rightarrow\D_3^{\B}$ has a right adjoint given by
\[
\D_3^{\B}\xrightarrow{\;\;Y\;\widetilde{\vartriangleright}\;-\;\;}\D_1^{\B\op\times\B}\xrightarrow{\;\;\int_{\B}\;\;}\D_1.
\]
\item For any $X\in\D_1(\A)$, the map $-\;\widetilde{\vartriangleleft}\;X:\D_3\rightarrow\D_2^{\A\op}$ has a left adjoint given by
\[
\D_2^{\A\op}\xrightarrow{\;\;X\;\widetilde{\otimes}\;-\;\;}\D_3^{\A\times\A\op}\xrightarrow{\;\;\int^{\A\op}\;\;}\D_3
\]
and for any $Z\in\D_3(\C)$, the map $Z\;\widetilde{\vartriangleleft\op}\;-:\D_1\rightarrow(\D_2\op)^{\C\op}$ has a right adjoint given by
\[
(\D_2\op)^{\C\op}\xrightarrow{\;\;-\;\widetilde{\vartriangleright}\;Z\;\;}\D_1^{\C\op\times\C}\xrightarrow{\;\;\int_\C\;\;}\D_1.
\]
\item For any $Y\in\D_2(\B)$, the map $Y\;\widetilde{\vartriangleright}\;-:\D_3\rightarrow\D_1^{\B\op}$ has a left adjoint given by
\[
\D_1^{\B\op}\xrightarrow{\;\;-\;\widetilde{\otimes}\;Y\;\;}\D_2^{\B\op\times\B}\xrightarrow{\;\;\int^{\B}\;\;}\D_3
\]
and for any $Z\in\D_3(\C)$, the map $-\;\widetilde{\vartriangleright\op}\;Z:\D_2\rightarrow(\D_1\op)^{\C\op}$ has a right adjoint given by
\[
(\D_1\op)^{\C\op}\xrightarrow{\;\;Z\;\widetilde{\vartriangleleft}\;-\;\;}\D_2^{\C\times\C\op}\xrightarrow{\;\;\int_{\C\op}\;\;}\D_2.
\]
\end{enumerate}
\end{theorem}
\begin{proof}
Given maps $\otimes$, $\vartriangleright$ and $\vartriangleleft$ as in the statement of the theorem, consider the maps below:
\begin{align*}
\vartriangleleft\op:\; & \D_3\op\times\D_1\rightarrow\D_2\op\\
\otimes\op:\; & \D_1\op\times\D_2\op\rightarrow\D_3\op\\
\vartriangleright:\; & \D_2\op\times\D_3\rightarrow\D_1
\end{align*}
By~\cref{Opposite of external product}, $\vartriangleleft\op$ is cocontinuous and $\otimes\op$ is continuous, so this triple satisfies the hypotheses of the theorem. Condition $(2)$ for the triple $(\otimes,\vartriangleright,\vartriangleleft)$ is exactly condition $(1)$ for the new triple $(\vartriangleleft\op,\otimes\op,\vartriangleright)$, $(3)$ for $(\otimes,\vartriangleright,\vartriangleleft)$ is $(2)$, and $(1)$ for $(\otimes,\vartriangleright,\vartriangleleft)$ is $(3)$. Thus, to prove the theorem, it suffices to prove that $(1)$ implies $(2)$.

So, suppose we have maps $(\otimes,\vartriangleright,\vartriangleleft)$ as in the statement of the theorem, that satisfy condition $(1)$. Given $X\in\D(\A)$, consider the map
\[
\D_2^{\A\op}\xrightarrow{\;\;X\;\widetilde{\otimes}\;-\;\;}\D_3^{\A\times\A\op}\xrightarrow{\;\;\int^{\A\op}\;\;}\D_3.
\]
By $(1)$, this has a right adjoint given by the composite
\[
\D_3\xrightarrow{\;\;\partial^{\A_{2,3}\op}\;\;}\D_3^{\A_2\times\A_3\op}\xrightarrow{\;\;-\;\widetilde{\vartriangleleft}\;X\;\;}\D_2^{\A_2\times\A_3\op\times\A_1\op}\xrightarrow{\;\;\int_{\A_{1,2}\op}\;\;}\D_2^{\A_3\op},
\]
where $\A=\A_1=\A_2=\A_3$ as in~\cref{A=A_1=A_2 notation}. Since $\vartriangleleft$ is continuous, $-\;\widetilde{\vartriangleleft}\;X$ commutes with $\partial^{\A\op}$, so the composite above is isomorphic to:
\[
\D_3\xrightarrow{\;\;-\;\widetilde{\vartriangleleft}\;X\;\;}\D_2^{\A_1\op}\xrightarrow{\;\;\partial^{\A_{2,3}\op}\;\;}\D_2^{\A_1\op\times\A_2\times\A_3\op}\xrightarrow{\;\;\int_{\A_{1,2}\op}\;\;}\D_2^{\A_3\op}
\]
By~\cref{Interaction of delta and integral}, this composite is isomorphic to $-\;\widetilde{\vartriangleleft}\;X$.

Let $Z\in\D_3(\C)$, and consider the map $Z\;\widetilde{\vartriangleleft\op}\;-:\D_1\rightarrow(\D_2\op)^{\C\op}$. The component of this map at $\A$ is the opposite of the component of $Z\;\widetilde{\vartriangleleft}\;-:\D_1\op\rightarrow\D_2^{\C}$ at $\A\op$:
\[
Z\;\widetilde{\vartriangleleft}\;-:\D_1(\A)\op\rightarrow\D_2(\C\times\A\op)
\]
Given $X\in\D_1(\A)$ and $Y\in\D_2(\C\times\A\op)$, we have a string of natural isomorphisms:
\begin{align*}
\D_2(\C\times\A\op)(Y,Z\;\widetilde{\vartriangleleft}\;X) & \iso\D_3(\C)(\mathsmaller{\int_{\A\op}}X\;\widetilde{\otimes}\;Y,Z)\\
\; & \iso\D_3(\C\times\A\times\A\op)(X\;\widetilde{\otimes}\;Y,\partial^{\A\op}Z)\\
\; & \iso\D_1(\A)(X,\mathsmaller{\int_{\C\times\A\op}}Y\;\widetilde{\vartriangleright}\;\partial^{\A\op}Z) 
\end{align*}
The first follows from the description of the left adjoint to $-\;\widetilde{\vartriangleleft}\;X$ which we have just proved; the second is by definition of $\partial^{\A\op}$; the third follows from the description of the right adjoint to $-\;\widetilde{\otimes}\;Y$.

Since $\vartriangleright$ is continuous, $Y\;\widetilde{\vartriangleright}\;-$ commutes with $\partial^{\A\op}$. Applying the Fubini theorem of~\cref{Fubini theorem} to the end $\int_{\C\times\A\op}$, we can use~\cref{Interaction of delta and integral} to cancel $\int_{\A\op}$ with $\partial^{\A\op}$. This leaves us with the desired description of the right adjoint to $Z\;\widetilde{\vartriangleleft\op}\;-$:
\[
(\D_2\op)^{\C\op}\xrightarrow{\;\;-\;\widetilde{\vartriangleright}\;Z\;\;}\D_1^{\C\op\times\C}\xrightarrow{\;\;\int_\C\;\;}\D_1 \qedhere
\]
\end{proof}

The theorem below is a consequence of the results in~\cite[Section 9]{GPS14}:

\begin{theorem}\label{Characterising two-variable left adjoints}
A map $\otimes:\D_1\times\D_2\rightarrow\D_3$ is a two-variable left adjoint if and only if we can construct derivator maps 
\begin{align*}
\vartriangleright: & \D_2\op\times\D_3\rightarrow\D_1\\
\vartriangleleft: & \D_3\times\D_1\op\rightarrow\D_2
\end{align*}
such that $\otimes$, $\vartriangleleft$ and $\vartriangleright$ satisfy the equivalent conditions of Theorem~\ref{2-variable adjunction characterisation}. 
\end{theorem}

In light of~\cref{Characterising two-variable left adjoints}, we will call a triple of derivator maps as in~\cref{2-variable adjunction characterisation} a \textbf{two-variable adjunction}. We denote them by $(\otimes,\vartriangleright,\vartriangleleft):\D_1\times\D_2\rightarrow\D_3$. 

It follows that~\cref{2-variable adjunction characterisation} gives a complete description of the various adjunctions that arise from a $2$-variable adjunction. We will use these descriptions repeatedly, especially in~\cref{Section Closed actions of monoidal derivators} and~\cref{Section Cotensors in closed E-modules}.

\begin{remark}
Given a two-variable adjunction $(\otimes,\vartriangleright,\vartriangleleft):\D_1\times\D_2\rightarrow\D_3$, the proof of~\cref{2-variable adjunction characterisation} implies that $(\vartriangleleft\op,\otimes\op,\vartriangleright):\D_3\op\times\D_1\rightarrow\D_2\op$ and $(\vartriangleright\op,\vartriangleleft,\otimes\op):\D_2\times\D_3\op\rightarrow\D_1\op$ are also two-variable adjunctions. These are the \textbf{cycled} two-variable adjunctions of~\cite[Section 9]{GPS14}.
\end{remark}

We finish this section with the definition of the cancelling variant of a two-variable derivator map:

\begin{definition}\label{Definition of cancelling product on morphisms}
Given a two-variable map $\otimes:\D_1\times\D_2\rightarrow\D_3$, categories $\A$ and $\B$, and objects $X\in\D_1(\A\op)$ and $Y\in\D_2(\B)$, we will use the following notation for the \textbf{cancelling tensor product} that appears in~\cref{2-variable adjunction characterisation}:
\[
X\otimes_{\A} -:\D_2^{\A}\xrightarrow{\;\;X\;\widetilde{\otimes}\;-\;\;}\D_3^{\A\op\times\A}\xrightarrow{\;\;\int^{\A}\;\;}\D_3
\]
\[
-\otimes_\B Y:\D_1^{\B\op}\xrightarrow{\;\;-\;\widetilde{\otimes}\;Y\;\;}\D_2^{\B\op\times\B}\xrightarrow{\;\;\int^{\B}\;\;}\D_3
\]
Together these maps induce a derivator map: 
\[
\otimes_\A:\D_1^{\A\op}\times\D_2^\A\rightarrow\D_3
\]
Suppose we have a functor $u:\A\rightarrow\B$. Pasting the isomorphism from~\cref{Lifting the external product to a derivator map} to the map $\int^u$ of~\cref{Delta and end on morphisms} gives the following canonical modification:

\begin{center}
\begin{tikzcd}
\D_1^{\B\op}\times\D_2^\B\arrow{dd}[left]{(u\op)^*\times u^*} \arrow[bend left=25,rrd,"\otimes_\B" above right,""{name=U}] & & & [-25pt] & [-15pt] \D_1^{\B\op}\times\D_2^\B \arrow{rr}[above]{\widetilde{\otimes}}\arrow{dd}[left]{(u\op)^*\times u^*} && 
 \D_3^{\B\op\times\B}\arrow[bend left=25,rrd,"\int^\B" above right,""{name=R}]  \arrow{dd}[left]{(u\op\times u)^*} \\
&& \D_3 & := &&&&& \D_3 \\
\D_1^{\A\op}\times\D_2^\A \arrow[bend right=25,rur,"\otimes_\A" below right]
 \arrow[Rightarrow,from=3-1,to=U,"\otimes_u" above left, shorten >=1.1cm,shorten <=1.1cm,shift right=2] &&&& \D_1^{\A\op}\times\D_2^\A \arrow{rr}[below]{\widetilde{\otimes}} &\;& 
 \D_3^{\A\op\times\A}\arrow[bend right=25,rur,"\int^\A" below right]
 
\arrow[Rightarrow,from=3-7,to=R,"\int^u" above left, shorten >=1cm,shorten <=1cm,shift right=2]

\arrow[Rightarrow,from=1-7,to=3-6,"\iso" above left,near end,shorten >=0.3cm,shorten <=1.8cm, shift right=11]
\end{tikzcd}
\end{center}
\end{definition}

\section{Monoidal derivators}\label{Section Monoidal derivators}

In this brief section we recall the definition of monoidal derivators, which can be found in~\cite{Groth11,GPS14}. The definition is a direct analogue of the familiar definition of monoidal category.

\begin{definition}\label{Definition of (symmetric) monoidal prederivator}
Let $\E$ be a prederivator. We say $\E$ is a \textbf{monoidal prederivator} if it is equipped with a product $\otimes:\E\times\E\rightarrow\E$, and a unit object $\mathbb{1}\in\E(\0)$, together with isomorphisms: 

\begin{center}
\begin{tikzcd}
\E\arrow[rrr, bend left=70, "-\widetilde{\otimes}\mathbb{1}" above,""{name=U, below}]\arrow[rrr, bend right=70, "\mathbb{1}\widetilde{\otimes}-" below, ""{name=D}]\arrow[rrr, equal, ""{name=R}]
&&& \E
\arrow[Rightarrow,from=U,to=R,shorten >=0.1cm,shorten <=0.1cm,"\rho" right=1.5, "\iso" left, shift right]
\arrow[Rightarrow,from=D,to=R,shorten >=0.2cm,shorten <=0.2cm,"\lambda" right=1.5, "\iso" left, shift left]
\end{tikzcd}
\hspace{3cm}
\begin{tikzcd}
 \E\times\E\times\E \arrow[rr,"\E\times\otimes"]\arrow[dd,"\otimes\times\E" left] && 
 \E\times\E \arrow[dd,"\otimes" right] \\
 \\
 \E\times\E \arrow[rr,"\otimes" below]\arrow[rruu,Rightarrow,"\iso" below right,"\alpha" above left,shorten >=1cm,shorten <=1cm] && 
 \E
\end{tikzcd}
\end{center} 

These must satisfy the familiar coherence conditions, as in~\cite{Groth11} and~\cite[Chapter 1.1]{Kelly82}. We say that $\E$ is a \textbf{symmetric monoidal prederivator} if, in addition, we have a natural isomorphism
\begin{center}
\begin{tikzcd}
 \E\times\E \arrow[rr,"\sigma" above,"\iso" below, bend left=10]\arrow[dddr, bend right=20,"\otimes" below left,""{name=L}] &&
 \E\times\E \arrow[dddl,"\otimes" below right, bend left=20]\\
\\\\
&\E
\arrow[Rightarrow,from=1-3,to=L, shorten >=0.8cm,shorten <=0.8cm,"\iso" above left,"\tau" below right=1.5,shift left=2.5]
\end{tikzcd}
\end{center}
where $\sigma$ is the canonical twist. This map must also satisfy coherence conditions, as in~\cite{Groth11} and~\cite[Chapter 1.4]{Kelly82}.
\end{definition} 

\begin{remark}
To give a monoidal structure on a prederivator $\E$ is equivalent to giving a factorisation of the $2$-functor $\E:\Cat\op\rightarrow\CAT$ through the $2$-category of monoidal categories. Similarly, a symmetric monoidal structure on $\E$ corresponds to a factorisation through the $2$-category of symmetric monoidal categories. See~\cite{Groth11} for details.
\end{remark}

\begin{definition}
Let $\E$ be a derivator, equipped with the structure of a monoidal prederivator. We say that $\E$ is a \textbf{monoidal derivator} if the product $\otimes:\E\times\E\rightarrow\E$ is cocontinuous. If, in addition, the product is a two-variable left adjoint, we call $\E$ a \textbf{closed monoidal derivator}.

If the underlying monoidal structure is symmetric, then we call $\E$ a \textbf{symmetric monoidal derivator}.
\end{definition}

\begin{example}
If $\E$ is a monoidal derivator, then so is the shifted derivator $\E^{\A}$, for any category $\A$. The product $\otimes:\E^{\A}\times\E^{\A}\rightarrow\E^{\A}$ is given by the shifted product on $\E$. See~\cite[Example 3.24]{GPS14}.
\end{example}

\begin{example}
For any monoidal model category $\M$, the associated derivator $\dHo(\M)$ is a closed monoidal derivator. See~\cite[Example 3.23]{GPS14} and~\cite[Proposition 6.1]{Cisinski08}. In particular, the derivator of spaces $\dHo(\sSet)$ is monoidal, as is the derivator of spectra $\dHo(\Spt)$. Note that the monoidal structure on $\dHo(\Spt)$ can be constructed directly from the monoidal structure on $\dHo(\sSet)$, without going via a monoidal model category of spectra. See~\cite[Section 10]{Heller97} for this approach.
\end{example}

\section{Closed actions of monoidal derivators}\label{Section Closed actions of monoidal derivators}

In this section we recall the definition of modules over a monoidal derivator $\E$, and give several examples. These are studied in~\cite{Groth11,GPS14,GS17}. Closed $\E$-modules --- that is, those for which the $\E$-action is a two-variable left adjoint --- are of particular importance to us. These provide archetypal examples of $\E$-prederivators, and we will revisit them repeatedly in~\cref{Chapter E-categories} and~\cref{Chapter Enriched derivators}. We recall the definition in \cref{Closed E-module definition}.

\begin{definition}\label{Definition of E-module}
Let $\E$ be a monoidal derivator. A \textbf{(left) action} of $\E$ on a derivator $\D$ is given by a cocontinuous map $\otimes:\E\times\D\rightarrow\D$, together with isomorphisms: 

\begin{center}
\begin{tikzcd}
\D\arrow[rrr, bend left=50, "\mathbb{1}\widetilde{\otimes}-" above,""{name=U, below}]\arrow[rrr, equal, ""{name=R},bend right=45]
&&& \D
\arrow[Rightarrow,from=U,to=R,shorten >=0.3cm,shorten <=0.3cm,"\lambda" right=1.5, "\iso" left, shift right]

\end{tikzcd}
\hspace{3cm}
\begin{tikzcd}
 \E\times\E\times\D \arrow[rr,"\E\times\otimes"]\arrow[dd,"\otimes\times\D" left] && 
 \E\times\D \arrow[dd,"\otimes" right] \\
 \\
 \E\times\D \arrow[rr,"\otimes" below]\arrow[rruu,Rightarrow,"\iso" below right,"\alpha" above left,shorten >=1cm,shorten <=1cm] && 
 \D
\end{tikzcd}
\end{center} 

These must satisfy coherence conditions, as in~\cite{Groth11} and~\cite[Section 1]{JK01}, which are analogous to those satisfied by a monoidal product. If $\D$ is equipped with an $\E$-action, we call $\D$ an \textbf{$\E$-module}.
\end{definition}

\begin{example}
Any symmetric monoidal derivator $\E$ is a left (and right) module over itself.
\end{example}

\begin{definition}\label{E-module map definition}
Let $\D_1$ and $\D_2$ be $\E$-modules. A derivator map $F:\D_1\rightarrow\D_2$ is called an \textbf{$\E$-module morphism}, or is said to \textbf{preserve tensors}, if it is equipped with an isomorphism:
\begin{center}
\begin{tikzcd}
 \E\times\D_1 \arrow{rr}[above]{\E\times F}\arrow{dd}[left]{\otimes} && 
 \E\times\D_2 \arrow{dd}[right]{\otimes} \\
 
 \\
 
 \D_1 \arrow{rr}[below]{F}\arrow[uurr,Leftarrow,"\varphi" above left,"\iso" below right,shorten >=1.1cm,shorten <=1.1cm] && 
 \D_2

\end{tikzcd}
\end{center} 
This isomorphism must satisfy the following coherence conditions, as in~\cite{Groth11}:
\begin{enumerate}
\item 
For any category $\A$, and any $X\in\D_1(\A)$ and $Y,Z\in\E(\A)$, the diagram below must commute:
\begin{center}
\begin{tikzcd}
(Z\otimes Y)\otimes FX\arrow{dd}[left]{\varphi} \arrow{rr}[above]{\alpha} && Z\otimes (Y\otimes FX)\arrow{d}[right]{Z\otimes \varphi} \\
 
&& Z\otimes F(Y\otimes X)\arrow{d}[right]{\varphi}\\
 
F((Z\otimes Y)\otimes X)\arrow{rr}[below]{F(\alpha)} && F(Z\otimes (Y\otimes X))
\end{tikzcd}
\end{center} 
\item
For any $X\in\D_1(\A)$, the diagram below must commute:
\begin{center}
\begin{tikzcd}
 \mathbb{1} \widetilde{\otimes} F(X) \arrow[bend right,ddr,"\lambda" below left]
 \arrow{rr}[above]{\varphi} && 
 F(\mathbb{1} \widetilde{\otimes} X) \arrow[bend left,ddl,"F(\lambda)" below right]
 \\\\
& FX
\end{tikzcd}
\end{center} 
\end{enumerate}
\end{definition}

\begin{definition}\label{E-module modification definition}
Let $F,G:\D_1\rightarrow\D_2$ be $\E$-module morphisms. A modification $\beta:F\Rightarrow G$ is called an \textbf{$\E$-module modification}, or is said to \textbf{respect tensors}, if the diagram below commutes, for any category $\A$, and any $X\in\D_1(\A)$ and $Y\in\E(\A)$:
\begin{center}
\begin{tikzcd}
 Y\otimes F(X) \arrow{rr}[above]{Y\otimes\beta_{X}}\arrow{dd}[left]{\varphi} && 
 Y\otimes G(X) \arrow{dd}[right]{\varphi} \\\\
 F(Y\otimes X) \arrow{rr}[below]{\beta_{Y\otimes X}} && 
 G(Y\otimes X)
\end{tikzcd}
\end{center} 
\end{definition}

\begin{definition}
Given two $\E$-modules $\D_1$ and $\D_2$, let $\Hom^\E(\D_1,\D_2)$ denote the category whose objects are $\E$-module maps from $\D_1$ to $\D_2$, and whose morphisms are $\E$-module modifications. Similarly, let $\Hom^\E_!(\D_1,\D_2)$ denote the full subcategory of $\Hom^\E(\D_1,\D_2)$ on the cocontinuous $\E$-module maps.
\end{definition}

\begin{example}\label{Derivators are sSet-modules stable derivators are Spt-modules}
Using the universal property of $\dHo(\sSet)$ described in \cref{Universal property of spaces}, we can see that any derivator $\D$ has a unique module structure over $\dHo(\sSet)$. The action 
\[
\otimes:\dHo(\sSet)\times\D\rightarrow\D
\]
is the map described at the end of \cref{Section Two-variable maps}; that is, it is the essentially unique cocontinuous map such that
\[
\Delta^0\;\widetilde{\otimes}\;-\iso\id:\D\rightarrow\D.
\] 
The coherent structure isomorphisms can be constructed using the universal property. By similar arguments, any cocontinuous derivator map $F:\D_1\rightarrow\D_2$ is a $\dHo(\sSet)$-module morphism.

Similarly, for any stable derivator $\D$, the map 
\[
\wedge:\dHo(\Spt)\times\D\rightarrow\D
\]
described at the end of \cref{Section Pointed and stable derivators} is a unique $\dHo(\Spt)$-action on $\D$, and any cocontinuous map between stable derivators is a $\dHo(\Spt)$-module map.
\end{example}

\begin{definition}\label{Closed E-module definition}
An $\E$-module $\D$ is called a \textbf{closed $\E$-module} if the action $\otimes:\E\times\D\rightarrow\D$ is a two-variable left adjoint. In this case, we will denote the associated maps given by Theorem~\ref{Characterising two-variable left adjoints} as follows:
\begingroup
\addtolength{\jot}{0.5em}
\begin{align*}
\map_\D(-,-)&:\D\op\times\D\rightarrow\E\\
\vartriangleleft&:\D\times\E\op\rightarrow\D
\end{align*}
\endgroup
We refer to $\otimes$ as the \textbf{tensor product} and $\vartriangleleft$ as the \textbf{cotensor product} on $\D$. We will call $\map_\D(-,-)$ the \textbf{mapping space} or \textbf{mapping object} morphism for $\D$.
\end{definition}

\begin{example}\label{E is an E-module}
Any closed symmetric monoidal derivator $\E$ is a closed module over itself. In this case, the cotensor product can be given in terms of the mapping space:
\begin{center}
\begin{tikzcd}[column sep=small]
\vartriangleleft:\E\times\E\op \arrow{rr}[above]{\sigma}[below]{\iso} && 
\E\op\times\E \arrow{rrrrr}{\map_{\E}(-,-)} &&&&& \E
\end{tikzcd}
\end{center} 

To verify this description of $\vartriangleleft$, we can apply Theorem~\ref{2-variable adjunction characterisation} and check that the composite above has the desired properties. This is immediate once we observe that, for any category $\A$ and any $X\in\E(\A)$, the symmetry isomorphism $\tau$ and the isomorphism of~\cref{Coends over A and Aop} induce an isomorphism between the cancelling products
\begin{align*}
X\otimes_{\A\op} -&:\E^{\A\op}\rightarrow\E
\\
-\otimes_\A X&:\E^{\A\op}\rightarrow\E
\end{align*}
of~\cref{Definition of cancelling product on morphisms}. This isomorphism is discussed in greater detail in~\cref{Symmetry for the cancelling tensor}. 
\end{example}

\begin{example}\label{Any triangulated derivator is a closed Spt module}
Let $\D$ be a triangulated derivator. Since $\D$ is stable, it has a unique $\dHo(\Spt)$-module structure, by \cref{Derivators are sSet-modules stable derivators are Spt-modules}. Using the explicit construction of the action in \cite{Cisinski08,Coley18,Heller97}, we can check that, for any category $\A$ and any $X\in\Ho(\Spt^\A)$, the map
\[
X\;\widetilde{\wedge}\;-:\D\rightarrow\D^\A
\]
has a right adjoint. On the other hand, for any category $\B$ and any $Y\in\D(\B)$, since $\Ho(\Spt)$ is compactly generated, \cref{Adjoint functor theorem for triangulated derivators} implies that the cocontinuous map
\[
-\;\widetilde{\wedge}\;Y:\dHo(\Spt)\rightarrow\D^\B
\]
has a right adjoint. Thus, any triangulated derivator $\D$ has a unique closed $\dHo(\Spt)$-module structure. See~\cite[Appendix A.3]{CT11} and~\cite[Section 4.4]{GR19} for more details.
\end{example}

\begin{example}
Let $\M$ be a monoidal model category and let $\N$ be an $\M$-enriched model category. See~\cite{GM11} for a discussion of enriched model categories. By~\cite[Example 3.23]{GPS14}, the derivator $\dHo(\N)$ is a closed $\dHo(\M)$-module. See also~\cite[Examples 4.75]{GR19}.
\end{example}

\begin{example}
For any model category $\M$, by \cref{Derivators are sSet-modules stable derivators are Spt-modules}, the derivator $\dHo(\M)$ is a $\dHo(\sSet)$-module. It follows from \cite[Section 6]{Cisinski08} that this action is closed. Mapping spaces can be constructed using cosimplicial frames, as in \cite[Chapter 5]{Hovey07}.
\end{example}

\begin{example}\label{Shift of an E-module is an E-module}
Given any $\E$-module $\D$, and any category $\A$, the shifted derivator $\D^\A$ has an $\E$-module structure given by
\[
\otimes:\E\times\D^{\A}\xrightarrow{\;\;\p^*\times\D^\A\;\;}\E^{\A}\times\D^{\A}\xrightarrow{\;\;\otimes\;\;}\D^{\A},
\]
where $\otimes:\E^{\A}\times\D^{\A}\rightarrow\D^{\A}$ is the action on $\D$ shifted by $\A$. Note that, in terms of the original action on $\D$, this map is simply the external tensor product
\[
\widetilde{\otimes}:\E\times\D^{\A}\rightarrow\D^{\A}.
\]
The maps $\alpha$ and $\lambda$ are inherited from $\D$. 

Moreover, if $\D$ is a closed $\E$-module, then so is $\D^\A$, by~\cite[Example 8.15]{GPS14}. The associated maps from Theorem~\ref{Characterising two-variable left adjoints} are given by
\begin{align*}
\map_{\D^\A}(-,-)&:(\D\op)^{\A\op}\times\D^\A\xrightarrow{\;\;\widetilde{\map}_{\D}(-,-)\;\;}\E^{\A\op\times\A}\xrightarrow{\;\;\int_{\A}\;\;}\E
\\\\
\vartriangleleft&:\D^\A\times\E\op\xrightarrow{\;\;\D^\A\times\p^*\;\;}\D^\A\times (\E\op)^\A\xrightarrow{\;\;\vartriangleleft\;\;}\D^\A,
\end{align*}
where $\vartriangleleft:\D^\A\times (\E\op)^\A\rightarrow\D^\A$ is the shifted map associated to $\D$. To see this directly, apply Theorem~\ref{2-variable adjunction characterisation} to these three maps.
\end{example}

\begin{lemma}\label{u* and alpha* respect the E-action}
Let $\D$ be an $\E$-module and let $u:\A\rightarrow\B$ be a functor. Then the derivator map 
\[
u^*:\D^\B\rightarrow\D^\A
\]
has a canonical $\E$-module morphism structure. Moreover, for any natural transformation $\kappa:u\Rightarrow v$, the modification $\kappa^*:u^*\Rightarrow v^*$ is an $\E$-module modification.
\end{lemma}
\begin{proof}
Using the description from~\cref{Shift of an E-module is an E-module} of the $\E$-action on the shifted derivator as an external tensor product, we require an isomorphism:
\begin{center}
\begin{tikzcd}
 \E\times\D^\B \arrow{rr}[above]{\E\times u^*}\arrow{dd}[left]{\widetilde{\otimes}} && 
 \E\times\D^\A \arrow{dd}[right]{\widetilde{\otimes}} \\
 
 \\
 
 \D^\B \arrow{rr}[below]{u^*}\arrow[uurr,Leftarrow,"\varphi" above left,"\iso" below right,shorten >=1.1cm,shorten <=1.1cm] && 
 \D^\A

\end{tikzcd}
\end{center} 
We take the isomorphism of~\cref{Lifting the external product to a derivator map} to be this structure isomorphism. The coherence conditions of~\cref{E-module map definition} follow from the fact that the structure isomorphisms $\alpha$ and $\lambda$ for the $\E$-module structure on $\D$ are modifications. 

Given a natural transformation $\kappa:u\Rightarrow v$, we need to check that $\kappa^*:u^*\Rightarrow v^*$ satisfies~\cref{E-module modification definition}. This follows from Axiom $3$ of~\cref{Pseudonatural transformation definition}, for the prederivator map $\otimes:\E\times\D\rightarrow\D$.
\end{proof}

\section{The cancelling tensor product}\label{Section The cancelling tensor product}

In this section, building on work in~\cite{GPS14}, we revisit the cancelling tensor product of~\cref{Definition of cancelling product on morphisms}, associated to any $\E$-module. In~\cref{Coherence for the cancelling version of the closed E-action}, we recall that this product is associative and unital, with coherence induced by the original action. In particular, for any monoidal derivator $\E$, we obtain a bicategory $\Prof(\E)$, which plays a vital role as the enriching bicategory in~\cref{Chapter E-categories} and~\cref{Chapter Enriched derivators}. For basic bicategory definitions see~\cite{Leinster98}. We conclude this section by showing that cocontinuous $\E$-module maps and $\E$-module modifications respect the cancelling tensor. The results in this section, in particular the list of commutative diagrams that we collect here, will be used as we develop the theory of $\E$-categories in~\cref{Chapter E-categories}.

\begin{definition}\label{Definition of h on objects and morphisms}
Let $\E$ be a monoidal derivator. For any category $\A$, consider the map
\[
\partial_\A:\E(\0)\rightarrow\E(\A\op\times\A)
\]
of~\cref{Ends and coends definition}. Denote the image of $\mathbb{1}\in\E(\0)$ under this map by $\h_\A\in\E(\A\op\times\A)$. In~\cite{GR19}, this object is called the \textbf{Yoneda bimodule} or the \textbf{identity profunctor}. Similarly, given a functor $u:\A\rightarrow\B$, consider the modification $\partial_u$ of~\cref{Delta and end on morphisms}. Denote the component of this map at $\mathbb{1}$ by
\[
\h_u:\h_\A\rightarrow (u\op\times u)^*\h_B
\]
in $\E(\A\op\times\A)$.
\end{definition}


\begin{proposition}\label{Coherence for the cancelling version of the closed E-action}
Let $\E$ be a monoidal derivator. For any categories $\A$ and $\B$ and any object $X\in\E(\A\op\times\B)$, the unit isomorphisms $\lambda$ and $\rho$ of~\cref{Definition of (symmetric) monoidal prederivator} induce isomorphisms:
\begin{align*}
\lambda&:\h_\B\otimes_\B X\xrightarrow{\;\;\;\iso\;\;\;} X\\
\rho&: X\otimes_\A \h_\A\xrightarrow{\;\;\;\iso\;\;\;} X
\end{align*}
Similarly, given additional categories $\C$ and $\cD$, and $Y\in\E(\B\op\times\C)$ and $Z\in\E(\C\op\times\cD)$, the associativity isomorphism $\alpha$ induces an isomorphism
\[
\alpha:(Z\otimes_\C Y)\otimes_\B X \xrightarrow{\;\;\;\iso\;\;\;} Z\otimes_\C (Y\otimes_\B X)
\]
in $\E(\A\op\times\cD)$. More generally, there are analogous isomorphisms $\alpha$ and $\lambda$ for any given $\E$-module $\D$.

These satisfy the following coherence conditions:
\begin{enumerate}
\item Given categories $\A$, $\B$, $\C$, $\cD$ and $\cE$, and objects $X\in\D(\A\op\times\B)$, $Y\in\E(\B\op\times\C)$, $Z\in\E(\C\op\times\cD)$ and $W\in\E(\cD\op\times\cE)$, the diagram below commutes: 
\begin{center}
\begin{tikzcd}[column sep=small]
((W\otimes_{\cD}Z)\otimes_\C Y)\otimes_\B X\arrow[dd,"\alpha" left]\arrow[rr,"\alpha\;\otimes_\B X" above]
&& (W\otimes_{\cD}(Z\otimes_\C Y))\otimes_\B X\arrow[dd,"\alpha" right]\\\\
(W\otimes_{\cD}Z)\otimes_\C (Y\otimes_\B X)\arrow[rd,"\alpha" below left, bend right=15]&& W\otimes_{\cD}((Z\otimes_\C Y)\otimes_\B X)\arrow[ld,"W \otimes_{\cD}\;\alpha" below right, bend left=15]
\\
& W\otimes_{\cD}(Z\otimes_\C (Y\otimes_\B X))
\end{tikzcd}
\end{center}

\item Given $X\in\D(\A\op\times\B)$ and $Y\in\E(\B\op\times\C)$, the diagram below commutes:
\begin{center}
\begin{tikzcd}[column sep=small]
 (Y\otimes_\B \h_\B)\otimes_\B X\arrow[bend right,ddr,"\rho\;\otimes_\B X" below left]
 \arrow{rr}[above]{\alpha} && 
 Y\otimes_\B (\h_\B\otimes_\B X) \arrow[bend left,ddl,"Y\otimes_\B\;\lambda" below right]
 \\\\
& Y\otimes_\B X
\end{tikzcd}
\end{center} 
\end{enumerate}
\end{proposition}
\begin{proof}
This is proved in~\cite[Theorem 5.9]{GPS14} in the case of $\D=\E$, and the proof in the general case carries over unchanged. We recall the construction of the maps $\alpha$, $\lambda$ and $\rho$.

Given $X\in\D(\A\op\times\B)$, the map $\lambda:\h_\B\otimes_\B X\xrightarrow{\;\;\iso\;\;} X$ is the component at $X$ of the modification below:

\begin{center}
\begin{tikzcd}
&& \D^{\B_1\times\B_2\op\times\B_3}\arrow[rdd, bend left,"\int^{\B_{12}}" above right] \\\\
 \D^{\B_1}\arrow[bend right,ddrr,equal,""{name=A}]\arrow[bend left,ddrr,"\mathbb{1}\widetilde{\otimes}-" above=3,""{name=B}]
 \arrow[bend left,rruu,"\partial_{\B_{23}}\mathbb{1}\widetilde{\otimes}-" above left,""{name=C}] &&& 
 \D^{\B_3} \arrow[bend left,ddl,equal,""{name=R}]
 \\\\
&& \D^{\B_1}\arrow[uuuu,"\partial_{\B_{23}}" left]

\arrow[Rightarrow,from=B,to=A,shorten >=0.2cm,shorten <=0.3cm,"\lambda" below right, "\iso" above left, shift right]
\arrow[Rightarrow,from=C,to=5-3,shorten >=2.6cm,shorten <=0.8cm,"\iso" above right, near start, shift left=4]
\arrow[Rightarrow,from=1-3,to=R,shorten >=1.2cm,shorten <=1.3cm,"\iso" above right,shift left]
\end{tikzcd}
\end{center} 
Here we have $\B=\B_1=\B_2=\B_3$, as in~\cref{A=A_1=A_2 notation}, and the unlabelled isomorphisms come from the cocontinuity of $\otimes$ and~\cref{Interaction of delta and integral}. For $X\in\E(\A\op\times\B)$, the isomorphism $\rho$ is obtained similarly.

For $X\in\D(\A\op\times\B)$, $Y\in\E(\B\op\times\C)$ and $Z\in\E(\C\op\times\cD)$, the associativity isomorphism from $(Z\otimes_\C Y)\otimes_\B X = \int^\B(\int^\C (Z\widetilde{\otimes} Y)\widetilde{\otimes} X)$ to $(Z\otimes_\C Y)\otimes_\B X = \int^\C Z\widetilde{\otimes} \int^\B(Y\widetilde{\otimes} X)$ is given by the following composite:
\begin{center}
\begin{tikzcd}
\int^\B(\int^\C (Z\widetilde{\otimes} Y)\widetilde{\otimes} X) \arrow[r,"\iso"]
& \int^\B\int^\C (Z\widetilde{\otimes} Y)\widetilde{\otimes} X \arrow[r,"\iso"] & \int^\C\int^\B (Z\widetilde{\otimes} Y)\widetilde{\otimes} X\arrow[d,"\int^\C\int^\B \alpha" right=2]\\
&& \int^\C\int^\B Z\widetilde{\otimes} (Y\widetilde{\otimes} X)\arrow[d,"\iso"]\\
&& \int^\C Z\widetilde{\otimes} \int^\B(Y\widetilde{\otimes} X)
\end{tikzcd}
\end{center}
The unlabelled isomorphisms follow by the cocontinuity of $\otimes$ and~\cref{Fubini theorem}. Note that, just as $\lambda$ and $\rho$ are modifications, the map $\alpha$ above is the component of a modification between certain derivator maps. 
\end{proof}

\begin{remark}\label{Definition of Prof(E)}
\cref{Coherence for the cancelling version of the closed E-action} can be rephrased as follows, which is how it appears in~\cite[Theorem 5.9]{GPS14}. For any monoidal derivator $\E$, the following data forms a bicategory $\Prof(\E)$:
\begin{itemize}
\item The objects of $\Prof(\E)$ are small categories.
\item Given categories $\A$ and $\B$, the hom-category from $\A$ to $\B$ is $\E(\B\op\times\A)$.
\item Composition is given by the cancelling product:
\[
\otimes_{\B} :\E(\B\op\times\A)\times\E(\C\op\times\B)\rightarrow\E(\C\op\times\A)
\]
\item For any category $\A$, the identity on $\A$ is given by
\[
\h_{\A}\in\E(\A\op\times\A).
\]
\end{itemize}
We call $\Prof(\E)$ the \textbf{bicategory of profunctors} associated to $\E$. This is the terminology of~\cite{GPS14}; it is called the \textbf{bicategory of bimodules} in~\cite{GR19}.
\end{remark}


\begin{notation}\label{Cancelling tensor cancels on the outside}
From this point onwards, we will take the convention that the cancelling tensor cancels the outside variables, as in~\cref{Definition of Prof(E)}:
\[
\otimes_{\B} :\E(\B\op\times\C)\times\E(\A\op\times\B)\rightarrow\E(\A\op\times\C)
\]
\end{notation}

\begin{example}\label{E shifted by Aop x A is monoidal}
Given any monoidal derivator $\E$ and any category $\A=\A_i$,~\cref{Coherence for the cancelling version of the closed E-action} implies that the shifted derivator $\E^{\A\op\times\A}$ is monoidal, with unit given by $\h_\A\in\E(\A\op\times\A)$, and tensor given by the cancelling tensor product:
\[
\otimes_{\A_{1,4}} :\E^{\A_1\op\times\A_2}\times\E^{\A_3\op\times\A_4}\rightarrow\E^{\A_3\op\times\A_2}
\]
If $\E$ is closed monoidal then so is $\E^{\A\op\times\A}$. In this case,~\cref{2-variable adjunction characterisation} provides a description of the required adjoints. Note, however, that if $\E$ is symmetric monoidal it does not follow that $\E^{\A\op\times\A}$ is symmetric; see \cref{Symmetry for the cancelling tensor} for a description of the structure induced on shifted derivators by the symmetry isomorphism from $\E$.
\end{example}

\begin{remark}
In addition to the commutative diagrams of~\cref{Coherence for the cancelling version of the closed E-action}, for any $\E$-module $\D$, any categories $\A$, $\B$ and $\C$, and any objects $X\in\D(\A\op\times\B)$ and $Y\in\E(\B\op\times\C)$, the diagram below commutes:
\begin{center}
\begin{tikzcd}
 (\h_\C\otimes_\C Y)\otimes_\B X\arrow[bend right,ddr,"\lambda\;\otimes_\B X" below left]
 \arrow{rr}[above]{\alpha} && 
 \h_\C\otimes_\C (Y\otimes_\B X) \arrow[bend left,ddl,"\lambda" below right]
 \\\\
& Y\otimes_\B X
\end{tikzcd}
\end{center} 
In the case of $\D=\E$, in light of~\cref{Definition of Prof(E)}, this is a special case of the coherence theorem for bicategories, as in~\cite{Leinster98}. In general, the commutativity of this diagram follows from the commutative diagrams in~\cref{Coherence for the cancelling version of the closed E-action}. The proof of this is essentially the same as the proof for monoidal categories in~\cite{Kelly64}.
\end{remark}

\begin{lemma}\label{Symmetry for the cancelling tensor}
Let $\E$ be a symmetric monoidal derivator. For any categories $\A$, $\B$ and $\C$, and any $X\in\E(\B\times\A\op)$ and $Y\in\E(\C\times\B\op)$, the symmetry isomorphism of~\cref{Definition of (symmetric) monoidal prederivator}  induces a map 
\[
\tau:\sigma^*Y\otimes_\B\sigma^*X\xrightarrow{\;\;\;\iso\;\;\;}\sigma^*(X\otimes_{\B\op}Y)
\]
that satisfies the following coherence conditions:
\begin{enumerate}
\item Given categories $\A$, $\B$, $\C$ and $\cD$, and objects $X\in\E(\B\times\A\op)$, $Y\in\E(\C\times\B\op)$ and $Z\in\E(\cD\times\C\op)$, the diagram below commutes: 
\begin{center}
\begin{tikzcd}
(\sigma^*Z\otimes_\C\sigma^*Y)\otimes_\B \sigma^*X\arrow{dd}[left]{\tau\;\mathsmaller{\otimes_\B} \sigma^*X} \arrow{rr}[above]{\alpha} && \sigma^*Z\otimes_\C(\sigma^*Y\otimes_\B \sigma^*X)\arrow{dd}[right]{\sigma^*Z\mathsmaller{\otimes_\C}\;\tau} \\\\
 
\sigma^*(Y\otimes_{\C\op}Z)\otimes_\B \sigma^*X\arrow{dd}[left]{\tau} && \sigma^*Z\otimes_\C\sigma^*(X\otimes_{\B\op} Y)\arrow{dd}[right]{\tau}\\\\
 
\sigma^*(X\otimes_{\B\op} (Y\otimes_{\C\op} Z)) && \sigma^*((X\otimes_{\B\op} Y)\otimes_{\C\op} Z)\arrow{ll}[below]{\sigma^*(\alpha)}
\end{tikzcd}
\end{center} 
\item Given $X\in\E(\B\times\A\op)$, the diagram below commutes
\begin{center}
\begin{tikzcd}
 \sigma^*X\otimes_\A \h_\A \arrow{rr}[above]{\iso}\arrow{dd}[left]{\rho} && 
 \sigma^*X\otimes_\A \sigma^*\h_{\A\op} \arrow{dd}[right]{\tau} \\\\
 \sigma^*X  && 
 \sigma^*(\h_{\A\op}\otimes_{\A\op} X)\arrow{ll}[below]{\sigma^*(\lambda)}
\end{tikzcd}
\end{center} 
where the isomorphism along the top is induced by the canonical isomorphism
\[
\h_\A\xrightarrow{\;\;\iso\;\;}\sigma^*\h_{\A\op},
\]
an instance of~\cref{Coends over A and Aop} for $\partial_\A$.
\item Given $X\in\E(\B\times\A\op)$, the diagram below commutes:
\begin{center}
\begin{tikzcd}
 \h_\B\otimes_\B \sigma^*X \arrow{rr}[above]{\iso}\arrow{dd}[left]{\lambda} && 
 \sigma^*\h_{\B\op}\otimes_\B \sigma^*X \arrow{dd}[right]{\tau} \\\\
 \sigma^*X  && 
 \sigma^*(X\otimes_{\B\op} \h_{\B\op})\arrow{ll}[below]{\sigma^*(\rho)}
\end{tikzcd}
\end{center} 
\end{enumerate}
\end{lemma}
\begin{proof}
For any categories $\J$ and $\K$, the symmetry isomorphism $\tau$ induces an isomorphism:
\begin{center}
\begin{tikzcd}
 \E^\J\times\E^\K \arrow{rr}[above]{\widetilde{\otimes}}\arrow{dd}[left]{\sigma} && 
 \E^{\J\times\K} \arrow{dd}[right]{\sigma^*} \\
 
 \\
 
  \E^\K\times\E^\J \arrow{rr}[below]{\widetilde{\otimes}}\arrow[uurr,Rightarrow,"\tau" above left,"\iso" below right,shorten >=1.1cm,shorten <=1.1cm] && 
 \E^{\K\times\J}

\end{tikzcd}
\end{center} 
Taking $\K=\J\op$, we get an induced map on the cancelling tensor products, using the isomorphism of \cref{Coends over A and Aop}:
\begin{center}
\begin{tikzcd}
\E^\J\times\E^{\J\op}\arrow{dd}[left]{\sigma} \arrow[bend left=25,rrd,"\otimes_{\J\op}" above right,""{name=U}] & & & [-25pt] & [-15pt] \E^\J\times\E^{\J\op} \arrow{rr}[above]{\widetilde{\otimes}}\arrow{dd}[left]{\sigma} && 
 \E^{\J\times\J\op}\arrow[bend left=25,rrd,"\int^{\J\op}" above right,""{name=R}]  \arrow{dd}[left]{\sigma^*} \\
&& \E & := &&&&& \E \\
\E^{\J\op}\times\E^\J \arrow[bend right=25,rur,"\otimes_\J" below right]
 \arrow[Rightarrow,from=3-1,to=U,"\iso" above left,"\tau" below right, shorten >=1.1cm,shorten <=1.1cm,shift right=2] &&&& \E^{\J\op}\times\E^\J \arrow{rr}[below]{\widetilde{\otimes}} &\;& 
 \E^{\J\op\times\J}\arrow[bend right=25,rur,"\int^\J" below right]
 
\arrow[Rightarrow,from=3-7,to=R,"\iso" above left, shorten >=1cm,shorten <=1cm,shift right=2]

\arrow[Leftarrow,from=1-7,to=3-6,"\iso" above left,near end,"\tau" right=4,shorten >=0.5cm,shorten <=1.6cm, shift right=8]
\end{tikzcd}
\end{center}
The map we require can be obtained as a shifted version of this modification. The coherence conditions follow from the coherence conditions on the original symmetry isomorphism \linebreak$\tau:\otimes\circ\sigma\Rightarrow\otimes$.
\end{proof}

\begin{remark}
If $\E$ is a symmetric monoidal derivator, then~\cref{Symmetry for the cancelling tensor} endows the maps $\sigma^*$ with the structure of an isomorphism between the bicategory $\Prof(\E)$ and its opposite $\Prof(\E)\op$; that is, the bicategory with the same objects, and with hom-category from $\A$ to $\B$ given by $\E(\A\op\times\B)$. See~\cite{Leinster98} for basic bicategorical definitions.

Explicitly, this map takes each small category $\A\in\Prof(\E)$ to its opposite $\A\op\in\Prof(\E)\op$, and the action on hom-categories is given by:
\[
\sigma^*:\E(\B\op\times\A)\xrightarrow{\;\;\iso\;\;}\E(\A\times\B\op)
\]
\end{remark}

We will now turn our attention to the interaction of $\E$-module maps and $\E$-module modifications with the cancelling tensors.

\begin{lemma}\label{Cocontinuous E-module maps preserve the cancelling product}
Let $\D_1$ and $\D_2$ be $\E$-modules, and let $F:\D_1\rightarrow\D_2$ be a cocontinuous $\E$-module morphism. Then for any categories $\A$, $\B$ and $\C$, and any $X\in\D_1(\A\op\times\B)$ and $Y\in\E(\B\op\times\C)$, the isomorphism $\varphi$ of~\cref{E-module map definition} induces an isomorphism
\[
\varphi:Y\otimes_\B FX\xrightarrow{\;\;\;\iso\;\;\;} F(Y\otimes_\B X)
\]
in $\D_2(\A\op\times\C)$. These satisfy the following coherence conditions, which are exact analogues of those in~\cref{E-module map definition}:
\begin{enumerate}
\item 
For any categories $\A$, $\B$, $\C$ and $\cD$, and any $X\in\D_1(\A\op\times\B)$, $Y\in\E(\B\op\times\C)$ and $Z\in\E(\C\op\times\cD)$, the diagram below commutes:
\begin{center}
\begin{tikzcd}
(Z\otimes_\C Y)\otimes_\B FX\arrow{dddd}[left]{\varphi} \arrow{rr}[above]{\alpha} && Z\otimes_\C (Y\otimes_\B FX)\arrow{dd}[right]{Z\mathsmaller{\otimes_\C}\;\varphi} \\\\
 
&& Z\otimes_\C F(Y\otimes_\B X)\arrow{dd}[right]{\varphi}\\\\
 
F((Z\otimes_\C Y)\otimes_\B X)\arrow{rr}[below]{F(\alpha)} && F(Z\otimes_\C (Y\otimes_\B X))
\end{tikzcd}
\end{center} 
\item
For any $X\in\D_1(\A\op\times\B)$, the diagram below commutes:
\begin{center}
\begin{tikzcd}
 \h_\B \otimes_\B F(X) \arrow[bend right,ddr,"\lambda" below left]
 \arrow{rr}[above]{\varphi} && 
 F(\h_\B \otimes_\B X) \arrow[bend left,ddl,"F(\lambda)" below right]
 \\\\
& FX
\end{tikzcd}
\end{center} 
\end{enumerate}
\end{lemma}
\begin{proof}
Given $X\in\D_1(\A\op\times\B)$ and $Y\in\E(\B\op\times\C)$, the map $\varphi:Y\otimes_\B FX\xrightarrow{\;\iso\;} F(Y\otimes_\B X)$ is given by the composite
\[
\int^\B(Y\widetilde{\otimes}FX)\xrightarrow{\;\;\;\;\;\int^\B\varphi\;\;\;\;\;}\int^\B F(Y\widetilde{\otimes}X)\xrightarrow{\;\;\;\;\;\iso\;\;\;\;\;}F\int^\B (Y\widetilde{\otimes}X),
\]
where the first map is induced by the structure isomorphism of~\cref{E-module map definition}, and the second follows by the cocontinuity of $F$. 

To see that the diagrams commute, we can rewrite $\varphi$ as above, and $\alpha$ and $\lambda$ as in the proof of~\cref{Coherence for the cancelling version of the closed E-action}.  Using~\cref{Modifications respect coends} and~\cref{Prederivator maps respect Fubini isomorphism}, the first diagram reduces to $\int^\B\int^\C$ applied to the diagram below, in $\D_2(\B\op\times\B\times\C\op\times\C\times\A\op\times\cD)$:
\begin{center}
\begin{tikzcd}
(Z\widetilde{\otimes} Y)\widetilde{\otimes} FX\arrow{dd}[left]{\varphi} \arrow{rr}[above]{\alpha} && Z\widetilde{\otimes} (Y\widetilde{\otimes} FX)\arrow{d}[right]{Z\widetilde{\otimes} \varphi} \\
 
&& Z\widetilde{\otimes} F(Y\widetilde{\otimes} X)\arrow{d}[right]{\varphi}\\
 
F((Z\widetilde{\otimes} Y)\widetilde{\otimes} X)\arrow{rr}[below]{F(\alpha)} && F(Z\widetilde{\otimes} (Y\widetilde{\otimes} X))
\end{tikzcd}
\end{center} 
This commutes by the first axiom of~\cref{E-module map definition}. Similarly, using~\cref{Derivator maps respect delta coend cancellation}, the second diagram reduces to the second axiom of~\cref{E-module map definition}.
\end{proof}

\begin{lemma}\label{E-module modifications respect the cancelling product}
Let $F,G:\D_1\rightarrow\D_2$ be cocontinuous $\E$-module maps, and let  $\beta:F\Rightarrow G$ be an $\E$-module modification. Then for any categories $\A$, $\B$ and $\C$, and any $X\in\D_1(\A\op\times\B)$ and $Y\in\E(\B\op\times\C)$, the diagram below commutes:
\begin{center}
\begin{tikzcd}
 Y\otimes_\B F(X) \arrow{rr}[above]{Y\mathsmaller{\otimes_\B}\;\beta_{X}}\arrow{dd}[left]{\varphi} && 
 Y\otimes_\B G(X) \arrow{dd}[right]{\varphi} \\\\
 F(Y\otimes_\B X) \arrow{rr}[below]{\beta_{Y\mathsmaller{\otimes_\B} X}} && 
 G(Y\otimes_\B X)
\end{tikzcd}
\end{center} 
\end{lemma}
\begin{proof}
Let $X\in\D_1(\A\op\times\B)$ and $Y\in\E(\B\op\times\C)$. The diagram we want to commute may be rewritten as follows, using the definition of $\varphi$ from the proof of~\cref{Cocontinuous E-module maps preserve the cancelling product}:
\begin{center}
\begin{tikzcd}
\int^\B(Y\widetilde{\otimes}FX) \arrow{rr}[above]{\int^\B(Y\widetilde{\otimes}\beta_{X})}\arrow{dd}[left]{\int^\B\varphi} && 
 \int^\B(Y\widetilde{\otimes}GX) \arrow{dd}[right]{\int^\B\varphi} \\\\
 \int^\B F(Y\widetilde{\otimes} X) \arrow{dd}[left]{\iso}\arrow{rr}[below]{\int^\B\beta_{Y\mathsmaller{\widetilde{\otimes}} X}} && 
 \int^\B G(Y\widetilde{\otimes} X)\arrow{dd}[right]{\iso}\\\\
  F\int^\B(Y\widetilde{\otimes} X) \arrow{rr}[below]{\beta_{\int^\B (Y\mathsmaller{\widetilde{\otimes}} X)}} && 
 G\int^\B(Y\widetilde{\otimes} X)
\end{tikzcd}
\end{center} 
The square at the top commutes by (the external version) of the $\E$-module modification condition of~\cref{E-module modification definition}. The second square commutes by~\cref{Modifications respect coends}.
\end{proof}

\begin{remark}\label{Coherence for the cancelling product gives coherence for the internal product}
In this section, we have proved a number of coherence results for the cancelling tensor product, resulting from the analogous coherence for the internal or external tensor product. On the other hand, suppose we have an $\E$-module $\D$, categories $\A$ and $\B$, and $X\in\D(\A)$ and $Y\in\E(\B)$. Using the isomorphism $Y\otimes_\B \partial_\B X\iso Y\;\widetilde{\otimes}\;X$, coherence results for the cancelling tensor product also imply the analogous results for the external product, and, using \cref{Bimorphisms}, the internal product. 
\end{remark}

\section{The maps $\otimes_u$ and $\h_u$}\label{Section The maps otimes u and h u}

In this section, we investigate the action of $\partial$ and $\int$ on morphisms. Specifically, given a functor $u:\A\rightarrow\B$, we revisit the maps $\partial_u$ and $\int^u$ of~\cref{Delta and end on morphisms}, and the associated maps $\h_u$ of~\cref{Definition of h on objects and morphisms} and $\otimes_u$ of~\cref{Definition of cancelling product on morphisms}. We conclude this section with a lemma describing the coherence  between the maps $\h_u$ and $\otimes_u$ in any $\E$-module. We will need this result, and others in this section, frequently in~\cref{Chapter Enriched derivators}. In particular, we use the results in this section in the proof of~\cref{E-modules induce E-prederivators}, showing that any closed $\E$-module induces an $\E$-prederivator.

Consider the commutative diagram of~\cref{tw(u) definition}:
\vspace{-3em}
\begin{center}
\begin{equation}\label{Diagram tw(u)}
\begin{tikzcd}[baseline=(current  bounding  box.center)]
\tw(\A)\arrow{r}[above]{(\mathrm{s},\mathrm{t})}\arrow{dd}[left]{\tw(u)} & \A\op\times\A\arrow{rr}[above]{\A\op\times u} && \A\op\times\B \arrow{dd}[right]{u\op\times \B} \\\\
 \tw(\B) \arrow{rrr}[below]{(\mathrm{s},\mathrm{t})} &&& 
\B\op\times\B
\end{tikzcd}
\end{equation}
\end{center} 
By~\cite[Lemma 5.4]{GPS14}, for any derivator $\D$, the commutative diagram (\ref{Diagram tw(u)}) induces an isomorphism:
\begin{center}
\begin{tikzcd}
\D^{\A\op\times\B}\arrow[ddrr,bend right,equal,""{name=L,above right}]\arrow[rr,"(u\op\times\B)_!" above] && \D^{\B\op\times\B}\arrow[Leftarrow,to=L,"\eta" above left,shorten >=0.6cm,shorten <=0.8cm] \arrow[rrrr,"(\mathrm{t}\op\text{,}\;\mathrm{s}\op)^*"]\arrow[dd,"(u\op\times\B)^*" right] && 
 &&\D^{\tw(\B)\op}\arrow[ddrr,"\p_!" above right, bend left,""{name=R,below left}] \arrow[dd,"(\tw(u)\op)^*" left]\\
\\
&&\D^{\A\op\times\B}\arrow[rr,"(\A\op\times u)^*" below] && 
 \D^{\A\op\times\A}\arrow[rr,"(\mathrm{t}\op\text{,}\;\mathrm{s}\op)^*" below] &&\D^{\tw(\A)\op}\arrow[Rightarrow,to=R,shorten >=0.5cm,shorten <=0.7cm] \arrow[rr,"\p_!" below]
 && \D 
\end{tikzcd}
\end{center}
Here the unlabelled $2$-cell is the canonical modification of~\cref{Canonical map to colimit}. Thus, for any $X\in\D(\A\op\times\B)$, we have a canonical isomorphism:
\vspace{-2.5em}
\begin{center}
\begin{equation}\label{Diagram coend and left Kan extension}
\begin{tikzcd}[baseline=(current  bounding  box.center)]
\int^\A(\A\op\times u)^*X\arrow{r}[above]{\iso} & \int^\B(u\op\times\B)_! X
\end{tikzcd}
\end{equation}
\end{center} 
Note that, if we paste the counit 
\begin{center}
\begin{tikzcd}
\D^{\B\op\times\B}\arrow[ddrr,equal, bend left,""{name=R,below left}] \arrow[dd,"(u\op\times\B)^*" left]\\
 
\\
\D^{\A\op\times\B}\arrow[Rightarrow,to=R,"\epsilon" above left, shorten >=0.5cm,shorten <=0.7cm] \arrow[rr,"(u\op\times\B)_!" below]
 && \D^{\B\op\times\B} 
\end{tikzcd}
\end{center}
onto the diagram above, then we obtain the modification $\int^u$ of~\cref{Delta and end on morphisms}. Thus, for any $Y\in\D(\B\op\times\B)$, we may factor $\int^u Y$ as follows:
\vspace{-3em}
\begin{center}
\begin{equation}\label{Diagram coend to the u}
\begin{tikzcd}[baseline=(current  bounding  box.center)]
\int^u Y:\int^\A(\A\op\times u)^*(u\op\times\B)^*Y\arrow[r,"\iso"]
& \int^\B(u\op\times \B)_!(u\op\times\B)^*Y\arrow[rr,"\int^\B \epsilon" above]  &[-12pt] &[-12pt] \int^\B Y
\end{tikzcd}
\end{equation}
\end{center}

Similarly, the commutative diagram (\ref{Diagram tw(u)}) induces an isomorphism
\begin{center}
\begin{tikzcd}
 && 
 \D^{\tw(\B)} \arrow[dd,"\tw(u)^*"]
\arrow[rrrr,"\text{(s,t)}_!"] &&&& 
 \D^{\B\op\times\B} \arrow[Leftarrow,ddllll,shorten >=2.6cm,shorten <=2.7cm] \arrow[dd,"(u\op\times \B)^*"]\\
\D\arrow[urr,"\p^*" above left,bend left=20]\arrow[drr,"\p^*" below left,bend right=20] \\
&&
\D^{\tw(\A)}\arrow[rr,"\text{(s,t)}_!" below] && 
\D^{\A\op\times\A}\arrow[rr,"(\A\op\times u)_!" below] && \D^{\A\op\times\B}
\end{tikzcd}
\end{center}
where the nonidentity $2$-cell is the canonical modification of~\cref{D-exact squares}. The component of this map at $Z\in\D(\0)$ gives a canonical isomorphism: 
\begin{equation}\label{Diagram Left Kan extensions and delta}
(\A\op\times u)_!\partial_\A Z\xrightarrow{\;\;\;\iso\;\;\;}(u\op\times\B)^*\partial_\B Z
\end{equation}
As above, if we paste the unit 
\begin{center}
\begin{tikzcd}
\D^{\A\op\times\A}\arrow[ddrr,bend right,equal,""{name=L,above right}]\arrow[rr,"(\A\op\times u)_!" above] && \D^{\A\op\times\B}\arrow[Leftarrow,to=L,"\eta" above left,shorten >=0.6cm,shorten <=0.8cm] \arrow[dd,"(\A\op\times u)^*" right] 
\\\\
&&\D^{\A\op\times\A}
\end{tikzcd}
\end{center}
onto this diagram, then we obtain $\partial_u$ of~\cref{Delta and end on morphisms}. Thus, for any $Z\in\D(\0)$, we may factor $\partial_u Z$ as follows:
\begin{center}
\begin{tikzcd}
\partial_u Z:\partial_\A Z\arrow[r,"\eta"]
& (\A\op\times u)^*(\A\op\times u)_!\partial_\A Z\arrow[r,"\iso" above] & (\A\op\times u)^*(u\op\times \B)^*\partial_\B Z
\end{tikzcd}
\end{center}
Finally, suppose $\D$ is an $\E$-module, and consider the map 
\begin{center}
\begin{tikzcd}
 \E^{\B\op}\times\D^\B\arrow{dd}[left]{(u\op)^*\times u^*} \arrow[bend left,rrd,"\otimes_\B" above right,""{name=U}] & \\
 && \D \\
 \E^{\A\op}\times\D^\A \arrow[bend right,rur,"\otimes_\A" below right]
 \arrow[Rightarrow,from=3-1,to=U,"\otimes_u" above left, shorten >=0.9cm,shorten <=0.9cm,shift right=2]
\end{tikzcd}
\end{center} 
of~\cref{Definition of cancelling product on morphisms}. For any objects $X\in\D(\B)$ and $Y\in\E(\A\op)$, the isomorphism (\ref{Diagram coend and left Kan extension}) above induces a canonical isomorphism $Y\otimes_\A u^*X\xrightarrow{\;\iso\;\;} (u\op)_!Y\otimes_\B X$ as follows:
\vspace{-2em}
\begin{center}
\begin{equation}\label{Diagram cancelling tensor and left Kan extension}
\begin{tikzcd}[baseline=(current  bounding  box.center)]
\int^\A(Y\;\widetilde{\otimes}\;u^*X)\arrow[r,"\iso"]
& \int^\A(\A\op\times u)^*(Y\;\widetilde{\otimes}\;X)\arrow[r,"\iso" above] & \int^\B(u\op\times \B)_!(Y\;\widetilde{\otimes}\;X)\arrow[d,"\iso" right] \\
&&  \int^\B((u\op)_!Y\;\widetilde{\otimes}\;X)
\end{tikzcd}
\end{equation}
\end{center}
The first isomorphism in this composite is the structure isomorphism for $Y\;\widetilde{\otimes}\;-$, the second is an instance of (\ref{Diagram coend and left Kan extension}), and the final isomorphism follows by the cocontinuity of $\otimes$. Using this isomorphism (\ref{Diagram cancelling tensor and left Kan extension}) and our description of $\int^u$ in (\ref{Diagram coend to the u}), for any $X\in\D(\B)$ and $Z\in\E(\B\op)$, we can factor $Z\otimes_u X$ as follows: 
\vspace{-3em}
\begin{center}
\begin{equation}\label{Diagram tensor cancelling over u}
\begin{tikzcd}[baseline=(current  bounding  box.center)]
Z\otimes_u X:(u\op)^*Z\otimes_\A u^*X\arrow[r,"\iso"]
& (u\op)_!(u\op)^*Z\otimes_\B X\arrow[rr,"\epsilon\;\otimes_\B X" above] && Z\otimes_\B X
\end{tikzcd}
\end{equation}
\end{center}

\begin{remark}\label{Cancelling tensor and epsilon}
Let $u:\A\rightarrow\B$ be a functor. For any closed $\E$-module $\D$, and any $X\in\D(\B)$, the canonical isomorphism 
\begin{center}
\begin{tikzcd}
\E^{\A\op}\arrow[ddrr,bend right,"-\otimes_\A u^*X" below left,""{name=L,above right}]\arrow[rr,"(u\op)_!" above] && \E^{\B\op}\arrow[Leftarrow,to=L,"\iso" left=5,shorten >=0.3cm,shorten <=0.8cm] \arrow[dd,"-\otimes_\B X" right] 
\\\\
&&\D
\end{tikzcd}
\end{center}
of (\ref{Diagram cancelling tensor and left Kan extension}) is conjugate to the isomorphism 
\begin{center}
\begin{tikzcd}
\E^{\A\op}\arrow[leftarrow,ddrr,bend right,"\widetilde{\map}_{\D}(u^*X\text{,}-)" below left,""{name=L,above right}]\arrow[leftarrow,rr,"(u\op)^*" above] && \E^{\B\op}\arrow[Rightarrow,to=L,"\iso" left=5,"\gamma" below=4,shorten >=0.3cm,shorten <=0.8cm] \arrow[leftarrow,dd,"\widetilde{\map}_{\D}(X\text{,}-)" right] 
\\\\
&&\D
\end{tikzcd}
\end{center}
induced by the structure isomorphism for $\map_{\D}(-,-)$. This appears in the proof of \cite[Theorem 9.1]{GPS14}, but it can also be checked directly. Using this fact and the description of $\otimes_u$ in (\ref{Diagram tensor cancelling over u}), for any category $\C$, and any $Y\in\D(\C)$, the diagram below commutes:
\begin{center}
\begin{tikzcd}
(u\op\times\C)^*\widetilde{\map}_\D(X,Y)\otimes_{\A} u^*X\arrow{rr}[above]{\gamma\;\mathsmaller{\otimes_{\A}} u^*X}\arrow{dd}[left]{\widetilde{\map}_\D(X,Y)\otimes_u X} && 
 \widetilde{\map}_\D(u^*X,Y)\otimes_{\A} u^*X\arrow{dd}[right]{\epsilon} \\\\
 \widetilde{\map}_\D(X,Y)\otimes_{\B} X\arrow{rr}[below]{\epsilon} && 
Y
\end{tikzcd}
\end{center} 
\end{remark}

\begin{remark}\label{Maps respects delta u}
Let $u:\A\rightarrow\B$ be a functor, and let $F:\D_1\rightarrow\D_2$ be a derivator map. Using~\cref{Derivator maps respect homotopy exact squares}, it follows that $F$ respects $\int^u$ and $\partial_u$. For example, in the case of $\partial_u$, this means that the diagram below commutes, for any $X\in\D_1(\0)$:
\begin{center}
\begin{tikzcd}
\partial_\A FX\arrow[dd,"\partial_u FX" left]\arrow{rr} && F\partial_\A X\arrow[dd,"F\partial_u X" right]\\\\
(u\op\times u)^*\partial_\B FX\arrow[r] & (u\op\times u)^* F\partial_\B X\arrow[r,"\iso" below] & F(u\op\times u)^*\partial_\B X
\end{tikzcd}
\end{center}
Similarly, suppose we have $\E$-modules $\D_1$ and $\D_2$, and a cocontinuous $\E$-module morphism $F:\D_1\rightarrow\D_2$. For any categories $\A$, $\B$, $\C$ and $\cD$, any functor $u:\A\rightarrow\B$, and objects $X\in\D_1(\C\op\times\B)$ and $Y\in\E(\B\op\times\cD)$, the diagram below commutes:
\begin{center}
\begin{tikzcd}[column sep=small]
(u\op\times\cD)^*Y\otimes_\A (\C\op\times u)^*FX\arrow{dddd}[left]{Y\otimes_u FX} \arrow{rr}[above]{\iso} && (u\op\times\cD)^*Y\otimes_\A F(\C\op\times u)^*X\arrow{dd}[right]{\varphi} \\\\
 
&& F((u\op\times\cD)^*Y\otimes_\A (\C\op\times u)^*X)\arrow{dd}[right]{F(Y\otimes_u X)}\\\\
 
Y\otimes_\B FX\arrow{rr}[below]{\varphi} && F(Y\otimes_\B X)
\end{tikzcd}
\end{center} 
The map $\varphi$ is the canonical isomorphism of~\cref{Cocontinuous E-module maps preserve the cancelling product}.
\end{remark}

\begin{example}\label{Cancelling tensor and associativity on morphism level}
Let $\D$ be a closed $\E$-module and let $X\in\D(\A)$. Using~\cref{Coherence for the cancelling version of the closed E-action}, the derivator map $-\otimes_\A X:\E^{\A\op}\rightarrow\D$ is an $\E$-module map. Applying~\cref{Maps respects delta u}, for any functor $v:\B\rightarrow\C$, and any objects $Y\in\E(\A\op\times\C)$ and $Z\in\E(\C\op\times\cD)$, we obtain the following commutative diagram:
\begin{center}
\begin{tikzcd}[column sep=small]
((v\op\times \cD)^*Z\otimes_\B (\A\op\times v)^*Y)\otimes_\A X\arrow{dddd}[left]{(Z\otimes_v Y)\otimes_\A X} \arrow{rr}[above]{\alpha} && (v\op\times \cD)^*Z\otimes_\B( (\A\op\times v)^*Y\otimes_\A X)\arrow{dd}[right]{\iso} \\\\
 
&& (v\op\times \cD)^*Z\otimes_\B v^*(Y\otimes_\A X)\arrow{dd}[right]{Z\otimes_v (Y\otimes_\A X)}\\\\
 
(Z\otimes_\C Y)\otimes_\A X\arrow[rr,"\alpha" below]
&& Z\otimes_\C (Y\otimes_\A X)
\end{tikzcd}
\end{center} 
\end{example}

Suppose we have a functor $u:\A\rightarrow\B$. The following lemma shows that the isomorphisms of~\cref{Interaction of delta and integral} respect the canonical morphism $\int^{\A_{1,2}}\partial_{\A_{2,3}}\Rightarrow\int^{\B_{1,2}}\partial_{\B_{2,3}}$ induced by $u$. We use this to prove~\cref{Interaction of cancelling tensor with h on the morphism level}, which expresses the coherence between $\h_u$ and $\otimes_u$ in an $\E$-module.

\begin{proposition}\label{Interaction of delta and integral on the morphism level}
For any derivator $\D$, and any category $\A=\A_i$, consider the isomorphism

\begin{center}
\begin{tikzcd}[row sep=tiny]
 && 
\D^{\A_1\times\A_2\op\times\A_3} \arrow[rrd,"\int^{\A_{1,2}}" above right, bend left=10]
\\
\D^{\A_1}\arrow[rrrr,equal,bend right=20,""{name=D}]\arrow[urr,"\partial_{\A_{2,3}}" above left,bend left=10]&&&&\D^{\A_3} 
\arrow[Rightarrow,from=1-3,to=D,"\iso",shorten >=0.2cm,shorten <=0.2cm]
\end{tikzcd}
\end{center}
of~\cref{Interaction of delta and integral}. Given a functor $u:\A\rightarrow\B$, the pasting diagram

\begin{center}
\begin{tikzcd}

\D^{\B_1}\arrow[dddrr,"\partial_{\A}" below left,bend right=10,""{name=L,above right}]\arrow[dddd,"u^*" left]\arrow[rrrrrr,equal,bend left=30,""{name=U}]\arrow[rrr,"\partial_{\B_{2,3}}" above]&&&
\D^{\B_1\times\B_2\op\times\B_3}\arrow[dddl,"\mathsmaller{(\B\times u\op\times u)^*}" below right,bend right=10]
\arrow[dr,"\mathsmaller{(\B\times\B\op\times u)^*}" below left]\arrow[Rightarrow,to=U,"\iso",shorten >=0.4cm,shorten <=0.4cm]\arrow[rrr,"\int^{\B_{1,2}}" above] 
&&&\D^{\B_3}\arrow[dddd,"u^*" right] 
\\
&&&& \D^{\B\times\B\op\times\A}\arrow[Rightarrow,rru,"\iso" below right,"\gamma^{-1}" above left,shorten >=0.7cm,shorten <=0.4cm,shift right=3]
\arrow[dddl,"\mathsmaller{(u\times u\op\times\A)^*}" above left,bend left=10]
\arrow[dddrr,"\int^{\B}" above right,bend left=10,""{name=R,below left}]
\\
\\
&& \D^{\B\times\A\op\times\A}\arrow[dr,"\mathsmaller{(u\times\A\op\times\A)^*}" above right]
\\
\D^{\A_1}\arrow[Rightarrow,rru,"\iso" below right,"\gamma^{-1}" above left,shorten >=0.7cm,shorten <=0.7cm,shift left=3]
\arrow[rrrrrr,equal,bend right=30,""{name=D}]\arrow[rrr,"\partial_{\A_{2,3}}" below]&&&
\D^{\A_1\times\A_2\op\times\A_3}\arrow[Leftarrow,to=D,"\iso" left,shorten >=0.4cm,shorten <=0.4cm]\arrow[rrr,"\int^{\A_{1,2}}" below] 
&&&\D^{\A_3} 
\arrow[Rightarrow,from=L,to=1-4,"\partial_{u}" above left,shorten >=2cm,shorten <=1.7cm]
\arrow[Rightarrow,from=5-4,to=R,"\int^{u}" above left,shorten >=1.9cm,shorten <=1.4cm,shift right]

\end{tikzcd}
\end{center}
reduces to $\id_{u^*}$.

\end{proposition}
\begin{proof}
Expanding the maps $\partial_u$ and $\int^u$ using~\cref{Delta and end on morphisms}, the large rectangle in the pasting diagram above reduces to the following:
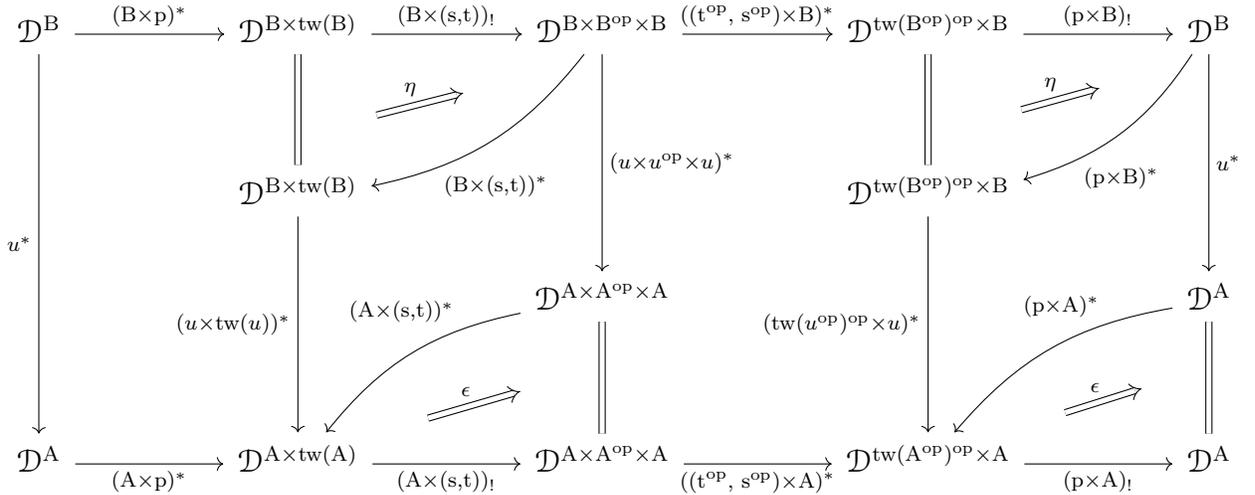
\begin{center}
\begin{tikzcd}

\D^{\B}\arrow[ddddd,"u^*" left]\arrow[rr,"(\B\times\p)^*" above]&&
\D^{\B\times\tw(\B)}\arrow[dd,equal,""{name=L,right}]\arrow[rr,"(\B\times\text{(s,t)})_!" above] &&
\D^{\B\times\B\op\times\B}\arrow[ddll,bend left=20,"(\B\times\text{(s,t)})^*" below right,near end]
\arrow[ddd,"(u\times u\op\times u)^*" right]\arrow[rr,"\text{((t}\op\text{, s}\op)\times\B)^*" above] 
&&
\D^{\tw(\B\op)\op\times\B}\arrow[rr,"(\p\times\B)_!" above]
\arrow[dd,equal,""{name=U,right}]
&&\D^{\B}\arrow[ddd,"u^*" right]
\arrow[ddll,bend left=20,"(\p\times\B)^*" below right,near end] 
\\
\\
&&\D^{\B\times\tw(\B)}\arrow[ddd,"(u\times\tw(u))^*" left] &&&& \D^{\tw(\B\op)\op\times\B}\arrow[ddd,"(\tw(u\op)\op\times u)^*" left]
\\
&&&&\D^{\A\times\A\op\times\A}\arrow[ddll,bend right=20,"(\A\times\text{(s,t)})^*" above left,near start]
\arrow[dd,equal,""{name=R,left}] &&&& \D^{\A}\arrow[ddll,bend right=20,"(\p\times\A)^*" above left,near start]\arrow[dd,equal,""{name=D,left}]
\\
\\
\D^{\A}\arrow[rr,"(\A\times\p)^*" below]&&
\D^{\A\times\tw(\A)}\arrow[rr,"(\A\times\text{(s,t)})_!" below]&&
\D^{\A\times\A\op\times\A}\arrow[rr,"\text{((t}\op\text{, s}\op)\times\A)^*" below] 
&&
\D^{\tw(\A\op)\op\times\A}\arrow[rr,"(\p\times\A)_!" below]
&&\D^{\A} 

\arrow[Rightarrow,from=L,to=1-5,shift right=3,"\eta" above left,shorten >=0.9cm,shorten <=0.8cm]
\arrow[Rightarrow,from=U,to=1-9,shift right=3,"\eta" above left,shorten >=1.1cm,shorten <=1cm]
\arrow[Rightarrow,from=6-3,to=R,shift left,"\epsilon" above left,shorten >=0.9cm,shorten <=0.8cm]
\arrow[Rightarrow,from=6-7,to=D,shift left,"\epsilon" above left,shorten >=0.7cm,shorten <=0.6cm]
\end{tikzcd}
\end{center}

Recall the construction of the isomorphisms $\int^{\A_{1,2}}\partial_{\A_{2,3}}\iso\id$ and $\int^{\B_{1,2}}\partial_{\B_{2,3}}\iso\id$ described in~\cref{Interaction of delta and integral}, in particular the pullbacks $\mathrm{P}_\A$ and $\mathrm{P}_\B$. The functor $u:\A\rightarrow\B$ induces a map $\mathrm{P}_u:\mathrm{P}_\A\rightarrow\mathrm{P}_\B$ making the diagram below commute:
\begin{center}
\begin{equation}\label{Diagram Definition of P u}
\begin{tikzcd}[baseline=(current  bounding  box.center),column sep=tiny]
\mathrm{P}_\A\arrow[dd,"\mathrm{n}_\A" left]\arrow[rr,"\mathrm{m}_\A" above]\arrow[dr,"\mathrm{P}_u" above right]
&& \A\times\tw(\A)\arrow[dr,bend left=20,"u\times\tw(u)" above right] \\
& \mathrm{P}_\B\arrow[rrdd,phantom,"\lrcorner" very near start] \arrow{rr}[above]{\mathrm{m}_\B}\arrow{dd}[left]{\mathrm{n}_\B} && 
 \B\times\tw(\B)\arrow{dd}[right]{\B\times (\mathrm{s},\mathrm{t})} \\
\tw(\A\op)\op\times\A\arrow[dr,bend right=22,"\tw(u\op)\op\times u" below left]  
 \\
& \tw(\B\op)\op\times\B \arrow{rr}[below]{(\mathrm{t}\op,\mathrm{s}\op)\times\B} && 
 \B\times\B\op\times\B
\end{tikzcd}
\end{equation}
\end{center}
Explicitly, this functor takes an object $a\xrightarrow{\;\;g\;\;} b\xrightarrow{\;\;f\;\;} c$ in $\mathrm{P}_\A$ to $u(a)\xrightarrow{\;\;u(g)\;\;} u(b)\xrightarrow{\;\;u(f)\;\;} u(c)$ in $\mathrm{P}_\B$. If we paste the isomorphisms
\[
(\mathrm{n}_\A)_!\circ (\mathrm{m}_\A)^*\iso ((\mathrm{t}\op,\mathrm{s}\op)\times\A)^*\circ (\A\times (\mathrm{s},\mathrm{t}))_!
\]
\[
(\mathrm{n}_\B)_!\circ (\mathrm{m}_\B)^*\iso ((\mathrm{t}\op,\mathrm{s}\op)\times\B)^*\circ (\B\times (\mathrm{s},\mathrm{t}))_!
\]
onto the top and bottom of the pasting diagram above, then the commutative diagram (\ref{Diagram Definition of P u}) allows us to simplify and obtain the following:

\begin{center}
\begin{tikzcd}

\D^{\B}\arrow[dddddd,"u^*" left]\arrow[rr,"(\B\times\p)^*" above]&&
\D^{\B\times\tw(\B)}\arrow[dddddd,"(u\times\tw(u))^*" right]\arrow[rr,"(\mathrm{m}_\B)^*" above] &&
\D^{\mathrm{P}_\B}\arrow[dd,equal,""{name=A}]\arrow[rr,"(\mathrm{n}_\B)_!" above] 
&&
\D^{\tw(\B\op)\op\times\B}\arrow[rr,"(\p\times\B)_!" above]
\arrow[dd,equal,""{name=U,right}]
&&\D^{\B}\arrow[dddd,"u^*" right]
\arrow[ddll,bend left=20,"(\p\times\B)^*" below right,near end] 
\\
\\
&&&&\D^{\mathrm{P}_\B}\arrow[dddd,"(\mathrm{P}_u)^*" left] && \D^{\tw(\B\op)\op\times\B}\arrow[dd,"(\tw(u\op)\op\times u)^*" left]\arrow[ll,"(\mathrm{n}_\B)^*" above]
\\\\
&&&&&&\D^{\tw(\A\op)\op\times\A}\arrow[ddll,bend right=20,"(\mathrm{n}_\A)^*" above left,near start]
\arrow[dd,equal,""{name=R,left}] && \D^{\A}\arrow[ll,"(\p\times\A)^*" above]\arrow[dd,equal,""{name=D,left}]
\\
\\
\D^{\A}\arrow[rr,"(\A\times\p)^*" below]&&
\D^{\A\times\tw(\A)}\arrow[rr,"(\mathrm{m}_\A)^*" below]&&
\D^{\mathrm{P}_\A}\arrow[rr,"(\mathrm{n}_\A)_!" below] 
&&
\D^{\tw(\A\op)\op\times\A}\arrow[rr,"(\p\times\A)_!" below]
&&\D^{\A}

\arrow[Rightarrow,from=A,to=U,"\eta" above,shorten >=1.4cm,shorten <=1.5cm]
\arrow[Rightarrow,from=U,to=1-9,shift right=3,"\eta" above left,shorten >=1.1cm,shorten <=1cm]
\arrow[Rightarrow,from=7-5,to=R,shift left,"\epsilon" above left,shorten >=1cm,shorten <=1cm]
\arrow[Rightarrow,from=R,to=D,left,"\epsilon" above,shorten >=1.3cm,shorten <=1.4cm]
\end{tikzcd}
\end{center}
We may now paste the remaining isomorphisms onto the top and bottom of this rectangle. The resultant pasting diagram reduces easily to the identity, using the fact that the two pasting diagrams below are equal:
\begin{center}
\begin{tikzcd}
 \mathrm{P}_\A\arrow{r}[above]{\mathrm{m}_\A}\arrow{d}[left]{\mathrm{n}_\A} & 
 \A\times\tw(\A)\arrow{r}[above]{\A\times \p} & \A\arrow[Rightarrow,ddll,"\vartheta_\A\;\;\;" above=2, shorten >=2.5cm,shorten <=2.5cm]\arrow[dd,equal] & \mathrm{P}_\A\arrow[d,"\mathrm{P}_u" right] \\
 \tw(\A\op)\op\times\A\arrow{d}[left]{\p\times\A} &&& \mathrm{P}_\B \arrow{r}[above]{\mathrm{m}_\B}\arrow{d}[left]{\mathrm{n}_\B} & 
 \B\times\tw(\B)\arrow{r}[above]{\B\times \p} & \B\arrow[dd,equal]\arrow[Rightarrow,ddll,"\vartheta_\B\;\;\;" above=2, shorten >=2.5cm,shorten <=2.5cm]
 \\
 \A\arrow[rr,equal] && 
 \A\arrow[d,"u" right] & \tw(\B\op)\op\times\B\arrow{d}[left]{\p\times\B}  \\
 && \B & \B\arrow[rr,equal] && \B
\end{tikzcd}
\end{center}
To see this, note that the component of the first at the object $a\xrightarrow{\;\;g\;\;} b\xrightarrow{\;\;f\;\;} c$ in $\mathrm{P}_\A$ is $u(f\circ g)$, and the component of the second is $u(f)\circ u(g)$.
\end{proof}

\begin{lemma}\label{Interaction of cancelling tensor with h on the morphism level}
Let $\D$ be an $\E$-module, let $u:\A\rightarrow\B$ be a functor, and let $X\in\D(\B)$. The diagram below commutes:
\begin{center}
\begin{tikzcd}
\h_\A\otimes_\A u^*X\arrow{dddd}[left]{\lambda} \arrow{rr}[above]{\h_u\otimes_\A u^*X} && (u\op\times u)^*\h_\B\otimes_\A u^*X\arrow{dd}[right]{(\B\op\times u)^*\h_\B\otimes_u X} \\\\
 
&& (\B\op\times u)^*\h_\B\otimes_\B X\arrow{dd}[right]{\iso}\\\\
 
u^*X 
&& u^*(\h_\B\otimes_\B X)\arrow{ll}[below]{u^*(\lambda)}
\end{tikzcd}
\end{center} 
\end{lemma}
\begin{proof}
To see that this diagram commutes, use the definition of $\lambda$ in~\cref{Coherence for the cancelling version of the closed E-action}, and the definitions of $\h_u$ and $\otimes_u$, to rewrite the diagram in terms of $\partial_u$ and $\int^u$. Using cocontinuity and pseudonaturality of $\otimes$, we can then pull the instances of $\partial_u$, $\int^u$ and $u^*$ out of the tensor product. From here, the commutativity follows immediately from~\cref{Interaction of delta and integral on the morphism level}.
\end{proof}

\section{Cotensors in closed $\E$-modules}\label{Section Cotensors in closed E-modules}

In this section, we consider cotensors in a closed $\E$-module $\D$. If $\E$ is symmetric monoidal, we show that these are part of a closed $\E$-action on the opposite derivator $\D\op$. The main result of this section is~\cref{Left adjoint preserves tensors iff right adjoint preserves cotensors}, which proves a derivator analogue of the well-known fact that a left adjoint preserves tensors if and only if its right adjoint preserves cotensors. In particular, this result implies that, for any $X\in\D(\A)$, the morphism $\widetilde{\map}_{\D}(X,-):\D\rightarrow\E^{\A\op}$ preserves cotensors. This is the content of~\cref{map(X -) preserves cotensors}. Throughout this section, let $\E$ be a symmetric monoidal derivator.

\begin{example}\label{Opposite of a closed E-module is an E-module}
If $\D$ is a closed $\E$-module, then $\D\op$ is a (right) $\E$-module via the map 
\[
\vartriangleleft\op:\D\op\times\E\rightarrow\D\op.
\]
See~\cite[Section 4]{GS17}. The structure isomorphisms are induced, using the adjunctions of~\cref{2-variable adjunction characterisation}, by the coherent isomorphisms for the cancelling tensor of~\cref{Coherence for the cancelling version of the closed E-action}. Using the symmetry isomorphism for $\E$, this right $\E$-module structure induces a left $\E$-module structure. The $\E$-module structure is closed, with the required adjoints given by~\cref{2-variable adjunction characterisation}.
\end{example}

\begin{remark}\label{mapping space in the opposite closed E-module}
Let $\D$ be a closed $\E$-module, let $\A$ be a category, and let $X\in\D(\A)$. Using the descriptions of the various adjoints in~\cref{2-variable adjunction characterisation}, we have an isomorphism
\[
\widetilde{\map}_{\D\op}(X,-)\iso\widetilde{\map}_{\D}(-,X):\D\op\rightarrow\E^{\A}.
\]
\end{remark}

\begin{definition}
Let $\D_1$ and $\D_2$ be closed $\E$-modules. We say that a derivator map $G:\D_1\rightarrow\D_2$ \textbf{preserves cotensors} if its opposite $G\op:\D_1\op\rightarrow\D_2\op$ is an $\E$-module map for the $\E$-module structures of~\cref{Opposite of a closed E-module is an E-module}. Similarly, a modification between cotensor-preserving maps is said to \textbf{respect cotensors} if its opposite respects tensors.
\end{definition}

\begin{example}\label{u* preserves cotensors}
Let $\D$ be a closed $\E$-module and let $u:\A\rightarrow\B$ be a functor. Consider the derivator map $u^*:\D^\B\rightarrow\D^\A$. The opposite of this map is equal to
\[
(u\op)^*:(\D\op)^{\B\op}\rightarrow(\D\op)^{\A\op}.
\]
By~\cref{u* and alpha* respect the E-action}, this is an $\E$-module map. Thus, $u^*:\D^\B\rightarrow\D^\A$ preserves cotensors as well as tensors.
\end{example}

\begin{proposition}\label{Left adjoint preserves tensors iff right adjoint preserves cotensors}
Let $\D_1$ and $\D_2$ be closed $\E$-modules, and suppose we have an adjunction:
\begin{center}
\begin{tikzcd}
\D_1\arrow[rr, bend left=40, "F" above,""{name=U, below}]\arrow[rr,leftarrow, bend right=40, "R" below, ""{name=D}]
&& \D_2
\arrow[phantom,from=U,to=D,"\bot"]
\end{tikzcd}
\end{center}
Then $F$ preserves tensors if and only if $R$ preserves cotensors.
\end{proposition}
\begin{proof}
If the left adjoint $F$ preserves tensors, we will show that the right adjoint $R$ preserves cotensors. Once we prove this, applying the result to the opposite adjunction 
\begin{center}
\begin{tikzcd}
(\D_2)\op\arrow[rr, bend left=40, "R\op" above,""{name=U, below}]\arrow[rr,leftarrow, bend right=40, "F\op" below, ""{name=D}]
&& (\D_1)\op
\arrow[phantom,from=U,to=D,"\bot"]
\end{tikzcd}
\end{center}
proves the reverse implication as well.

Suppose $F$ preserves tensors. Taking the opposite of the structure isomorphism in~\cref{E-module map definition}, we require a coherent isomorphism:
\begin{center}
\begin{tikzcd}
 \D_2\times\E\op \arrow{rr}[above]{R\times \E\op}\arrow{dd}[left]{\vartriangleleft} && 
 \D_1\times\E\op \arrow{dd}[right]{\vartriangleleft} \\
 
 \\
 
 \D_2 \arrow{rr}[below]{R}\arrow[uurr,Rightarrow,"\iso" above left,shorten >=1.3cm,shorten <=1.3cm] && 
 \D_1

\end{tikzcd}
\end{center} 
Using~\cref{Cartesian closure of PDer}, to give the desired isomorphism, we may equivalently provide isomorphisms
\begin{center}
\begin{tikzcd}
 \D_2 \arrow{rr}[above]{R}\arrow{dd}[left]{-\;\widetilde{\vartriangleleft}\;X} && 
 \D_1\arrow{dd}[right]{-\;\widetilde{\vartriangleleft}\;X} \\
 
 \\
 
 \D_2^{\B\op} \arrow{rr}[below]{R}\arrow[uurr,Rightarrow,"\iso" above left,shorten >=1cm,shorten <=1cm] && 
 \D_1^{\B\op}

\end{tikzcd}
\end{center} 
for each category $\B$, and each $X\in\E(\B)$, such that these isomorphisms organise into a modification in $X$. Take these to be conjugate to the isomorphism between left adjoints:
\begin{center}
\begin{tikzcd}
 \D_2  && 
 \D_1\arrow[ll,"F" above] \\
 
 \\
 
 \D_2^{\B\op} \arrow[uu,"X\otimes_{\B\op}-" left]\arrow[uurr,Leftarrow,"\iso" above left,"\varphi^{-1}" below right,shorten >=1cm,shorten <=1cm] && 
 \D_1^{\B\op}\arrow[ll,"F" below]\arrow[uu,"X\otimes_{\B\op}-" right]

\end{tikzcd}
\end{center} 
This isomorphism is defined in~\cref{Cocontinuous E-module maps preserve the cancelling product}. The description of the left adjoints follows from~\cref{2-variable adjunction characterisation}.

Let $u:\A\rightarrow\B$ be a functor. To see that the isomorphisms above form a modification in $X$, we need the pasting diagrams below to be equal:
\begin{center}
\begin{tikzcd}
&& \D_2^{\A\op}\arrow[rr,"R" above] && \D_1^{\A\op} \\\\
 \D_2\arrow[bend right,ddrr,"R" below left]
 \arrow[bend left,rruu,"-\;\widetilde{\vartriangleleft}\;u^*X" above left,""{name=C}]
 \arrow[rr,"-\;\widetilde{\vartriangleleft}\;X" below] && 
 \D_2^{\B\op}\arrow[uu,"(u\op)^*" right]\arrow[rr,"R" below] 
 && \D_1^{\B\op}\arrow[uu,"(u\op)^*" right] \arrow[Rightarrow,lluu,"\iso" above right,shorten >=0.8cm,shorten <=0.9cm,shift right]
 \\\\
&& \D_1 \arrow[rruu,bend right,"-\;\widetilde{\vartriangleleft}\;X" below right]

\arrow[Leftarrow,from=C,to=5-3,shorten >=2.8cm,shorten <=0.5cm,"\iso" above right, near start, shift left=4]
\arrow[Leftarrow,from=C,to=5-3,shorten >=0.3cm,shorten <=3cm,"\iso" above right, very near end, shift left=4]
\end{tikzcd}
\hspace{0.5cm}
\begin{tikzcd}
&& \D_2^{\A\op}\arrow[rr,"R" above] && \D_1^{\A\op} \\\\
 \D_2\arrow[bend right,ddrr,"R" below left]
 \arrow[bend left,rruu,"-\;\widetilde{\vartriangleleft}\;u^*X" above left,""{name=T}]
 &&&& \D_1^{\B\op}\arrow[uu,"(u\op)^*" right] 
 \\\\
&& \D_1 \arrow[rruuuu,bend left,"-\;\widetilde{\vartriangleleft}\;u^*X" above left]\arrow[rruu,bend right,"-\;\widetilde{\vartriangleleft}\;X" below right,""{name=C}]
\arrow[Rightarrow,from=C,to=1-3,shorten >=2.3cm,shorten <=0.8cm,"\iso" above right, near start,shift right]
\arrow[Rightarrow,from=5-3,to=T,shorten >=1.4cm,shorten <=1.7cm,"\iso" above right, shift left]
\end{tikzcd}

\end{center} 
We can check this equality by taking mates under the adjunctions $F(X\otimes_{\B\op}-)\dashv R(-)\;\widetilde{\vartriangleleft}\;X$ and $u^*X\otimes_{\A\op}F(-)\dashv R(-\;\widetilde{\vartriangleleft}\;u^*X)$. 

Using the definition of the isomorphism $R(-\;\widetilde{\vartriangleleft}\;X)\iso R(-)\;\widetilde{\vartriangleleft}\;X$ above, and the fact that
\begin{center}
\begin{tikzcd}
&& \D_2^{\A\op}\\\\
 \D_2
 \arrow[bend left,rruu,"-\;\widetilde{\vartriangleleft}\;u^*X" above left,""{name=C}]
 \arrow[rr,"-\;\widetilde{\vartriangleleft}\;X" below] && 
 \D_2^{\B\op}\arrow[uu,"(u\op)^*" right]
\arrow[Leftarrow,from=C,to=3-3,shorten >=0.6cm,shorten <=0.5cm,"\iso" above right]
\end{tikzcd}
\end{center} 
is conjugate to the canonical isomorphism
\begin{center}
\begin{tikzcd}
&& \D_2^{\A\op}\\\\
 \D_2
 \arrow[leftarrow,bend left,rruu,"u^*X\mathsmaller{\otimes_{\A\op}}-" above left,""{name=C}]
 \arrow[leftarrow,rr,"X\mathsmaller{\otimes_{\B\op}}-" below] && 
 \D_2^{\B\op}\arrow[leftarrow,uu,"(u\op)_!" right]
\arrow[Rightarrow,from=C,to=3-3,shorten >=0.6cm,shorten <=0.5cm,"\iso" above right]
\end{tikzcd}
\end{center} 
the equality follows by the commutative diagram in~\cref{Maps respects delta u} for $F$.

The coherence conditions for $R$ follow from the coherence conditions for $F$, once again using conjugacy.
\end{proof}

\begin{lemma}\label{Modification respects tensors iff conjugate respects cotensors}
Suppose we have closed $\E$-modules $\D_1$ and $\D_2$, and a pair of adjunctions
\begin{center}
\begin{tikzcd}
\D_1\arrow[rr, bend left=40, "F_1" above,""{name=U, below}]\arrow[rr,leftarrow, bend right=40, "R_1" below, ""{name=D}]
&& \D_2
\arrow[phantom,from=U,to=D,"\bot"]
\end{tikzcd}
\hspace{1.5cm}
\begin{tikzcd}
\D_1\arrow[rr, bend left=40, "F_2" above,""{name=U, below}]\arrow[rr,leftarrow, bend right=40, "R_2" below, ""{name=D}]
&& \D_2
\arrow[phantom,from=U,to=D,"\bot"]
\end{tikzcd}
\end{center}
such that the left adjoints $F_1$ and $F_2$ preserve tensors. Then a modification $\theta:F_1\Rightarrow F_2$ respects tensors if and only if its conjugate $\vartheta:R_2\Rightarrow R_1$ respects cotensors.
\end{lemma}
\begin{proof}
Let $\theta:F_1\Rightarrow F_2$ be a modification. Suppose $\theta$ respects tensors. We will show that its conjugate $\vartheta:R_2\Rightarrow R_1$ respects cotensors. As in the proof of~\cref{Left adjoint preserves tensors iff right adjoint preserves cotensors}, the converse follows from this result by taking opposites.

Let $\A$ be a category and let $X\in\E(\A)$. We need to show that the pasting diagrams below are equal:
\begin{center}
\begin{tikzcd}
 \D_2 \arrow[rr,"-\;\widetilde{\vartriangleleft}\;X"]\arrow[dd,"R_1" right,bend left=40,""{name=R,left}]\arrow[dd,"R_2" left,bend right=40,""{name=L,right}] && 
 \D_2^{\A\op}\arrow[Leftarrow,ddll,"\iso" above=3,shorten >=0.9cm,shorten <=0.7cm,shift left=2] \arrow[dd,"R_1"] 
 && \D_2 \arrow[rr,"-\;\widetilde{\vartriangleleft}\;X"]\arrow[dd,"R_2" left] 
 && \D_2^{\A\op}\arrow[Leftarrow,ddll,"\iso" above left,shorten >=0.8cm,shorten <=0.9cm, shift right=2] \arrow[dd,"R_2" left,bend right=40,""{name=U,right}]\arrow[dd,"R_1" right,bend left=40,""{name=D,left}]
 \\\\
\D_1\arrow[rr,"-\;\widetilde{\vartriangleleft}\;X" below] 
&& \D_1^{\A\op}
&& \D_1\arrow[rr,"-\;\widetilde{\vartriangleleft}\;X" below] 
&& \D_1^{\A\op}
\arrow[Leftarrow,from=R,to=L,shorten >=0.1cm,shorten <=0.1cm,"\vartheta" above,shift left]
\arrow[Leftarrow,from=D,to=U,shorten >=0.1cm,shorten <=0.1cm,"\vartheta" above,shift left]
\end{tikzcd}
\end{center}
Taking conjugates, and using the definition of the isomorphisms $R_i(-\;\widetilde{\vartriangleleft}\;X)\iso R_i(-)\;\widetilde{\vartriangleleft}\;X$ from the proof of~\cref{Left adjoint preserves tensors iff right adjoint preserves cotensors}, these pasting diagrams are equal if and only if the pasting diagrams below are equal:
\begin{center}
\begin{tikzcd}
 \D_2 \arrow[leftarrow,rr,"X\otimes_{\A\op}-"]\arrow[leftarrow,dd,"F_1" right,bend left=40,""{name=R,left}]\arrow[leftarrow,dd,"F_2" left,bend right=40,""{name=L,right}] && 
 \D_2^{\A\op}\arrow[Rightarrow,ddll,"\iso" above=3,"\varphi" below right,shorten >=0.9cm,shorten <=0.7cm,shift left=2] \arrow[leftarrow,dd,"F_1"] 
 && \D_2 \arrow[leftarrow,rr,"X\otimes_{\A\op}-"]\arrow[leftarrow,dd,"F_2" left] 
 && \D_2^{\A\op}\arrow[Rightarrow,ddll,"\iso" above left,"\varphi" below right,shorten >=0.8cm,shorten <=0.9cm, shift right=2] \arrow[leftarrow,dd,"F_2" left,bend right=40,""{name=U,right}]\arrow[leftarrow,dd,"F_1" right,bend left=40,""{name=D,left}]
 \\\\
\D_1\arrow[leftarrow,rr,"X\otimes_{\A\op}-" below] 
&& \D_1^{\A\op}
&& \D_1\arrow[leftarrow,rr,"X\otimes_{\A\op}-" below] 
&& \D_1^{\A\op}
\arrow[Rightarrow,from=R,to=L,shorten >=0.1cm,shorten <=0.1cm,"\theta" above,shift left]
\arrow[Rightarrow,from=D,to=U,shorten >=0.1cm,shorten <=0.1cm,"\theta" above,shift left]
\end{tikzcd}
\end{center}
Since $\theta$ respects tensors, these two pasting diagrams are equal by~\cref{E-module modifications respect the cancelling product}.
\end{proof}

\begin{example}\label{map(X -) preserves cotensors}
Let $\D$ be a closed $\E$-module, and let $X\in\D(\A)$. The derivator map $-\otimes_\A X:\E^{\A\op}\rightarrow\D$ preserves tensors. To see this, note that the external component of the structure isomorphism from~\cref{E-module map definition} has the form $(W\widetilde{\otimes}Z)\otimes_\A X\xrightarrow{\;\;\iso\;\;}W\widetilde{\otimes}(Z\otimes_\A X)$, for any $Z\in\E(\A\op\times\B)$ and $W\in\E(\C)$. This is given by the composite below:
\begin{center}
\begin{tikzcd}[column sep=40]
\int^\A(W\widetilde{\otimes} Z)\widetilde{\otimes} X \arrow[r,"\int^\A \alpha"]
& \int^\A W\widetilde{\otimes} (Z\widetilde{\otimes} X) \arrow[r,"\iso"] &  W\widetilde{\otimes} \int^\A(Z\widetilde{\otimes} X) 
\end{tikzcd}
\end{center}
The second isomorphism in the composite above comes from the cocontinuity of $\otimes$. The internal component of this map can be recovered using \cref{Bimorphisms}. By \cref{Left adjoint preserves tensors iff right adjoint preserves cotensors}, it follows that 
\[
\widetilde{\map}_{\D}(X,-):\D\rightarrow\E^{\A\op}
\]
preserves cotensors. 
\end{example}

\begin{example}
Let $\D$ be a closed $\E$-module and let $u:\A\rightarrow\B$ be a functor. By \cref{u* and alpha* respect the E-action} and~\cref{u* preserves cotensors}, the map $u^*:\D^\B\rightarrow\D^\A$ preserves both tensors and cotensors. Thus, by \cref{Left adjoint preserves tensors iff right adjoint preserves cotensors}, its left adjoint $u_!:\D^\A\rightarrow\D^\B$ preserves tensors, and its right adjoint $u_*:\D^\A\rightarrow\D^\B$ preserves cotensors.
\end{example}

\section{A representability theorem for triangulated $\E$-modules}\label{Section A representability theorem for triangulated E-modules}

We begin this section with a characterisation of cocontinuous $\E$-module maps from $\E^{\A\op}$ to a closed $\E$-module $\D$; using the results of \cref{Section The cancelling tensor product}, it is easy to see that any such map is isomorphic to $-\otimes_\A X:\E^{\A\op}\rightarrow\D$, for some object $X\in\D(\A)$. When $\E$ is triangulated, we can combine this observation with \cref{Left adjoint preserves tensors iff right adjoint preserves cotensors}, to prove a representability theorem for triangulated closed $\E$-modules. This is \cref{Representability for perfectly generated triangulated derivator}, the final result of this chapter.

\begin{proposition}\label{Universal property of E shifted by A}
Let $\D$ be an $\E$-module and let $\A$ be a category. The functor 
\[
\Hom^\E_!(\E^{\A\op},\D)\rightarrow\D(\A),
\]
given by evaluation at $\h_\A\in\E(\A\op\times\A)$, is an equivalence.
\end{proposition}
\begin{proof}
We claim that the functor
\begin{center}
\begin{tikzcd}[row sep=small]
\D(\A)\arrow[r] & \Hom^\E_!(\E^{\A\op},\D)\\
X \arrow[mapsto,r,shorten >=0.8cm,shorten <=0.6cm] & -\otimes_\A X
\end{tikzcd}
\end{center}
is an inverse equivalence for evaluation at $\h_\A$. To see this, we must show that both composites are naturally isomorphic to the identity.

Suppose we have an object $X\in\D(\A)$. The isomorphisms 
\[
\lambda:\h_\A\otimes_\A X\xrightarrow{\;\;\;\iso\;\;\;} X
\]
of \cref{Coherence for the cancelling version of the closed E-action} are natural in $X$, so these provide one of the required isomorphisms. 

For the other isomorphism, let $F:\E^{\A\op}\rightarrow\D$ be a cocontinuous $\E$-module map. For any category $\B$ and any $Z\in\E(\A\op\times\B)$, consider the isomorphism
\[
Z\otimes_\A F\h_\A\xrightarrow{\;\;\;\varphi\;\;\;} F(Z\otimes_\A \h_\A)\xrightarrow{\;\;\;F(\rho)\;\;\;} FZ,
\]
where $\varphi$ is the map of \cref{Cocontinuous E-module maps preserve the cancelling product}, and $\rho$ is the map of \cref{Coherence for the cancelling version of the closed E-action}. This is a modification in $Z$; we must check that it respects tensors. It suffices to check this on the external product, but this follows from the first commutative diagram of \cref{Cocontinuous E-module maps preserve the cancelling product}, using \cref{Coherence for the cancelling product gives coherence for the internal product}. Thus, we have an isomorphism
\begin{center}
\begin{tikzcd}
-\otimes_\A F\h_\A\arrow[r,Rightarrow,"\iso"] & F
\end{tikzcd}
\end{center}
in $\Hom^\E_!(\E^{\A\op},\D)$. Finally, we need to show that this isomorphism is natural in $F$. That is, given cocontinuous $\E$-module maps $F,G:\E^{\A\op}\rightarrow\D$ and an $\E$-module modification $\beta:F\Rightarrow G$, we need the diagram below to commute, for any $Z\in\E(\A\op\times\B)$:
\begin{center}
\begin{tikzcd}
Z\otimes_\A F\h_\A\arrow{r}{\varphi}\arrow[dd,"Z\otimes_\A \beta_{\h_\A}" left] & F(Z\otimes_\A \h_\A)\arrow{r}{F(\rho)} & FZ\arrow[dd,"\beta_Z" right]\\\\
Z\otimes_\A G\h_\A\arrow{r}[below]{\varphi} & G(Z\otimes_\A \h_\A)\arrow{r}[below]{G(\rho)} & GZ
\end{tikzcd}
\end{center} 
This commutes by~\cref{E-module modifications respect the cancelling product}.
\end{proof}

\begin{definition}\label{Representable map from a closed E-module}
Let $\D$ be a closed $\E$-module, and let $\A$ be a category. We say a derivator map $F:\D\op\rightarrow\E^{\A}$ is \textbf{representable} if there is an object $X\in\D(\A)$ and an isomorphism:
\[
F\iso\widetilde{\map}_{\D}(-,X):\D\op\rightarrow\E^{\A}
\]
There is also a dual concept for maps of the form $F:\D\rightarrow\E^{\A\op}$.
\end{definition}

\begin{theorem}\label{Representability for perfectly generated triangulated derivator}
Let $\E$ be a symmetric monoidal derivator, let $\D$ be a closed $\E$-module, and let $\A$ be a category. Suppose that $\E$ and $\D$ are triangulated, and that, for any category $\C$, the triangulated category $\D(\C)$ satisfies Brown representability. Then a derivator map
\[
F:\D\op\rightarrow\E^{\A}
\]
is representable if and only if it is continuous and preserves cotensors.
\end{theorem}
\begin{proof}
Note that a continuous, cotensor-preserving map from $\D\op$ to $\E^{\A}$ takes homotopy colimits in $\D$ to homotopy limits in $\E^\A$, and tensors in $\D$ to cotensors in $\E^\A$. For the forward implication, let $X\in\D(\A)$, and consider the derivator map
\[
\widetilde{\map}_{\D}(-,X):\D\op\rightarrow\E^{\A}.
\]
This map is continuous, and it preserves cotensors by \cref{mapping space in the opposite closed E-module} and \cref{map(X -) preserves cotensors}.

On the other hand, suppose we have a continuous, cotensor-preserving map $F:\D\op\rightarrow\E^\A$. By \cref{Adjoint functor theorem for triangulated derivators} (applied to $F\op$), this map must have a left adjoint $G:\E^\A\rightarrow\D\op$. By \cref{Left adjoint preserves tensors iff right adjoint preserves cotensors}, $G$ preserves tensors. Thus, by \cref{Universal property of E shifted by A}, we have an isomorphism
\[
G\iso G\h_\A \vartriangleleft_\A -:\E^\A\rightarrow \D\op.
\]
Here $\vartriangleleft_\A$ is the cancelling version of the $\E$-action on $\D\op$ from \cref{Opposite of a closed E-module is an E-module}. But we know that the right adjoint of this map is $\widetilde{\map}_{\D}(-,G\h_\A):\D\op\rightarrow\E^{\A}$. Thus, this map $\widetilde{\map}_{\D}(-,G\h_\A)$ must be isomorphic to $F$.
\end{proof}

Note that, by~\cref{Perfectly generated triangulated categories satisfy representability} and~\cref{Perfectly generated triangulated derivator}, the previous theorem applies to any triangulated derivator $\D$ whose underlying category is perfectly generated. 

In particular, applying this theorem with $\E=\dHo(\Spt)$, we get the following special case: given a triangulated derivator $\D$ whose underlying category is perfectly generated, a derivator map
\[
F:\D\op\rightarrow\dHo(\Spt)
\]
is representable if and only if it is continuous.

\chapter{$\E$-Categories}\label{Chapter E-categories}

In this chapter, we study $\E$-categories, a precursor to the $\E$-prederivators and $\E$-derivators of~\cref{Chapter Enriched derivators}. In~\cref{Section Basic definitions of E-categories} we develop the basic theory  and give a number of examples; in particular, in~\cref{E-modules are E-categories}, we prove that any closed $\E$-module gives rise to an associated $\E$-category. In~\cref{Section The yoneda lemma and adjunctions for E-categories}, we prove a Yoneda lemma for $\E$-categories, and use this to study $\E$-category adjunctions. The Yoneda lemma is extremely useful, and we use it repeatedly as we develop the theory of $\E$-prederivators in~\cref{Chapter Enriched derivators}. We end this chapter with~\cref{Section Transferring enrichments}, in which we study monoidal morphisms, and prove that enrichment can be transferred along a monoidal adjunction. 

Throughout the rest of the thesis, unless otherwise specified, $\E$ will denote a closed symmetric monoidal derivator.

\section{Basic definitions}\label{Section Basic definitions of E-categories}

In this section, we define $\E$-categories, and develop the basic theory that we will need in subsequent sections. Many of the results in this section are formally similar to the development of basic enriched category theory in~\cite{Kelly82}. Note, in fact, that an $\E$-category in the sense of~\cref{E-category definition} is a $\Prof(\E)$-category in the sense of~\cite{Street83}; that is, a category enriched over the bicategory $\Prof(\E)$ of~\cref{Definition of Prof(E)}. However, a number of fundamental constructions, such as the mapping space $\E$-morphisms of~\cref{mapping space is an E-morphism}, are not definable over a general enriching bicategory. Thus, the majority of the results in this section do not follow from any general results that we know of in the bicategorical literature, so we include a complete account of the theory, starting from the definition. 

\begin{definition}\label{E-category definition}
An \textbf{$\E$-category} $\eA$ consists of the following data:
\begin{itemize}
\item For each small category $\A$, a (large) set of objects $\dA_0(\A)$.
\item For any two objects $X\in\dA_0(\A)$ and $Y\in\dA_0(\B)$, an object $\widetilde{\map}_{\dA}(X,Y)\in\E(\A\op\times\B)$.
\item For any three objects $X\in\dA_0(\A)$, $Y\in\dA_0(\B)$ and $Z\in\dA_0(\C)$, a map
\[
\circ:\widetilde{\map}_{\dA}(Y,Z)\otimes_{\B}\widetilde{\map}_{\dA}(X,Y)\rightarrow\widetilde{\map}_{\dA}(X,Z)
\]
in $\E(\A\op\times\C)$, which we call \textbf{composition}.
\item For every object $X\in\dA_0(\A)$, a map
\[
j:\h_\A\rightarrow\widetilde{\map}_{\dA}(X,X)
\]
in $\E(\A\op\times\A)$, which we call the \textbf{unit}.
\end{itemize}
These must satisfy the following coherence conditions:
\begin{enumerate}
\item For any $X\in\dA_0(\A)$, $Y\in\dA_0(\B)$, $Z\in\dA_0(\C)$ and $W\in\dA_0(\mathrm{D})$, the diagram below commutes:
\begin{center}
\begin{tikzcd}[column sep=-1cm]
\widetilde{\map}_{\dA}(Z,W)\otimes_{\C}(\widetilde{\map}_{\dA}(Y,Z)\otimes_{\B}\widetilde{\map}_{\dA}(X,Y))\arrow[leftarrow,rd,"\alpha",controls={+(5,0)and +(-1,0.8)}]
\arrow{ddd}[left]{\widetilde{\map}_{\dA}(Z,W)\otimes_{\C}\;\circ}\\
 
& (\widetilde{\map}_{\dA}(Z,W)\otimes_{\C}\widetilde{\map}_{\dA}(Y,Z))\otimes_{\B}\widetilde{\map}_{\dA}(X,Y)\arrow{dd}{\circ\;\otimes_{\B}\widetilde{\map}_{\dA}(X,Y)}\\\\
 
\widetilde{\map}_{\dA}(Z,W)\otimes_{\C}\widetilde{\map}_{\dA}(X,Z)\arrow[ddr,"\circ" below left, bend right=20] &
\widetilde{\map}_{\dA}(Y,W)\otimes_{\B}\widetilde{\map}_{\dA}(X,Y)\arrow{dd}[right]{\circ}  \\\\
 
&\widetilde{\map}_{\dA}(X,W)
\end{tikzcd}
\end{center} 
\item For any $X\in\dA_0(\A)$ and $Y\in\dA_0(\B)$, the diagram below commutes:
\begin{center}
\begin{tikzcd}
\h_\B\otimes_{\B}\widetilde{\map}_{\dA}(X,Y)\arrow{rr}{j\otimes_{\B}\widetilde{\map}_{\dA}(X,Y)}\arrow[bend right=20]{rdd}[below left]{\lambda} & & \widetilde{\map}_{\dA}(Y,Y)\otimes_{\B}\widetilde{\map}_{\dA}(X,Y)\arrow[ldd,"\circ", bend left=20]\\\\
& \widetilde{\map}_{\dA}(X,Y)
\end{tikzcd}
\end{center} 
\item For any $X\in\dA_0(\A)$ and $Y\in\dA_0(\B)$, the diagram below commutes:
\begin{center}
\begin{tikzcd}
\widetilde{\map}_{\dA}(X,Y)\otimes_{\A}\h_\A\arrow{rr}{\widetilde{\map}_{\dA}(X,Y)\otimes_{\A} j}\arrow[bend right=20]{rdd}[below left]{\rho} & & \widetilde{\map}_{\dA}(X,Y)\otimes_{\A}\widetilde{\map}_{\dA}(X,X)\arrow[ldd,"\circ", bend left=20]\\\\
& \widetilde{\map}_{\dA}(X,Y)
\end{tikzcd}
\end{center} 
\end{enumerate}
Here the maps $\alpha$, $\lambda$ and $\rho$ are those of~\cref{Coherence for the cancelling version of the closed E-action}.
\end{definition}

\begin{definition}\label{E-morphism definition}
Let $\eA$ and $\eB$ be $\E$-categories. An \textbf{$\E$-morphism} $F:\eA\rightarrow\eB$ consists of the following data:
\begin{itemize}
\item For any category $\A$, and any $X\in\dA_0(\A)$, an object $FX\in\dB_0(\A)$.
\item For any objects $X\in\dA_0(\A)$ and $Y\in\dA_0(\B)$, a map
\[
F:\widetilde{\map}_{\dA}(X,Y)\rightarrow\widetilde{\map}_{\dB}(FX,FY)
\]
in $\E(\A\op\times\B)$.
\end{itemize}
These must satisfy the following coherence conditions:
\begin{enumerate}
\item For any $X\in\dA_0(\A)$, $Y\in\dA_0(\B)$ and $Z\in\dA_0(\C)$, the diagram below commutes:
\begin{center}
\begin{tikzcd}
\widetilde{\map}_{\dA}(Y,Z)\otimes_{\B}\widetilde{\map}_{\dA}(X,Y)\arrow{rr}[above]{\circ}\arrow{dd}[left]{F\otimes_{\B}F} && 
 \widetilde{\map}_{\dA}(X,Z) \arrow{dd}[right]{F} \\\\
 \widetilde{\map}_{\dB}(FY,FZ)\otimes_{\B}\widetilde{\map}_{\dB}(FX,FY)\arrow{rr}[below]{\circ} && 
\widetilde{\map}_{\dB}(FX,FZ)
\end{tikzcd}
\end{center} 
\item For any $X\in\dA_0(\A)$, the diagram below commutes:
\begin{center}
\begin{tikzcd}
\;&&& \widetilde{\map}_{\dA}(X,X)\arrow{dd}[right]{F} \\
 
\h_\A\arrow[bend left=13]{rrru}[above left]{j}\arrow[bend right=15]{rrrd}[below left]{j}\\
 
\;&&& \widetilde{\map}_{\dB}(FX,FX)
\end{tikzcd}
\end{center} 
\end{enumerate}
\end{definition}

\begin{definition}\label{E-natural transformation definition}
Let $F,G:\eA\rightarrow\eB$ be $\E$-morphisms. An \textbf{$\E$-natural transformation} $\beta:F\Rightarrow G$ consists of maps
\[
\beta_X:\h_\A\rightarrow\widetilde{\map}_{\dB}(FX,GX)
\]
in $\E(\A\op\times\A)$, for every category $\A$, and every $X\in\dA_0(\A)$. These must make the diagram below commute, for any $X\in\dA_0(\A)$ and $Y\in\dA_0(\B)$:

\begin{center}
\begin{tikzcd}
\h_\B\otimes_\B\widetilde{\map}_{\dA}(X,Y)\arrow{rr}[above]{\beta_Y\;\otimes_\B F} && \widetilde{\map}_{\dB}(FY,GY)\otimes_\B\widetilde{\map}_{\dB}(FX,FY)\arrow{dd}[right]{\circ} \\\\
 
\widetilde{\map}_{\dA}(X,Y)\arrow{uu}[left]{\lambda^{-1}}\arrow{dd}[left]{\rho^{-1}} &&\widetilde{\map}_{\dB}(FX,GY)\\\\
 
\widetilde{\map}_{\dA}(X,Y)\otimes_\A\h_\A\arrow{rr}[below]{G\;\otimes_\A\beta_X} && \widetilde{\map}_{\dB}(GX,GY)\otimes_\A\widetilde{\map}_{\dB}(FX,GX)\arrow{uu}[right]{\circ}
\end{tikzcd}
\end{center} 
\end{definition}

The following construction can be carried out for categories enriched over any bicategory (see~\cite[Section 2]{Street83}):

\begin{remark}\label{A([0]) is E([0])-enriched}\label{Underlying categories of an E-category}
Let $\eA$ be an $\E$-category, and let $\A$ be a category. It is immediate from \cref{E-category definition} that the set of objects $\dA_0(\A)$ are the objects of an $\E(\A\op\times\A)$-enriched category, where $\E(\A\op\times\A)$ is equipped with the monoidal structure induced by \cref{E shifted by Aop x A is monoidal}. Given objects $X,Y\in\dA_0(\A)$, the mapping object between them is given by $\widetilde{\map}_{\dA}(X,Y)\in\E(\A\op\times\A)$, and units and composition are inherited from the $\E$-category structure. We will denote this $\E(\A\op\times\A)$-enriched category by $\eA(\A)$. 

Denote the underlying category of this $\E(\A\op\times\A)$-enriched category by $\dA(\A)$. To avoid confusion with the established meaning of the underlying category of a prederivator, we will call $\dA(\A)$ the \textbf{category induced by $\dA$ at $\A$.} 

Explicitly, $\dA(\A)$ is defined as follows: objects are given by the set $\dA_0(\A)$, and, for any $X,Y\in\dA_0(\A)$, we have
\[
\dA(\A)(X,Y)=\E(\A\op\times\A)(\h_\A,\widetilde{\map}_{\dA}(X,Y)).
\]
For any object $X\in\dA_0(\A)$, the identity map is given by the unit 
\[
j:\h_\A\rightarrow\widetilde{\map}_{\dA}(X,X)
\]
in $\E(\A\op\times\A)(\h_\A,\widetilde{\map}_{\dA}(X,X))$. Given objects $X,Y,Z\in\dA_0(\A)$ and maps
\[
f:\h_\A\rightarrow\widetilde{\map}_{\dA}(X,Y)
\]
\[
g:\h_\A\rightarrow\widetilde{\map}_{\dA}(Y,Z)
\]
their composite is given by the following:
\begin{center}
\begin{tikzcd}
\h_\A\arrow{r}{\iso} & \h_\A\otimes_\A \h_\A\arrow{r}{\;\;\;g\;\otimes_\A f\;\;\;\;\;} & \widetilde{\map}_{\dA}(Y,Z)\otimes_{\A}\widetilde{\map}_{\dA}(X,Y)\arrow{r}{\circ} & \widetilde{\map}_{\dA}(X,Z)
\end{tikzcd}
\end{center} 
The unnamed isomorphism is $\lambda^{-1}=\rho^{-1}:\h_\A\rightarrow\h_\A\otimes_\A\h_\A$.

Similarly, any $\E$-morphism $F:\eA\rightarrow\eB$ induces an $\E(\A\op\times\A)$-functor $F:\eA(\A)\rightarrow\eB(\A)$. The underlying functor of this $\E(\A\op\times\A)$-functor is called the \textbf{functor induced by $F$ at $\A$}. It takes any object $X\in\dA(\A)$ to $FX\in\dB(\A)$, and given a map 
\[
f:\h_\A\rightarrow\widetilde{\map}_{\dA}(X,Y)
\]
in $\dA(\A)$, its image is given by
\begin{center}
\begin{tikzcd}
\h_\A\arrow{r}{f} & \widetilde{\map}_{\dA}(X,Y)\arrow{r}{F} & \widetilde{\map}_{\dB}(FX,FY).
\end{tikzcd}
\end{center} 

Suppose we have $\E$-morphisms $F,G:\eA\rightarrow\eB$. Any $\E$-natural transformation $\beta:F\Rightarrow G$ induces an $\E(\A\op\times\A)$-natural transformation
\begin{center}
\begin{tikzcd}
\eA(\A)\arrow[r, bend left=50, "F" above,""{name=U, below}]\arrow[r, bend right=50, "G" below, ""{name=D}]
& \eB(\A).
\arrow[Rightarrow,from=U,to=D,shorten >=0.1cm,shorten <=0.1cm,"\beta"]
\end{tikzcd}
\end{center}
For any $X\in\dA(\A)$, the component
\[
\beta_X:\h_\A\rightarrow\widetilde{\map}_{\dB}(FX,GX)
\]
may be thought of as a map in $\dB(\A)$. It follows immediately from the $\E$-naturality condition that these form a natural transformation between the functors induced by $F$ and $G$, which we call the \textbf{natural transformation induced by $\beta$ at $\A$.}
\end{remark}

\begin{lemma}\label{2-category of E-categories}
$\E$-Categories, $\E$-morphisms and $\E$-natural transformations form a $2$-category $\ECat$. Moreover, any category $\A$ induces a $2$-functor:
\begin{align*}
\ECat &\rightarrow \; \CAT\\
 \eA \;\; &\mapsto \;\;\; \dA(\A)
\end{align*}
\end{lemma}
\begin{proof}
We need to define composition for $\E$-morphisms and $\E$-natural transformations. First, suppose we have $\E$-morphsims $F:\eA\rightarrow\eB$ and $G:\eB\rightarrow\eC$. The composite $G\circ F$ takes an object $X\in\dA_0(\A)$ to $GFX\in\dC_0(\A)$, and the action on mapping spaces is simply given by
\begin{center}
\begin{tikzcd}
\widetilde{\map}_{\dA}(X,Y)\arrow{r}{F} & \widetilde{\map}_{\dB}(FX,FY)\arrow{r}{G} & \widetilde{\map}_{\dC}(GFX,GFY)
\end{tikzcd}
\end{center} 
for any $X\in\dA_0(\A)$ and $Y\in\dA_0(\B)$.

Given $\E$-natural transformations $\alpha:F\Rightarrow G$ and $\beta: G\Rightarrow H$, their vertical composite $\beta\cdot\alpha:F\Rightarrow H$ has component at $X\in\dA_0(\A)$ given by the composite $\beta_X\circ\alpha_X$ in $\dB(\A)$. Note that an $\E$-natural transformation is an isomorphism if and only if each of its components is an isomorphism. 

Given $\E$-morphisms $F,H:\eA\rightarrow\eB$ and $G,K:\eB\rightarrow\eC$, suppose we have $\E$-natural transformations $\alpha:F\Rightarrow H$ and $\beta:G\Rightarrow K$. Their horizontal composite $\beta\circ\alpha$ has component at $X\in\dA_0(\A)$ given by 
\[
\beta_{HX}\circ G(\alpha_X)=K(\alpha_X)\circ\beta_{FX}
\]
in $\dC(\A)$.

It is routine to check that the composite of $\E$-morphisms is an $\E$-morphism, and that the vertical and horizontal composites above define $\E$-natural transformations. The $2$-category axioms for this data follow from the corresponding facts for the $2$-category $\CAT$.
\end{proof}


\begin{example}\label{Full sub-E-category}
Let $\eA$ be an $\E$-category. For each category $\A$, suppose we have a set of objects $\dB_0(\A)\subseteq\dA_0(\A)$. The \textbf{full sub-$\E$-category} of $\eA$ on these objects is the $\E$-category $\eB$ defined as follows:
\begin{itemize}
\item For each category $\A$, the objects are $\dB_0(\A)$.
\item For any two objects $X\in\dB_0(\A)$ and $Y\in\dB_0(\B)$, the mapping object is $\widetilde{\map}_{\dA}(X,Y)\in\E(\A\op\times\B)$.
\item Units and composition are inherited from $\eA$.
\end{itemize}
The $\E$-category axioms for $\eB$ follow immediately from the axioms for $\eA$. Note that, for any category $\A$, the induced category $\dB(\A)$ is the full subcategory of $\dA(\A)$ on the objects $\dB_0(\A)$.
\end{example}

\begin{definition}\label{Opposite E-category}
Let $\eA$ be an $\E$-category. We may form the \textbf{opposite} $\E$-category $\eA\op$ as follows:
for any category $\A$, we define 
\[
\dA\op_0(\A)=\dA_0(\A\op),
\]
and, given $X\in\dA_0(\A\op)$ and $Y\in\dA_0(\B\op)$, define
\[
\widetilde{\map}_{\dA\op}(X,Y)=\sigma^*\widetilde{\map}_{\dA}(Y,X)\in\E(\A\op\times\B),
\]
where $\sigma:\A\op\times\B\xrightarrow{\;\;\iso\;\;}\B\times\A\op$ is the canonical isomorphism.

For any $X\in\dA_0(\A\op)$, the unit is given by the composite below
\[
\h_\A\xrightarrow{\;\;\;\iso\;\;\;}\sigma^*\h_{\A\op}\xrightarrow{\;\;\;\sigma^*j\;\;\;}\sigma^*\widetilde{\map}_{\dA}(X,X)
\]
where the unnamed isomorphism is an instance of~\cref{Coends over A and Aop}, as in~\cref{Symmetry for the cancelling tensor}.

Similarly, given objects $X\in\dA_0(\A\op)$, $Y\in\dA_0(\B\op)$ and $Z\in\dA_0(\C\op)$, composition is as follows:
\begin{center}
\begin{tikzcd}
\sigma^*\widetilde{\map}_{\dA}(Z,Y)\otimes_{\B}\sigma^*\widetilde{\map}_{\dA}(Y,X)\arrow[r,"\tau"]
& \sigma^*(\widetilde{\map}_{\dA}(Y,X)\otimes_{\B\op}\widetilde{\map}_{\dA}(Z,Y))\arrow[d,"\sigma^*(\circ)" right]\\
& \sigma^*\widetilde{\map}_{\dA}(Z,X)
\end{tikzcd}
\end{center}
Here $\tau$ is the canonical isomorphism of~\cref{Symmetry for the cancelling tensor}.
\end{definition}

The $\E$-category axioms for $\eA\op$ follow easily from the axioms for $\eA$, using the coherence of~\cref{Symmetry for the cancelling tensor}.

\begin{remark}
For any $\E$-category $\eA$, and any category $\A$, we have an isomorphism of categories
\[
\dA\op(\A)\iso\dA(\A\op)\op.
\]
To see this, note that the objects of both categories are the same by definition. Moreover, given two objects $X,Y\in\dA_0(\A\op)$, we have:
\begin{align*}
\dA\op(\A)(X,Y) & = \E(\A\op\times\A)(\h_\A,\sigma^*\widetilde{\map}_{\dA}(X,Y))\\
 & \iso \E(\A\op\times\A)(\sigma^*\h_{\A\op},\sigma^*\widetilde{\map}_{\dA}(Y,X))\\
 & \iso \E(\A\times\A\op)(\h_{\A\op},\widetilde{\map}_{\dA}(Y,X))\\
 & = \dA(\A\op)(Y,X)
\end{align*}
It is easy to check that this bijection respects identities and composition,  using~\cref{Symmetry for the cancelling tensor}. 
\end{remark}

\begin{remark}
In the obvious way, we can extend the opposite $\E$-category construction of~\cref{Opposite E-category} to $\E$-morphisms and $\E$-natural transformations. We obtain a $2$-functor on $\ECat$, which preserves the direction of $1$-cells but reverses the direction of $2$-cells.
\end{remark}

\begin{theorem}\label{E-modules are E-categories}
Let $\D$ be a closed $\E$-module, with action: 
\begin{align*}
\otimes&:\E\times\D\rightarrow\D\\
\vartriangleleft&:\D\times\E\op\rightarrow\D\\
\map_\D(-,-)&:\D\op\times\D\rightarrow\E
\end{align*}
Then there is an associated $\E$-category $\eD$ such that, for any category $\A$, the induced category $\D(\A)$ of~\cref{Underlying categories of an E-category} recovers the value at $\A$ of the derivator $\D$.
\end{theorem}
\begin{proof}
The $\E$-category $\eD$ is defined as follows. For any category $\A$, the set $\D_0(\A)$ is the set of objects of $\D(\A)$. Given $X\in\D_0(\A)$ and $Y\in\D_0(\B)$, the mapping object $\widetilde{\map}_{\D}(X,Y)\in\E(\A\op\times\B)$ is their image under the functor:
\[
\widetilde{\map}_\D(-,-):\D(\A)\op\times\D(\B)\rightarrow\E(\A\op\times\B)
\]

To describe composition and units, let $X\in\D(\A)$, and consider the adjunction
\begin{center}
\begin{tikzcd}
\E^{\A\op}\arrow[rr, bend left=40, "-\;\otimes_\A X" above,""{name=U, below}]\arrow[rr,leftarrow, bend right=40, "\widetilde{\map}_\D(X\text{,}-)" below, ""{name=D}]
&& \D,
\arrow[phantom,from=U,to=D,"\bot"]
\end{tikzcd}
\end{center}
described in~\cref{2-variable adjunction characterisation}. Given any $X\in\D(\A)$, the unit 
\[
j:\h_\A\rightarrow\widetilde{\map}_\D(X,X)
\]
is adjunct under this adjunction to the canonical isomorphism
\[
\lambda:\h_\A\otimes_\A X\rightarrow X
\]
of~\cref{Coherence for the cancelling version of the closed E-action}.
Similarly, given objects $X\in\D(\A)$, $Y\in\D(\B)$ and $Z\in\D(\C)$, composition is adjunct to the following:
\begin{center}
\begin{tikzcd}
(\widetilde{\map}_{\D}(Y,Z)\otimes_{\B}\widetilde{\map}_{\D}(X,Y))\otimes_\A X\arrow[r,"\alpha"]
& \widetilde{\map}_{\D}(Y,Z)\otimes_{\B}(\widetilde{\map}_{\D}(X,Y)\otimes_\A X)\arrow[d,"\widetilde{\map}_{\D}(Y\text{,}Z)\;\mathsmaller{\otimes_{\B}}\;\epsilon" right]\\
& \widetilde{\map}_{\D}(Y,Z)\otimes_{\B} Y\arrow[d,"\epsilon" right]\\
& Z
\end{tikzcd}
\end{center}

To see that this data satisfies the axioms of~\cref{E-category definition}, we can replace each diagram that we need to commute by its adjunct. The first axiom then follows using the first commutative diagram of~\cref{Coherence for the cancelling version of the closed E-action}. The second and third axioms of~\cref{E-category definition} follow using the second commutative diagram in~\cref{Coherence for the cancelling version of the closed E-action}. 

For any category $\A$, we still need to show that the induced category of~\cref{Underlying categories of an E-category} recovers the value of $\D$ at $\A$. Given $X,Y\in\D(\A)$, consider the isomorphsims:
\begin{align*}
\D(\A)(X,Y) & \iso \D(\A)(\h_\A\otimes_\A X,Y) \\
 & \iso \E(\A\op\times\A)(\h_\A,\widetilde{\map}_{\D}(X,Y))
\end{align*}
Given a map $f:X\rightarrow Y$ in $\D(\A)$, write $\tilde{f}:\h_\A\rightarrow\widetilde{\map}_{\D}(X,Y)$ for the corresponding map in $\E(\A\op\times\A)$. Thus, $\tilde{f}$ is the unique map making the diagram below commute:
\begin{center}
\begin{tikzcd}
\h_\A\otimes_{\A} X\arrow{rr}[above]{\lambda}\arrow{dd}[left]{\tilde{f}\otimes_{\A}X} && 
 X \arrow{dd}[right]{f} \\\\
 \widetilde{\map}_{\D}(X,Y)\otimes_\A X\arrow{rr}[below]{\epsilon} && 
Y
\end{tikzcd}
\end{center} 
We claim that the description of identity and composition in~\cref{Underlying categories of an E-category} corresponds to identity and composition in $\D(\A)$ under these bijections. 

By definition, given any $X\in\D(\A)$, we have $\widetilde{\id_X}=j:\h_\A\rightarrow\widetilde{\map}_{\D}(X,X)$. This shows that the bijection respects identities. To see that it respects composition, suppose we have composable maps $X\xrightarrow{\;\;f\;\;} Y\xrightarrow{\;\;g\;\;} Z$ in $\D(\A)$. Using the definition of composition from \cref{Underlying categories of an E-category}, we need the diagram below to commute:
\begin{center}
\begin{tikzcd}
\h_\A\arrow{r}{\rho^{-1}}\arrow[drrr,"\widetilde{(g\circ f)}" below left, bend right=10] & \h_\A\otimes_\A \h_\A\arrow{rr}{\tilde{g}\otimes_\A \tilde{f}} && \widetilde{\map}_{\D}(Y,Z)\otimes_{\A}\widetilde{\map}_{\D}(X,Y)\arrow{d}{\circ} \\
&&& \widetilde{\map}_{\D}(X,Z)
\end{tikzcd}
\end{center} 
Taking the adjunct under the adjunction $-\otimes_\A X\dashv \widetilde{\map}_{\D}(X,-)$, and using the definition of composition in $\eD$ given above, we can see that this commutes.
\end{proof}

By~\cref{Shift of an E-module is an E-module}, for any category $\J$ and any closed $\E$-module $\D$, the shifted derivator $\D^\J$ is a closed $\E$-module. Thus, by~\cref{E-modules are E-categories}, $\D^\J$ gives rise to an associated $\E$-category. In particular, by~\cref{E is an E-module}, this is the case for $\E^\J$. More generally, we will prove in~\cref{Section Transferring enrichments} that, for any $\E$-category $\eA$ and any category $\J$, we may form a shifted $\E$-category $\eA^\J$.

\begin{proposition}\label{mapping space is an E-morphism}
Let $\eA$ be an $\E$-category, let $\A$ be a category, and let $X\in\dA_0(\A)$. The mapping objects in $\eA$ induce an $\E$-morphism:
\[
\widetilde{\map}_{\dA}(X,-):\eA\rightarrow\eE^{\A\op}
\]
\end{proposition}
\begin{proof}
Let $X\in\dA_0(\A)$. On objects, the $\E$-morphism 
\[
\widetilde{\map}_{\dA}(X,-):\eA\rightarrow\eE^{\A\op}
\]
takes $Y\in\dA_0(\B)$ to $\widetilde{\map}_{\dA}(X,Y)\in\E(\A\op\times\B)$. Given a further object $Z\in\dA_0(\C)$, we need a map
\[
\widetilde{\map}_{\dA}(X,-):\widetilde{\map}_{\dA}(Y,Z)\rightarrow\widetilde{\map}_{\E^{\A\op}}(\widetilde{\map}_{\dA}(X,Y),\widetilde{\map}_{\dA}(X,Z))
\]
in $\E(\B\op\times\C)$. Consider the adjunction below:
\begin{center}
\begin{tikzcd}
\E^{\B\op}\arrow[rrr, bend left=40, "-\;\otimes_\B \widetilde{\map}_{\dA}(X\text{,}Y)" above,""{name=U, below}]\arrow[rrr,leftarrow, bend right=40, "\widetilde{\map}_{\E^{\A\op}}(\widetilde{\map}_{\dA}(X\text{,}Y)\text{,}-)" below, ""{name=D}]
&&& \;\E^{\A\op}
\arrow[phantom,from=U,to=D,"\bot"]
\end{tikzcd}
\end{center}
Under this adjunction, take the structure map to be adjunct to composition:
\[
\circ:\widetilde{\map}_{\dA}(Y,Z)\otimes_{\B}\widetilde{\map}_{\dA}(X,Y)\rightarrow\widetilde{\map}_{\dA}(X,Z)
\]
To prove that this satisfies the axioms of~\cref{E-morphism definition}, replace the diagrams that we need to commute by their adjuncts, and use the description of units and composition in $\eE^{\A\op}$ from~\cref{E-modules are E-categories}. The second axiom of~\cref{E-morphism definition} for $\widetilde{\map}_{\dA}(X,-)$ corresponds precisely to the second axiom of~\cref{E-category definition} for $\eA$. Similarly, the first axiom of~\cref{E-morphism definition} corresponds to the first of~\cref{E-category definition}.
\end{proof}

\begin{definition}\label{Representable E-morphism definition}
Let $F:\eA\rightarrow\eE^{\A\op}$ be an $\E$-morphism. We say that $F$ is \textbf{representable} if there is an object $X\in\dA_0(\A)$ and an $\E$-natural isomorphism:
\[
F\iso\widetilde{\map}_{\dA}(X,-):\eA\rightarrow\eE^{\A\op}
\]
\end{definition}

Let $\eA$ be an $\E$-category, let $\A$ and $\B$ be categories, and let $X\in\dA_0(\A)$. As in~\cref{Underlying categories of an E-category}, consider the functor
\[
\widetilde{\map}_{\dA}(X,-):\dA(\B)\rightarrow\E(\A\op\times\B)
\]
induced by the $\E$-morphism $\widetilde{\map}_{\dA}(X,-)$. Given any map $f:\h_\B\rightarrow\widetilde{\map}_{\dA}(Y,Z)$ in $\dA(\B)$, write 
\[
\widetilde{\map}_{\dA}(X,f):\widetilde{\map}_{\dA}(X,Y)\rightarrow\widetilde{\map}_{\dA}(X,Z)
\]
for its image in $\E(\A\op\times\B)$. Using the description of $\widetilde{\map}_{\dA}(X,-)$ from~\cref{mapping space is an E-morphism}, this is the composite below:
\begin{center}
\begin{tikzcd}
\widetilde{\map}_{\dA}(X,Y)\arrow[r,"\lambda^{-1}"]
& \h_\B\otimes_{\B}\widetilde{\map}_{\dA}(X,Y)\arrow[rrr,"f\otimes_\B\widetilde{\map}_{\dA}(X\text{,}Y)" above] &&& \widetilde{\map}_{\dA}(Y,Z)\otimes_{\B}\widetilde{\map}_{\dA}(X,Y)\arrow[d,"\circ" right]\\
&&&& \widetilde{\map}_{\dA}(X,Z)
\end{tikzcd}
\end{center}
Given any object $Y\in\dA_0(\B)=\dA\op_0(\B\op)$, we will write
\[
\widetilde{\map}_{\dA}(-,Y):=\widetilde{\map}_{\dA\op}(Y,-):\eA\op\rightarrow\eE^{\B}.
\]
As above, any map $g:\h_\A\rightarrow\widetilde{\map}_{\dA}(X,Z)$ in $\dA(\A)$ induces a map
\[
\widetilde{\map}_{\dA}(g,Y):\widetilde{\map}_{\dA}(Z,Y)\rightarrow\widetilde{\map}_{\dA}(X,Y)
\]
in $\E(\A\op\times\B)$. Explicitly, this is the following composite:
\begin{center}
\begin{tikzcd}
\widetilde{\map}_{\dA}(Z,Y)\arrow[r,"\rho^{-1}"]
& \widetilde{\map}_{\dA}(Z,Y)\otimes_\A\h_\A\arrow[rrr,"\widetilde{\map}_{\dA}(Z\text{,}Y)\otimes_\A g" above] &&& \widetilde{\map}_{\dA}(Z,Y)\otimes_\A\widetilde{\map}_{\dA}(X,Z)\arrow[d,"\circ" right]\\
&&&& \widetilde{\map}_{\dA}(X,Y)
\end{tikzcd}
\end{center}

\begin{remark}\label{Extracting f from map(f X)}
Let $\eA$ be an $\E$-category, let $\A$ be a category, and let $f:\h_\A\rightarrow\widetilde{\map}_{\dA}(X,Y)$ be a map in $\dA(\A)$. The unit axioms for $\eA$ imply that the diagram below commutes:
\begin{center}
\begin{tikzcd}
\h_\A\arrow{rrr}[above]{j}\arrow{dd}[left]{j}\arrow{rrrdd}[above right]{f} &&& 
 \widetilde{\map}_{\dA}(X,X) \arrow{dd}[right]{\widetilde{\map}_{\dA}(X,f)} \\\\
 \widetilde{\map}_{\dA}(Y,Y) \arrow{rrr}[below]{\widetilde{\map}_{\dA}(f,Y)} &&& 
\widetilde{\map}_{\dA}(X,Y)
\end{tikzcd}
\end{center} 
So, precomposing with $j$, we can recover the map $f$ from $\widetilde{\map}_{\dA}(X,f)$ or $\widetilde{\map}_{\dA}(f,Y)$.
\end{remark}

\begin{remark}
Suppose $\eD$ is the $\E$-category associated to a closed $\E$-module $\D$,  and let $X\in\D(\A)$ and $Y\in\D(\B)$. Then, for any category $\C$, the functors
\begin{align*}
\widetilde{\map}_{\D}(X,-)&:\D(\C)\rightarrow\E(\A\op\times\C)\\
\widetilde{\map}_{\D}(-,Y)&:\D(\C)\op\rightarrow\E(\C\op\times\B),
\end{align*}
induced by the $\E$-morphisms $\widetilde{\map}_{\D}(X,-)$ and $\widetilde{\map}_{\D}(-,Y)$, recover the component functors of the original derivator maps:
\begin{align*}
\widetilde{\map}_{\D}(X,-)&:\D\rightarrow\E^{\A\op}\\
\widetilde{\map}_{\D}(-,Y)&:\D\op\rightarrow\E^\B
\end{align*}
\end{remark}

Using the descriptions above of $\widetilde{\map}_{\dA}(X,-)$ and $\widetilde{\map}_{\dA}(-,Y)$, we can give a useful characterisation of $\E$-naturality:

\begin{lemma}\label{Rephrasing E-naturality}
Let $\eA$ and $\eB$ be $\E$-categories, and let $F,G:\eA\rightarrow\eB$ be $\E$-morphisms. Suppose we have a collection of maps
\[
\beta_X:\h_\A\rightarrow\widetilde{\map}_{\dB}(FX,GX),
\]
for every category $\A$, and every $X\in\dA_0(\A)$. These maps form an $\E$-natural transformation $\beta:F\Rightarrow G$ if and only if the diagram below commutes, for any $X\in\dA_0(\A)$ and $Y\in\dA_0(\B)$:
\begin{center}
\begin{tikzcd}
\widetilde{\map}_{\dA}(X,Y)\arrow{rrr}[above]{F}\arrow{dd}[left]{G} &&& 
 \widetilde{\map}_{\dB}(FX,FY) \arrow{dd}[right]{\widetilde{\map}_{\dB}(FX,\beta_Y)} \\\\
 \widetilde{\map}_{\dB}(GX,GY)\arrow{rrr}[below]{\widetilde{\map}_{\dB}(\beta_X,GY)} &&& 
\widetilde{\map}_{\dB}(FX,GY)
\end{tikzcd}
\end{center} 
\end{lemma}
\begin{proof}
The diagram in~\cref{E-natural transformation definition} expressing $\E$-naturality is the same as this diagram, once we expand $\widetilde{\map}_{\dB}(FX,\beta_Y)$ and $\widetilde{\map}_{\dB}(\beta_X,GY)$.
\end{proof}

Using this characterisation, we can now prove the following two lemmas:

\begin{lemma}\label{map(f -) is E-natural}
Let $\eA$ be an $\E$-category, let $\A$ be a category, and let $f:\h_\A\rightarrow\widetilde{\map}_{\dA}(X,Y)$ be a map in $\dA(\A)$. Then the maps
\[
\widetilde{\map}_{\dA}(f,Z):\widetilde{\map}_{\dA}(Y,Z)\rightarrow\widetilde{\map}_{\dA}(X,Z)
\]
are $\E$-natural in $Z\in\dA_0(\B)$.
\end{lemma}
\begin{proof}
Fix a map $f:\h_\A\rightarrow\widetilde{\map}_{\dA}(X,Y)$ in $\dA(\A)$. By~\cref{Rephrasing E-naturality}, we need the diagram below to commute, for any $Z\in\dA_0(\B)$ and $W\in\dA_0(\C)$:
\begin{center}
\begin{tikzcd}[column sep=0cm]
\widetilde{\map}_{\dA}(Z,W)\arrow[rd,"\widetilde{\map}_{\dA}(Y\text{,}-)",bend left=12]
\arrow{dd}[left]{\widetilde{\map}_{\dA}(X,-)}\\
 
& \widetilde{\map}_{\E^{\A\op}}(\widetilde{\map}_{\dA}(Y,Z),\widetilde{\map}_{\dA}(Y,W)) \arrow{dd}{\widetilde{\map}_{\E^{\A\op}}(\widetilde{\map}_{\dA}(Y,Z),\widetilde{\map}_{\dA}(f,W))}\\
 
\widetilde{\map}_{\E^{\A\op}}(\widetilde{\map}_{\dA}(X,Z),\widetilde{\map}_{\dA}(X,W))\arrow[dr,"\widetilde{\map}_{\E^{\A\op}}(\widetilde{\map}_{\dA}(f\text{,}Z)\text{,}\widetilde{\map}_{\dA}(X\text{,}W))" below left, bend right=12] 
\\
&\widetilde{\map}_{\E^{\A\op}}(\widetilde{\map}_{\dA}(Y,Z),\widetilde{\map}_{\dA}(X,W))
\end{tikzcd}
\end{center} 
Under the adjunction
\begin{center}
\begin{tikzcd}
\E(\B\op\times\C)\arrow[rr, bend left=20, "-\;\otimes_\B \widetilde{\map}_{\dA}(Y\text{,}Z)" above,""{name=U, below}]\arrow[rr,leftarrow, bend right=20, "\widetilde{\map}_{\E^{\A\op}}(\widetilde{\map}_{\dA}(Y\text{,}Z)\text{,}-)" below=2, ""{name=D}]
&& \;\E(\A\op\times\C)
\arrow[phantom,from=U,to=D,"\bot"]
\end{tikzcd}
\end{center}
this reduces to the square:
\begin{center}
\begin{tikzcd}
\widetilde{\map}_{\dA}(Z,W)\otimes_\B \widetilde{\map}_{\dA}(Y,Z)\arrow{rr}[above]{\circ}\arrow{dd}[left]{\widetilde{\map}_{\dA}(Z,W)\otimes_\B \widetilde{\map}_{\dA}(f,Z)} && 
 \widetilde{\map}_{\dA}(Y,W) \arrow{dd}[right]{\widetilde{\map}_{\dA}(f,W)} \\\\
\widetilde{\map}_{\dA}(Z,W)\otimes_\B \widetilde{\map}_{\dA}(X,Z)\arrow{rr}[below]{\circ} && 
\widetilde{\map}_{\dA}(X,W)
\end{tikzcd}
\end{center} 
Using the descriptions of $\widetilde{\map}_{\dA}(f,Z)$ and $\widetilde{\map}_{\dA}(f,W)$ given above, we can see that this diagram commutes.
\end{proof}

\begin{remark}\label{E-category yoneda embedding definition}
For any $\E$-category $\eA$, and any category $\A$, the assignment
\begin{align*}
y:\dA(\A)\op &\longrightarrow \;\ECat(\eA,\eE^{\A\op})\\
 X\;\;\;\; &\;\mapsto \;\;\; \widetilde{\map}_{\dA}(X,-)
\end{align*}
is functorial. Two $\E$-natural transformations are equal if and only if their components are equal, so the functoriality of $y$ follows by the functoriality of the map
\[
\widetilde{\map}_{\dA}(-,W):\dA(\A)\op\rightarrow\E(\A\op\times\B),
\]
for any category $\B$ and any $W\in\dA_0(\B)$.
\end{remark}

\begin{lemma}\label{F:map(X Y) -> map(FX FY) is E-natural}
Let $F:\eA\rightarrow\eB$ be an $\E$-morphism. The maps
\[
F:\widetilde{\map}_{\dA}(X,Y)\rightarrow\widetilde{\map}_{\dB}(FX,FY)
\]
are $\E$-natural in both $X\in\dA_0(\A)$ and $Y\in\dA_0(\B)$. 
\end{lemma}
\begin{proof}
Fix $X\in\dA_0(\A)$. We will prove that the maps 
\[
F:\widetilde{\map}_{\dA}(X,Y)\rightarrow\widetilde{\map}_{\dB}(FX,FY)
\]
are $\E$-natural in $Y\in\dA_0(\B)$; $\E$-naturality in the first variable is dual. Given $Y\in\dA_0(\B)$ and $Z\in\dA_0(\C)$, we need the diagram below to commute:
\begin{center}
\begin{tikzcd}[column sep=0cm]
\widetilde{\map}_{\dA}(Y,Z)\arrow[rd,"\widetilde{\map}_{\dA}(X\text{,}-)",bend left=12]\arrow{dd}[left]{F}\\

& \widetilde{\map}_{\E^{\A\op}}(\widetilde{\map}_{\dA}(X,Y),\widetilde{\map}_{\dA}(X,Z)) \arrow{dddd}{\widetilde{\map}_{\E^{\A\op}}(\widetilde{\map}_{\dA}(X,Y),F)}\\

\widetilde{\map}_{\dB}(FY,FZ)\arrow{dd}[left]{\widetilde{\map}_{\dB}(FX,-)}
\\\\
 
\widetilde{\map}_{\E^{\A\op}}(\widetilde{\map}_{\dB}(FX,FY),\widetilde{\map}_{\dB}(FX,FZ))\arrow[dr,"\widetilde{\map}_{\E^{\A\op}}(F\text{,}\widetilde{\map}_{\dB}(FX\text{,}FZ))" below left, bend right=12] 
\\
&\widetilde{\map}_{\E^{\A\op}}(\widetilde{\map}_{\dA}(X,Y),\widetilde{\map}_{\dB}(FX,FZ))
\end{tikzcd}
\end{center} 
Under the adjunction
\begin{center}
\begin{tikzcd}
\E(\B\op\times\C)\arrow[rr, bend left=20, "-\;\otimes_\B \widetilde{\map}_{\dA}(X\text{,}Y)" above,""{name=U, below}]\arrow[rr,leftarrow, bend right=20, "\widetilde{\map}_{\E^{\A\op}}(\widetilde{\map}_{\dA}(X\text{,}Y)\text{,}-)" below=2, ""{name=D}]
&& \;\E(\A\op\times\C)
\arrow[phantom,from=U,to=D,"\bot"]
\end{tikzcd}
\end{center}
this corresponds exactly to the first axiom of~\cref{E-morphism definition} for the $\E$-morphism $F$.
\end{proof}

We end this section by proving that cocontinuous $\E$-module maps, and $\E$-module modifications between them, induce $\E$-morphisms and $\E$-natural transformations, using the coherence results at the end of~\cref{Section The cancelling tensor product}.

\begin{proposition}\label{Cocontinuous E-module maps induce E-morphisms}
Let $\D_1$ and $\D_2$ be closed $\E$-modules. Then any cocontinuous $\E$-module map $F:\D_1\rightarrow\D_2$ induces an $\E$-morphism
\[
F:\eD_1\rightarrow\eD_2
\]
on the associated $\E$-categories of~\cref{E-modules are E-categories}. Moreover, for any category $\A$, the functor $F:\D_1(\A)\rightarrow\D_2(\A)$ induced by this $\E$-morphism is the component at $\A$ of the original derivator map.
\end{proposition}
\begin{proof}
The $\E$-morphism $F:\eD_1\rightarrow\eD_2$ takes any object $X\in\D_1(\A)$ to its image under the derivator map, $FX\in\D_2(\A)$. Let $X\in\D_1(\A)$ and $Y\in\D_1(\B)$, and consider the adjunction
\begin{center}
\begin{tikzcd}
\E^{\A\op}\arrow[rr, bend left=40, "-\;\otimes_\A FX" above,""{name=U, below}]\arrow[rr,leftarrow, bend right=40, "\widetilde{\map}_{\D_2}(FX\text{,}-)" below, ""{name=D}]
&& \;\D_2.
\arrow[phantom,from=U,to=D,"\bot"]
\end{tikzcd}
\end{center}
Under this adjunction, define the structure map 
\[
F:\widetilde{\map}_{\D_1}(X,Y)\rightarrow\widetilde{\map}_{\D_2}(FX,FY)
\]
to be adjunct to the composite below:
\[
\widetilde{\map}_{\D_1}(X,Y)\otimes_{\A}FX\xrightarrow{\;\;\varphi\;\;}F(\widetilde{\map}_{\D_1}(X,Y)\otimes_{\A} X)\xrightarrow{\;\;F(\epsilon)\;\;}FY
\]
Here the map $\varphi$ is the canonical isomorphism of~\cref{Cocontinuous E-module maps preserve the cancelling product}.

To prove that this satisfies the axioms of~\cref{E-morphism definition}, we replace the diagrams that need to commute by their adjuncts. Using the description of units and composition in $\eD_2$, from~\cref{E-modules are E-categories}, the first axiom of~\cref{E-morphism definition} follows from the first commutative diagram of~\cref{Cocontinuous E-module maps preserve the cancelling product}, and the second axiom follows from the second commutative diagram.

To see that the functor of~\cref{Underlying categories of an E-category} recovers the original functor $F:\D_1(\A)\rightarrow\D_2(\A)$, suppose we have a map $f:X\rightarrow Y$ in $\D_1(\A)$. Using~\cref{E-modules are E-categories}, we need to check that its image under the prederivator map $F$ is adjunct to the composite below (up to the isomorphism $\lambda:\h_\A\otimes_\A FX\xrightarrow{\;\iso\;} FX$):
\begin{center}
\begin{tikzcd}
\h_\A\arrow{r}{j} & \widetilde{\map}_{\D_1}(X,X)\arrow{rr}{\widetilde{\map}_{\D_1}(X\text{,}f)} && \widetilde{\map}_{\D_1}(X,Y)\arrow{r}{F} & \widetilde{\map}_{\D_2}(FX,FY)
\end{tikzcd}
\end{center} 
This is immediate, using the definition of the unit in $\eD_2$, from~\cref{E-modules are E-categories}, and the second commutative diagram of~\cref{Cocontinuous E-module maps preserve the cancelling product}. 
\end{proof}

\begin{lemma}\label{E-module modifications induce E-natural maps}
Let $F,G:\D_1\rightarrow\D_2$ be cocontinuous $\E$-module maps between closed $\E$-modules. Then any $\E$-module modification
\begin{center}
\begin{tikzcd}
\D_1\arrow[r, bend left=50, "F" above,""{name=U, below}]\arrow[r, bend right=50, "G" below, ""{name=D}]
& \D_2
\arrow[Rightarrow,from=U,to=D,shorten >=0.1cm,shorten <=0.1cm,"\beta", shift right=0.5]
\end{tikzcd}
\end{center}
induces an $\E$-natural transformation between the associated $\E$-morphisms of~\cref{Cocontinuous E-module maps induce E-morphisms}.
\end{lemma}
\begin{proof}
Suppose we have categories $\A$ and $\B$, and objects $X\in\D_1(\A)$ and $Y\in\D_1(\B)$. We need to check that the diagram of~\cref{Rephrasing E-naturality} commutes; equivalently, we may consider its adjunct under the adjunction below:
\begin{center}
\begin{tikzcd}
\E^{\A\op}\arrow[rr, bend left=40, "-\;\otimes_\A FX" above,""{name=U, below}]\arrow[rr,leftarrow, bend right=40, "\widetilde{\map}_{\D_2}(FX\text{,}-)" below, ""{name=D}]
&& \;\D_2
\arrow[phantom,from=U,to=D,"\bot"]
\end{tikzcd}
\end{center}
That diagram commutes, using the description of $F$ and $G$ from~\cref{Cocontinuous E-module maps induce E-morphisms}, and the fact that $\beta$ respects the cancelling tensor product, as in~\cref{E-module modifications respect the cancelling product}.
\end{proof}

\begin{remark}
Suppose $F:\D_1\rightarrow\D_2$ is a continuous map between closed $\E$-modules. If $F$ preserves cotensors then the dual of~\cref{Cocontinuous E-module maps induce E-morphisms} implies that $F$ induces an $\E$-category map $F:\eD_1\rightarrow\eD_2$. Given a map $F$ that is both continuous and cocontinuous, and preserves tensors and cotensors, there are two potentially distinct enrichments on $F$. Without coherence between the isomorphisms that express the preservation of tensors and cotensors, these two induced enrichments need not agree.
\end{remark}


\section{The Yoneda lemma and adjunctions for $\E$-categories}\label{Section The yoneda lemma and adjunctions for E-categories}

We begin this section with~\cref{Yoneda lemma for E-categories}, an $\E$-category analogue of the Yoneda lemma. As is the case in other settings, this result is extremely useful, and we will use it repeatedly in our discussion of enriched derivators in~\cref{Chapter Enriched derivators}. We also use it immediately to prove~\cref{Characterising adjunctions in E-Cat}, which gives a convenient characterisation of adjunctions in the $2$-category $\ECat$, reminiscent of the familiar result for enriched categories in~\cite[Chapter 1]{Kelly82}. We finish this section with~\cref{E-module adjunctions induce E-category adjunctions}, showing that an adjunction between closed $\E$-modules whose left adjoint preserves tensors induces an $\E$-category adjunction. This result is important for the proof, in~\cref{Chapter Enriched derivators}, that closed $\E$-modules induce enriched derivators.

\begin{theorem}[Yoneda lemma for $\E$-categories]\label{Yoneda lemma for E-categories}
Let $\eA$ be an $\E$-category, let $\A$ be a category, and let $X\in\dA_0(\A)$. Let $F:\eA\rightarrow\eE^{\A\op}$ be an $\E$-morphism. Then we have a natural bijection:
\[
\ECat(\eA,\eE^{\A\op})(\widetilde{\map}_{\dA}(X,-),F)\iso\E(\A\op\times\A)(\h_\A,FX)
\]
\end{theorem}
\begin{proof}
Let $\beta:\widetilde{\map}_{\dA}(X,-)\Rightarrow F$ be an $\E$-natural transformation. This determines a map in $\E(\A\op\times\A)$ as follows:
\begin{center}
\begin{tikzcd}
\h_\A\arrow{r}{j} & \widetilde{\map}_{\dA}(X,X)\arrow{r}{\beta_X} & FX
\end{tikzcd}
\end{center} 
On the other hand, given $f:\h_\A\rightarrow FX$ in $\E(\A\op\times\A)$, we can construct the following family of maps, one for each $Y\in\dA_0(\B)$:
\begin{center}
\begin{tikzcd}
\widetilde{\map}_{\dA}(X,Y)\arrow[r,"F"]
& \widetilde{\map}_{\E^{\A\op}}(FX,FY)\arrow[rrr,"\widetilde{\map}_{\E^{\A\op}}(f\text{,}FY)" above] &&& \widetilde{\map}_{\E^{\A\op}}(\h_\A,FY)\arrow[d,"\iso" left,"\varrho" right]\\
&&&& FY
\end{tikzcd}
\end{center}
Here the isomorphism $\varrho:\widetilde{\map}_{\E^{\A\op}}(\h_\A,-)\xrightarrow{\;\iso\;}\id$ is conjugate to the inverse of the unit isomorphism $\rho^{-1}:\id \xrightarrow{\;\iso\;}-\otimes_\A \h_\A$. Thus, for any $Z\in\E(\A\op\times\B)$, the diagram below commutes:
\begin{center}
\begin{equation}\label{Diagram defining varrho}
\begin{tikzcd}[baseline=(current  bounding  box.center)]
\widetilde{\map}_{\E^{\A\op}}(\h_\A,Z)\otimes_\A \h_\A \arrow[rrrd,"\epsilon" above right, bend left=15]\arrow{dd}[left]{\rho}  \\
&&& Z \arrow[leftarrow,llld,"\varrho" below right, bend left=15]\\
\widetilde{\map}_{\E^{\A\op}}(\h_\A,Z)
\end{tikzcd}
\end{equation}
\end{center} 
We claim that this map $\varrho$ is $\E$-natural in $Z\in\E^{\A\op}(\B)$, in which case the composite above is $\E$-natural in $Y\in\dA_0(\B)$, by~\cref{map(f -) is E-natural} and~\cref{F:map(X Y) -> map(FX FY) is E-natural}.

To see this, let $Z\in\E(\A\op\times\B)$ and $W\in\E(\A\op\times\C)$. By \cref{Rephrasing E-naturality}, we need the diagram below to commute:
\begin{center}
\begin{tikzcd}
\widetilde{\map}_{\E^{\A\op}}(Z,W)\arrow{rrr}[above]{\widetilde{\map}_{\E^{\A\op}}(\h_\A,-)}\arrow[rrrdd,"\widetilde{\map}_{\E^{\A\op}}(\varrho\text{,}W)" below left, bend right=15]&&& 
 \widetilde{\map}_{\E^{\A\op}}(\widetilde{\map}_{\E^{\A\op}}(\h_\A,Z),\widetilde{\map}_{\E^{\A\op}}(\h_\A,W)) \arrow{dd}[right]{\widetilde{\map}_{\E^{\A\op}}(\widetilde{\map}_{\E^{\A\op}}(\h_\A,Z),\varrho)} \\\\
&&& 
\widetilde{\map}_{\E^{\A\op}}(\widetilde{\map}_{\E^{\A\op}}(\h_\A,Z),W)
\end{tikzcd}
\end{center} 
Taking adjuncts under the adjunction
\begin{center}
\begin{tikzcd}
\E^{\B\op}\arrow[rrr, bend left=30, "-\otimes_\B \widetilde{\map}_{\E^{\A\op}}(\h_\A\text{,}Z)" above,""{name=U, below}]\arrow[rrr,leftarrow, bend right=30, "\widetilde{\map}_{\E^{\A\op}}(\widetilde{\map}_{\E^{\A\op}}(\h_\A\text{,}Z)\text{,}-)" below=2, ""{name=D}]
&&& \;\E^{\A\op},
\arrow[phantom,from=U,to=D,"\bot"]
\end{tikzcd}
\end{center}
this corresponds to the square below, using the description of $\widetilde{\map}_{\E^{\A\op}}(\h_\A,-)$ given in \cref{mapping space is an E-morphism}:
\begin{center}
\begin{tikzcd}
\widetilde{\map}_{\E^{\A\op}}(Z,W)\otimes_\B \widetilde{\map}_{\E^{\A\op}}(\h_\A,Z)\arrow{rr}[above]{\circ}\arrow{dd}[left]{\widetilde{\map}_{\E^{\A\op}}(Z,W)\otimes_\B\varrho} && 
 \widetilde{\map}_{\E^{\A\op}}(\h_\A,W) \arrow{dd}[right]{\varrho} \\\\
\widetilde{\map}_{\E^{\A\op}}(Z,W)\otimes_\B Z\arrow{rr}[below]{\epsilon} && 
W
\end{tikzcd}
\end{center} 
We can expand the vertical arrows in this diagram using the equation $\varrho=\epsilon\circ\rho^{-1}$ from (\ref{Diagram defining varrho}). Using the description of composition in $\eE^{\A\op}$ from~\cref{E-modules are E-categories}, we can then see that this diagram commutes.

Thus, we have well-defined functions in both directions. We will now show that they are mutually inverse.

First, suppose we have a map $f:\h_\A\rightarrow FX$ in $\E(\A\op\times\A)$. We need to show that the diagram below commutes:
\begin{center}
\begin{tikzcd}
\h_\A\arrow[r,"j"]\arrow[rrrrrd,"f" below left,bend right=8] & \widetilde{\map}_{\dA}(X,X)\arrow[r,"F"]
& \widetilde{\map}_{\E^{\A\op}}(FX,FX)\arrow[rrr,"\widetilde{\map}_{\E^{\A\op}}(f\text{,}FX)" above] &&& \widetilde{\map}_{\E^{\A\op}}(\h_\A,FX)\arrow[d,"\varrho" right]\\
&&&&& FX
\end{tikzcd}
\end{center}
Using~\cref{Extracting f from map(f X)}, this diagram is equal to the diagram below:
\vspace{-2em}
\begin{center}
\begin{equation}\label{Diagram Yoneda lemma}
\begin{tikzcd}[baseline=(current  bounding  box.center)]
\h_\A\arrow[rr,"\tilde{f}"]\arrow[rrd,"f" below left,bend right=8] && \widetilde{\map}_{\E^{\A\op}}(\h_\A,FX)\arrow[d,"\varrho" right]\\
&& FX
\end{tikzcd}
\end{equation}
\end{center}
Here $\tilde{f}$ is the map corresponding to $f:\h_\A\rightarrow FX$, as in the proof of \cref{E-modules are E-categories}. That is, $\tilde{f}$ makes the diagram below commute:
\begin{center}
\begin{tikzcd}
\h_\A\otimes_{\A}\h_\A\arrow{rr}[above]{\lambda}\arrow{dd}[left]{\tilde{f}\otimes_{\A}\h_\A} && 
 \h_\A \arrow{dd}[right]{f} \\\\
 \widetilde{\map}_{\D}(\h_\A,FX)\otimes_\A \h_\A\arrow{rr}[below]{\epsilon} && 
FX
\end{tikzcd}
\end{center} 
Using this commutative diagram, and (\ref{Diagram defining varrho}), it is easy to see that (\ref{Diagram Yoneda lemma}) commutes, using the fact that $\lambda=\rho:\h_\A\otimes\h_\A\rightarrow\h_\A$.

On the other hand, suppose we have an $\E$-natural transformation $\beta:\widetilde{\map}_{\dA}(X,-)\Rightarrow F$. Given any $Y\in\dA_0(\B)$, we need the diagram below to commute:
\begin{center}
\begin{tikzcd}
\widetilde{\map}_{\dA}(X,Y)\arrow[r,"F"]\arrow[rrrrrdd,"\beta_Y" below left,bend right=8] & \widetilde{\map}_{\E^{\A\op}}(FX,FY)\arrow[rrrr,"\widetilde{\map}_{\E^{\A\op}}(\beta_X\text{,}FY)"]
&&&& \widetilde{\map}_{\E^{\A\op}}(\widetilde{\map}_{\dA}(X,X),FY)\arrow[d,"\widetilde{\map}_{\E^{\A\op}}(j\text{,}FY)" right]\\ 
&&&&& \widetilde{\map}_{\E^{\A\op}}(\h_\A,FY)\arrow[d,"\varrho" right]\\
&&&&& FY
\end{tikzcd}
\end{center}
This diagram commutes, using~\cref{Rephrasing E-naturality} for $\beta$, the definition of the structure map $\widetilde{\map}_{\dA}(X,-)$ from \cref{mapping space is an E-morphism}, the commutative diagram (\ref{Diagram defining varrho}), and the third axiom of~\cref{E-category definition} for $\eA$.

Thus, these functions form a bijection
\[
\ECat(\eA,\eE^{\A\op})(\widetilde{\map}_{\dA}(X,-),F)\iso\E(\A\op\times\A)(\h_\A,FX).
\]
We still need to check that the bijection is natural in both $X$ and $F$; note that both sides are indeed functorial in $X$ and $F$, by~\cref{E-category yoneda embedding definition}. 

It is easy to check that this bijection is natural in $F$. On the other hand, suppose we have a map $f:X\rightarrow Y$ in $\dA(\A)$. Naturality in the first variable amounts to the commutativity of the diagram below, for any $F\in\ECat(\eA,\eE^{\A\op})$, and any $\E$-natural map $\beta:\widetilde{\map}_{\dA}(X,-)\Rightarrow F$:
\begin{center}
\begin{tikzcd}
\h_\A\arrow{rrr}[above]{j}\arrow{dd}[left]{j} &&& \widetilde{\map}_{\dA}(X,X)\arrow{rr}[above]{\beta_X} && FX \arrow{dd}[right]{F(f)} \\\\
 \widetilde{\map}_{\dA}(Y,Y) \arrow{rrr}[below]{\widetilde{\map}_{\dA}(f,Y)} &&& 
\widetilde{\map}_{\dA}(X,Y)\arrow{rr}[below]{\beta_Y} && FY
\end{tikzcd}
\end{center} 
This diagram commutes, using~\cref{Extracting f from map(f X)}, and the naturality of the induced natural transformation $\beta$.
\end{proof}

\begin{remark}\label{Another description of Yoneda}
Let $\eA$ be an $\E$-category and let $\A$ be a category. Let $X\in\dA_0(\A)$, let $F:\eA\rightarrow\eE^{\A\op}$ be an $\E$-morphism, and suppose we have a map $f:\h_\A\rightarrow FX$ in $\E(\A\op\times\A)$.  Consider the $\E$-natural map $\beta:\widetilde{\map}_{\dA}(X,-)\Rightarrow F$ determined by~\cref{Yoneda lemma for E-categories}. The component of $\beta$ at $Y\in\dA_0(\B)$, given in~\cref{Yoneda lemma for E-categories}, can also be described as follows:
\begin{center}
\begin{tikzcd}
\widetilde{\map}_{\dA}(X,Y)\arrow[r,"\rho^{-1}"]
& \widetilde{\map}_{\dA}(X,Y)\otimes_\A\h_\A\arrow[rr,"F\otimes_\A f" above] && \widetilde{\map}_{\E^{\A\op}}(FX,FY)\otimes_\A FX\arrow[d,"\epsilon" right]\\
&&& FY
\end{tikzcd}
\end{center}
\end{remark}

\begin{corollary}\label{E-category yoneda embedding is fully faithful}
For any $\E$-category $\eA$ and any category $\A$, the functor
\[
y:\dA(\A)\op\rightarrow\ECat(\eA,\eE^{\A\op})
\]
of~\cref{E-category yoneda embedding definition} is fully faithful.
\end{corollary}
\begin{proof}
This map is faithful by~\cref{Extracting f from map(f X)}. To see that it is full, suppose we have objects $X,Y\in\dA_0(\A)$, and an $\E$-natural transformation
\[
\beta:\widetilde{\map}_{\dA}(Y,-)\Rightarrow \widetilde{\map}_{\dA}(X,-).
\]
Consider the corresponding map $f:\h_\A\rightarrow\widetilde{\map}_{\dA}(X,Y)$ determined by~\cref{Yoneda lemma for E-categories}: 
\[
f:\h_\A\xrightarrow{\;\;j\;\;}\widetilde{\map}_{\dA}(Y,Y)\xrightarrow{\;\;\beta_Y\;\;}\widetilde{\map}_{\dA}(X,Y)
\]
We may think of this as a map $f:X\rightarrow Y$ in $\dA(\A)$. Using~\cref{Another description of Yoneda} and~\cref{mapping space is an E-morphism}, the component of $\beta$ at $Z\in\dA_0(\B)$ is the following composite:
\begin{center}
\begin{tikzcd}
\widetilde{\map}_{\dA}(Y,Z)\arrow[r,"\rho^{-1}"]
& \widetilde{\map}_{\dA}(Y,Z)\otimes_\A\h_\A\arrow[rrr,"\widetilde{\map}_{\dA}(Y\text{,}Z)\otimes_\A f" above] &&& \widetilde{\map}_{\dA}(Y,Z)\otimes_\A\widetilde{\map}_{\dA}(X,Y)\arrow[d,"\circ" right]\\
&&&& \widetilde{\map}_{\dA}(X,Z)
\end{tikzcd}
\end{center}
But this is exactly $\widetilde{\map}_{\dA}(f,Z)$. Thus, we have:
\[
\beta=\widetilde{\map}_{\dA}(f,-):\widetilde{\map}_{\dA}(Y,-)\Rightarrow \widetilde{\map}_{\dA}(X,-)
\]
\end{proof}

\begin{definition}
An $\E$-category \textbf{adjunction} is an adjunction in the $2$-category $\ECat$, in the sense of~\cite{KS74}. Thus, an $\E$-category adjunction consists of $\E$-categories $\eA$ and $\eB$, a left adjoint $\E$-morphism $F:\eA\rightarrow\eB$ and a $\E$-morphism right adjoint $G:\eB\rightarrow\eA$, together with the counit $\E$-natural transformation $\epsilon:F\circ G\Rightarrow \id_{\eB}$ and the unit $\E$-natural transformation $\eta:\id_{\eA}\Rightarrow G\circ F$. These must satisfy the \textbf{triangle identities} $(G\circ\epsilon)\cdot (\eta\circ G)=\id_G$ and $(\epsilon\circ F)\cdot (F\circ \eta)=\id_F$.
\end{definition}

\begin{theorem}\label{Characterising adjunctions in E-Cat}
The data of an adjunction in $\ECat$ is equivalent to the following: a pair of $\E$-morphisms $F:\eA\rightarrow\eB$ and $G:\eB\rightarrow\eA$, and a family of isomorphisms
\[
\widetilde{\map}_{\dB}(FX,Y)\iso\widetilde{\map}_{\dA}(X,GY)
\]
$\E$-natural in both $X\in\dA_0(\A)$ and $Y\in\dB_0(\B)$.
\end{theorem}
\begin{proof}
Let $F:\eA\rightarrow\eB$ and $G:\eB\rightarrow\eA$  be $\E$-morphisms, and suppose we have a family of maps as below, for each category $\A$ and each $X\in\dA_0(\A)$:
\[
\eta_X:\h_\A\rightarrow\widetilde{\map}_{\dA}(X,GFX)
\]
By~\cref{Yoneda lemma for E-categories}, each map $\eta_X$ determines an $\E$-natural transformation:
\begin{center}
\begin{tikzcd}[row sep=tiny]

\eB\arrow[rrrr,"\widetilde{\map}_{\dB}(FX\text{,}-)" above,bend left=25,""{name=D}]\arrow[drr,"G" below left,bend right=10]&&&&\eE^{\A\op}\\
 && 
\eA \arrow[rru,"\widetilde{\map}_{\dA}(X\text{,}-)" below right, bend right=10]
\arrow[Leftarrow,from=2-3,to=D,"\Omega_{\mathsmaller{{X,-}}}" right=2,shorten >=0.4cm,shorten <=0.2cm]
\end{tikzcd}
\end{center}
Using~\cref{Another description of Yoneda}, the component of this map at $Y\in\dB_0(\B)$ is given by the composite
\begin{center}
\begin{tikzcd}
\Omega_{X,Y}:\widetilde{\map}_{\dB}(FX,Y)\arrow[r,"G"]
& \widetilde{\map}_{\dA}(GFX,GY)\arrow[rrr,"\widetilde{\map}_{\dA}(\eta_X\text{,}GY)" above] &&& \widetilde{\map}_{\dA}(X,GY).
\end{tikzcd}
\end{center}
We claim that the family $\eta_X:\h_\A\rightarrow \widetilde{\map}_{\dA}(X,GFX)$ is $\E$-natural in $X\in\dA_0(\A)$ if and only if these maps $\Omega_{X,Y}$ are $\E$-natural in $X$.

First, suppose $\eta:\id_{\eA} \Rightarrow G\circ F$ is $\E$-natural. Using~\cref{F:map(X Y) -> map(FX FY) is E-natural} for $G$, and the description of horizontal composition for $\E$-natural transformations in~\cref{2-category of E-categories}, it follows immediately that $\Omega_{X,Y}$ is $\E$-natural in $X$.

Conversely, suppose that for any category $\B$, and any $Y\in\dB_0(\B)$, the map
\begin{center}
\begin{tikzcd}[row sep=tiny]
 && 
\eB\op \arrow[rrd,"\widetilde{\map}_{\dB}(-\text{,}Y)" above right, bend left=10]
\\
\eA\op\arrow[rrrr,"\widetilde{\map}_{\dA}(-\text{,}GY)" below,bend right=25,""{name=D}]\arrow[urr,"F\op" above left,bend left=10]&&&&\eE^{\B}
\arrow[Rightarrow,from=1-3,to=D,"\Omega_{\mathsmaller{{-,Y}}}",shorten >=0.2cm,shorten <=0.2cm]
\end{tikzcd}
\end{center}
is $\E$-natural. The commutative diagram of~\cref{Rephrasing E-naturality} that expresses this $\E$-naturality corresponds by adjointness to the following commutative diagram, for any $X\in\dA_0(\A)$ and $Z\in\dA_0(\C)$:
\begin{center}
\begin{tikzcd}[row sep=small]
\widetilde{\map}_{\dB}(FZ,Y)\otimes_\C\widetilde{\map}_{\dA}(X,Z)\arrow{ddd}[left]{G\otimes_\C\widetilde{\map}_{\dA}(X,Z)} \arrow{rrr}[above]{\widetilde{\map}_{\dB}(FZ,Y)\otimes_\C F} &&& \widetilde{\map}_{\dB}(FZ,Y)\otimes_\C\widetilde{\map}_{\dB}(FX,FZ)\arrow{dd}[right]{\circ} \\\\
 
&&&\widetilde{\map}_{\dB}(FX,Y)\arrow{dd}[right]{G}\\

\widetilde{\map}_{\dA}(GFZ,GY)\otimes_\C\widetilde{\map}_{\dA}(X,Z)\arrow{ddd}[left]{\widetilde{\map}_{\dA}(\eta_Z,GY)\otimes_\C\widetilde{\map}_{\dA}(X,Z)} \\

&&&\widetilde{\map}_{\dA}(GFX,GY)\arrow{dd}[right]{\widetilde{\map}_{\dA}(\eta_X,GY)}\\\\
 
\widetilde{\map}_{\dA}(Z,GY)\otimes_\C\widetilde{\map}_{\dA}(X,Z)\arrow{rrr}[below]{\circ} &&& \widetilde{\map}_{\dA}(X,GY)
\end{tikzcd}
\end{center} 
In particular, take $X\in\dA_0(\A)$, $Z\in\dA_0(\C)$ and $Y=FZ\in\dB_0(\C)$.  Precomposing with the composite
\begin{center}
\begin{tikzcd}
\widetilde{\map}_{\dA}(X,Z)\arrow[r,"\lambda^{-1}"]
& \h_\C\otimes_\C\widetilde{\map}_{\dA}(X,Z)\arrow[rrr,"j\otimes_\C\widetilde{\map}_{\dA}(X\text{,}Z)" above] &&& \widetilde{\map}_{\dB}(FZ,FZ)\otimes_\A \widetilde{\map}_{\dA}(X,Z)
\end{tikzcd}
\end{center}
this commutative diagram reduces to the following:
\begin{center}
\begin{tikzcd}
\widetilde{\map}_{\dA}(X,Z)\arrow[rr,"F"]\arrow[rrrrdd,"\widetilde{\map}_{\dA}(X\text{,}\eta_Z)" below left,bend right=8] && \widetilde{\map}_{\dB}(FX,FZ)\arrow[rr,"G"]
&& \widetilde{\map}_{\dA}(GFX,GFZ)\arrow[dd,"\widetilde{\map}_{\dA}(\eta_X\text{,}GFZ)" right]\\\\
&&&& \widetilde{\map}_{\dA}(X,GFZ)
\end{tikzcd}
\end{center}
But this is the diagram of~\cref{Rephrasing E-naturality}, expressing the $\E$-naturality of $\eta$.

So the maps $\eta_X$ are $\E$-natural in $X$ if and only if the maps $\Omega_{X,Y}$ are $\E$-natural in $X$. Dually, a family of maps 
\[
\epsilon_Y:\h_\B\rightarrow\widetilde{\map}_{\dB}(FGY,Y)
\]
corresponds to maps, $\E$-natural in $X\in\dA_0(\A)$:
\begin{center}
\begin{tikzcd}
\Lambda_{X,Y}:\widetilde{\map}_{\dA}(X,GY)\arrow[r,"F"]
& \widetilde{\map}_{\dB}(FX,FGY)\arrow[rrr,"\widetilde{\map}_{\dB}(FX\text{,}\epsilon_Y)" above] &&& \widetilde{\map}_{\dB}(FX,Y)
\end{tikzcd}
\end{center}
These maps $\Lambda_{X,Y}$ are $\E$-natural in $Y$ if and only if the maps $\epsilon_Y$ are $\E$-natural in $Y$.

Given this data, the equations $\Lambda_{X,Y}\circ\Omega_{X,Y}=\id$ and $\Omega_{X,Y}\circ\Lambda_{X,Y}=\id$ are equivalent to the triangle identities for $\epsilon$ and $\eta$.
\end{proof} 

\begin{theorem}\label{Defining adjoints representably}
Let $G:\eB\rightarrow\eA$ be an $\E$-morphism. Suppose that, for every category $\A$ and any object $X\in\dA_0(\A)$, there is an object $FX\in\dB_0(\A)$ representing the composite $\E$-morphism below:
\begin{center}
\begin{tikzcd}
\eB\arrow[r,"G"]
& \eA\arrow[rr,"\widetilde{\map}_{\dA}(X\text{,}-)" above] && \eE^{\A\op}
\end{tikzcd}
\end{center}
Then there is a unique way to extend $F$ to an $\E$-morphism $F:\eA\rightarrow\eB$, such that $F$ gives a left adjoint to $G:\eB\rightarrow\eA$.
\end{theorem}
\begin{proof}
On objects, the $\E$-morphism $F:\eA\rightarrow\eB$ takes $X\in\dA_0(\A)$ to the given object $FX\in\dB_0(\A)$.  For any two objects $X\in\dA_0(\A)$ and $Z\in\dA_0(\C)$, we require the structure map
\[
F:\widetilde{\map}_{\dA}(X,Z)\rightarrow\widetilde{\map}_{\dB}(FX,FZ)
\]
in $\E(\A\op\times\C)$. 

By assumption, for any $X\in\dA_0(\A)$, we have an isomorphism
\[
\Omega_{X,Y}:\widetilde{\map}_{\dB}(FX,Y)\xrightarrow{\;\iso\;}\widetilde{\map}_{\dA}(X,GY),
\]
$\E$-natural in $Y\in\dB_0(\B)$. By~\cref{Yoneda lemma for E-categories}, this map is given by the composite
\begin{center}
\begin{tikzcd}
\Omega_{X,Y}:\widetilde{\map}_{\dB}(FX,Y)\arrow[r,"G"]
& \widetilde{\map}_{\dA}(GFX,GY)\arrow[rrr,"\widetilde{\map}_{\dA}(\eta_X\text{,}GY)" above] &&& \widetilde{\map}_{\dA}(X,GY),
\end{tikzcd}
\end{center}
for a map $\eta_X:X\rightarrow GFX$, as in the proof of~\cref{Characterising adjunctions in E-Cat}. Using these maps, we define the structure map as follows, for $X\in\dA_0(\A)$ and $Z\in\dA_0(\C)$:
\begin{center}
\begin{tikzcd}
F:\widetilde{\map}_{\dA}(X,Z)\arrow[rrr,"\widetilde{\map}_{\dA}(X\text{,}\eta_Z)"]
&&& \widetilde{\map}_{\dA}(X,GFZ)\arrow[rr,"\Omega_{X,FZ}^{-1}" above] && \widetilde{\map}_{\dA}(FX,FZ)
\end{tikzcd}
\end{center}
Note that this definition is necessary to make the diagram below commute:
\begin{center}
\begin{equation}\label{Diagram would-be E-naturality of eta}
\begin{tikzcd}[baseline=(current  bounding  box.center)]
\widetilde{\map}_{\dA}(X,Z)\arrow[rr,"F"]\arrow[rrrrdd,"\widetilde{\map}_{\dA}(X\text{,}\eta_Z)" below left,bend right=8] && \widetilde{\map}_{\dB}(FX,FZ)\arrow[rr,"G"]
&& \widetilde{\map}_{\dA}(GFX,GFZ)\arrow[dd,"\widetilde{\map}_{\dA}(\eta_X\text{,}GFZ)" right]\\\\
&&&& \widetilde{\map}_{\dA}(X,GFZ)
\end{tikzcd}
\end{equation}
\end{center}
If we show that $F$ is indeed an $\E$-morphism, then this diagram will express $\E$-naturality for the maps $\eta_X:X\rightarrow GFX$. Thus, once we verify the axioms of~\cref{E-morphism definition}, it will follow that $F$ is the unique left adjoint to $G$.

To verify the first axiom, note that the commutativity of the diagram (\ref{Diagram would-be E-naturality of eta}) above implies that the diagram below commutes, for any $X\in\dA_0(\A)$, $Z\in\dA_0(\C)$ and $Y\in\dB_0(\B)$:
\begin{center}
\begin{equation}\label{Diagram would-be E-naturality of omega}
\begin{tikzcd}[column sep=0cm,baseline=(current  bounding  box.center)]
\widetilde{\map}_{\dA}(X,Z)\arrow[rd,"\widetilde{\map}_{\dA}(-\text{,}GY)",bend left=12]\arrow{dd}[left]{F}\\
& \widetilde{\map}_{\E^\B}(\widetilde{\map}_{\dA}(Z,GY),\widetilde{\map}_{\dA}(X,GY)) \arrow{dddd}{\widetilde{\map}_{\E^{\B}}(\Omega_{Z,Y},\widetilde{\map}_{\dA}(X,GY))}\\
\widetilde{\map}_{\dB}(FX,FZ)\arrow{dd}[left]{\widetilde{\map}_{\dB}(-,Y)}
\\\\
\widetilde{\map}_{\E^{\B}}(\widetilde{\map}_{\dB}(FZ,Y),\widetilde{\map}_{\dB}(FX,Y))\arrow[dr,"\widetilde{\map}_{\E^\B}(\widetilde{\map}_{\dB}(FZ\text{,}Y)\text{,}\Omega_{X\text{,}Y})" below left, bend right=12] 
\\
&\widetilde{\map}_{\E^\B}(\widetilde{\map}_{\dB}(FZ,Y),\widetilde{\map}_{\dA}(X,GY))
\end{tikzcd}
\end{equation}
\end{center} 
If we knew that $F$ was an $\E$-morphism, this diagram would express the $\E$-naturality of the maps $\Omega_{X,Y}$ in $X$. The commutativity of this diagram follows by the same argument as in the proof of~\cref{Characterising adjunctions in E-Cat}. It follows from the commutativity of the diagram (\ref{Diagram would-be E-naturality of omega}) and the $\E$-naturality of $\widetilde{\map}_{\dA}(-,\eta_Z)$, as in~\cref{map(f -) is E-natural}, that the composite
\begin{center}
\begin{tikzcd}
F:\widetilde{\map}_{\dA}(X,Z)\arrow[rrr,"\widetilde{\map}_{\dA}(X\text{,}\eta_Z)"]
&&& \widetilde{\map}_{\dA}(X,GFZ)\arrow[rr,"\Omega_{X,FZ}^{-1}" above] && \widetilde{\map}_{\dA}(FX,FZ)
\end{tikzcd}
\end{center}
also satisfies the would-be $\E$-naturality condition in $X$. That is, the diagram below commutes, for $X\in\dA_0(\A)$, $Z\in\dA_0(\C)$ and $W\in\dA_0(\cD)$:
\begin{center}
\begin{tikzcd}[column sep=0cm]
\widetilde{\map}_{\dA}(X,Z)\arrow[rd,"\widetilde{\map}_{\dA}(-\text{,}W)",bend left=12]\arrow{dd}[left]{F}\\

& \widetilde{\map}_{\E^\cD}(\widetilde{\map}_{\dA}(Z,W),\widetilde{\map}_{\dA}(X,W)) \arrow{dddd}{\widetilde{\map}_{\E^{\cD}}(\widetilde{\map}_{\dA}(Z,W),F)}\\

\widetilde{\map}_{\dB}(FX,FZ)\arrow{dd}[left]{\widetilde{\map}_{\dB}(-,FW)}
\\\\
 
\widetilde{\map}_{\E^{\cD}}(\widetilde{\map}_{\dB}(FZ,FW),\widetilde{\map}_{\dB}(FX,FW))\arrow[dr,"\widetilde{\map}_{\E^\cD}(F\text{,}\widetilde{\map}_{\dB}(FX\text{,}FW))" below left, bend right=12] 
\\
&\widetilde{\map}_{\E^\cD}(\widetilde{\map}_{\dA}(Z,W),\widetilde{\map}_{\dB}(FX,FW))
\end{tikzcd}
\end{center} 
Taking adjuncts, this diagram reduces to the first axiom of~\cref{E-morphism definition} for $F$. Essentially, this is the proof of~\cref{F:map(X Y) -> map(FX FY) is E-natural} in reverse.

The second axiom of~\cref{E-morphism definition} for $F$ follows easily from~\cref{Extracting f from map(f X)} and the corresponding axiom for $G$.
\end{proof}

Suppose we have a adjunction 
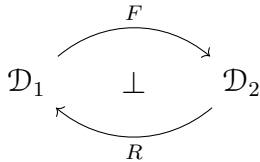
\begin{center}
\begin{tikzcd}
\D_1\arrow[rr, bend left=40, "F" above,""{name=U, below}]\arrow[rr,leftarrow, bend right=40, "R" below, ""{name=D}]
&& \D_2
\arrow[phantom,from=U,to=D,"\bot"]
\end{tikzcd}
\end{center}
between closed $\E$-modules. By~\cref{Left adjoint preserves tensors iff right adjoint preserves cotensors}, $F$ preserves tensors if and only if $R$ preserves cotensors. If this is the case, then, by~\cref{Cocontinuous E-module maps induce E-morphisms}, $F$ and $R$ both induce $\E$-morphisms on the associated $\E$-categories. We will now show that these $\E$-morphisms form an adjunction:

\begin{proposition}\label{E-module adjunctions induce E-category adjunctions}
Let $\D_1$ and $\D_2$ be closed $\E$-modules, and suppose we have an adjunction:
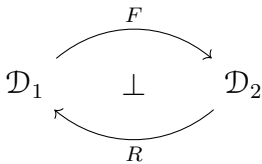
\begin{center}
\begin{tikzcd}
\D_1\arrow[rr, bend left=40, "F" above,""{name=U, below}]\arrow[rr,leftarrow, bend right=40, "R" below, ""{name=D}]
&& \D_2
\arrow[phantom,from=U,to=D,"\bot"]
\end{tikzcd}
\end{center}
If $F$ preserves tensors, then the induced $\E$-morphisms form an $\E$-category adjunction.
\end{proposition}
\begin{proof}
Let $\A$ be a category and let $X\in\D_1(\A)$. Consider the isomorphism 
\begin{center}
\begin{tikzcd}
\E^{\A\op}\arrow[ddrr,bend right,"-\otimes_\A FX" below left,""{name=L,above right}]\arrow[rr,"-\mathsmaller{\otimes_\A} X" above] && \D_1\arrow[Leftarrow,to=L,"\varphi" left=5,"\iso" below=4,shorten >=0.3cm,shorten <=0.7cm] \arrow[dd,"F" right] 
\\\\
&&\D_2
\end{tikzcd}
\end{center}
of~\cref{Cocontinuous E-module maps preserve the cancelling product}. The conjugate of this map gives an isomorphism between the right adjoints:
\begin{center}
\begin{tikzcd}
\E^{\A\op}\arrow[leftarrow,ddrrr,bend right,"\widetilde{\map}_{\D_2}(FX\text{,}-)" below left,""{name=L,above right}]\arrow[leftarrow,rrr,"\widetilde{\map}_{\D_1}(X\text{,}-)" above] &&& \D_1\arrow[Rightarrow,to=L,"\iso" left=5,"\Omega" below=4,shorten >=0.3cm,shorten <=0.8cm] \arrow[leftarrow,dd,"R" right] 
\\\\
&&&\D_2
\end{tikzcd}
\end{center}
By~\cref{Modification respects tensors iff conjugate respects cotensors}, and the dual of~\cref{E-module modifications induce E-natural maps}, this isomorphism induces an $\E$-natural isomorphism between the induced $\E$-category maps. Thus, we obtain an isomorphism
\[
\Omega_{X,Y}:\widetilde{\map}_{\D_2}(FX,Y)\xrightarrow{\;\iso\;}\widetilde{\map}_{\D_1}(X,RY),
\]
$\E$-natural in $Y\in\D_2(\B)$. By~\cref{Defining adjoints representably}, there is a unique way to define an $\E$-morphism structure on $F$, such that $F$ is a left adjoint to $R$. Specifically, for $X\in\D_1(\A)$ and $Z\in\D_1(\C)$, we have:
\begin{center}
\begin{tikzcd}
F:\widetilde{\map}_{\D_1}(X,Z)\arrow[rrr,"\widetilde{\map}_{\D_1}(X\text{,}\eta_Z)"]
&&& \widetilde{\map}_{\D_1}(X,RFZ)\arrow[rr,"\Omega_{X,FZ}^{-1}" above] && \widetilde{\map}_{\D_2}(FX,FZ)
\end{tikzcd}
\end{center} 
Here the map $\eta_Z:Z\rightarrow RFZ$ is obtained using~\cref{Yoneda lemma for E-categories}, but it is easy to check that this is the unit of the unenriched adjunction. It follows that this composite is adjunct to
\[
\widetilde{\map}_{\D_1}(X,Y)\otimes_{\A}FX\xrightarrow{\;\;\varphi\;\;}F(\widetilde{\map}_{\D_1}(X,Y)\otimes_{\A} X)\xrightarrow{\;\;F(\epsilon)\;\;}FY
\]
under the adjunction $-\otimes_{\A}FX\dashv \widetilde{\map}_{\D_2}(FX,-)$. Thus, the $\E$-morphism structure on $F$ induced by~\cref{Defining adjoints representably} is the original structure of~\cref{Cocontinuous E-module maps induce E-morphisms}.
\end{proof}


\section{Transferring enrichments}\label{Section Transferring enrichments}

In this section we study monoidal morphisms, starting with~\cref{Monoidal maps respect the cancelling tensor}. This is an analogue of~\cref{Cocontinuous E-module maps preserve the cancelling product}, for monoidal maps rather than module maps. In~\cref{Transferring enrichment along a monoidal left adjoint}, we use this result to show that we can transfer enrichment along a monoidal adjunction. This construction has a number of applications, but in particular, we use it to define shifted $\E$-categories in~\cref{Shifted E-category is E-enriched}.

\begin{definition}\label{Monoidal morphism definition}
Let $\E_1$ and $\E_2$ be monoidal derivators. A derivator map $F:\E_1\rightarrow\E_2$ is called a \textbf{monoidal morphism} if it is equipped with an isomorphism
\[
\xi:F\mathbb{1}\xrightarrow{\;\;\iso\;\;}\mathbb{1}
\]
in $\E_2(\0)$, and an isomorphism:
\begin{center}
\begin{tikzcd}
 \E_1\times\E_1 \arrow{rr}[above]{F\times F}\arrow{dd}[left]{\otimes} && 
 \E_2\times\E_2 \arrow{dd}[right]{\otimes} \\
 
 \\
 
 \E_1 \arrow{rr}[below]{F}\arrow[uurr,Rightarrow,"\chi" above left,"\iso" below right,shorten >=1.1cm,shorten <=1.1cm] && 
 \E_2

\end{tikzcd}
\end{center} 
These isomorphisms must satisfy familiar coherence conditions, which precisely mirror the axioms for monoidal functors, as in~\cite{Borceux94}.
\end{definition}

\begin{lemma}\label{Monoidal maps respect the cancelling tensor}
Let $F:\E_1\rightarrow\E_2$ be a cocontinuous monoidal morphism. For any categories $\A$, $\B$ and $\C$, and any $X\in\E_1(\A\op\times\B)$ and $Y\in\E_1(\B\op\times\C)$, the isomorphism $\chi$ of~\cref{Monoidal morphism definition} induces an isomorphism
\[
\chi:F(Y\otimes_\B X)\xrightarrow{\;\;\;\iso\;\;\;} FY\otimes_\B FX
\]
in $\E_2(\A\op\times\C)$. Similarly, the isomorphism $\xi$ induces an isomorphism
\[
\xi:F\h_\A\xrightarrow{\;\;\;\iso\;\;\;} \h_\A
\]
in $\E_2(\A\op\times\A)$, for any category $\A$. 

These satisfy the following coherence conditions:
\begin{enumerate}
\item Given categories $\A$, $\B$, $\C$ and $\cD$, and objects $X\in\E_1(\A\op\times\B)$, $Y\in\E_1(\B\op\times\C)$ and $Z\in\E_1(\C\op\times\cD)$, the diagram below commutes: 
\begin{center}
\begin{tikzcd}
F((Z\otimes_\C Y)\otimes_\B X)\arrow{dd}[left]{\chi} \arrow{rr}[above]{F(\alpha)} && F(Z\otimes_\C (Y\otimes_\B X))\arrow{dd}[right]{\chi} \\\\
 
F(Z\otimes_\C Y)\otimes_\B F X\arrow{dd}[left]{\chi\;\mathsmaller{\otimes_{\B}}\;F X} 
&& FZ\otimes_\C F(Y\otimes_\B X)\arrow{dd}[right]{FZ\;\mathsmaller{\otimes_\C}\;\chi}\\\\
 
(FZ\otimes_\C FY)\otimes_\B FX \arrow{rr}[below]{\alpha}
&& FZ\otimes_\C (FY\otimes_\B FX)
\end{tikzcd}
\end{center} 
\item Given $X\in\E(\A\op\times\B)$, the diagram below commutes:
\begin{center}
\begin{tikzcd}
 F(X\otimes_\A \h_\A) \arrow{rr}[above]{\chi}\arrow{dd}[left]{F(\rho)} && 
 FX\otimes_\A F\h_\A \arrow{dd}[right]{FX\mathsmaller{\otimes_\A}\;\xi} \\\\
 FX  && 
 FX\otimes_\A \h_\A\arrow{ll}[below]{\rho}
\end{tikzcd}
\end{center} 
\item Given $X\in\E(\A\op\times\B)$, the diagram below commutes:
\begin{center}
\begin{tikzcd}
 F(\h_\B\otimes_\B X) \arrow{rr}[above]{\chi}\arrow{dd}[left]{F(\lambda)} && 
 F\h_\B \otimes_{\B} F X \arrow{dd}[right]{\xi\;\mathsmaller{\otimes_\B}FX} \\\\
 FX  && 
 \h_\B \otimes_\B F X\arrow{ll}[below]{\lambda}
\end{tikzcd}
\end{center} 
\end{enumerate}
\end{lemma}
\begin{proof}
Let $X\in\E_1(\A\op\times\B)$ and $Y\in\E_1(\B\op\times\C)$. The map $\chi:F(Y\otimes_\B X)\xrightarrow{\;\;\iso\;\;} FY\otimes_\B FX$ is given by the composite
\[
F\int^\B(Y\widetilde{\otimes}X)\xrightarrow{\;\;\;\;\;\iso\;\;\;\;\;}\int^\B F(Y\widetilde{\otimes}X)\xrightarrow{\;\;\;\;\;\int^\B\chi\;\;\;\;\;}\int^\B (FY\widetilde{\otimes}FX),
\]
where the first isomorphism is induced by the cocontinuity of $F$.

For any category $\A$, the isomorphism $\xi:F\h_\A\xrightarrow{\;\;\iso\;\;} \h_\A$ is given by the composite
\[
F\partial_\A\mathbb{1}\xrightarrow{\;\;\;\iso\;\;\;}\partial_\A F\mathbb{1}\xrightarrow{\;\;\;\partial_\A \xi\;\;\;} \partial_\A\mathbb{1}
\]
where the first isomorphism follows from the cocontinuity of $F$.

Each diagram in the statement corresponds to an axiom for the monoidal morphism $F$, as in~\cref{Monoidal morphism definition}. The commutativity of each diagram can be reduced easily to the corresponding axiom, using essentially the same arguments as in the proof of~\cref{Cocontinuous E-module maps preserve the cancelling product}.
\end{proof}

\begin{remark}
Let $F:\E_1\rightarrow\E_2$ be a cocontinuous monoidal morphism. Then $F$ induces a bicategory map 
\[
\Prof(F):\Prof(\E_1)\rightarrow\Prof(\E_2)
\]
between the bicategories of~\cref{Definition of Prof(E)}. Explicitly, the map is the identity on objects, and, given categories $\A$ and $\B$, the map on hom-categories is $F:\E_1(\B\op\times\A)\rightarrow\E_2(\B\op\times\A)$. The structure isomorphisms are $\chi$ and $\xi$ of~\cref{Monoidal maps respect the cancelling tensor}.
\end{remark}

\begin{proposition}\label{Transferring enrichment along a monoidal left adjoint}
Let $\E_1$ and $\E_2$ be closed symmetric monoidal derivators, and suppose we have an adjunction
\begin{center}
\begin{tikzcd}
\E_1\arrow[rr, bend left=40, "F" above,""{name=U, below}]\arrow[rr,leftarrow, bend right=40, "R" below, ""{name=D}]
&& \;\E_2
\arrow[phantom,from=U,to=D,"\bot"]
\end{tikzcd}
\end{center}
such that the left adjoint $F:\E_1\rightarrow\E_2$ is monoidal. Let $\eA$ be an $\E_2$-category. Then we can construct an $\E_1$-category with the same objects as $\eA$, and, for $X\in\dA_0(\A)$ and $Y\in\dA_0(\B)$, with mapping objects given by
\[
R\;\widetilde{\map}_{\dA}(X,Y)\in\E_1(\A\op\times\B).
\]
\end{proposition}
\begin{proof}
Units and composition for the $\E_1$-enrichment are defined as follows. Given any\linebreak$X\in\dA_0(\A)$, the unit
\[
j:\h_\A\rightarrow R\;\widetilde{\map}_{\dA}(X,X)
\]
is adjunct to the composite
\begin{center}
\begin{tikzcd}
F\h_\A\arrow[r,"\xi"] & \h_\A \arrow[r,"j"]
& \widetilde{\map}_{\dA}(X,X),
\end{tikzcd}
\end{center}
where $\xi$ is the isomorphism of~\cref{Monoidal maps respect the cancelling tensor}, and $j$ is the unit for the $\E_2$-enrichment on $\eA$. Similarly, given $X\in\dA_0(\A)$, $Y\in\dA_0(\B)$ and $Z\in\dA_0(\C)$, composition is adjunct to the following, where $\chi$ is the isomorphism of~\cref{Monoidal maps respect the cancelling tensor}:
\begin{center}
\begin{tikzcd}
F(R\;\widetilde{\map}_{\dA}(Y,Z)\otimes_{\B}R\;\widetilde{\map}_{\dA}(X,Y))\arrow[r,"\chi"]
& FR\;\widetilde{\map}_{\dA}(Y,Z)\otimes_{\B}FR\;\widetilde{\map}_{\dA}(X,Y)\arrow[d,"\epsilon\;\mathsmaller{\otimes_{\B}}\;\epsilon" right]\\
& \widetilde{\map}_{\dA}(Y,Z)\otimes_{\B}\widetilde{\map}_{\dA}(X,Y) \arrow[d,"\circ" right]\\
& \widetilde{\map}_{\dA}(X,Z)
\end{tikzcd}
\end{center}

Using adjointness, and the commutative diagrams of~\cref{Monoidal maps respect the cancelling tensor}, the $\E_1$-category axioms of~\cref{E-category definition} reduce easily to the corresponding $\E_2$-category axioms for $\eA$.
\end{proof}

Given a monoidal left adjoint $F:\E_1\rightarrow\E_2$ as in~\cref{Transferring enrichment along a monoidal left adjoint}, and an $\E_2$-category $\eA$, we will continue to denote the associated $\E_1$-category by $\eA$. This is partially justified by the following:

\begin{remark}\label{Underlying category of transferred enrichment}
Let $F:\E_1\rightarrow\E_2$ be a monoidal left adjoint, and let $\eA$ be an $\E_2$-category. Then for any category $\A$, the induced category $\dA(\A)$, calculated with respect to the induced $\E_1$-enrichment of~\cref{Transferring enrichment along a monoidal left adjoint}, agrees with that calculated using the original $\E_2$-enrichment. To see this, suppose we have $X,Y\in\dA_0(\A)$, and consider the isomorphisms:
\begin{align*}
\E_2(\A\op\times\A)(\h_\A,\widetilde{\map}_{\dA}(X,Y)) & \iso \E_2(\A\op\times\A)(F\h_\A,\widetilde{\map}_{\dA}(X,Y))\\
 & \iso \E_1(\A\op\times\A)(\h_\A,R\;\widetilde{\map}_{\dA}(X,Y))
\end{align*}
It is easy to see that these isomorphisms preserve composition and units, using~\cref{Monoidal maps respect the cancelling tensor}.
\end{remark}


\begin{proposition}\label{Interaction of cancelling tensor with delta}
Let $\E$ be a monoidal derivator. Then for any category $\J$, the map
\[
\partial_\J:\E\rightarrow\E^{\J\op\times\J}
\]
is a monoidal morphism, where $\E^{\J\op\times\J}$ is considered with the monoidal structure of~\cref{E shifted by Aop x A is monoidal}.
\end{proposition}
\begin{proof}
The structure isomorphisms for $\partial_\J$ are as follows. Define
\[
\xi=\id:\partial_\J\mathbb{1}\rightarrow\h_\J
\]
in $\E(\J\op\times\J)$, and define $\chi$ to be the composite below:
\begin{center}
\begin{tikzcd}[row sep=tiny]
\E\times\E\arrow[dddddddd,"\otimes" left]\arrow[rrr,"\partial_{\J_{1,2}}\times\E" above] &&& \E^{\J_1\op\times\J_2}\times\E\arrow[dddddddd,"\widetilde{\otimes}" left]\arrow[rrr,"\E^{\J_1\op\times\J_2}\times\partial_{\J_{3,4}}" above,] &&& \E^{\J_1\op\times\J_2}\times\E^{\J_3\op\times\J_4}\arrow[ddddd,"\widetilde{\otimes}" right]
\\\\\\\\\\ &&&&&& 
\E^{\J_1\op\times\J_2\times\J_3\op\times\J_4} \arrow[rddd,"\int^{\J_{1,4}}" above right, bend left=10]
\\\\\\
\E\arrow[Rightarrow,rrruuuuuuuu,"\iso" above left,shorten >=1.8cm,shorten <=1.8cm]\arrow[rrr,"\partial_{\J_{1,2}}" below]  &&&
\E^{\J_1\op\times\J_2}\arrow[Rightarrow,rrruuuuuuuu,"\iso" above left,shorten >=2cm,shorten <=2cm,shift left=5]\arrow[rrrr,equal,bend right=20]\arrow[uuurrr,"\partial_{\J_{3,4}}" below right,bend left=10]&&&&\E^{\J_3\op\times\J_2}
\\\\ &&&&&& \;

\arrow[Rightarrow,from=6-7,to=11-7,"\iso",shorten >=0.4cm,shorten <=0.4cm]
\end{tikzcd}
\end{center}
Here the first two isomorphisms follow from the cocontinuity of $\otimes$, as in~\cref{Cocontinuity of tensor as a modification}. The third is the isomorphism of~\cref{Interaction of delta and integral}. Using~\cref{Cancelling As and Fubini}, the pasting diagram above is equal to the one below:
\begin{center}
\begin{tikzcd}[row sep=tiny]
\E\times\E\arrow[dddddddd,"\otimes" left]\arrow[rrr,"\E\times\partial_{\J_{3,4}}" above] &&& \E\times\E^{\J_3\op\times\J_4}\arrow[dddddddd,"\widetilde{\otimes}" left]\arrow[rrr,"\partial_{\J_{1,2}}\times\E^{\J_3\op\times\J_4}" above,] &&& \E^{\J_1\op\times\J_2}\times\E^{\J_3\op\times\J_4}\arrow[ddddd,"\widetilde{\otimes}" right]
\\\\\\\\\\ &&&&&& 
\E^{\J_1\op\times\J_2\times\J_3\op\times\J_4} \arrow[rddd,"\int^{\J_{1,4}}" above right, bend left=10]
\\\\\\
\E\arrow[Rightarrow,rrruuuuuuuu,"\iso" above left,shorten >=1.8cm,shorten <=1.8cm]\arrow[rrr,"\partial_{\J_{3,4}}" below]  &&&
\E^{\J_3\op\times\J_4}\arrow[Rightarrow,rrruuuuuuuu,"\iso" above left,shorten >=2cm,shorten <=2cm,shift left=5]\arrow[rrrr,equal,bend right=20]\arrow[uuurrr,"\partial_{\J_{1,2}}" below right,bend left=10]&&&&\E^{\J_3\op\times\J_2}
\\\\ &&&&&& \;

\arrow[Rightarrow,from=6-7,to=11-7,"\iso",shorten >=0.4cm,shorten <=0.4cm]
\end{tikzcd}
\end{center}
Thus, this diagram gives a second description of $\chi$. Using these descriptions, and the definitions of $\lambda$ and $\rho$ from~\cref{Coherence for the cancelling version of the closed E-action}, it is easy to see that the diagrams below commute, for any $X\in\E(\A)$:
\begin{center}
\begin{tikzcd}
 \partial_{\J}(X\;\widetilde{\otimes}\;\mathbb{1})\arrow{rr}[above]{\chi}\arrow{dd}[left]{\partial_{\J}(\rho)} && 
 \partial_{\J_{1,2}}X \otimes_{\J_{1,4}}\partial_{\J_{3,4}}\mathbb{1} \arrow[dd, equal]\\\\
 \partial_{\J}X  && 
 \partial_{\J_{1,2}}X \otimes_{\J_{1,4}}\h_{\J_{3,4}}\arrow{ll}[below]{\rho}
\end{tikzcd}
\hspace{0.1cm}
\begin{tikzcd}
 \partial_{\J}(\mathbb{1}\;\widetilde{\otimes}\;X)\arrow{rr}[above]{\chi}\arrow{dd}[left]{\partial_{\J}(\lambda)} && 
 \partial_{\J_{1,2}}\mathbb{1} \otimes_{\J_{1,4}}\partial_{\J_{3,4}}X \arrow[dd, equal]\\\\
 \partial_{\J}X  && 
 \h_{\J_{1,2}}\otimes_{\J_{1,4}}\partial_{\J_{3,4}}X\arrow{ll}[below]{\lambda}
\end{tikzcd}
\end{center} 
These express the unit axioms for $\partial_\J$. It remains to show that the diagram below commutes, for any $X,Y,Z\in\E(\A)$:
\begin{center}
\begin{tikzcd}
\partial_{\J}((Z\otimes Y)\otimes X)\arrow{dd}[left]{\chi} \arrow{rr}[above]{\partial_{\J}(\alpha)} && \partial_{\J}(Z\otimes (Y\otimes X))\arrow{dd}[right]{\chi} \\\\
 
\partial_{\J_{3,2}}(Z\otimes Y)\otimes_{\J_{3,6}} \partial_{\J_{5,6}} X\arrow{dd}[left]{\chi\;\mathsmaller{\otimes_{\J_{3,6}}}\;\partial_{\J_{5,6}} X} 
&& \partial_{\J_{1,2}}Z\otimes_{\J_{1,4}} \partial_{\J_{5,4}}(Y\otimes_\B X)\arrow{dd}[right]{\partial_{\J_{1,2}}Z\;\mathsmaller{\otimes_{\J_{1,4}}}\;\chi}\\\\
 
(\partial_{\J_{1,2}}Z\otimes_{\J_{1,4}} \partial_{\J_{3,4}}Y)\otimes_{\J_{3,6}} \partial_{\J_{5,6}}X \arrow{rr}[below]{\alpha}
&& \partial_{\J_{1,2}}Z\otimes_{\J_{1,4}} (\partial_{\J_{3,4}}Y\otimes_{\J_{3,6}} \partial_{\J_{5,6}}X)
\end{tikzcd}
\end{center} 
Note that, using the shifted derivator, we may assume $X,Y,Z\in\E(\0)$. Using the definition of $\alpha$ from~\cref{Coherence for the cancelling version of the closed E-action}, and~\cref{Modifications respect coends} and~\cref{Derivator maps respect delta coend cancellation}, we can see that this diagram commutes.
\end{proof}

\begin{example}\label{Shifted E-category enriched over E Jop x J}
Let $\eA$ be an $\E$-category. For any category $\J$, we may form an~\linebreak$\E^{\J\op\times\J}$-category $\eA^\J$ as follows. For any category $\A$, we define 
\[
\dA^\J_0(\A)=\dA_0(\J\times\A).
\]
Given $X\in\dA^\J_0(\A)$ and $Y\in\dA^\J_0(\B)$, define
\[
\widetilde{\map}_{\dA^\J}(X,Y)=\widetilde{\map}_{\dA}(X,Y)\in\E(\J\op\times\J\times\A\op\times\B).
\]
Units and composition are inherited from the $\E$-category structure on $\eA$, and the coherence conditions of~\cref{E-category definition} carry over immediately. Note that, for any category $\A$, we have an isomorphism
\begin{equation}\label{Diagram undelying category of shifted E-prederivator}
\dA^\J(\A)\iso\dA(\J\times\A)
\end{equation}
between the category induced by the $\E^{\J\op\times\J}$-category $\eA^\J$ at $\A$, and the category induced by the $\E$-category $\eA$ at $\J\times\A$. 
\end{example}

\begin{example}\label{Shifted E-category is E-enriched}
Let $\eA$ be an $\E$-category, let $\J$ be a category, and consider the~\linebreak$\E^{\J\op\times\J}$-category $\eA^\J$ of~\cref{Shifted E-category enriched over E Jop x J}. By~\cref{Transferring enrichment along a monoidal left adjoint} and~\cref{Interaction of cancelling tensor with delta}, we may transfer the $\E^{\J\op\times\J}$-enrichment of $\eA^\J$ along the adjunction
\begin{center}
\begin{tikzcd}
\E\arrow[rr, bend left=40, "\partial_\J" above,""{name=U, below}]\arrow[rr,leftarrow, bend right=40, "\int_\J" below, ""{name=D}]
&& \;\;\;\;\E^{\J\op\times\J}.
\arrow[phantom,from=U,to=D,"\bot"]
\end{tikzcd}
\end{center} 
Thus, we obtain the \textbf{shifted} $\E$-category $\eA^\J$. Explicitly, for any category $\A$, we have~\linebreak$\dA^\J_0(\A)=\dA_0(\J\times\A)$, and, given $X\in\dA^\J_0(\A)$ and $Y\in\dA^\J_0(\B)$, we have
\[
\widetilde{\map}_{\dA^\J}(X,Y)=\int_\J\widetilde{\map}_{\dA}(X,Y)\in\E(\A\op\times\B).
\]
\end{example}

\begin{remark}\label{Any A(J) is E(0)-enriched for an E-prederivator}
For any $\E$-category $\eA$ and any categories $\J$ and $\A$, \cref{Underlying category of transferred enrichment} and the isomorphism (\ref{Diagram undelying category of shifted E-prederivator}) give rise to an isomorphism
\[
\dA^\J(\A)\iso\dA(\J\times\A),
\]
where $\dA^\J(\A)$ is the category induced by the $\E$-category $\eA^\J$ at $\A$, and $\dA(\J\times\A)$ is the category induced by the $\E$-category $\eA$ at $\J\times\A$. 
Thus, taking $\A=\0$ and using \cref{A([0]) is E([0])-enriched}, the induced category $\dA(\J)$ is canonically $\E(\0)$-enriched, as well as $\E(\J\op\times\J)$-enriched.
\end{remark}


\chapter{Enriched Derivators}\label{Chapter Enriched derivators}

In this chapter, we study additional structure and properties that we can ask for in an $\E$-category. In \cref{Section E-prederivators}, we study $\E$-prederivators, which we introduce in \cref{E-prederivator definition}. These are $\E$-categories equipped with a notion of pullback along functors; we show in \cref{E-prederivators induce prederivators} that these pullbacks are part of a prederivator structure on the induced categories of \cref{Underlying categories of an E-category}. In \cref{Section The 2-category of E-prederivators}, we show that $\E$-morphisms and $\E$-natural transformations between $\E$-prederivators induce prederivator maps and modifications. Using this, and the Yoneda lemma of \cref{Yoneda lemma for E-categories}, we prove a representability theorem for $\E$-prederivators in \cref{E-morphism is representable if underlying morphism is}. Finally, in \cref{Section E-derivators}, we define weighted homotopy limits and colimits in an $\E$-category, and use these to define $\E$-derivators in \cref{(Left) E-derivator definition}. We show, in \cref{Closed E-modules induce E-derivators}, that the $\E$-category associated to a closed $\E$-module is an $\E$-derivator. Finally, in \cref{(Left) E-derivators induce (left) derivators}, we show that any $\E$-derivator induces a derivator.

\section{$\E$-Prederivators}\label{Section E-prederivators}

In this section, we introduce $\E$-prederivators, which are $\E$-categories endowed with extra structure. In particular, given a functor $u:\A\rightarrow\B$ and an $\E$-prederivator $\eA$, we are able to form a pullback along $u$ in $\eA$. These pullbacks form part of a prederivator structure on the categories induced by $\eA$. We record this fact in \cref{E-prederivators induce prederivators}. We also show, in \cref{In an E-prederivator map(X -) is a prederivator map}, that the mapping objects in an $\E$-prederivator induce prederivator maps. Finally, in \cref{E-modules induce E-prederivators}, we show that, given a closed $\E$-module $\D$, the associated $\E$-category $\eD$ is an $\E$-prederivator. 

Let $\eA$ be an $\E$-category, and suppose we have a functor $u:\A\rightarrow\B$. For $\eA$ to be an $\E$-prederivator, we want $\eA$ to have a notion of pullback along $u$. Suppose we have an object $X\in\dA_0(\B)$. Using~\cref{u* and alpha* respect the E-action} and~\cref{Cocontinuous E-module maps induce E-morphisms}, we can form the following $\E$-morphism: 
\vspace{-1em}
\begin{center}
\begin{tikzcd}[column sep=tiny]
\eA\arrow{rrrrrrrrr}{\widetilde{\map}_{\dA}(X,-)} &&&&&&&&& \eE^{\B\op}\arrow{rrrrr}{(u\op)^*} &&&&&\eE^{\A\op}
\end{tikzcd}
\end{center} 
The pullback of $X$ along $u$ should be an object $u^*X\in\dA_0(\A)$ that represents this $\E$-morphism. Dually, this same object should also represent the $\E$-morphism below:
\begin{center}
\begin{tikzcd}[column sep=tiny]
\eA\op\arrow{rrrrrrrrr}{\widetilde{\map}_{\dA}(-,X)} &&&&&&&&& \eE^{\B}\arrow{rrrrr}{u^*} &&&&&\eE^{\A}
\end{tikzcd}
\end{center} 
This motivates the following definition:

\begin{definition}\label{E-prederivator definition}
Let $\eA$ be an $\E$-category.  We call $\eA$ an \textbf{$\E$-prederivator} if we have the following data:
\begin{itemize}
\item
For any functor $u:\A\rightarrow\B$ and any $X\in\dA_0(\B)$, an object $u^*X\in\dA_0(\A)$.
\item
For any functors $u:\A\rightarrow\B$ and $v:\C\rightarrow\cD$, and any $X\in\dA_0(\B)$ and $Y\in\dA_0(\cD)$, an isomorphism
\[
\gamma^{u,v}:(u\op\times v)^*\widetilde{\map}_{\dA}(X,Y)\xrightarrow{\;\;\;\iso\;\;\;}\widetilde{\map}_{\dA}(u^*X,v^*Y)
\]
in $\E(\A\op\times\C)$.
\end{itemize}
These must satisfy the following axioms:
\begin{enumerate}
\item
For any category $\A$ and any $X\in\dA_0(\A)$, we have $\id^*X=X\in\dA_0(\A)$. Given an additional category $\B$, and any $Y\in\dA_0(\B)$, we have:
\[
\gamma^{\id,\id}=\id:\widetilde{\map}_{\dA}(X,Y)\rightarrow\widetilde{\map}_{\dA}(X,Y)
\] 

Moreover, given composable functors $\A\xrightarrow{\;\;u\;\;}\B\xrightarrow{\;\;v\;\;}\C$, and $X\in\dA_0(\C)$ we have $(v\circ u)^*X=u^*v^*X$, and given additional functors $\cD\xrightarrow{\;\;w\;\;}\mathrm{E}\xrightarrow{\;\;z\;\;}\mathrm{F}$, and $Y\in\dA_0(\mathrm{F})$, the diagram below commutes:
\begin{center}
\begin{tikzcd}
(u\op\times w)^*(v\op\times z)^*\widetilde{\map}_{\dA}(X,Y)\arrow{rrr}[above]{(u\op\times w)^*\gamma^{v,z}}\arrow[d,equal] &&& 
 (u\op\times w)^*\widetilde{\map}_{\dA}(v^*X,z^*Y) \arrow{dd}[right]{\gamma^{u,w}} \\
 ((v\circ u)\op\times (z\circ w))^*\widetilde{\map}_{\dA}(X,Y)\arrow[bend right=20]{rrrdd}[below left]{\gamma^{v\circ u,z\circ w}}
 \\
&&&\widetilde{\map}_{\dA}(u^*v^*X,w^*z^*Y)\arrow[d,equal] \\
&&&\widetilde{\map}_{\dA}((v\circ u)^*X,(z\circ w)^*Y) 
\end{tikzcd}
\end{center} 

\item
For any $u:\A\rightarrow\B$ and $X\in\dA_0(\B)$, the maps
\[
\gamma^{u,\id}:(u\op\times \C)^*\widetilde{\map}_{\dA}(X,Y)\xrightarrow{\;\;\;\iso\;\;\;}\widetilde{\map}_{\dA}(u^*X,Y)
\]
are $\E$-natural in $Y\in\dA_0(\C)$. For any $v:\C\rightarrow\cD$ and $Y\in\dA_0(\cD)$, the maps
\[
\gamma^{\id,v}:(\A\op\times v)^*\widetilde{\map}_{\dA}(X,Y)\xrightarrow{\;\;\;\iso\;\;\;}\widetilde{\map}_{\dA}(X,v^*Y)
\]
are $\E$-natural in $X\in\dA_0(\A)$.

\item
For any functor $u:\A\rightarrow\B$ and any $X\in\dA_0(\B)$, the diagram below commutes:
\begin{center}
\begin{tikzcd}
\h_\A\arrow{rr}{\h_u}\arrow[bend right=15]{rrrrdd}[below left]{j} && (u\op\times u)^*\h_\B\arrow{rr}{(u\op\times u)^*j} &&(u\op\times u)^*\widetilde{\map}_{\dA}(X,X)\arrow{dd}{\gamma^{u,u}}\\\\
&&&& \widetilde{\map}_{\dA}(u^*X,u^*X) 
\end{tikzcd}
\end{center} 
Here $\h_u$ is the map from~\cref{Definition of h on objects and morphisms}. 
\item
For any categories $\A$, $\B$, $\C$ and $\cD$, any functor $v:\B\rightarrow\C$, and any $X\in\dA_0(\A)$, $Y\in\dA_0(\C)$ and $Z\in\dA_0(\cD)$, the diagram below commutes:
\begin{center}
\begin{tikzcd}[column sep=0cm]
(v\op\times\cD)^*\widetilde{\map}_{\dA}(Y,Z)\otimes_\B (\A\op\times v)^*\widetilde{\map}_{\dA}(X,Y)\arrow[rd,"\gamma^{v,\id}\;\mathsmaller{\otimes_\B}\gamma^{\id,v}",controls={+(5,0)and +(-1,0.8)}]
\arrow{ddd}[left]{\widetilde{\map}_{\dA}(Y,Z)\;\otimes_v\;\widetilde{\map}_{\dA}(X,Y)}\\
 
& \widetilde{\map}_{\dA}(v^*Y,Z)\otimes_\B \widetilde{\map}_{\dA}(X,v^*Y)\arrow{ddd}{\circ}\\\\
 
\widetilde{\map}_{\dA}(Y,Z)\otimes_\C \widetilde{\map}_{\dA}(X,Y)\arrow[dr,"\circ" below left, bend right=12] 
\\
&\widetilde{\map}_{\dA}(X,Z)
\end{tikzcd}
\end{center} 
Here $\otimes_v$ is the map from~\cref{Definition of cancelling product on morphisms}.
\end{enumerate}
\end{definition}

\begin{remark}\label{Rephrasing E-naturality in definition of E-prederivator}
In~\cref{E-prederivator definition} Axiom $2$, the $\E$-naturality conditions amount to the following: 
\begin{enumerate}
\item
For any categories $\A$, $\B$, $\C$ and $\cD$, any functor $u:\A\rightarrow\B$, and any $X\in\dA_0(\B)$, $Y\in\dA_0(\C)$ and $Z\in\dA_0(\cD)$, the diagram below commutes:
\begin{center}
\begin{tikzcd}[column sep=tiny]
(u\op\times\cD)^*(\widetilde{\map}_{\dA}(Y,Z)\otimes_\C\widetilde{\map}_{\dA}(X,Y))\arrow{dddd}[left]{(u\op\times\cD)^*(\circ)} \arrow{rrr}[above]{\iso} &&& \widetilde{\map}_{\dA}(Y,Z)\otimes_\C(u\op\times\C)^*\widetilde{\map}_{\dA}(X,Y)\arrow{dd}[right]{\widetilde{\map}_{\dA}(Y,Z)\mathsmaller{\otimes_\C}\;\gamma^{u,\id}} \\\\
 
&&&\widetilde{\map}_{\dA}(Y,Z)\otimes_\C\widetilde{\map}_{\dA}(u^*X,Y)\arrow{dd}[right]{\circ}\\\\
 
(u\op\times\cD)^*\widetilde{\map}_{\dA}(X,Z)\arrow{rrr}[below]{\gamma^{u,\id}} &&& \widetilde{\map}_{\dA}(u^*X,Z)
\end{tikzcd}
\end{center} 
\item
For any categories $\A$, $\B$, $\C$ and $\cD$, any functor $w:\C\rightarrow\cD$, and any $X\in\dA_0(\A)$, $Y\in\dA_0(\B)$ and $Z\in\dA_0(\cD)$, the diagram below commutes:
\begin{center}
\begin{tikzcd}[column sep=tiny]
(\A\op\times w)^*(\widetilde{\map}_{\dA}(Y,Z)\otimes_\B\widetilde{\map}_{\dA}(X,Y))\arrow{dddd}[left]{(\A\op\times w)^*(\circ)} \arrow{rrr}[above]{\iso} &&& (\B\op\times w)^*\widetilde{\map}_{\dA}(Y,Z)\otimes_\B\widetilde{\map}_{\dA}(X,Y)\arrow{dd}[right]{\gamma^{\id,w}\;\mathsmaller{\otimes_\B}\widetilde{\map}_{\dA}(X,Y)} \\\\
 
&&&\widetilde{\map}_{\dA}(Y,w^*Z)\otimes_\B\widetilde{\map}_{\dA}(X,Y)\arrow{dd}[right]{\circ}\\\\
 
(\A\op\times w)^*\widetilde{\map}_{\dA}(X,Z)\arrow{rrr}[below]{\gamma^{\id,w}} &&& \widetilde{\map}_{\dA}(X,w^*Z)
\end{tikzcd}
\end{center} 
\end{enumerate}
This follows, using adjointness, from the $\E$-naturality conditions in the form described in~\cref{Rephrasing E-naturality}.
\end{remark}


\begin{example}\label{Maximal sub-E-prederivator}
Let $\eA$ be an $\E$-prederivator, and let $\mathcal{L}\subseteq\dA_0(\A)$ be a set of objects in $\dA_0(\0)$. For any category $\A$, consider the set
\[
\dB_0(\A) = \{X\in\dA_0(\A) \;|\; a^*X\in\mathcal{L} \;\; \forall a\in\A\}.
\]
Note that $\dB_0(\0)=\mathcal{L}$. As in \cref{Full sub-E-category}, we may consider the full sub-$\E$-category $\dB$ on these objects. We claim that this $\E$-category $\dB$ is an $\E$-prederivator. 

To see this, let $u:\A\rightarrow\B$ be a functor and let $X\in\dB_0(\B)$; we need to give an object $u^*X\in\dB_0(\A)$. Consider the object $u^*X\in\dA_0(\A)$. Using Axiom $1$ of \cref{E-prederivator definition}, it follows that this object $u^*X$ is in $\dB_0(\A)$. We take this to be the required object. The structure isomorphisms $\gamma$ are also inherited from $\dA$. The $\E$-prederivator axioms for $\dB$ follow from the axioms for $\dA$. We call $\dB$ the \textbf{maximal sub-$\E$-prederivator on $\mathcal{L}$}.
\end{example}

Given an $\E$-prederivator, we want to show that the induced categories of~\cref{Underlying categories of an E-category} organise into a prederivator. We begin to give the required structure in the following lemma; the final statement appears in~\cref{E-prederivators induce prederivators}.

\begin{lemma}\label{The functor u* from an E-prederivator}
Let $\eA$ be an $\E$-prederivator and let $u:\A\rightarrow\B$ be a functor. Then the assignment
\begin{align*}
\dA_0(\B) &\longrightarrow \;\dA_0(\A)\\
 X\;\; &\;\mapsto \;\;\; u^*X
\end{align*}
extends to a functor $u^*:\dA(\B)\rightarrow\dA(\A)$ on the induced categories.
\end{lemma}
\begin{proof}
Let $f:\h_\B\rightarrow\widetilde{\map}_{\dA}(X,Y)$ be a map in $\dA(\B)$. For any category $\C$ and any object $Z\in\dA_0(\C)$, consider the composite below:
\begin{center}
\begin{tikzcd}[column sep=tiny]
\widetilde{\map}_{\dA}(u^*Y,Z)\arrow[rrrrr,"(\gamma^{u,\id})^{-1}"]
&&&&& (u\op\times\C)^*\widetilde{\map}_{\dA}(Y,Z)\arrow[rrrrrrrrrrrrr,"(u\op\times\C)^*\widetilde{\map}_{\dA}(f\text{,}Z)" above] &&&&&&&&&&&&& (u\op\times\C)^*\widetilde{\map}_{\dA}(X,Z)\arrow[d,"\gamma^{u,\id}" right]\\
&&&&&&&&&&&&&&&&&& \widetilde{\map}_{\dA}(u^*X,Z)
\end{tikzcd}
\end{center}
By~\cref{map(f -) is E-natural}, this map is $\E$-natural in $Z\in\dA_0(\C)$. Using~\cref{E-category yoneda embedding is fully faithful}, we define $u^*f:\h_\A\rightarrow\widetilde{\map}_{\dA}(u^*X,u^*Y)$ to be the unique map in $\dA(\A)$ representing this $\E$-natural transformation. That is, $u^*f$ is the unique map that makes the diagram below commute, for any category $\C$ and any $Z\in\dA_0(\C)$:
\vspace{-2em}
\begin{center}
\begin{equation}\label{Diagram defining u*}
\begin{tikzcd}[baseline=(current  bounding  box.center)]
(u\op\times\C)^*\widetilde{\map}_{\dA}(Y,Z)\arrow[dd,"\gamma^{u,\id}" left]\arrow[rrrr,"(u\op\times\C)^*\widetilde{\map}_{\dA}(f\text{,}Z)" above]
&&&& (u\op\times\C)^*\widetilde{\map}_{\dA}(X,Z)\arrow[dd,"\gamma^{u,\id}" right]\\\\
\widetilde{\map}_{\dA}(u^*Y,Z)\arrow[rrrr,"\widetilde{\map}_{\dA}(u^*f\text{,}Z)" below]&&&& \widetilde{\map}_{\dA}(u^*X,Z)
\end{tikzcd}
\end{equation}
\end{center}
Functoriality of this construction follows by the uniqueness, using~\cref{E-category yoneda embedding definition}.
\end{proof}

\begin{remark}
In the proof of \cref{The functor u* from an E-prederivator}, to define $u^*:\dA(\B)\rightarrow\dA(\A)$ and prove it is functorial, it suffices to use the Yoneda lemma for the category $\dA(\A)$, rather than the $\E$-category Yoneda lemma of \cref{Yoneda lemma for E-categories}. However, \cref{Yoneda lemma for E-categories} implies that the diagram (\ref{Diagram defining u*}) commutes, and we will use this repeatedly in the rest of the chapter, which is why we used the $\E$-category Yoneda lemma rather than the unenriched Yoneda lemma in this proof.
\end{remark}

\begin{remark}\label{Explicit description of u*f in an E-prederivator}
Suppose we have an $\E$-prederivator $\eA$ and a functor $u:\A\rightarrow\B$. Given a map $f:\h_\B\rightarrow\widetilde{\map}_{\dA}(X,Y)$ in $\dA(\B)$, the map $u^*f:\h_\A\rightarrow\widetilde{\map}_{\dA}(u^*X,u^*Y)$ of~\cref{The functor u* from an E-prederivator} can be explicitly described as follows:
\begin{center}
\begin{tikzcd}
\h_\A\arrow[r,"\h_u"]
& (u\op\times u)^*\h_\B\arrow[rr,"(u\op\times u)^*f" above] && (u\op\times u)^*\widetilde{\map}_{\dA}(X,Y)\arrow[r,"\gamma^{u,u}" above] & \widetilde{\map}_{\dA}(u^*X,u^*Y)
\end{tikzcd}
\end{center}
This follows from Axioms $1$ and $3$ of~\cref{E-prederivator definition}, using~\cref{Extracting f from map(f X)}.
\end{remark}

In the proof of~\cref{The functor u* from an E-prederivator}, the action of $u^*:\dA(\B)\rightarrow\dA(\A)$ on maps is defined representably using the $\E$-natural isomorphisms $\gamma^{u,\id}$. However, the explicit description in~\cref{Explicit description of u*f in an E-prederivator}, implies that $u^*$ can also be defined using the dual isomorphisms $\gamma^{\id,u}$:

\begin{remark}\label{Dual description of u*f}
Let $\eA$ be an $\E$-prederivator, let $u:\A\rightarrow\B$ be a functor, and let \linebreak~$f:\h_\B\rightarrow\widetilde{\map}_{\dA}(X,Y)$ be a map in $\dA(\B)$. Then $u^*f:\h_\A\rightarrow\widetilde{\map}_{\dA}(u^*X,u^*Y)$ makes the diagram below commute, for any category $\C$ and any object $Z\in\dA_0(\C)$:
\begin{center}
\begin{tikzcd}
(\C\op\times u)^*\widetilde{\map}_{\dA}(Z,X)\arrow[dd,"\gamma^{\id,u}" left]\arrow[rrrr,"(\C\op\times u)^*\widetilde{\map}_{\dA}(Z\text{,}f)" above]
&&&& (\C\op\times u)^*\widetilde{\map}_{\dA}(Z,Y)\arrow[dd,"\gamma^{\id,u}" right]\\\\
\widetilde{\map}_{\dA}(Z,u^*X)\arrow[rrrr,"\widetilde{\map}_{\dA}(Z\text{,}u^*f)" below]&&&& \widetilde{\map}_{\dA}(Z,u^*Y)
\end{tikzcd}
\end{center}
Note that this property characterises $u^*f$, by~\cref{E-category yoneda embedding is fully faithful}. We can check this equality using~\cref{Extracting f from map(f X)}, and the explicit description of $u^*f$ from~\cref{Explicit description of u*f in an E-prederivator}.
\end{remark}

We now give the corresponding results for natural transformations.

\begin{lemma}\label{Natural transformation omega* from an E-prederivator}
Let $\eA$ be an $\E$-prederivator, and let 
\begin{center}
\begin{tikzcd}
\A\arrow[rr, bend left=45, "u" above,""{name=U, below}]\arrow[rr, bend right=45, "v" below, ""{name=D}]
&& \B
\arrow[Rightarrow,from=U,to=D,shorten >=0.1cm,shorten <=0.1cm,"\kappa"]
\end{tikzcd}
\end{center}
be a natural transformation. This induces a canonical natural transformation
\begin{center}
\begin{tikzcd}
\dA(\A)\arrow[leftarrow,r, bend left=50, "u^*" above,""{name=U, below}]\arrow[leftarrow,r, bend right=50, "v^*" below, ""{name=D}]
& \dA(\B)
\arrow[Rightarrow,from=U,to=D,shorten >=0.1cm,shorten <=0.1cm,"\kappa^*", shift right]
\end{tikzcd}
\end{center}
between the functors of~\cref{The functor u* from an E-prederivator}.
\end{lemma}
\begin{proof}
Let $X\in\dA_0(\B)$. We need to define the component of $\kappa^*$ at $X$. For any category $\C$ and any object $Z\in\dA_0(\C)$, consider the composite below:
\begin{center}
\begin{tikzcd}
\widetilde{\map}_{\dA}(v^*X,Z)\arrow[rr,"(\gamma^{v,\id})^{-1}"]
&& (v\op\times\C)^*\widetilde{\map}_{\dA}(X,Z)\arrow[rrr,"(\kappa\op\times\C)^*_{\widetilde{\map}_{\dA}(X\text{,}Z)}" above] &&& (u\op\times\C)^*\widetilde{\map}_{\dA}(X,Z)\arrow[d,"\gamma^{u,\id}" right]\\
&&&&& \widetilde{\map}_{\dA}(u^*X,Z)
\end{tikzcd}
\end{center}
This map is $\E$-natural in $Z\in\dA_0(\C)$, by~\cref{u* and alpha* respect the E-action} and~\cref{E-module modifications induce E-natural maps}. Thus, by~\cref{E-category yoneda embedding is fully faithful}, we may define $\kappa^*_X:\h_\A\rightarrow\widetilde{\map}_{\dA}(u^*X,v^*X)$ to be the unique map in $\dA(\A)$ representing this $\E$-natural transformation. That is, $\kappa^*_X$ is the unique map that makes the diagram below commute, for any category $\C$ and any $Z\in\dA_0(\C)$:
\begin{center}
\begin{tikzcd}
(v\op\times\C)^*\widetilde{\map}_{\dA}(X,Z)\arrow[dd,"\gamma^{v,\id}" left]\arrow[rrrr,"(\kappa\op\times\C)^*_{\widetilde{\map}_{\dA}(X\text{,}Z)}" above]
&&&& (u\op\times\C)^*\widetilde{\map}_{\dA}(X,Z)\arrow[dd,"\gamma^{u,\id}" right]\\\\
\widetilde{\map}_{\dA}(v^*X,Z)\arrow[rrrr,"\widetilde{\map}_{\dA}(\kappa^*_X\text{,}Z)" below]&&&& \widetilde{\map}_{\dA}(u^*X,Z)
\end{tikzcd}
\end{center}
Suppose we have a map $f:\h_\B\rightarrow\widetilde{\map}_{\dA}(X,Y)$ in $\dA(\B)$. To check naturality, we need to show that $v^*(f)\circ\kappa^*_X$ and $\kappa^*_Y\circ u^*(f)$ represent the same $\E$-natural transformation.  This follows easily from the definitions of $u^*$ and $v^*$ in~\cref{The functor u* from an E-prederivator}.
\end{proof}

As in~\cref{Dual description of u*f}, we also have a dual description of the map $\kappa^*_X$:

\begin{remark}
Let $\eA$ be an $\E$-prederivator, let $u,v:\A\rightarrow\B$ be functors, let $\kappa:u\Rightarrow v$ be a natural transformation, and let $X\in\dA_0(\B)$. For any category $\C$ and any $Z\in\dA_0(\C)$, the diagram below commutes:
\begin{center}
\begin{tikzcd}
(\C\op\times u)^*\widetilde{\map}_{\dA}(Z,X)\arrow[dd,"\gamma^{\id,u}" left]\arrow[rrrr,"(\C\op\times\kappa)^*_{\widetilde{\map}_{\dA}(Z\text{,}X)}" above]
&&&& (\C\op\times v)^*\widetilde{\map}_{\dA}(Z,X)\arrow[dd,"\gamma^{\id,v}" right]\\\\
\widetilde{\map}_{\dA}(Z,u^*X)\arrow[rrrr,"\widetilde{\map}_{\dA}(Z\text{,}\;\kappa^*_X)" below]&&&& \widetilde{\map}_{\dA}(Z,v^*X)
\end{tikzcd}
\end{center}
To check this, we can use~\cref{Extracting f from map(f X)}; the diagram above commutes if and only if the diagram below commutes:
\begin{center}
\begin{tikzcd}[column sep=small]
\widetilde{\map}_{\dA}(v^*X,v^*X)\arrow[rrr,"(\gamma^{v,\id})^{-1}"]
&&& (v\op\times\A)^*\widetilde{\map}_{\dA}(X,v^*X)\arrow[rrrr,"(\kappa\op\times\A)^*" above] &&&& (u\op\times\A)^*\widetilde{\map}_{\dA}(X,v^*X)\arrow[d,"\gamma^{u,\id}" right]\\
\h_\A\arrow[u,"j" left]\arrow[d,"j" left] &&&&&&& \widetilde{\map}_{\dA}(u^*X,v^*X)\\

\widetilde{\map}_{\dA}(u^*X,u^*X)\arrow[rrr,"(\gamma^{\id,u})^{-1}" below]
&&& (\A\op\times u)^*\widetilde{\map}_{\dA}(u^*X,X)\arrow[rrrr,"(\A\op\times\kappa)^*" below] &&&& (A\op\times v)^*\widetilde{\map}_{\dA}(u^*X,X)\arrow[u,"\gamma^{\id,v}" right]
\end{tikzcd}
\end{center}

Using Axioms $1$ and $3$ of~\cref{E-prederivator definition}, we can reduce the commutativity of this diagram to the commutativity of the diagram below:
\begin{center}
\begin{tikzcd}
\h_\A \arrow{rr}[above]{\h_u}\arrow{dd}[left]{\h_v} && 
 (u\op\times u)^*\h_\B \arrow{dd}[right]{(u\op\times \kappa)^*} \\\\
 (v\op\times v)^*\h_\B \arrow{rr}[below]{(\kappa\op\times v)^*} && 
 (u\op\times v)^*\h_\B 
\end{tikzcd}
\end{center}
This diagram commutes by~\cref{Interaction of delta with natural transformations}.
\end{remark}

\begin{theorem}\label{E-prederivators induce prederivators}
Any $\E$-prederivator $\eA$ induces a prederivator $\dA:\Cat\op\rightarrow\CAT$, defined as follows:
\begin{center}
\begin{tikzcd}
\A\arrow[rr, bend left=50, "u" above,""{name=U, below}]\arrow[rr, bend right=50, "v" below, ""{name=D}]
&& \B
& \mapsto
& \dA(\A)
& \dA(\B)\arrow[l,bend left=50,"v^*" below,""{name=L,above}]\arrow[l,bend right=50,"u^*" above,""{name=R,below}]
\arrow[Rightarrow,from=U,to=D,shorten >=0.1cm,shorten <=0.1cm,"\kappa"]
\arrow[Rightarrow,from=R,to=L,shorten >=0.1cm,shorten <=0.1cm,"\kappa^*"]
\end{tikzcd}
\end{center}
Here $u^*$ is the functor of~\cref{The functor u* from an E-prederivator}, and $\kappa^*$ is the natural transformation of~\cref{Natural transformation omega* from an E-prederivator}.
\end{theorem}
\begin{proof}
We need to check that this assignment is $2$-functorial. This follows from the corresponding fact for $\E$, using Axiom $1$ of~\cref{E-prederivator definition}. To illustrate how this works, we will prove that $\dA$ preserves composition of functors. 

Suppose we have composable functors $\A\xrightarrow{\;\;u\;\;}\B\xrightarrow{\;\;v\;\;}\C$, and an object $X\in\dA_0(\C)$. By definition, we have $(v\circ u)^*X=u^*v^*X$. Suppose we have a map $f:\h_\C\rightarrow\widetilde{\map}_{\dA}(X,Y)$ in $\dA(\C)$. We want to show that $(v\circ u)^*f=u^*v^*f$. By definition of the pullback functors in~\cref{The functor u* from an E-prederivator}, we must show that the diagram below commutes, for any $Z\in\dA_0(\cD)$:
\begin{center}
\begin{tikzcd}[column sep=tiny,row sep=large]
\widetilde{\map}_{\dA}((v\circ u)^*X,Z)\arrow[d,"(\gamma^{u,\id})^{-1}" left]\arrow[rr,"(\gamma^{v\circ u,\id})^{-1}" above] && ((v\circ u)\op\times\cD)^*\widetilde{\map}_{\dA}(X,Z) \arrow[d,"((v\circ u)\op\times\cD)^*\widetilde{\map}_{\dA}(f\text{,}Z)"]
\\
(u\op\times\cD)^*\widetilde{\map}_{\dA}(v^*X,Z)\arrow[dd,"(u\op\times\cD)^*(\gamma^{v,\id})^{-1}" left] && ((v\circ u)\op\times\cD)^*\widetilde{\map}_{\dA}(Y,Z)\arrow{d}{\gamma^{v\circ u,\id}}\\
&& \widetilde{\map}_{\dA}((v\circ u)^*Y,Z)
\\
 
(u\op\times\cD)^*(v\op\times\cD)^*\widetilde{\map}_{\dA}(X,Z)\arrow[rrd,"(u\op\times\cD)^*(v\op\times\cD)^*\widetilde{\map}_{\dA}(f\text{,}Z)" below left,bend right=10] && (u\op\times\cD)^*\widetilde{\map}_{\dA}(v^*Y,Z)\arrow[u,"\gamma^{u,\id}" right]
\\
&& (u\op\times\cD)^*(v\op\times\cD)^*\widetilde{\map}_{\dA}(Y,Z)\arrow[u,"(u\op\times\cD)^*(\gamma^{v,\id})" right] 
\end{tikzcd}
\end{center} 

This commutes by Axiom $2$ of~\cref{E-prederivator definition}.
\end{proof}

We will refer to the prederivator $\dA$ of~\cref{E-prederivators induce prederivators} as the prederivator \textbf{induced} by $\eA$.

\begin{lemma}
Let $\eA$ be an $\E$-prederivator. The opposite $\E$-category $\eA\op$ carries a natural $\E$-prederivator structure, such that the induced prederivator $\dA\op$ is the opposite of the prederivator induced by $\eA$. 
\end{lemma}
\begin{proof}
Let $u:\A\rightarrow\B$ be a functor and let $X\in\dA\op_0(\B)=\dA_0(\B\op)$. Define $u^*X$ with respect to $\eA\op$ to be $(u\op)^*X\in\dA_0(\A\op)$ with respect to $\eA$. 

Given functors $u:\A\rightarrow\B$ and $v:\C\rightarrow\cD$, and $X\in\dA_0(\B\op)$ and $Y\in\dA_0(\cD\op)$, we define $\gamma^{u,v}$ with respect to $\eA\op$ to be the isomorphism
\[
(u\op\times v)^*\sigma^*\widetilde{\map}_{\dA}(Y,X)=\sigma^*(v\times u\op)^*\widetilde{\map}_{\dA}(Y,X)\xrightarrow{\;\;\;\sigma^*(\gamma^{v\op,u\op})\;\;\;}\sigma^*\widetilde{\map}_{\dA}((v\op)^*Y,(u\op)^*X)
\]
with respect to $\eA$. It is easy to check that this data satisfies~\cref{E-prederivator definition}, using the corresponding facts for $\eA$, and the respect of $\sigma^*$ for $\h_u$ and $\otimes_u$.
\end{proof}

\begin{proposition}\label{In an E-prederivator map(X -) is a prederivator map}
Let $\eA$ be an $\E$-prederivator, let $\A$ and $\B$ be categories, and let $X\in\dA_0(\A)$ and $Y\in\dA_0(\B)$. For any category $\C$, consider the functors
\[
\widetilde{\map}_{\dA}(X,-):\dA(\C)\rightarrow\E(\A\op\times\C)
\]
\[
\widetilde{\map}_{\dA}(-,Y):\dA(\C)\op\rightarrow\E(\C\op\times\B)
\]
induced by the representable $\E$-morphisms $\widetilde{\map}_{\dA}(X,-)$ and $\widetilde{\map}_{\dA}(-,Y)$. These form the components of prederivator maps
\[
\widetilde{\map}_{\dA}(X,-):\dA\rightarrow\E^{\A\op}
\]
\[
\widetilde{\map}_{\dA}(-,Y):\dA\op\rightarrow\E^\B
\]
with structure isomorphisms given by the maps $\gamma$ of~\cref{E-prederivator definition}. 
\end{proposition}
\begin{proof}
Let $u:\C\rightarrow\cD$ be a functor, and let $X\in\dA_0(\A)$ and $Y\in\dA_0(\B)$. The structure isomorphism for $\widetilde{\map}_{\dA}(X,-):\dA\rightarrow\E^{\A\op}$ has component at $Z\in\dA_0(\cD)$ given by:
\[
\gamma^{\id,u}:(\A\op\times u)^*\widetilde{\map}_{\dA}(X,Z)\xrightarrow{\;\;\;\iso\;\;\;}\widetilde{\map}_{\dA}(X,u^*Z)
\]
The structure isomorphism for $\widetilde{\map}_{\dA}(-,Y):\dA\op\rightarrow\E^\B$ has component at $Z\in\dA_0(\cD)$ given by:
\[
\gamma^{u,\id}:(u\op\times \B)^*\widetilde{\map}_{\dA}(Z,Y)\xrightarrow{\;\;\;\iso\;\;\;}\widetilde{\map}_{\dA}(u^*Z,Y)
\]
We will outline a proof for the second map $\widetilde{\map}_{\dA}(-,Y)$. The other is analogous.

First, we need to check that the maps $\gamma^{u,\id}$ are natural in $Z\in\dA(\cD)\op$. That is, for any map $f:\h_{\cD}\rightarrow\widetilde{\map}_{\dA}(Z,W)$ in $\dA(\cD)$, we need the diagram below to commute:
\begin{center}
\begin{tikzcd}
(u\op\times\B)^*\widetilde{\map}_{\dA}(W,Y)\arrow[dd,"\gamma^{u,\id}" left]\arrow[rrrr,"(u\op\times\B)^*\widetilde{\map}_{\dA}(f\text{,}Y)" above]
&&&& (u\op\times\B)^*\widetilde{\map}_{\dA}(Z,Y)\arrow[dd,"\gamma^{u,\id}" right]\\\\
\widetilde{\map}_{\dA}(u^*W,Y)\arrow[rrrr,"\widetilde{\map}_{\dA}(u^*f\text{,}Y)" below]&&&& \widetilde{\map}_{\dA}(u^*Z,Y)
\end{tikzcd}
\end{center}
This holds by the commutativity of the diagram (\ref{Diagram defining u*}) in~\cref{The functor u* from an E-prederivator}.

We now need to check that $\gamma^{u,\id}$ satisfies the axioms of~\cref{Pseudonatural transformation definition}. Axioms $1$ and $2$ follow immediately from Axiom $1$ of~\cref{E-prederivator definition}. Finally, given a natural transformation $\kappa$, Axiom $3$ of~\cref{Pseudonatural transformation definition} follows from the definition of $\kappa^*$ in~\cref{Natural transformation omega* from an E-prederivator}.
\end{proof}

\begin{remark}\label{In an E-prederivator map(- -) is a prederivator map}
Let $\eA$ be an $\E$-prederivator. The prederivator maps of~\cref{In an E-prederivator map(X -) is a prederivator map} are the external components of a prederivator map
\[
\map_\dA(-,-):\dA\op\times\dA\rightarrow\E.
\]
Given a category $\A$ and objects $X\in\dA\op(\A)$ and $Y\in\dA(\A)$, we have
\[
\map_\dA(X,Y)=\delta^*\widetilde{\map}_{\dA}(X,Y)\in\E(\A),
\]
where $\delta:\A\rightarrow\A\times\A$ is the diagonal map, as in~\cref{Cartesian closure of PDer}.
\end{remark}

\begin{theorem}\label{E-modules induce E-prederivators}
Let $\D$ be a closed $\E$-module. The associated $\E$-category $\eD$ has a canonical $\E$-prederivator structure, such that the induced prederivator recovers the original $\D$.
\end{theorem}
\begin{proof}
Let $u:\A\rightarrow\B$ be a functor, and let $X\in\D(\B)$. To satisfy~\cref{E-prederivator definition}, we require an object $u^*X\in\D(\A)$. We take this object to be the image of $X$ under the functor $u^*:\D(\B)\rightarrow\D(\A)$, coming from the prederivator structure on $\D$.

Given functors $u:\A\rightarrow\B$ and $v:\C\rightarrow\cD$, and $X\in\D(\B)$ and $Y\in\D(\cD)$, the required isomorphism
\[
\gamma^{u,v}:(u\op\times v)^*\widetilde{\map}_{\D}(X,Y)\xrightarrow{\;\;\;\iso\;\;\;}\widetilde{\map}_{\D}(u^*X,v^*Y)
\]
is induced by the structure isomorphisms of the prederivator map
\[
\map_\D(-,-):\D\op\times\D\rightarrow\E,
\]
as in~\cref{Lifting the external product to a derivator map}. With these definitions, once we show that we do indeed obtain an $\E$-prederivator, it is immediate that the induced prederivator must recover the original prederivator. We will now check the axioms of~\cref{E-prederivator definition}.

Axiom $1$ follows immediately from Axioms $1$ and $2$ of~\cref{Pseudonatural transformation definition}. For Axiom $2$, suppose we have functors $u:\A\rightarrow\B$ and $v:\C\rightarrow\cD$, and objects $X\in\D(\B)$ and $Y\in\D(\cD)$, and consider the maps below:
\[
\gamma^{u,\id}:(u\op\times \cE)^*\widetilde{\map}_{\D}(X,Z)\xrightarrow{\;\;\;\iso\;\;\;}\widetilde{\map}_{\D}(u^*X,Z)
\]
\[
\gamma^{\id,v}:(\cE\op\times v)^*\widetilde{\map}_{\D}(Z,Y)\xrightarrow{\;\;\;\iso\;\;\;}\widetilde{\map}_{\D}(Z,v^*Y)
\]
We must show that these maps are $\E$-natural in $Z\in\D(\cE)$. We will prove this for the first map $\gamma^{u,\id}$; the proof for $\gamma^{\id,v}$ is dual. First, note that $\gamma^{u,\id}$ is the component of a modification in $Z$. Moreover, by~\cref{u* preserves cotensors} and~\cref{map(X -) preserves cotensors}, the source and target of this modification preserve cotensors. Thus, if we can prove that $\gamma^{u,\id}$
respects cotensors, $\E$-naturality will follow by the dual of~\cref{E-module modifications induce E-natural maps}. Therefore, given any object $W\in\E(\C)$, we need the diagram below to commute:
\begin{center}
\begin{tikzcd}
(u\op\times\C\op\times\cE)^*\widetilde{\map}_{\D}(X,Z\;\widetilde{\vartriangleleft}\;W)\arrow{dddd}[left]{\gamma^{u,\id}} \arrow{rr}[above]{\iso} && (u\op\times\C\op\times\cE)^*\widetilde{\map}_{\E}(W,\widetilde{\map}_{\D}(X,Z))\arrow{dd}[right]{\iso} \\\\
 
&& \widetilde{\map}_{\E}(W,(u\op\times\cE)^*\widetilde{\map}_{\D}(X,Z))\arrow{dd}[right]{\widetilde{\map}_{\E}(W,\;\gamma^{u,\id})}\\\\
 
\widetilde{\map}_{\D}(u^*X,Z\;\widetilde{\vartriangleleft}\;W)\arrow[rr,"\iso" below]
&& \widetilde{\map}_{\E}(W,\widetilde{\map}_{\D}(u^*X,Z))
\end{tikzcd}
\end{center}
This diagram commutes; it expresses the fact that the isomorphism 
\[
\widetilde{\map}_{\D}(X,Z\;\widetilde{\vartriangleleft}\;W)\iso\widetilde{\map}_{\E}(W,\widetilde{\map}_{\D}(X,Z))
\]
of~\cref{map(X -) preserves cotensors} is a modification in $X$.
We will now consider Axiom $3$ of~\cref{E-prederivator definition}. Given a functor $u:\A\rightarrow\B$ and $X\in\D(\B)$, the diagram in Axiom $3$ commutes if and only if its adjunct diagram below commutes:
\begin{center}
\begin{tikzcd}
\h_\A\otimes_\A u^*X\arrow{rr}{\h_u\otimes_\A u^*X}\arrow[bend right=10]{rrrrrddd}[below left]{\lambda} && (u\op\times u)^*\h_\B\otimes_\A u^*X\arrow{rrr}{(u\op\times u)^*j\otimes_\A u^*X} &&&(u\op\times u)^*\widetilde{\map}_{\D}(X,X)\otimes_\A u^*X\arrow{dd}{\gamma^{u,u}\otimes_\A u^*X}\\\\
&&&&& \widetilde{\map}_{\D}(u^*X,u^*X)\otimes_\A u^*X\arrow{d}{\epsilon} \\
&&&&& u^*X
\end{tikzcd}
\end{center} 
Here we have used the definition of the unit $j:\h_\A\rightarrow\widetilde{\map}_{\D}(u^*X,u^*X)$, from~\cref{E-modules are E-categories}, to simplify the left hand branch of the diagram. Using~\cref{Interaction of cancelling tensor with h on the morphism level} and~\cref{Cancelling tensor and epsilon}, this diagram reduces to the definition of the unit $j:\h_\A\rightarrow\widetilde{\map}_{\D}(X,X)$, and so it commutes.

Finally, consider Axiom $4$ of~\cref{E-prederivator definition}. Once again, we can show the required diagram commutes by taking adjuncts, and using the definition of composition in $\eD$, from~\cref{E-modules are E-categories}. The diagram we obtain commutes, using~\cref{Cancelling tensor and associativity on morphism level} and~\cref{Cancelling tensor and epsilon}.
\end{proof}

\begin{remark}
Let $\D$ be a closed $\E$-module, let $\A$ and $\B$ be categories, and let $X\in\D(\A)$ and $Y\in\D(\B)$. Then the prederivator maps
\[
\widetilde{\map}_{\D}(X,-):\D\rightarrow\E^{\A\op}
\]
\[
\widetilde{\map}_{\D}(-,Y):\D\op\rightarrow\E^{\B}
\]
of~\cref{In an E-prederivator map(X -) is a prederivator map} recover the original prederivator maps coming from the closed $\E$-module structure. The same is true for $\map_\D(-,-):\D\op\times\D\rightarrow\E$ of~\cref{In an E-prederivator map(- -) is a prederivator map}.
\end{remark}

\section{The 2-category of $\E$-prederivators}\label{Section The 2-category of E-prederivators}

In the first half of this section, we show that $\E$-morphisms and $\E$-natural transformations respect the extra structure on $\E$-prederivators. In light of this, in \cref{2-category of E-prederivators}, we take these to be the $1$-cells and $2$-cells in the $2$-category of $\E$-prederivators. It follows immediately that the Yoneda lemma of \cref{Yoneda lemma for E-categories} carries over to $\E$-prederivators, and we use this in the second half of this section to prove \cref{E-morphism is representable if underlying morphism is}, a representability theorem for $\E$-morphisms of the form $F:\eA\rightarrow\eE$. In \cref{Representability for closed E-modules}, we apply this theorem to the $\E$-prederivator $\eD$ associated to a closed $\E$-module $\D$ to show that representability for a derivator map $F:\D\rightarrow\E$ can be deduced from representability theorems for the underlying category $\D(\0)$.

\begin{proposition}\label{E-morphisms induce prederivator maps}
Let $\eA$ and $\eB$ be $\E$-prederivators, and let $F:\eA\rightarrow\eB$ be an $\E$-morphism between the underlying $\E$-categories. For any functor $u:\A\rightarrow\B$ and any $X\in\dA_0(\B)$, we have a canonical isomorphism:
\[
\phi^u:u^*FX\xrightarrow{\;\;\;\iso\;\;\;}Fu^*X
\]
These are the structure isomorphisms for a prederivator map $F:\dA\rightarrow\dB$, with component at $\A$ given by the induced functor $F:\dA(\A)\rightarrow\dB(\A)$. Moreover, this is the unique prederivator map structure on these functors with the property that, for any $X\in\dA_0(\A)$ and $Y\in\dA_0(\B)$, the map
\[
F:\widetilde{\map}_{\dA}(X,Y)\rightarrow\widetilde{\map}_{\dB}(FX,FY)
\]
is a modification in both variables.
\end{proposition}
\begin{proof}
Let $u:\A\rightarrow\B$ be a functor and let $X\in\dA_0(\B)$. For any category $\C$ and any object $Z\in\dA_0(\C)$, consider the composite below:
\begin{center}
\begin{tikzcd}[column sep=small]
\widetilde{\map}_{\dA}(u^*X,Z)\arrow[rrr,"(\gamma^{u,\id})^{-1}"]
&&& (u\op\times\C)^*\widetilde{\map}_{\dA}(X,Z)\arrow[rrrr,"(u\op\times\C)^*F" above] &&&& (u\op\times\C)^*\widetilde{\map}_{\dB}(FX,FZ)\arrow[d,"\gamma^{u,\id}" right]\\
&&&&&&& \widetilde{\map}_{\dB}(u^*FX,FZ)
\end{tikzcd}
\end{center}
By~\cref{F:map(X Y) -> map(FX FY) is E-natural} this map is $\E$-natural in $Z\in\dA_0(\C)$. Thus, by~\cref{Yoneda lemma for E-categories}, this determines a unique map
\[
\phi^u_X:\h_\A\rightarrow\widetilde{\map}_{\dB}(u^*FX,Fu^*X)
\] 
in $\dB(\A)$. Using~\cref{Another description of Yoneda}, the map $\phi^u$ is characterised by the fact that, for any category $\C$ and any $Z\in\dA_0(\C)$, the diagram below commutes:
\begin{center}
\begin{tikzcd}
(u\op\times \C)^*\widetilde{\map}_{\dA}(X,Z)\arrow[dd,"\gamma^{u,\id}" left]\arrow[rrrr,"(u\op\times\C)^*F" above]
&&&& (u\op\times\C)^*\widetilde{\map}_{\dB}(FX,FZ)\arrow[dddd,"\gamma^{u,\id}" right]\\\\
\widetilde{\map}_{\dA}(u^*X,Z)\arrow[dd,"F" left]
\\\\
\widetilde{\map}_{\dB}(Fu^*X,FZ)\arrow[rrrr,"\widetilde{\map}_{\dB}(\phi^u_X\text{,}FZ)" below]&&&& \widetilde{\map}_{\dB}(u^*FX,FZ)
\end{tikzcd}
\end{center}
Dually, for any category $\C$ and any $Z\in\dA_0(\C)$, consider the following composite:
\begin{center}
\begin{tikzcd}[column sep=small]
\widetilde{\map}_{\dA}(Z,u^*X)\arrow[rrr,"(\gamma^{\id,u})^{-1}"]
&&& (\C\op\times u)^*\widetilde{\map}_{\dA}(Z,X)\arrow[rrrr,"(\C\op\times u)^*F" above] &&&& (\C\op\times u)^*\widetilde{\map}_{\dB}(FZ,FX)\arrow[d,"\gamma^{\id,u}" right]\\
&&&&&&& \widetilde{\map}_{\dB}(FZ,u^*FX)
\end{tikzcd}
\end{center}
This map is $\E$-natural in $Z\in\dA_0(\C)$, so, by~\cref{Yoneda lemma for E-categories}, it determines a map
\[
\psi^u_X:\h_\A\rightarrow\widetilde{\map}_{\dB}(Fu^*X,u^*FX)
\] 
in $\dB(\A)$, which is unique such that the diagram below commutes, for any category $\C$ and any $Z\in\dA_0(\C)$:
\begin{center}
\begin{tikzcd}
(\C\op\times u)^*\widetilde{\map}_{\dA}(Z,X)\arrow[dd,"\gamma^{\id,u}" left]\arrow[rrrr,"(\C\op\times u)^*F" above]
&&&& (\C\op\times u)^*\widetilde{\map}_{\dB}(FZ,FX)\arrow[dddd,"\gamma^{\id,u}" right]\\\\
\widetilde{\map}_{\dA}(Z,u^*X)\arrow[dd,"F" left]
\\\\
\widetilde{\map}_{\dB}(FZ,Fu^*X)\arrow[rrrr,"\widetilde{\map}_{\dB}(FZ\text{,}\;\psi^u_X)" below]&&&& \widetilde{\map}_{\dB}(FZ,u^*FX)
\end{tikzcd}
\end{center}
We claim that $\psi^u_X$ is inverse to $\phi^u_X$, and that the isomorphisms $\phi^u$ organise into structure isomorphisms for $F$. Note that once these facts are established, the commutative diagrams above express the fact that
\[
F:\widetilde{\map}_{\dA}(X,Y)\rightarrow\widetilde{\map}_{\dB}(FX,FY)
\]
is a modification in both variables. The uniqueness statement in the theorem follows from this observation.

We will start by showing that $\phi^u$ satisfies the conditions of~\cref{Pseudonatural transformation definition}. First, we must show that the maps $\psi^u_X$ are natural in $X\in\dA(\B)$. Let $f:\h_\B\rightarrow\widetilde{\map}_{\dA}(X,Y)$ be a map in $\dA(\B)$. To show that the naturality condition $Fu^*f\circ \phi^u_X=\phi^u_Y\circ u^*Ff$ is satisfied, we may equivalently show that the diagram below commutes, for any category $\C$ and any $Z\in\dA_0(\C)$:
\begin{center}
\begin{tikzcd}
\widetilde{\map}_{\dA}(u^*Y,Z)\arrow[r,"F"] & \widetilde{\map}_{\dB}(Fu^*Y,FZ)\arrow[dd,"\widetilde{\map}_{\dB}(Fu^*f\text{,}FZ)" left]\arrow[rrr,"\widetilde{\map}_{\dB}(\phi^u_Y\text{,}FZ)" above]
&&& \widetilde{\map}_{\dB}(u^*FY,FZ)\arrow[dd,"\widetilde{\map}_{\dB}(u^*Ff\text{,}FZ)" right]\\\\
& \widetilde{\map}_{\dB}(Fu^*X,FZ)\arrow[rrr,"\widetilde{\map}_{\dB}(\phi^u_X\text{,}FZ)" below]&&& \widetilde{\map}_{\dB}(u^*FX,FZ)
\end{tikzcd}
\end{center}
This diagram commutes, using the definition of $\phi^u$ above, and the commutative diagram (\ref{Diagram defining u*}) that defines $u^*$.

To verify Axiom $1$ of~\cref{Pseudonatural transformation definition}, for any category $\A$ and any $X\in\dA_0(\A)$, we must show that the unit $j:\h_\A\rightarrow \widetilde{\map}_{\dB}(FX,FX)$ satisfies the defining property of $\phi^\id_X$. This is immediate once we unravel the definitions.

Similarly, suppose we have composable functors $\A\xrightarrow{\;\;u\;\;}\B\xrightarrow{\;\;v\;\;}\C$, and let $X\in\dA_0(\C)$. To verify Axiom $2$ of~\cref{Pseudonatural transformation definition}, we must check that $\phi^u_{v^*X}\circ u^*\phi^v_X$ satisfies the defining property of $\phi^{v\circ u}_X$. That is, we need the diagram below to commute, for any category $\cD$ and any $Z\in\dA_0(\cD)$:
\begin{center}
\begin{tikzcd}
((v\circ u)\op\times \cD)^*\widetilde{\map}_{\dA}(X,Z)\arrow[dd,"\gamma^{v\circ u,\id}" left]\arrow[rrrr,"((v\circ u)\op\times \cD)^*F" above]
&&&& ((v\circ u)\op\times \cD)^*\widetilde{\map}_{\dB}(FX,FZ)\arrow[dd,"\gamma^{v\circ u,\id}" right]\\\\
\widetilde{\map}_{\dA}(u^*v^*X,Z)\arrow[dd,"F" left] &&&& \widetilde{\map}_{\dB}(u^*v^*FX,FZ)
\\\\
\widetilde{\map}_{\dB}(Fu^*v^*X,FZ)\arrow[rrrr,"\widetilde{\map}_{\dB}(\phi^u_{v^*X}\text{,}FZ)" below]&&&& \widetilde{\map}_{\dB}(u^*Fv^*X,FZ)\arrow[uu,"\widetilde{\map}_{\dB}(u^*\phi^v_{X}\text{,}FZ)" right]
\end{tikzcd}
\end{center}
This diagram commutes, using Axiom $1$ of~\cref{E-prederivator definition}, the defining property of $\phi^u$ given above, and the commutative diagram (\ref{Diagram defining u*}).

Finally, suppose we have functors $u,v:\A\rightarrow\B$ and a natural transformation $\kappa:u\Rightarrow v$. To check Axiom $3$ of~\cref{Pseudonatural transformation definition}, we must verify the equality $F(\kappa^*_X)\circ\phi^u_X=\phi^v_X\circ \kappa^*_{FX}$, for any $X\in\dA_0(\B)$. Equivalently, we may show that the diagram below commutes, for any category $\C$ and any $Z\in\dA_0(\C)$:
\begin{center}
\begin{tikzcd}
\widetilde{\map}_{\dA}(v^*X,Z)\arrow[r,"F"] & \widetilde{\map}_{\dB}(Fv^*X,FZ)\arrow[dd,"\widetilde{\map}_{\dB}(F(\kappa^*_X)\text{,}FZ)" left]\arrow[rrr,"\widetilde{\map}_{\dB}(\phi^v_X\text{,}FZ)" above]
&&& \widetilde{\map}_{\dB}(v^*FX,FZ)\arrow[dd,"\widetilde{\map}_{\dB}(\kappa^*_{FX}\text{,}FZ)" right]\\\\
& \widetilde{\map}_{\dB}(Fu^*X,FZ)\arrow[rrr,"\widetilde{\map}_{\dB}(\phi^u_X\text{,}FZ)" below]&&& \widetilde{\map}_{\dB}(u^*FX,FZ)
\end{tikzcd}
\end{center}
This diagram commutes, using the definition of $\phi^u$ given above, and the definition of $\kappa^*$ from~\cref{Natural transformation omega* from an E-prederivator}.

For any functor $u:\A\rightarrow\B$ and any $X\in\dA_0(\B)$, it remains to show that $\psi^u_X$ is inverse to $\phi^u_X$. To do this, note that we may describe $\psi^u_X$ and $\phi^u_X$ explicitly as follows:
\begin{center}
\begin{tikzcd}[column sep=tiny,row sep=small]
\phi^u_X:\h_\A\arrow[rr,"j"] && \widetilde{\map}_{\dA}(u^*X,u^*X)\arrow[rrrrr,"(\gamma^{u,\id})^{-1}"]
&&&&& (u\op\times\A)^*\widetilde{\map}_{\dA}(X,u^*X)\arrow[rd,"(u\op\times\A)^*F" above right,bend left=10]\\

&&&&&&&& (u\op\times\A)^*\widetilde{\map}_{\dB}(FX,Fu^*X)\arrow[ddd,"\gamma^{u,\id}" right]\\\\\\
&&&&&&&& \widetilde{\map}_{\dB}(u^*FX,Fu^*X)
\end{tikzcd}
\end{center}
\begin{center}
\begin{tikzcd}[column sep=tiny,row sep=small]
\psi^u_X:\h_\A\arrow[rr,"j"] && \widetilde{\map}_{\dA}(u^*X,u^*X)\arrow[rrrrr,"(\gamma^{\id,u})^{-1}"]
&&&&& (\A\op\times u)^*\widetilde{\map}_{\dA}(u^*X,X)\arrow[rd,"(\A\op\times u)^*F" above right,bend left=10]\\

&&&&&&&& (\A\op\times u)^*\widetilde{\map}_{\dB}(Fu^*X,FX)\arrow[ddd,"\gamma^{\id,u}" right]\\\\\\
&&&&&&&& \widetilde{\map}_{\dB}(Fu^*X,u^*FX)
\end{tikzcd}
\end{center}
Using the description of composition in $\dB(\A)$ from~\cref{Underlying categories of an E-category}, we see that $\psi^u_X\circ\phi^u_X=\id$, using Axiom $3$ of~\cref{E-prederivator definition}, and Axiom $2$ in the form given in~\cref{Rephrasing E-naturality in definition of E-prederivator}. Similarly, $\phi^u_X\circ\psi^u_X=\id$, using Axiom $4$ of~\cref{E-prederivator definition}. In both cases, we also use the axioms for $F$ from~\cref{E-morphism definition}.
\end{proof}

\begin{lemma}\label{E-natural maps induce modifications}
Let $\eA$ and $\eB$ be $\E$-prederivators, and $F,G:\eA\rightarrow\eB$ be $\E$-morphisms. Then any $\E$-natural transformation
\begin{center}
\begin{tikzcd}
\eA\arrow[rr, bend left=45, "F" above,""{name=U, below}]\arrow[rr, bend right=45, "G" below, ""{name=D}]
&& \eB
\arrow[Rightarrow,from=U,to=D,shorten >=0.2cm,shorten <=0.2cm,"\beta"]
\end{tikzcd}
\end{center}
induces a modification $\beta:F\Rightarrow G$ between the prederivator maps $F,G:\dA\rightarrow\dB$ of~\cref{E-morphisms induce prederivator maps}.
\end{lemma}
\begin{proof}
Suppose we have a functor $u:\A\rightarrow\B$ and an object $X\in\dA(\B)$. We need to show that the following diagram in $\dB(\A)$ commutes:
\begin{center}
\begin{tikzcd}
u^*FX \arrow{rr}[above]{u^*\beta_X}\arrow{dd}[left]{\phi^u_X} && 
 u^*GX \arrow{dd}[right]{\phi^u_X} \\\\
 Fu^*X \arrow{rr}[below]{\beta_{u^*X}} && 
 Gu^*X
\end{tikzcd}
\end{center}
Equivalently, we may show that the diagram below commutes, for any category $\C$ and any $Z\in\dA_0(\C)$:
\begin{center}
\begin{tikzcd}
\widetilde{\map}_{\dA}(u^*X,Z)\arrow[r,"G"] & \widetilde{\map}_{\dB}(Gu^*X,GZ)\arrow[dd,"\widetilde{\map}_{\dB}(\phi^u_X\text{,}GZ)" left]\arrow[rrr,"\widetilde{\map}_{\dB}(\beta_{u^*X}\text{,}GZ)" above]
&&& \widetilde{\map}_{\dB}(Fu^*X,GZ)\arrow[dd,"\widetilde{\map}_{\dB}(\phi^u_X\text{,}GZ)" right]\\\\
& \widetilde{\map}_{\dB}(u^*GX,GZ)\arrow[rrr,"\widetilde{\map}_{\dB}(u^*\beta_X\text{,}GZ)" below]&&& \widetilde{\map}_{\dB}(u^*FX,GZ)
\end{tikzcd}
\end{center}
Using the $\E$-naturality of $\beta$, in the form given in~\cref{Rephrasing E-naturality}, the definition of $u^*\beta_X$ from~\cref{The functor u* from an E-prederivator}, and the definition of $\phi^u_X$ from~\cref{E-morphisms induce prederivator maps}, we can see that this diagram commutes. 
\end{proof}

In light of~\cref{E-morphisms induce prederivator maps} and~\cref{E-natural maps induce modifications}, $\E$-morphisms and $\E$-natural transformations are the appropriate $1$-cells and $2$-cells for the $2$-category of $\E$-prederivators:

\begin{definition}\label{2-category of E-prederivators}
The \textbf{$2$-category of $\E$-prederivators} is the $2$-category $\EPDer$, with objects given by $\E$-prederivators, $1$-cells given by $\E$-morphisms and $2$-cells given by $\E$-natural transformations.
\end{definition}

\begin{remark}
Let $\eA$ be an $\E$-category. By \cref{2-category of E-prederivators}, the identity $\E$-morphism $\id_{\eA}$ induces an equivalence between any two $\E$-prederivator structures on $\eA$. Thus, being an $\E$-prederivator is a property of an $\E$-category, rather than extra structure. See \cref{Pullbacks are weighted colimits} for a related discussion.
\end{remark}

With this definition,~\cref{Yoneda lemma for E-categories} carries over immediately to an $\E$-prederivator version:

\begin{corollary}\label{Yoneda lemma for E-prederivators}
Let $\eA$ be an $\E$-prederivator, let $\A$ be a category, let $X\in\dA_0(\A)$, and let $F:\eA\rightarrow\eE^{\A\op}$ be an $\E$-morphism. We have a natural bijection:
\[
\EPDer(\eA,\eE^{\A\op})(\widetilde{\map}_{\dA}(X,-),F)\iso\E(\A\op\times\A)(\h_\A,FX)
\]
\end{corollary}

We finish this section with an application of \cref{Yoneda lemma for E-prederivators}. The following theorem reduces the representability of an $\E$-morphism $F$ to the representability of its underlying $\E(\0)$-functor of~\cref{A([0]) is E([0])-enriched}:

\begin{theorem}\label{E-morphism is representable if underlying morphism is}
Let $\eA$ be an $\E$-prederivator. An $\E$-morphism $F:\eA\rightarrow\eE$ is representable if and only if the $\E(\0)$-functor
\[
F:\eA(\0)\rightarrow\eE(\0)
\]
is representable as an $\E(\0)$-functor.
\end{theorem}
\begin{proof}
Let $F:\eA\rightarrow\eE$ be an $\E$-morphism. If $F$ is represented by an object $X\in\A_0(\0)$, then the $\E$-natural isomorphism
\begin{center}
\begin{tikzcd}
\eA\arrow[rrr, bend left=40, "\widetilde{\map}_{\dA}(X\text{,}-)" above,""{name=U, below}]\arrow[rrr, bend right=40, "F" below, ""{name=D}]
&&& \eE
\arrow[Rightarrow,from=U,to=D,shorten >=0.3cm,shorten <=0.3cm,"\iso" left=2]
\end{tikzcd}
\end{center}
induces an isomorphism between the corresponding $\E(\0)$-functors.

On the other hand, suppose the $\E(\0)$-functor $F:\eA(\0)\rightarrow\eE(\0)$ is representable. That is, there is an object $X\in\dA_0(\0)$ and an $\E(\0)$-natural isomorphism 
\begin{center}
\begin{tikzcd}
\eA(\0)\arrow[rr, bend left=40, "\widetilde{\map}_{\dA}(X\text{,}-)" above,""{name=U, below}]\arrow[rr, bend right=40, "F" below, ""{name=D}]
&& \eE(\0).
\arrow[Rightarrow,from=U,to=D,shorten >=0.3cm,shorten <=0.3cm,"\beta" left=2,"\iso" right=2]
\end{tikzcd}
\end{center}
By the (weak) Yoneda for enriched categories (see~\cite[Section 1.9]{Kelly82}), this map $\beta$ uniquely determines a map
\[
f:\mathbb{1}\rightarrow FX
\]
in $\E(\0)$. But, by \cref{Yoneda lemma for E-categories}, this map $f$ uniquely determines an $\E$-natural transformation
\begin{center}
\begin{tikzcd}
\eA\arrow[rrr, bend left=40, "\widetilde{\map}_{\dA}(X\text{,}-)" above,""{name=U, below}]\arrow[rrr, bend right=40, "F" below, ""{name=D}]
&&& \eE.
\arrow[Rightarrow,from=U,to=D,shorten >=0.3cm,shorten <=0.3cm,"\bar{\beta}" left=2]
\end{tikzcd}
\end{center}
Moreover, the $\E(\0)$-natural transformation induced by this $\E$-natural transformation is the original $\E(\0)$-natural transformation $\beta$. 

The $\E$-natural map $\bar{\beta}:\widetilde{\map}_{\dA}(X,-)\Rightarrow F$ is an isomorphism if and only if each component is an isomorphism; that is, if and only if the induced modification 
\begin{center}
\begin{tikzcd}
\dA\arrow[rrr, bend left=40, "\widetilde{\map}_{\dA}(X\text{,}-)" above,""{name=U, below}]\arrow[rrr, bend right=40, "F" below, ""{name=D}]
&&& \E
\arrow[Rightarrow,from=U,to=D,shorten >=0.3cm,shorten <=0.3cm]
\end{tikzcd}
\end{center}
is an isomorphism. Since $\E$ satisfies \textbf{Der 2}, this is the case if and only if the natural transformation on underlying categories
\begin{center}
\begin{tikzcd}
\dA(\0)\arrow[rr, bend left=40, "\widetilde{\map}_{\dA}(X\text{,}-)" above,""{name=U, below}]\arrow[rr, bend right=40, "F" below, ""{name=D}]
&& \E(\0)
\arrow[Rightarrow,from=U,to=D,shorten >=0.3cm,shorten <=0.3cm]
\end{tikzcd}
\end{center}
is an isomorphism, by \cref{modifications are isomorphisms iff they are on underlying category}. (Note that this step is where we need $\eA$ be an $\E$-prederivator, rather than a general $\E$-category.) This natural transformation is an isomorphism, since it underlies the original $\E(\0)$-natural isomorphism $\beta$.
\end{proof}

\begin{corollary}\label{Representability for closed E-modules}
Let $\D$ be a closed $\E$-module. A derivator map
\[
F:\D\rightarrow\E
\]
is representable in the sense of \cref{Representable map from a closed E-module} if and only if it is continuous and preserves cotensors, and the induced functor
\vspace{-3em}
\begin{center}
\begin{equation}\label{Underlying representable functor}
\begin{tikzcd}[column sep=small,baseline=(current  bounding  box.center)]
\D(\0)\arrow{rr}[above]{F} && \E(\0)\arrow{rrrr}[above]{\E(\0)(\mathbb{1},-)} &&&& \Set
\end{tikzcd}
\end{equation}
\end{center}
\vspace{-2em}
is representable.
\end{corollary}
\begin{proof}
As in \cref{Representability for perfectly generated triangulated derivator}, the forward implication is immediate. For the reverse implication, suppose $F:\D\rightarrow\E$ is a continuous, cotensor-preserving map. Note that the cotensor product on the derivator $\D$ induces cotensors on the $\E(\0)$-category $\eD(\0)$, and these are preserved by the $\E(\0)$-functor $F:\eD(\0)\rightarrow\eE(\0)$.

Consider the (unenriched) composite (\ref{Underlying representable functor}). Since $F:\eD(\0)\rightarrow\eE(\0)$ preserves cotensors, it follows from~\cite[Theorem 4.85]{Kelly82} that the $\E(\0)$-functor $F:\eD(\0)\rightarrow\eE(\0)$ is representable if and only if this unenriched functor (\ref{Underlying representable functor}) is representable.

By \cref{Cocontinuous E-module maps induce E-morphisms}, $F$ induces an $\E$-morphism $F:\eD\rightarrow\eE$ on the associated $\E$-prederivators. If (\ref{Underlying representable functor}) is representable, then \cref{E-morphism is representable if underlying morphism is} implies that this $\E$-morphism is representable. Thus, there is an object $X\in\D(\0)$, and an $\E$-natural isomorphism
\begin{center}
\begin{tikzcd}
\eD\arrow[rrr, bend left=35, "\widetilde{\map}_{\D}(X\text{,}-)" above,""{name=U, below}]\arrow[rrr, bend right=35, "F" below, ""{name=D}]
&&& \eE.
\arrow[Rightarrow,from=U,to=D,shorten >=0.3cm,shorten <=0.3cm,"\iso" left=2]
\end{tikzcd}
\end{center}
This $\E$-natural transformation induces the desired isomorphism between the induced derivator maps.
\end{proof}

\cref{Representability for closed E-modules} allows us to improve slightly on \cref{Representability for perfectly generated triangulated derivator}, in the case of maps of the form $F:\D\op\rightarrow\E$. Specifically, in the following theorem, we only require that the underlying category $\D(\0)$ satisfies Brown representability, rather than asking that $\D(\C)$ satisfies Brown representability, for every category $\C$.

\begin{corollary}\label{Representability for triangulated E-module II}
Let $\D$ be a closed $\E$-module. Suppose that $\E$ and $\D$ are triangulated, and that $\D(\0)$ satisfies Brown representability. Then a derivator map
\[
F:\D\op\rightarrow\E
\]
is representable in the sense of \cref{Representable map from a closed E-module} if and only if it is continuous and preserves cotensors.
\end{corollary}
\begin{proof}
Suppose $F:\D\op\rightarrow\E$ is a continuous, cotensor-preserving map, and consider the composite below:
\begin{center}
\begin{tikzcd}[column sep=small]
\D(\0)\op\arrow{rr}[above]{F} && \E(\0)\arrow{rrrr}[above]{\E(\0)(\mathbb{1},-)} &&&& \Ab
\end{tikzcd}
\end{center}
Since $F:\D(\0)\op\rightarrow\E(\0)$ is an exact map of triangulated categories, this composite is a cohomological functor. Moreover, it takes coproducts in $\D(\0)$ to products in $\Ab$. Thus, since $\D(\0)$ satisfies Brown representability, this composite is representable. Thus, the result follows by \cref{Representability for closed E-modules}.
\end{proof}

\section{$\E$-Derivators}\label{Section E-derivators}

In this section, we introduce $\E$-derivators and show that any closed $\E$-module gives rise to an $\E$-derivator. As a first step, we define $\E$-semiderivators in \cref{E-semiderivator definition}, followed by $\E$-derivators in \cref{(Left) E-derivator definition}. Just as derivators are, in particular, semiderivators that admit all homotopy limits and colimits, $\E$-derivators are $\E$-semiderivators that admit all weighted homotopy limits and colimits. However, in contrast to derivators, there are no further axioms we need to impose; in particular, we do not need any analogue of \textbf{Der 4}. Nonetheless, in \cref{(Left) E-derivators induce (left) derivators}, we prove that, for any $\E$-derivator $\eA$, the induced prederivator $\dA$ is a derivator. On the other hand, in \cref{Closed E-modules induce E-derivators}, we prove that the $\E$-prederivator associated to any closed $\E$-module is an $\E$-derivator. 

We begin this section with the definitions of weighted homotopy limits and colimits. These are studied in~\cite[Section 4.5]{GR19} and~\cite{GS17} in the context of $\E$-modules; if we apply \cref{Weighted colimits definition} in the $\E$-category associated to a closed $\E$-module, we recover these notions. However, in contrast to closed $\E$-modules, where all weighted homotopy limits and colimits always exist, in general $\E$-categories we may study particular weighted homotopy limits and colimits, in settings where all may not exist. In this way, it is preferable to work in the context of $\E$-categories and $\E$-prederivators, rather than restricting to closed $\E$-modules.

\begin{definition}\label{Weighted colimits definition}
Let $\eA$ be an $\E$-category, let $\A$ and $\B$ be categories, and let $X\in\dA_0(\A)$ and $W\in\E(\A\op\times\B)$. Consider the $\E$-morphism below:
\begin{center}
\begin{tikzcd}[column sep=small]
\eA\arrow{rrrrr}{\widetilde{\map}_{\dA}(X,-)} &&&&& \eE^{\A\op}\arrow{rrrrr}{\widetilde{\map}_{\E^{\A\op}}(W,-)} &&&&& \eE^{\B\op}
\end{tikzcd}
\end{center} 
If this map is representable, we call the representing object the \textbf{homotopy colimit of $X$ weighted by $W$}, and denote it by $W\otimes_\A X\in\dA_0(\B)$. Thus, $W\otimes_\A X$ is characterised by isomorphisms
\[
\widetilde{\map}_{\dA}(W\otimes_\A X,Z)\iso\widetilde{\map}_{\E^{\A\op}}(W,\widetilde{\map}_{\dA}(X,Z)),
\]
$\E$-natural in $Z\in\dA_0(\C)$.

Dually, given $Y\in\dA_0(\B)$ and $W\in\E(\A\op\times\B)$, the \textbf{homotopy limit of $Y$ weighted by $W$}, if it exists, is an object $Y\vartriangleleft_\B W\in\dA_0(\A)$ representing the $\E$-morphism below:
\begin{center}
\begin{tikzcd}[column sep=small]
\eA\op\arrow{rrrrr}{\widetilde{\map}_{\dA}(-,Y)} &&&&& \eE^{\B}\arrow{rrrrr}{\widetilde{\map}_{\E^{\B}}(W,-)} &&&&& \eE^{\A}
\end{tikzcd}
\end{center} 
\end{definition}


Let $\eA$ and $\eB$ be $\E$-categories and let $F:\eA\rightarrow\eB$ be an $\E$-morphism. Given $X\in\dA_0(\A)$ and $W\in\E(\A\op\times\B)$, suppose that the weighted homotopy colimits $W\otimes_\A X\in\dA_0(\B)$ and $W\otimes_\A FX\in\dB_0(\B)$ both exist, and consider the composite below, for any category $\C$, and any $Z\in\dA_0(\C)$:
\begin{center}
\begin{tikzcd}[column sep=small]
\widetilde{\map}_{\dA}(W\otimes_\A X,Z)\arrow[rr,"\iso"]
&& \widetilde{\map}_{\E^{\A\op}}(W,\widetilde{\map}_{\dA}(X,Z))\arrow[rd,"\widetilde{\map}_{\E^{\A\op}}(W\text{,}F)" above right,bend left=13] \\ &&&\widetilde{\map}_{\E^{\A\op}}(W,\widetilde{\map}_{\dB}(FX,FZ))\arrow[d,"\iso" right]\\
&&& \widetilde{\map}_{\dB}(W\otimes_\A FX,FZ)
\end{tikzcd}
\end{center}
By~\cref{F:map(X Y) -> map(FX FY) is E-natural}, this map is $\E$-natural in $Z$. Thus, using~\cref{Yoneda lemma for E-categories}, it determines a map
\[
\phi:\h_\B\rightarrow\widetilde{\map}_{\dB}(W\otimes_\A FX,F(W\otimes_\A X))
\] 
in $\dB(\B)$, unique such that the diagram below commutes, for any category $\C$ and any object $Z\in\dA_0(\C)$:
\begin{center}
\begin{tikzcd}[column sep=small]
\widetilde{\map}_{\E^{\A\op}}(W,\widetilde{\map}_{\dA}(X,Z))\arrow[dd,"\iso" left]\arrow[rrrrr,"\widetilde{\map}_{\E^{\A\op}}(W\text{,}F)" above]
&&&&& \widetilde{\map}_{\E^{\A\op}}(W,\widetilde{\map}_{\dB}(FX,FZ))\arrow[dddd,"\iso" right]\\\\
\widetilde{\map}_{\dA}(W\otimes_\A X,Z)\arrow[dd,"F" left]
\\\\
\widetilde{\map}_{\dB}(F(W\otimes_\A X),FZ)\arrow[rrrrr,"\widetilde{\map}_{\dB}(\phi\text{,}FZ)" below]&&&&& \widetilde{\map}_{\dB}(W\otimes_\A FX,FZ)
\end{tikzcd}
\end{center}
If this canonical map $\phi:W\otimes_\A FX\rightarrow F(W\otimes_\A X)$ is an isomorphism in $\dB(\B)$, then we say that $F$ \textbf{preserves} the weighted homotopy colimit.

\begin{example}\label{Pullbacks are weighted colimits}
Let $\eA$ be an $\E$-prederivator and let $u:\A\rightarrow\B$ be a functor. For any $X\in\dA_0(\B)$, and any $Z\in\dA_0(\C)$, consider the isomorphism below:
\begingroup
\addtolength{\jot}{0.7em}
\begin{align*}
\widetilde{\map}_{\dA}(u^*X,Z) &\iso (u\op\times\C)^*\widetilde{\map}_{\dA}(X,Z)\\
&\iso (u\op\times\C)^*\widetilde{\map}_{\E^{\B\op}}(\h_\B,\widetilde{\map}_{\dA}(X,Z))\\
&\iso\widetilde{\map}_{\E^{\B\op}}((\B\op\times u)^*\h_\B,\widetilde{\map}_{\dA}(X,Z))
\end{align*}
\endgroup
In the composite above, the first and last maps are instances of $\gamma^{u,\id}$ for the $\E$-prederivators $\eA$ and $\eE^{\B\op}$. The second map is induced by the isomorphism $\varrho:\widetilde{\map}_{\E^{\B\op}}(\h_\B,-)\xrightarrow{\;\iso\;}\id$. In the proof of~\cref{Yoneda lemma for E-categories}, this map is shown to be $\E$-natural; thus, the whole composite above is $\E$-natural in $Z\in\dA_0(\C)$. 

Therefore, in any $\E$-prederivator $\eA$ and for any $u:\A\rightarrow\B$, the object $u^*X\in\dA_0(\A)$ is the homotopy colimit of $X\in\dA_0(\B)$ weighted by $(\B\op\times u)^*\h_\B\in\E(\B\op\times\A)$. Moreover, given a second $\E$-prederivator $\eB$ and an $\E$-morphism $F:\eA\rightarrow\eB$, the canonical map
\[
\phi:(\B\op\times u)^*\h_\B\otimes_\B FX\rightarrow F((\B\op\times u)^*\h_\B\otimes_\B X)
\]
is the isomorphism $\phi^u_X$ of~\cref{E-morphisms induce prederivator maps}. Thus, any $\E$-morphism between $\E$-prederivators preserves weighted homotopy colimits of this form.
\end{example}

\begin{lemma}\label{Left adjoints preserve weighted colimits}
Let $\eA$ and $\eB$ be $\E$-categories, and suppose we have an adjunction:
\begin{center}
\begin{tikzcd}
\eA\arrow[rr, bend left=40, "F" above,""{name=U, below}]\arrow[rr,leftarrow, bend right=40, "G" below, ""{name=D}]
&& \;\eB
\arrow[phantom,from=U,to=D,"\bot"]
\end{tikzcd}
\end{center}
Let $X\in\dA_0(\A)$, let $W\in\E(\A\op\times\B)$, and suppose the weighted homotopy colimit $W\otimes_\A X\in\dA_0(\B)$ exists. Then $F(W\otimes_\A X)\in\dB_0(\B)$ is the homotopy colimit of $FX$ weighted by $W$. In particular, $F$ preserves any weighted homotopy colimit that exists in $\eA$.
\end{lemma}
\begin{proof}
Given $X\in\dA_0(\A)$ and $W\in\E(\A\op\times\B)$, we have the following string of isomorphisms, $\E$-natural in $Z\in\dB_0(\C)$:
\begingroup
\addtolength{\jot}{0.5em}
\begin{align*}
\widetilde{\map}_{\dB}(F(W\otimes_\A X),Z) &\iso\widetilde{\map}_{\dA}(W\otimes_\A X,GZ)\\
&\iso\widetilde{\map}_{\E^{\A\op}}(W,\widetilde{\map}_{\dA}(X,GZ))\\
&\iso\widetilde{\map}_{\E^{\A\op}}(W,\widetilde{\map}_{\dB}(FX,Z))
\end{align*}
\endgroup
Thus, $F(W\otimes_\A X)\in\dB_0(\B)$ has the defining property of $W\otimes_\A FX$, and it follows that $F$ preserves the weighted homotopy colimit.
\end{proof}

We will now begin to work towards the definition of $\E$-derivator. One property we desire of an $\E$-derivator is that the prederivator it induces should be a derivator. We observe below that this is already partway satisfied, for any $\E$-prederivator:

\begin{proposition}\label{E-prederivators satisfy Der 2}
For any $\E$-prederivator $\eA$, the induced prederivator $\dA$ satisfies \textbf{Der 2}.
\end{proposition}
\begin{proof}
Let $\A$ be a category, and let $f:\h_\A\rightarrow\widetilde{\map}_{\dA}(X,Y)$ be a map in $\dA(\A)$. We need to show that $f$ is an isomorphism if and only if, for every object $a\in\A$, the map $f_a:X_a\rightarrow Y_a$ is an isomorphism.

By~\cref{E-category yoneda embedding is fully faithful}, $f$ is an isomorphism if and only if the map
\[
\widetilde{\map}_{\dA}(Z,f):\widetilde{\map}_{\dA}(Z,X)\rightarrow\widetilde{\map}_{\dA}(Z,Y)
\]
is an isomorphism in $\E(\C\op\times\A)$, for every category $\C$ and every $Z\in\dA_0(\C)$.

For a fixed $Z\in\dA_0(\C)$, \textbf{Der 2} for the derivator $\E^{\C\op}$ implies that this map is an isomorphism if and only if, for every $a\in\A$, the map
\[
(\C\op\times a)^*\widetilde{\map}_{\dA}(Z,f):(\C\op\times a)^*\widetilde{\map}_{\dA}(Z,X)\rightarrow(\C\op\times a)^*\widetilde{\map}_{\dA}(Z,Y)
\]
is an isomorphism in $\E(\C\op)$.

For a fixed $a\in\A$, this map is an isomorphism if and only if 
\[
\widetilde{\map}_{\dA}(Z,f_a):\widetilde{\map}_{\dA}(Z,X_a)\rightarrow\widetilde{\map}_{\dA}(Z,Y_a)
\]
is an isomorphism, using the prederivator map structure for $\widetilde{\map}_{\dA}(Z,-)$.

Thus, $f$ is an isomorphism if and only if, for every $a\in\A$, every category $\C$, and every $Z\in\dA_0(\C)$, the map $\widetilde{\map}_{\dA}(Z,f_a):\widetilde{\map}_{\dA}(Z,X_a)\rightarrow\widetilde{\map}_{\dA}(Z,Y_a)$ is an isomorphism. By~\cref{E-category yoneda embedding is fully faithful}, this is true if and only if each map $f_a:X_a\rightarrow Y_a$ is an isomorphism in $\dA(\0)$.
\end{proof}

\begin{definition}\label{E-semiderivator definition}
Let $\eA$ be an $\E$-prederivator. We call $\eA$ an \textbf{$\E$-semiderivator} if the induced prederivator $\dA$ is a semiderivator; that is, $\dA$ satisfies \textbf{Der 1} and \textbf{Der 2}. Note that, by~\cref{E-prederivators satisfy Der 2}, this is the case if and only if $\dA$ satisfies \textbf{Der 1}.
\end{definition}

\begin{definition}\label{(Left) E-derivator definition}
Let $\eA$ be an $\E$-semiderivator. We say that $\eA$ is a \textbf{left $\E$-derivator} if $\eA$ admits all weighted homotopy colimits. Dually, we say that $\eA$ is a \textbf{right $\E$-derivator} if $\eA$ admits all weighted homotopy limits. We call $\eA$ an \textbf{$\E$-derivator} if it is both a left and right $\E$-derivator.
\end{definition}

Suppose $\eA$ is a left $\E$-derivator, let $\A$ be a category, and let $X\in\dA_0(\A)$. For any object $W\in\E(\A\op\times\B)$, we have a family of isomorphisms
\[
\widetilde{\map}_{\dA}(W\otimes_\A X,Y)\iso\widetilde{\map}_{\E^{\A\op}}(W,\widetilde{\map}_{\dA}(X,Y)),
\]
$\E$-natural in $Y\in\dA_0(\C)$. By~\cref{Defining adjoints representably}, the weighted homotopy colimits organise into an $\E$-morphism $-\otimes_\A X:\eE^{\A\op}\rightarrow\eA$, the left adjoint in an adjunction:
\begin{center}
\begin{tikzcd}
\eE^{\A\op}\arrow[rr, bend left=40, "-\;\otimes_\A X" above,""{name=U, below}]\arrow[rr,leftarrow, bend right=40, "\widetilde{\map}_{\dA}(X\text{,}-)" below, ""{name=D}]
&& \;\eA
\arrow[phantom,from=U,to=D,"\bot"]
\end{tikzcd}
\end{center}
Thus, an $\E$-semiderivator $\eA$ is a left $\E$-derivator if and only if, for any category $\A$ and any $X\in\dA_0(\A)$, the $\E$-morphism 
\[
\widetilde{\map}_{\dA}(X,-):\eA\rightarrow\eE^{\A\op}
\]
has a left adjoint. Dually, an $\E$-semiderivator $\eA$ is a right $\E$-derivator if and only if, for any category $\B$ and any $Y\in\dA_0(\B)$, the $\E$-morphism 
\[
\widetilde{\map}_{\dA}(-,Y):\eA\op\rightarrow\eE^{\B}
\]
has a left adjoint.

\begin{theorem}\label{Closed E-modules induce E-derivators}
Let $\D$ be a closed $\E$-module. The associated $\E$-prederivator $\eD$ is an $\E$-derivator.
\end{theorem}
\begin{proof}
Since $\D$ is a derivator, the associated $\E$-prederivator $\eD$ is an $\E$-semiderivator. We need to show that $\eD$ admits all weighted homotopy limits and colimits. 

Let $\A$ be a category and let $X\in\D(\A)$. By~\cref{E-module adjunctions induce E-category adjunctions}, the adjunction 
\begin{center}
\begin{tikzcd}
\E^{\A\op}\arrow[rr, bend left=40, "-\;\otimes_\A X" above,""{name=U, below}]\arrow[rr,leftarrow, bend right=40, "\widetilde{\map}_{\D}(X\text{,}-)" below, ""{name=D}]
&& \;\D
\arrow[phantom,from=U,to=D,"\bot"]
\end{tikzcd}
\end{center}
induces an adjunction between the corresponding $\E$-categories. Thus, $\eD$ admits all weighted homotopy colimits. Similarly, $\eD$ admits all weighted homotopy limits.
\end{proof}

The proof of~\cref{Closed E-modules induce E-derivators} justifies our notation for weighted homotopy limits and colimits: if $\D$ is a closed $\E$-module, given $X\in\D(\A)$ and $W\in\E(\A\op\times\B)$, the weighted homotopy colimit $W\otimes_\A X\in\D(\B)$ is the image of $W$ under the functor $-\otimes_\A X:\E(\A\op\times\B)\rightarrow\D(\B)$.

\begin{remark}
Let $F:\D_1\rightarrow\D_2$ be a cocontinuous $\E$-module map between closed $\E$-modules. Using the description of weighted homotopy colimits given in the proof of \cref{Closed E-modules induce E-derivators}, it follows easily that the associated $\E$-morphism $F:\eD_1\rightarrow\eD_2$ preserves all weighted homotopy colimits.
\end{remark}

\begin{example}\label{Compact objects give an E-prederivator}
We will now give an example of an enriched prederivator that admits some, but not all, weighted homotopy colimits. 

Let $\D$ be a triangulated derivator. By \cref{Any triangulated derivator is a closed Spt module}, $\D$ is a closed $\dHo(\Spt)$-module, and so we can consider the associated $\dHo(\Spt)$-derivator $\eD$. Let $\eD^c$ be the maximal sub-$\dHo(\Spt)$-prederivator of $\D$ on the compact objects of $\D(\0)$, in the sense of \cref{Maximal sub-E-prederivator}. Thus, for any category $\A$, the set $\D^{^c}_0(\A)$ consists of the objects $X\in\D(\A)$ such that $X_a\in\D(\0)$ is compact for every $a\in\A$.

We claim that an object $X\in\D(\A)$ lies in $\D^{^c}_0(\A)$ if and only if the derivator map
\[
\widetilde{\map}_{\D}(X,-):\D\rightarrow\dHo(\Spt^{\A\op})
\]
is cocontinuous. By \cite[Proposition 7.3]{Groth16} and \cite[Theorem 7.13]{PS16}, this map is cocontinuous if and only if the functor on underlying categories
\[
\widetilde{\map}_{\D}(X,-):\D(\0)\rightarrow\Ho(\Spt^{\A\op})
\]
preserves coproducts. By \cite[Lemma 1.34]{GM11}, this functor preserves coproducts if and only if its left adjoint
\[
-\wedge_\A X:\Ho(\Spt^{\A\op})\rightarrow\D(\0)
\]
takes the set of compact generators in $\Ho(\Spt^{\A\op})$ to compact objects. By \cref{Perfectly generated triangulated derivator}, 
$\{a_!\mathds{S} \;|\; a\in\A\op\}$ is a set of compact generators for $\Ho(\Spt^{\A\op})$, where $\mathds{S}\in\Ho(\Spt)$ is the sphere spectrum. The isomorphism (\ref{Diagram Left Kan extensions and delta}), applied to the map $a:\0\rightarrow\A\op$, gives an isomorphism
\[
a_!\mathds{S}\iso(\A\op\times a)^*\partial_\A\mathds{S}=(\A\op\times a)^*\h_\A,
\]
where, on the right hand side, we think of $a$ as an object of $\A$ rather than $\A\op$. Applying $-\wedge_\A X$, we have
\[
((\A\op\times a)^*\h_\A)\wedge_\A X\iso a^*(\h_\A\wedge_\A X)\iso X_a.
\]
This proves our claim.

Using this, suppose we have $X\in\D^{^c}_0(\A)$ and $W\in\Ho(\Spt^{\A\op\times\B})$, and consider the composite
\begin{center}
\begin{tikzcd}[column sep=small]
\eD\arrow{rrrrr}{\widetilde{\map}_{\dA}(X,-)} &&&&& \underline{\dHo}(\Spt^{\A\op})\arrow{rrrrrrr}{\widetilde{\map}_{\dHo(\Spt^{\A\op})}(W,-)} &&&&&&& \underline{\dHo}(\Spt^{\B\op}).
\end{tikzcd}
\end{center} 
This composite is represented by $W\wedge_\A X\in\D(\B)$; thus, if the induced derivator map is cocontinuous, then $W\wedge_\A X\in\D^{^c}_0(\B)$, and it follows that this object is the weighted homotopy colimit in $\eD^c$ as well as $\eD$. In particular, this is the case if 
\begin{center}
\begin{tikzcd}
\widetilde{\map}_{\dHo(\Spt^{\A\op})}(W,-):\dHo(\Spt^{\A\op})\arrow{r} &\dHo(\Spt^{\B\op})
\end{tikzcd}
\end{center} 
is cocontinuous; that is, if $(\A\op\times b)^*W\in\Ho(\Spt^{\A\op})$ is compact for each $b\in\B$. Thus, $\eD^c$ has homotopy colimits weighted by any pointwise compact object $W$. 

However, in general, $\eD^c$ does not admit all weighted homotopy colimits. This follows from the following theorem, since the induced prederivator $\D^c$ in general does not admit all coproducts, so in particular it is not a derivator.
\end{example}

\begin{theorem}\label{(Left) E-derivators induce (left) derivators}
If $\eA$ is a left $\E$-derivator, then the induced prederivator $\dA$ is a left derivator. Similarly, if $\eA$ is a right $\E$-derivator, then $\dA$ is a right derivator, and if $\eA$ is an $\E$-derivator, then $\dA$ is a derivator.
\end{theorem}
\begin{proof}
We will show that, given a left $\E$-derivator $\eA$, the induced prederivator $\dA$ is a left derivator. The corresponding result for right $\E$-derivators is dual, and the combination of these two results proves the corresponding result for $\E$-derivators.

Let $\eA$ be a left $\E$-derivator. By definition, $\dA$ is a semiderivator, so we need only show that $\dA$ admits all homotopy left Kan extensions, and prove the relevant part of \textbf{Der 4}.

We will start by showing that $\dA$ admits homotopy left Kan extensions. Let $u:\A\rightarrow\B$ be a functor and consider the object $(u\op\times\B)^*\h_\B\in\E(\A\op\times\B)$. For any object $X\in\dA(\A)$, denote the homotopy colimit of $X$ weighted by this object as follows:
\[
u_! X=(u\op\times\B)^*\h_\B\otimes_\A X\in\dA(\B)
\]
For any category $\C$, and any object $Z\in\A_0(\C)$, consider the isomorphisms below:
\begingroup
\addtolength{\jot}{0.7em}
\begin{align*}
\widetilde{\map}_{\dA}(u_!X,Z) &=\widetilde{\map}_{\dA}((u\op\times\B)^*\h_\B\otimes_\A X,Z)\\
&\iso\widetilde{\map}_{\E^{\A\op}}((u\op\times\B)^*\h_\B,\widetilde{\map}_{\dA}(X,Z))\\
&\iso\widetilde{\map}_{\E^{\B\op}}(\h_\B,(u\op\times\C)_*\widetilde{\map}_{\dA}(X,Z))\\
&\iso (u\op\times\C)_*\widetilde{\map}_{\dA}(X,Z)
\end{align*}
\endgroup
The first isomorphism follows from the definition of the weighted homotopy colimit. The second is associated to the following adjunction, obtained using~\cref{E-module adjunctions induce E-category adjunctions}:
\begin{center}
\begin{tikzcd}
\eE^{\B\op}\arrow[rr, bend left=40, "(u\op)^*" above,""{name=U, below}]\arrow[rr,leftarrow, bend right=40, "(u\op)_*" below, ""{name=D}]
&& \;\eE^{\A\op}
\arrow[phantom,from=U,to=D,"\bot"]
\end{tikzcd}
\end{center}
The final isomorphism $\varrho:\widetilde{\map}_{\E^{\B\op}}(\h_\B,-)\xrightarrow{\;\iso\;}\id$ is shown to be $\E$-natural in the proof of~\cref{Yoneda lemma for E-categories}. It follows that the entire composite is $\E$-natural in $Z\in\dA_0(\C)$. Denote this isomorphism by
\[
\varpi^u:(u\op\times\C)_*\widetilde{\map}_{\dA}(X,Z)\xrightarrow{\;\;\iso\;\;}\widetilde{\map}_{\dA}(u_!X,Z).
\]

We will now show that we can extend $u_!$ to a functor; the construction is similar to the definition of $u^*$ in~\cref{The functor u* from an E-prederivator}. Given a map $f:\h_\A\rightarrow\widetilde{\map}_{\dA}(X,Y)$ in $\dA(\A)$, we define $u_!f:\h_\B\rightarrow\widetilde{\map}_{\dA}(u_!X,u_!Y)$, using~\cref{E-category yoneda embedding is fully faithful}, to be the unique map in $\dA(\B)$ that makes the diagram below commute, for any $Z\in\dA_0(\C)$:
\begin{center}
\begin{tikzcd}
(u\op\times\C)_*\widetilde{\map}_{\dA}(Y,Z)\arrow[dd,"\varpi^u" left]\arrow[rrrr,"(u\op\times\C)_*\widetilde{\map}_{\dA}(f\text{,}Z)" above]
&&&& (u\op\times\C)_*\widetilde{\map}_{\dA}(X,Z)\arrow[dd,"\varpi^u" right]\\\\
\widetilde{\map}_{\dA}(u_!Y,Z)\arrow[rrrr,"\widetilde{\map}_{\dA}(u_!f\text{,}Z)" below]&&&& \widetilde{\map}_{\dA}(u_!X,Z)
\end{tikzcd}
\end{center}
As in~\cref{The functor u* from an E-prederivator}, it is easy to check that this construction is functorial. It remains to show that it is indeed a left adjoint to $u^*$. We will do this directly, by providing the unit and counit of the adjunction, and showing that they satisfy the triangle identities. 

Let $X\in\dA(\A)$. Using~\cref{E-category yoneda embedding is fully faithful}, we define $\eta_X:X\rightarrow u^*u_!X$ to be the unique map that makes the diagram below commute, for any category $\C$ and any $Z\in\dA_0(\C)$: 

\begin{center}
\begin{tikzcd}[row sep=small]
\widetilde{\map}_{\dA}(X,Z)\arrow[leftarrow,rrrdddddd,"\widetilde{\map}_{\dA}(\eta_X\text{,}Z)" below left, bend right]\arrow[leftarrow,rrr,"\epsilon_\mathsmaller{{\widetilde{\map}_{\dA}(X\text{,}Z)}}"]
&&& (u\op\times\C)^*(u\op\times\C)_*\widetilde{\map}_{\dA}(X,Z)\arrow[ddd,"(u\op\times\C)^*\varpi^u"] \\\\\\ &&&(u\op\times\C)^*\widetilde{\map}_{\dA}(u_!X,Z)\arrow[ddd,"\gamma^{u,\id}" right]\\\\\\
&&& \widetilde{\map}_{\dA}(u^*u_!X,Z)
\end{tikzcd}
\end{center}
To see that these maps form a natural transformation, let $f:\h_\A\rightarrow\widetilde{\map}_{\dA}(X,Y)$ be a map in $\dA(\A)$. We need to check that we have $\eta_Y\circ f=u^*u_!(f)\circ\eta_X$. Equivalently, for any $\C$ and any $Z\in\dA(\C)$, we need to show that this diagram commutes:
\begin{center}
\begin{tikzcd}
\widetilde{\map}_{\dA}(u^*u_!Y,Z)\arrow[dd,"\widetilde{\map}_{\dA}(\eta_Y\text{,}Z)" left]\arrow[rrr,"\widetilde{\map}_{\dA}(u^*u_!(f)\text{,}Z)" above]
&&& \widetilde{\map}_{\dA}(u^*u_!X,Z)\arrow[dd,"\widetilde{\map}_{\dA}(\eta_X\text{,}Z)" right]\\\\
\widetilde{\map}_{\dA}(Y,Z)\arrow[rrr,"\widetilde{\map}_{\dA}(f\text{,}Z)" below]&&& \widetilde{\map}_{\dA}(X,Z)
\end{tikzcd}
\end{center}
Using the definitions of $u_!$ and $\eta$ above, and the definition of $u^*$ in~\cref{The functor u* from an E-prederivator}, the commutativity of this diagram follows from the naturality of $\epsilon$.

Similarly, for any $W\in\dA(\B)$, define $\epsilon_W:u_!u^*W\rightarrow W$ to be the unique map that makes the diagram below commute, for any category $\C$ and any $Z\in\dA_0(\C)$: 
\begin{center}
\begin{tikzcd}[row sep=small]
\widetilde{\map}_{\dA}(W,Z)\arrow[rrrdddddd,"\widetilde{\map}_{\dA}(\epsilon_W\text{,}Z)" below left, bend right]\arrow[rrr,"\eta_\mathsmaller{{\widetilde{\map}_{\dA}(W\text{,}Z)}}"]
&&& (u\op\times\C)_*(u\op\times\C)^*\widetilde{\map}_{\dA}(W,Z)\arrow[ddd,"(u\op\times\C)_*\gamma^{u,\id}"] \\\\\\ &&&(u\op\times\C)_*\widetilde{\map}_{\dA}(u^*W,Z)\arrow[ddd,"\varpi^u" right]\\\\\\
&&& \widetilde{\map}_{\dA}(u_!u^*W,Z)
\end{tikzcd}
\end{center}
By a similar argument to the one above, these maps are natural in $W\in\dA(\B)$.

It remains to verify the triangle identities for $\eta$ and $\epsilon$. For the first identity, given $X\in\dA(\A)$, we need to check that we have $\epsilon_{u_!X}\circ u_!(\eta_X)=\id_{u_!X}$. Equivalently, for any $\C$ and any $Z\in\dA(\C)$, we need to show that the composite below is equal to the identity:
\begin{center}
\begin{tikzcd}
\widetilde{\map}_{\dA}(u_!X,Z)\arrow[rrr,"\widetilde{\map}_{\dA}(\epsilon_{u_!X}\text{,}Z)" above]
&&& \widetilde{\map}_{\dA}(u_!u^*u_!X,Z)\arrow[rrr,"\widetilde{\map}_{\dA}(u_!(\eta_X)\text{,}Z)" above]
&&& \widetilde{\map}_{\dA}(u_!X,Z)
\end{tikzcd}
\end{center}
Using the definitions of $u_!$, $\epsilon$ and $\eta$ above, and the definition of $u^*$ in~\cref{The functor u* from an E-prederivator}, this follows from the triangle identity for the adjunction
\begin{center}
\begin{tikzcd}
\E(\B\op\times\C)\arrow[rr, bend left=20, "(u\op\times\C)^*" above,""{name=U, below}]\arrow[rr,leftarrow, bend right=20, "(u\op\times\C)_*" below, ""{name=D}]
&& \;\E(\A\op\times\C).
\arrow[phantom,from=U,to=D,"\bot"]
\end{tikzcd}
\end{center}
The other triangle identity follows similarly. 

Thus, the induced prederivator $\dA$ admits all homotopy left Kan extensions. It remains to check the relevant part of \textbf{Der 4}. To do this, suppose we have a homotopy exact square:
\begin{center}
\begin{tikzcd}
 \A \arrow[rr,"u"]\arrow[dd,"v" left] && 
 \B \arrow[Rightarrow,ddll,"\kappa" above left,shorten >=0.9cm,shorten <=0.9cm, shift left] \arrow[dd,"w"] \\
 &\\
\J\arrow[rr,"z" below] && 
 \K
\end{tikzcd}
\end{center}
We will show that the canonical natural transformation 
\begin{center}
\begin{tikzcd}
\dA(\J)\arrow[ddrr,bend right,equal,""{name=L,above right}]\arrow[rr,leftarrow,"v_!" above] && \dA(\A)\arrow[Rightarrow,to=L,"\epsilon" above left,shorten >=0.5cm,shorten <=0.7cm] \arrow[leftarrow,rr,"u^*"]\arrow[leftarrow,dd,"v^*" left] && 
 \dA(\B)\arrow[ddrr,bend left,equal,""{name=R,below left}] \arrow[Rightarrow,ddll,"\kappa^*" above=3,shorten >=0.9cm,shorten <=0.8cm, shift left] \arrow[leftarrow,dd,"w^*"]\\\\
 
&&\dA(\J)\arrow[leftarrow,rr,"z^*" below] && 
 \dA(\K)\arrow[Leftarrow,to=R,"\eta" above left,shorten >=0.5cm,shorten <=0.7cm] \arrow[rr,leftarrow,"w_!" below]
 && \dA(\B) 
\end{tikzcd}
\end{center}
is an isomorphism. Applying this to the relevant comma square in~\cref{Derivator definition} gives us the result.

Thus, for any object $Y\in\dA(\B)$, we must show that the composite
\begin{center}
\begin{tikzcd}[column sep=25]
v_!u^*Y\arrow[rr,"v_!u^*(\eta_Y)"]
&& v_!u^*w^*w_!Y\arrow[rr,"v_!\kappa^*_{w_!Y}" above]
&& v_!v^*z^*w_!Y\arrow[rr,"\epsilon_{z^*w_!Y}" above]
&& z^*w_!Y
\end{tikzcd}
\end{center}
is an isomorphism in $\dA(\J)$. By~\cref{E-category yoneda embedding is fully faithful}, this map is an isomorphism if and only if, for any category $\C$ and any $Z\in\dA_0(\C)$, the composite
\begin{center}
\begin{tikzcd}[row sep=25]
\widetilde{\map}_{\dA}(z^*w_!Y,Z)\arrow[rrr,"\widetilde{\map}_{\dA}(\epsilon_{z^*w_!Y}\text{,}Z)"]
&&& \widetilde{\map}_{\dA}(v_!v^*z^*w_!Y,Z)\arrow[d,"\widetilde{\map}_{\dA}(v_!\kappa^*_{w_!Y}\text{,}Z)" right]\\
&&& \widetilde{\map}_{\dA}(v_!u^*w^*w_!Y,Z)\arrow[d,"\widetilde{\map}_{\dA}(v_!u^*(\eta_Y)\text{,}Z)" right]\\
&&& \widetilde{\map}_{\dA}(v_!u^*Y,Z)
\end{tikzcd}
\end{center}
is an isomorphism in $\E(\J\op\times\C)$. But, using the definitions of $\epsilon$ and $\eta$ above, and the definition of $\kappa^*$ in~\cref{Natural transformation omega* from an E-prederivator}, this composite is isomorphic to the component at $\widetilde{\map}_{\dA}(Y,Z)\in\E(\B\op\times\C)$ of the natural transformation below:
\begin{center}
\begin{tikzcd}
\E(\J\op\times\C)\arrow[ddrr,bend right=25,equal,""{name=L,above right}]\arrow[rr,leftarrow,"(v\op\times\C)_*" above] && \E(\A\op\times\C)\arrow[Leftarrow,to=L,"\eta" left=9,shorten >=0.3cm,shorten <=1cm, shift right=3] \arrow[leftarrow,rr,"(u\op\times\C)^*"]\arrow[leftarrow,dd,"(v\op\times\C)^*" left] && 
 \E(\B\op\times\C)\arrow[ddrr,bend left=25,equal,""{name=R,below left}] \arrow[Leftarrow,ddll,"(\kappa\op\times\C)^*" above left,shorten >=1cm,shorten <=1cm, shift left] \arrow[leftarrow,dd,"(w\op\times\C)^*"]\\\\
 
&&\E(\J\op\times\C)\arrow[leftarrow,rr,"(z\op\times\C)^*" below] && 
 \E(\K\op\times\C)\arrow[Rightarrow,to=R,"\epsilon" above=4,shorten >=0.4cm,shorten <=1cm,shift right=3] \arrow[rr,leftarrow,"(w\op\times\C)_*" below]
 && \E(\B\op\times\C)
\end{tikzcd}
\end{center}
Since the square $\kappa$ is homotopy exact, this natural transformation is an isomorphism. Thus, the original composite is also an isomorphism. 
\end{proof}

In light of~\cref{(Left) E-derivators induce (left) derivators}, if $\eA$ is an $\E$-derivator, we will refer to the induced prederivator $\dA$ as the \textbf{induced derivator}.

\begin{remark}\label{Functoriality of weighted homotopy colimits}
Let $\eA$ be a left $\E$-derivator. Given a category $\A$ and an object $X\in\dA(\A)$, the $\E$-morphism $-\otimes_\A X:\eE^{\A\op}\rightarrow\eA$ induces functors
\[
-\otimes_\A X:\E(\A\op\times\B)\rightarrow\dA(\B)
\]
for any category $\B$.

On the other hand, given an object $W\in\E(\A\op\times\B)$, the weighted homotopy colimits induce a functor:
\[
W\otimes_\A -:\dA(\A)\rightarrow\dA(\B)
\]
Given a map $f:\h_\A\rightarrow\widetilde{\map}_{\dA}(X,Y)$ in $\dA(\A)$, its image $W\otimes_\A f$ in $\dA(\B)$ is the unique map making the diagram below commute, for any category $\C$ and any $Z\in\dA_0(\C)$:
\begin{center}
\begin{tikzcd}
\widetilde{\map}_{\E^{\A\op}}(W,\widetilde{\map}_{\dA}(Y,Z))\arrow[dd,"\iso" left]\arrow[rrrr,"\widetilde{\map}_{\E^{\A\op}}(W\text{,}\widetilde{\map}_{\dA}(f\text{,}Z))" above]
&&&& \widetilde{\map}_{\E^{\A\op}}(W,\widetilde{\map}_{\dA}(X,Z))\arrow[dd,"\iso" right]\\\\
\widetilde{\map}_{\dA}(W\otimes_\A Y,Z)\arrow[rrrr,"\widetilde{\map}_{\dA}(W\otimes_\A f\text{,}Z)" below]&&&& \widetilde{\map}_{\dA}(W\otimes_\A X,Z)
\end{tikzcd}
\end{center}
Together, the functors above induce a two-variable functor:
\[
-\otimes_\A -:\E(\A\op\times\B)\times\dA(\A)\rightarrow\dA(\B)
\]

Let $u:\A\rightarrow\B$ be a functor. From the descriptions of $u^*$ and $u_!$ in \cref{Pullbacks are weighted colimits} and \cref{(Left) E-derivators induce (left) derivators}, we can check that we have natural isomorphisms:
\begin{align*}
u^*\iso (\B\op\times u)^*\h_\B\otimes_\B - &:\dA(\B)\rightarrow\dA(\A)\\
u_!\iso (u\op\times\B)^*\h_\B\otimes_\A - &:\dA(\A)\rightarrow\dA(\B)
\end{align*}
Using the isomorphism $(\A\op\times u)_!\h_\A\iso(u\op\times\B)^*\h_\B$ from \cref{Section The maps otimes u and h u}, we can also describe $u_!$ as follows:
\[
u_!\iso (\A\op\times u)_!\h_\A\otimes_\A - :\dA(\A)\rightarrow\dA(\B)
\]
In this way, we recover the formulas for $\E$-modules, given in~\cite{GS17} and~\cite{GR19}.
\end{remark}


  \newpage
  \phantomsection
  \addcontentsline{toc}{chapter}{Bibliography}
\bibliography{Thesis}{}

\begin{thebibliography}{10}

\bibitem{Borceux94}
Francis Borceux.
\newblock {\em Handbook of Categorical Algebra: Volume 2, Categories and
  Structures}, volume~50.
\newblock Cambridge University Press, 1994.

\bibitem{Cisinski03}
Denis-Charles Cisinski.
\newblock Images directes cohomologiques dans les cat{\'e}gories de modeles.
\newblock In {\em Annales Math{\'e}matiques Blaise Pascal}, volume~10, pages
  195--244, 2003.

\bibitem{Cisinski08}
Denis-Charles Cisinski.
\newblock Propri{\'e}t{\'e}s universelles et extensions de {K}an
  d{\'e}riv{\'e}es.
\newblock {\em Theory Appl. Categ}, 20(17):605--649, 2008.

\bibitem{CN08}
Denis-Charles Cisinski and Amnon Neeman.
\newblock Additivity for derivator {K}-theory.
\newblock {\em Advances in Mathematics}, 217(4):1381--1475, 2008.

\bibitem{CT11}
Denis-Charles Cisinski and Gon{\c{c}}alo Tabuada.
\newblock Non-connective {K}-theory via universal invariants.
\newblock {\em Compositio Mathematica}, 147(4):1281--1320, 2011.

\bibitem{Coley18}
Ian Coley.
\newblock Stabilization of derivators revisited.
\newblock {\em Journal of Homotopy and Related Structures}, pages 1--53, 2018.

\bibitem{DK80}
William Dwyer and Daniel Kan.
\newblock Simplicial localizations of categories.
\newblock {\em J. Pure Appl. Algebra}, 17(3):267--284, 1980.

\bibitem{Franke96}
Jens Franke.
\newblock Uniqueness theorems for certain triangulated categories possessing an
  {A}dams spectral sequence.
\newblock {\em K-theory Preprint Archives}, 139, 1996.

\bibitem{Gambino10}
Nicola Gambino.
\newblock Weighted limits in simplicial homotopy theory.
\newblock {\em Journal of Pure and Applied Algebra}, 214(7):1193--1199, 2010.

\bibitem{Groth11}
Moritz Groth.
\newblock Monoidal derivators and enriched derivators.
\newblock {\em preprint}, 2011.

\bibitem{Groth13}
Moritz Groth.
\newblock Derivators, pointed derivators and stable derivators.
\newblock {\em Algebraic \& Geometric Topology}, 13(1):313--374, 2013.

\bibitem{Groth16}
Moritz Groth.
\newblock Revisiting the canonicity of canonical triangulations.
\newblock {\em arXiv preprint arXiv:1602.04846}, 2016.

\bibitem{GPS14}
Moritz Groth, Kate Ponto, and Michael Shulman.
\newblock The additivity of traces in monoidal derivators.
\newblock {\em Journal of K-theory}, 14(3):422--494, 2014.

\bibitem{GPS14a}
Moritz Groth, Kate Ponto, and Michael Shulman.
\newblock Mayer-{V}ietoris sequences in stable derivators.
\newblock {\em Homology, Homotopy and Applications}, 16(1):265--294, 2014.

\bibitem{GR19}
Moritz Groth and Moritz Rahn.
\newblock Higher symmetries in abstract stable homotopy theories.
\newblock {\em arXiv preprint arXiv:1904.00580}, 2019.

\bibitem{GS17}
Moritz Groth and Mike Shulman.
\newblock Generalized stability for abstract homotopy theories.
\newblock {\em arXiv preprint arXiv:1704.08084}, 2017.

\bibitem{Grothendieck90}
Alexandre Grothendieck.
\newblock Les d{\'e}rivateurs.

\bibitem{GM11}
Bertrand Guillou and JP~May.
\newblock Enriched model categories and presheaf categories.
\newblock {\em arXiv preprint arXiv:1110.3567}, 2011.

\bibitem{Heller88}
Alex Heller.
\newblock {\em Homotopy Theories}, volume 383.
\newblock American Mathematical Soc., 1988.

\bibitem{Heller97}
Alex Heller.
\newblock Stable homotopy theories and stabilization.
\newblock {\em Journal of Pure and Applied Algebra}, 115(2):113--130, 1997.

\bibitem{Hirschhorn03}
Philip Hirschhorn.
\newblock {\em Model Categories and their Localizations}.
\newblock Number~99. American Mathematical Soc., 2003.

\bibitem{Hormann17}
Fritz H{\"o}rmann.
\newblock Fibered multiderivators and (co)homological descent.
\newblock {\em Theory and Applications of Categories}, 32(38):1258--1362, 2017.

\bibitem{Hovey07}
Mark Hovey.
\newblock {\em Model Categories}.
\newblock Number~63. American Mathematical Soc., 2007.

\bibitem{JK01}
George Janelidze and Max Kelly.
\newblock A note on actions of a monoidal category.
\newblock {\em Theory Appl. Categ}, 9(61-91):02, 2001.

\bibitem{Keller91}
Bernhard Keller.
\newblock Derived categories and universal problems.
\newblock {\em Communications in Algebra}, 19(3):699--747, 1991.

\bibitem{Kelly64}
Max Kelly.
\newblock On {M}ac {L}ane's conditions for coherence of natural
  associativities, commutativities, etc.
\newblock {\em Journal of Algebra}, 1(4):397--402, 1964.

\bibitem{Kelly82}
Max Kelly.
\newblock {\em Basic Concepts of Enriched Category Theory}, volume~64.
\newblock CUP Archive, 1982.

\bibitem{KS74}
Max Kelly and Ross Street.
\newblock Review of the elements of 2-categories.
\newblock In {\em Category Seminar}, pages 75--103. Springer, 1974.

\bibitem{Krause02}
Henning Krause.
\newblock A {B}rown representability theorem via coherent functors.
\newblock {\em Topology}, 41(4):853--861, 2002.

\bibitem{Krause08}
Henning Krause.
\newblock Localization theory for triangulated categories.
\newblock {\em arXiv preprint arXiv:0806.1324}, 2008.

\bibitem{Lagkas18a}
Ioannis Lagkas~Nikolos.
\newblock {\em Levelwise Modules and Localization in Derivators}.
\newblock PhD thesis, UCLA, 2018.

\bibitem{Leinster98}
Tom Leinster.
\newblock Basic bicategories.
\newblock {\em arXiv preprint math/9810017}, 1998.

\bibitem{Lenz18}
Tobias Lenz.
\newblock Homotopy (pre)derivators of cofibration categories and
  quasicategories.
\newblock {\em Algebraic \& Geometric Topology}, 18(6):3601--3646, 2018.

\bibitem{MacLane13}
Saunders Mac~Lane.
\newblock {\em Categories for the Working Mathematician}, volume~5.
\newblock Springer Science \& Business Media, 2013.

\bibitem{Neeman14}
Amnon Neeman.
\newblock {\em Triangulated Categories}, volume 148.
\newblock Princeton University Press, 2014.

\bibitem{PS16}
Kate Ponto and Michael Shulman.
\newblock The linearity of traces in monoidal categories and bicategories.
\newblock {\em Theory and Applications of Categories}, 31(23):594--689, 2016.

\bibitem{Quillen67}
Daniel Quillen.
\newblock {\em Homotopical Algebra}, volume~43.
\newblock Springer, 1967.

\bibitem{SS02}
Stefan Schwede and Brooke Shipley.
\newblock A uniqueness theorem for stable homotopy theory.
\newblock {\em Mathematische Zeitschrift}, 239(4):803--828, 2002.

\bibitem{Shulman06}
Michael Shulman.
\newblock Homotopy limits and colimits and enriched homotopy theory.
\newblock {\em arXiv preprint math/0610194}, 2006.

\bibitem{Stovicek09}
Jan {\v{S}}{\v{t}}ov{\'\i}{\v{c}}ek.
\newblock Compactly generated triangulated categories and the telescope
  conjecture.
\newblock 2009.

\bibitem{Street83}
Ross Street.
\newblock Enriched categories and cohomology.
\newblock {\em Quaestiones Mathematicae}, 6(1-3):265--283, 1983.

\bibitem{Tabuada08}
Gon{\c{c}}alo Tabuada.
\newblock Higher {K}-theory via universal invariants.
\newblock {\em Duke Mathematical Journal}, 145(1):121--206, 2008.

\bibitem{Toen03}
Betrand To{\"e}n.
\newblock Homotopical and higher categorical structures in algebraic geometry.
\newblock {\em arXiv preprint math/0312262}, 2003.

\end{thebibliography}
\bibliographystyle{plain}

\end{document}